\documentclass[11pt,leqno]{article}

\usepackage{pst-plot, pstricks,pstricks-add}
\usepackage{fancybox,amssymb,color,amsmath}
\usepackage{hyperref,srcltx,mathrsfs}
\usepackage[utf8]{inputenc}
\usepackage{amscd}
\usepackage{textcomp}

\long\def\forget#1{}

\usepackage{geometry}
\usepackage{amsmath}
\usepackage{amssymb}
\usepackage{amscd}
\usepackage{textcomp}
\usepackage{ifthen}

\def\theenumi{(\alph{enumi})}

\def\p@enumii{\theenumi}

\newcommand{\DS}{\displaystyle}
\newcommand{\TS}{\textstyle}
\newcommand{\SC}{\scriptstyle}
\newcommand{\SSC}{\scriptscriptstyle}

\usepackage{amsfonts}

\newcommand{\BA}{{\mathbb{A}}}

\newcommand{\BC}{{\mathbb{C}}}

\newcommand{\BF}{{\mathbb{F}}}
\newcommand{\BG}{{\mathbb{G}}}
\newcommand{\BH}{{\mathbb{H}}}

\newcommand{\BN}{{\mathbb{N}}}

\newcommand{\BP}{{\mathbb{P}}}
\newcommand{\BQ}{{\mathbb{Q}}}
\newcommand{\BR}{{\mathbb{R}}}

\newcommand{\BZ}{{\mathbb{Z}}}

\newcommand{\CB}{{\cal{B}}}
\newcommand{\CC}{{\cal{C}}}

\newcommand{\CE}{{\cal{E}}}
\newcommand{\CF}{{\cal{F}}}
\newcommand{\CG}{{\cal{G}}}
\newcommand{\CH}{{\cal{H}}}

\newcommand{\CL}{{\cal{L}}}
\newcommand{\CM}{{\cal{M}}}

\newcommand{\CO}{{\cal{O}}}
\newcommand{\CP}{{\cal{P}}}

\newcommand{\CR}{{\cal{R}}}

\newcommand{\CT}{{\cal{T}}}

\newcommand{\FC}{{\mathfrak{C}}}
\newcommand{\FD}{{\mathfrak{D}}}

\newcommand{\FU}{{\mathfrak{U}}}

\newcommand{\Fp}{{\mathfrak{p}}}
\newcommand{\Fq}{{\mathfrak{q}}}
\newcommand{\Fr}{{\mathfrak{r}}}

\DeclareMathOperator{\Aut}{Aut}
\newcommand{\Betti}{{\rm Betti}}
\DeclareMathOperator{\Cent}{Cent}

\DeclareMathOperator{\End}{End}
\DeclareMathOperator{\Ext}{Ext}
\DeclareMathOperator{\Frob}{Frob}
\DeclareMathOperator{\Gal}{Gal}
\DeclareMathOperator{\GL}{GL}
\DeclareMathOperator{\Gr}{Gr}
\DeclareMathOperator{\Koh}{H}

\DeclareMathOperator{\Hom}{Hom}
\newcommand{\CHom}{{\cal H}om}
\DeclareMathOperator{\Id}{Id}

\DeclareMathOperator{\Lie}{Lie}

\DeclareMathOperator{\Pic}{Pic}

\DeclareMathOperator{\QEnd}{QEnd}
\DeclareMathOperator{\QHom}{QHom}
\DeclareMathOperator{\Quot}{Quot}

\DeclareMathOperator{\Rep}{Rep}
\DeclareMathOperator{\Res}{Res}

\DeclareMathOperator{\Spm}{Sp}
\DeclareMathOperator{\Spec}{Spec}

\DeclareMathOperator{\Sym}{Sym}

\DeclareMathOperator{\Tr}{Tr}
\DeclareMathOperator{\Var}{V}

\newcommand{\ad}{{\rm ad}}
\newcommand{\alg}{{\rm alg}}

\DeclareMathOperator{\coim}{coim}
\DeclareMathOperator{\coker}{coker}

\newcommand{\dR}{{\rm dR}}
\newcommand{\et}{{\rm\acute{e}t}}

\DeclareMathOperator{\id}{\,id}
\DeclareMathOperator{\im}{im}

\renewcommand{\mod}{\;{\rm mod}\;}

\DeclareMathOperator{\rank}{rk}

\newcommand{\rig}{{\rm rig}}
\DeclareMathOperator{\rk}{rk}
\newcommand{\sep}{{\rm sep}}

\newcommand{\tors}{{\rm tors}}
\DeclareMathOperator{\weight}{wt}

\newbox\dotCOdbox
\newbox\dotCOtbox
\newbox\dotCOsbox
\newbox\dotCOssbox
\setbox\dotCOdbox=\vbox{\hbox{\kern3.5pt\bf.}\vskip-4.5pt\hbox{$\CO$}}
\setbox\dotCOtbox=\vbox{\hbox{\kern3.5pt\bf.}\vskip-4.5pt\hbox{$\CO$}}
\setbox\dotCOsbox=\vbox{\hbox{\kern2.1pt.}
                       \vskip-6.6pt\hbox{$\scriptstyle\CO$}}
\setbox\dotCOssbox=\vbox{\hbox{\kern1.6pt.}
                        \vskip-8.1pt\hbox{$\scriptscriptstyle\CO$}}
\newcommand{\dotCO}{{\mathchoice{\copy\dotCOdbox}
                               {\copy\dotCOtbox}
                               {\copy\dotCOsbox}
                               {\copy\dotCOssbox}}}

\newbox\dotCdbox
\newbox\dotCtbox
\newbox\dotCsbox
\newbox\dotCssbox
\setbox\dotCdbox=\vbox{\hbox{\kern3.5pt\bf.}\vskip-5.3pt\hbox{$C$}}
\setbox\dotCtbox=\vbox{\hbox{\kern3.5pt\bf.}\vskip-5.3pt\hbox{$C$}}
\setbox\dotCsbox=\vbox{\hbox{\kern2.1pt.}
                       \vskip-8.2pt\hbox{$\scriptstyle C$}}
\setbox\dotCssbox=\vbox{\hbox{\kern1.6pt.}
                        \vskip-9.6pt\hbox{$\scriptscriptstyle C$}}
\newcommand{\dotC}{{\mathchoice{\copy\dotCdbox}
                               {\copy\dotCtbox}
                               {\copy\dotCsbox}
                               {\copy\dotCssbox}}}

\newbox\dotFCdbox
\newbox\dotFCtbox
\newbox\dotFCsbox
\newbox\dotFCssbox
\setbox\dotFCdbox=\vbox{\hbox{\kern3pt\bf.}\vskip-5.5pt\hbox{$\FC$}}
\setbox\dotFCtbox=\vbox{\hbox{\kern3pt\bf.}\vskip-5.5pt\hbox{$\FC$}}
\setbox\dotFCsbox=\vbox{\hbox{\kern1.8pt.}
                       \vskip-7.5pt\hbox{$\scriptstyle\FC$}}
\setbox\dotFCssbox=\vbox{\hbox{\kern1.1pt.}
                        \vskip-10pt\hbox{$\scriptscriptstyle\FC$}}
\newcommand{\dotFC}{{\mathchoice{\copy\dotFCdbox}
                               {\copy\dotFCtbox}
                               {\copy\dotFCsbox}
                               {\copy\dotFCssbox}}}

\newbox\dotFDdbox
\newbox\dotFDtbox
\newbox\dotFDsbox
\newbox\dotFDssbox
\setbox\dotFDdbox=\vbox{\hbox{\kern3.5pt\bf.}\vskip-5.5pt\hbox{$\FD$}}
\setbox\dotFDtbox=\vbox{\hbox{\kern3.5pt\bf.}\vskip-5.5pt\hbox{$\FD$}}
\setbox\dotFDsbox=\vbox{\hbox{\kern2.1pt.}
                       \vskip-8.2pt\hbox{$\scriptstyle\FD$}}
\setbox\dotFDssbox=\vbox{\hbox{\kern1.6pt.}
                        \vskip-9.9pt\hbox{$\scriptscriptstyle\FD$}}
\newcommand{\dotFD}{{\mathchoice{\copy\dotFDdbox}
                               {\copy\dotFDtbox}
                               {\copy\dotFDsbox}
                               {\copy\dotFDssbox}}}

\renewcommand{\phi}{\varphi}
\renewcommand{\epsilon}{\varepsilon}

\def\longto{\longrightarrow}
\def\into{\hookrightarrow}
\let\onto\twoheadrightarrow
\def\isoto{\stackrel{}{\mbox{\hspace{1mm}\raisebox{+1.4mm}{$\SC\sim$}\hspace{-3.5mm}$\longrightarrow$}}}
\def\leftisoto{\stackrel{}{\mbox{\hspace{2.5mm}\raisebox{+1.4mm}{$\SC\sim$}\hspace{-5mm}$\longleftarrow$}}}
\def\longinto{\lhook\joinrel\longrightarrow}

\newbox\mybox
\def\arrover#1{\mathrel{
       \setbox\mybox=\hbox spread 1.4em{\hfil$\scriptstyle#1$\hfil}
       \vbox{\offinterlineskip\copy\mybox
             \hbox to\wd\mybox{\rightarrowfill}}}}

\newcommand{\linj}{ \mbox{\mathsurround=0pt \;$\longleftarrow $ \hspace{-1.07em} \raisebox{0.63ex} {\small $\supset$}\;}}

\newcommand{\longleftmapsto}{\:\mbox{\mathsurround=0pt $\longleftarrow$\hspace{-0.2em}\raisebox{0.1ex}{$\shortmid$}}\:}

\usepackage[curve,matrix,arrow,cmtip]{xy}

\newdir^{ (}{{}*!/-3pt/\dir^{(}}    
\newdir^{  }{{}*!/-3pt/\dir^{}}    
\newdir_{ (}{{}*!/-3pt/\dir_{(}}    
\newdir_{  }{{}*!/-3pt/\dir_{}}    

\let\setminus\smallsetminus

\newcommand{\es}{\enspace}
\newcommand{\open}{^\circ}

\newcommand{\dual}{^{\SSC\vee}}

\newcommand{\mal}{^{^\times}}
\newcommand{\fdot}{\;\,_{_{\bullet}}\,\;}
\newcommand{\dbl}{{\mathchoice{\mbox{\rm [\hspace{-0.15em}[}}
                              {\mbox{\rm [\hspace{-0.15em}[}}
                              {\mbox{\scriptsize\rm [\hspace{-0.15em}[}}
                              {\mbox{\tiny\rm [\hspace{-0.15em}[}}}}
\newcommand{\dbr}{{\mathchoice{\mbox{\rm ]\hspace{-0.15em}]}}
                              {\mbox{\rm ]\hspace{-0.15em}]}}
                              {\mbox{\scriptsize\rm ]\hspace{-0.15em}]}}
                              {\mbox{\tiny\rm ]\hspace{-0.15em}]}}}}
\newcommand{\dpl}{{\mathchoice{\mbox{\rm (\hspace{-0.15em}(}}
                              {\mbox{\rm (\hspace{-0.15em}(}}
                              {\mbox{\scriptsize\rm (\hspace{-0.15em}(}}
                              {\mbox{\tiny\rm (\hspace{-0.15em}(}}}}
\newcommand{\dpr}{{\mathchoice{\mbox{\rm )\hspace{-0.15em})}}
                              {\mbox{\rm )\hspace{-0.15em})}}
                              {\mbox{\scriptsize\rm )\hspace{-0.15em})}}
                              {\mbox{\tiny\rm )\hspace{-0.15em})}}}}
\newcommand{\invlim}[1][]{\ifthenelse{\equal{#1}{}}
{\DS \lim_{\longleftarrow}}
{\DS \lim_{\underset{#1}{\longleftarrow}}}
}
\newcommand{\dirlim}[1][]{\ifthenelse{\equal{#1}{}}
{\DS \lim_{\longrightarrow}}
{\DS \lim_{\underset{#1}{\longrightarrow}}}
}
\newcommand{\ul}[1]{{\underline{#1}}}
\newcommand{\ol}[1]{{\overline{#1}}}
\newcommand{\wt}[1]{{\widetilde{#1}}}
\newcommand{\wh}[1]{{\widehat{#1}}}

\usepackage{amsthm}
\theoremstyle{plain}
\newtheorem{theorem}{Theorem}[section]

\newtheorem{lemma}[theorem]{Lemma}

\newtheorem{corollary}[theorem]{Corollary}
\newtheorem{proposition}[theorem]{Proposition}

\newtheorem{question}[theorem]{Question}

\theoremstyle{definition}
\newtheorem{definition}[theorem]{Definition}

\newtheorem{bigexample}[theorem]{Example}
\newtheorem{remark}[theorem]{Remark}
\newtheorem{example}[theorem]{Example}

\usepackage{color}

\newcommand{\comment}[1]{}

\def\?{\ 
{\bf\color{red}???}\ 
\immediate\write16{}
\immediate\write16{Warning: There was still a question mark . . . }
\immediate\write16{}}

\def\olCM{{\,\,\overline{\!\!\mathcal{M}}}{}}
\def\olF{{\,\overline{\!F}}{}}
\def\olH{{\,\overline{\!H}}{}}
\def\olM{{\,\overline{\!M}}{}}
\def\olN{{\,\overline{\!N}}{}}
\def\oldM{{\,\overline{\!\dM}}{}}

\def\ulE{{\underline{E\!}\,}{}}
\def\ulH{{\underline{H\!}\,}{}}
\def\ulM{{\underline{M\!}\,}{}}
\def\ulN{{\underline{N\!}\,}{}}
\def\ulCE{{\underline{\mathcal{E}\!}\,}{}}
\def\ulCF{{\underline{\mathcal{F}\!}\,}{}}

\def\ulCH{{\underline{\mathcal{H}\!}\,}{}}
\def\ulTE{{\underline{\,\,\wt{\!\!E}\!}\,}{}}
\def\ulTH{{\underline{\,\wt{\!H}\!}\,}{}}
\def\ulTM{{\underline{\,\,\wt{\!\!M}\!}\,}{}}

\def\ulHM{{\underline{\widehat M\!}\,}{}}
\def\uldHM{{\underline{\widehat\dM\!}\,}{}}

\def\wtE{{\,\,\wt{\!\!E}}{}}

\def\whCB{{\widehat{\CB\,}\!}{}}

\newcommand{\charmorph}{c^*}

\newcommand{\con}[1][]{{\mathchoice
           {\TS\langle\frac{z}{\zeta^{#1}}\rangle[z^{-1}]}
           {\TS\langle\frac{z}{\zeta^{#1}}\rangle[z^{-1}]}
           {\SC\langle\frac{z}{\zeta^{#1}}\rangle[z^{-1}]}
           {\SSC\langle\frac{z}{\zeta^{#1}}\rangle[z^{-1}]}}}

\newcommand{\BOne} {{\mathchoice{\hbox{\rm1\kern-2.7pt l\kern.9pt}}
                              {\hbox{\rm1\kern-2.7pt l\kern.9pt}}
                              {\hbox{\scriptsize\rm1\kern-2.3pt l\kern.4pt}}
                              {\hbox{\scriptsize\rm1\kern-2.4pt l\kern.5pt}}}}
\def\UOne{\underline{\BOne}}
\def\dUOne{\underline{\check{\BOne}}}
\def\GlobMotRing{A_{\BC}[J^{-1}]}
\newcommand{\QHodgeCat}{{\mathchoice
           {\mbox{$Q$-{\tt HP}}}
           {\mbox{$Q$-{\tt HP}}}
           {\mbox{$\SC Q$-{\tt\scriptsize HP}}}
           {\mbox{$\SSC Q$-{\tt\tiny HP}}}}}
\newcommand{\Hodge}{\BH}
\newcommand{\ulHodge}{\ul{\BH}}
\newcommand{\AMotCat}{{\mathchoice
           {\mbox{$A$-{\tt Mot}}}
           {\mbox{$A$-{\tt Mot}}}
           {\mbox{$\SC A$-{\tt\scriptsize Mot}}}
           {\mbox{$\SSC A$-{\tt\tiny Mot}}}}}
\newcommand{\AMotCatIsog}{{\mathchoice
           {\mbox{$A$-{\tt MotI}}}
           {\mbox{$A$-{\tt MotI}}}
           {\mbox{$\SC A$-{\tt\scriptsize MotI}}}
           {\mbox{$\SSC A$-{\tt\tiny MotI}}}}}
\newcommand{\dualAMotCat}{{\mathchoice
           {\mbox{$A$-{\tt dMot}}}
           {\mbox{$A$-{\tt dMot}}}
           {\mbox{$\SC A$-{\tt\scriptsize dMot}}}
           {\mbox{$\SSC A$-{\tt\tiny dMot}}}}}
\newcommand{\dualAMotCatIsog}{{\mathchoice
           {\mbox{$A$-{\tt dMotI}}}
           {\mbox{$A$-{\tt dMotI}}}
           {\mbox{$\SC A$-{\tt\scriptsize dMotI}}}
           {\mbox{$\SSC A$-{\tt\tiny dMotI}}}}}
\def\dualAMotCat{A\text{-\tt dMot}}
\def\AMMotCat{A\text{-\tt MMot}}
\def\AMMotCatIsog{A\text{-\tt MMotI}}

\def\AUMotCatIsog{A\text{-\tt UMotI}}

\def\AMUMotCatIsog{A\text{-\tt MUMotI}}
\def\dualAMMotCat{A\text{-\tt dMMot}}
\def\dualAMMotCatIsog{A\text{-\tt dMMotI}}

\def\dualAMUMotCatIsog{A\text{-\tt dMUMotI}}
\def\dualAUMotCatIsog{A\text{-\tt dUMotI}}

\def\QVec{{\tt Vec}_Q}
\def\KVec{{\tt Vec}_K}
\def\RMod{{\tt Mod}_R}
\DeclareMathOperator{\Disc}{\FD_\BC}
\DeclareMathOperator{\PDisc}{\dotFD_\BC}

\def\sq{q}

\def\ssigma{\sigma}
\def\stau{{\tau}}
\def\sdtau{{\check{\tau}}}
\def\sdsigma{\check{\sigma}}
\def\df{\check{f}}
\def\dh{\check{h}}
\def\dm{\check{m}}
\def\dn{\check{n}}
\def\dH{{\check{H}}}
\def\dM{{\check{M}}}
\def\dN{{\check{N}}}
\def\uldM{{\underline{\dM\!}\,}{}}
\def\uldN{{\underline{\dN\!}\,}{}}

\def\llangle{\langle\!\langle}
\def\rrangle{\rangle\!\rangle}

\begin{document}


\author{Urs Hartl and Ann-Kristin Juschka}

\title{Pink's theory of Hodge structures and the Hodge conjecture over function fields}

\maketitle

\begin{abstract}
\noindent
In 1997 Richard Pink has clarified the concept of Hodge structures over function fields in positive characteristic, which today are called Hodge-Pink structures. They form a neutral Tannakian category over the underlying function field. He has defined Hodge realization functors from the uniformizable abelian $t$-modules and $t$-motives of Greg Anderson to Hodge-Pink structures. This allows one to associate with each uniformizable $t$-motive a Hodge-Pink group, analogous to the Mumford-Tate group of a smooth projective variety over the complex numbers. It further enabled Pink to prove the analog of the Mumford-Tate Conjecture for Drinfeld modules. Moreover, based on unpublished work of Pink and the first author, the second author proved in her Diploma thesis that the Hodge-Pink group equals the motivic Galois group of the $t$-motive as defined by Papanikolas and Taelman. This yields a precise analog of the famous Hodge Conjecture, which is an outstanding open problem for varieties over the complex numbers.

In this report we explain Pink's results on Hodge structures and the proof of the function field analog of the Hodge conjecture. The theory of $t$-motives has a variant in the theory of dual $t$-motives. We clarify the relation between $t$-motives, dual $t$-motives and $t$-modules. We also construct cohomology realizations of abelian $t$-modules and (dual) $t$-motives and comparison isomorphisms between them generalizing Gekeler's de Rham isomorphism for Drinfeld modules. 


\end{abstract}

\setcounter{tocdepth}{2}
\tableofcontents

 
\section{Introduction}
\setcounter{equation}{0}

According to Deligne~\cite[2.3.8]{DeligneHodge2}, a \emph{rational mixed Hodge structure} $\ulH$ consists of a finite dimensional $\BQ$-vector space $H$, an increasing filtration $W_\bullet H$ of $H$ by $\BQ$-subspaces, called the \emph{weight filtration}, and a decreasing filtration $F^\bullet H_\BC$ of $H_\BC := H\otimes_\BQ \BC$ by $\BC$-subspaces, called the \emph{Hodge filtration}, such that $\Gr_F^p\Gr_{\olF}^q\Gr_n^W H_\BC=(0)$ for $p+q\ne n$ where $\olF^q H_\BC$ is the complex conjugate subspace $\ol{F^q H_\BC}\subset H_\BC$. The rational mixed Hodge structures form a neutral Tannakian category \cite[Definition~II.2.19]{DM82} over $\BQ$, whose fiber functor sends a rational mixed Hodge structure $\ulH$ to its underlying $\BQ$-vector space \cite{Deligne94}. By Tannakian duality \cite[Theorem~II.2.11]{DM82} there is a linear algebraic group $\Gamma_\ulH$ over $\BQ$, called the \emph{Hodge group} of $\ulH$, such that the Tannakian subcategory $\llangle\ulH\rrangle$ generated by $\ulH$ is tensor equivalent to the category of $\BQ$-rational representations of $\Gamma_\ulH$. We give more details and explanations on Tannakian theory in Section~\ref{SectTannaka}.

If $X$ is a smooth projective variety over the complex numbers $\BC$, its Betti cohomology group $\Koh^n_\Betti(X,\BQ)$ is a $\BQ$-vector space. Via the de Rham isomorphism $\Koh^n_\Betti(X,\BQ)\otimes_\BQ\BC\cong\Koh^n_\dR(X/\BC)$ and the Hodge filtration on the latter, it becomes a rational (pure) Hodge structure. This provides a functor from smooth projective varieties over $\BC$ to rational mixed Hodge structures. Deligne \cite[\S\,8.2]{DeligneHodge3} extended this functor to separated schemes of finite type over $\BC$. If $X$ is smooth projective and $Z\subset X$ is a closed subscheme of codimension $p$ then $Z$ defines a cohomology class in $\Koh^{2p}_\Betti(X,\BQ)\cap F^p$. The Hodge conjecture \cite{Hodge52,GrothendieckHodge,Deligne06} states that every cohomology class in $\Koh^{2p}_\Betti(X,\BQ)\cap F^p$ arises from a $\BQ$-rational linear combination of closed subschemes of codimension $p$ in $X$.

Besides the Betti and de Rham cohomology, there are various other cohomology theories for $X$. They are linked to each other via comparison isomorphisms. This inspired Grothendieck to propose a universal cohomology theory he called ``motives'' \cite{GrothendieckStandard}. More precisely Grothendieck conjectured the existence of a Tannakian category of \emph{motives} such that the cohomology functors like $X\mapsto\Koh^n_\Betti(X,\BQ)$ and $X\mapsto\Koh^n_\dR(X/\BC)$ factor through this category of motives; see \cite{DemazureMotives,KleimanMotives,ManinMotives}. The motive associated with $X$ is denoted $h(X)$ and the various cohomology groups attached to $X$ are called the \emph{realizations} of the motive $h(X)$. In particular the Betti realization of $h(X)$ is $\ulH(X):=\bigoplus_{n=0}^{2\dim X}\Koh^n_\Betti(X,\BQ)$ equipped with its rational mixed Hodge structure. In terms of the conjectural category of motives, the Hodge conjecture is equivalent to the statement, that the Betti realization functor $\llangle h(X)\rrangle\to \llangle\ulH(X)\rrangle$ is a tensor equivalence, where $\llangle h(X)\rrangle$ is the Tannakian subcategory generated by $h(X)$. By Tannakian duality  $\llangle h(X)\rrangle$ is tensor equivalent to the category of $\BQ$-rational representations of a linear algebraic group $\Gamma_{h(X)}$ over $Q$ which is called the \emph{motivic Galois group} of $X$. The Betti realization functor corresponds to a homomorphism of algebraic groups $\Gamma_{\ulH(X)}\to \Gamma_{h(X)}$ over $\BQ$. By \cite[Proposition~2.21]{DM82} it is a closed immersion and the Hodge conjecture is equivalent to the statement that this homomorphism is an isomorphism.

\medskip

In this article we want to describe the function field analog of the above. There, a category of motives actually exists in the \emph{$t$-motives} of Anderson~\cite{Anderson86}. We slightly generalize them to $A$-motives in Section~\ref{SectMixedAMotives}. An $A$-motive has various cohomology realizations. In this article we explain the Betti, de Rham and $\ell$-adic realization. The $p$-adic and crystalline realization is discussed in the survey \cite{HartlKim} in this volume. In \cite{PinkHodge} Richard Pink invented mixed Hodge structures over function fields (which we call \emph{mixed Hodge-Pink structures}) as an analog of classical rational mixed Hodge structures. He discovered the crucial fact that instead of a Hodge filtration one needs finer information to obtain a Tannakian category. This information is given in terms of a \emph{Hodge-Pink lattice}. The definition is as follows. 

Let $\BF_p=\BZ/(p)$ for a prime $p$ and let $A=\BF_p[t]$ and $Q=\BF_p(t)$ be the polynomial ring and its fraction field. They are the analogs in the arithmetic of function fields of the integers $\BZ$ and the rational numbers $\BQ$. (The theory is actually developed for slightly more general rings $A$ and $Q$.) Let $Q_\infty=\BF_p\dpl \tfrac{1}{t}\dpr$ be the completion of $Q$ for the valuation $\infty$ of $Q$ which does not correspond to a maximal ideal of $A$. Let $\BC\supset Q_\infty$ be an algebraically closed, complete, rank one valued extension, for example the completion of an algebraic closure of $Q_\infty$. The fields $Q_\infty$ and $\BC$ are the analogs of the usual fields $\BR$ and $\BC$ of real, respectively complex numbers. We denote the image of $t$ in $\BC$ by $\theta$ and consider the ring $\BC\dbl t-\theta\dbr$ of formal power series in the ``variable'' $t-\theta$ and the embedding $Q\to\BC\dbl t-\theta\dbr$, $t\mapsto t=\theta+(t-\theta)$.

\bigskip

\noindent
{\bfseries Definitions~\ref{Def1.1} and \ref{Def1.5}.}\es
A \emph{mixed $Q$-Hodge-Pink structure} is a triple $\ulH=(H,W_\bullet H,\Fq)$ with
\begin{itemize}
\item $H$ a finite dimensional $Q$-vector space,
\item $W_\mu H\subset H$ for $\mu\in\BQ$ an exhaustive and separated increasing filtration by $Q$-subspaces, called the \emph{weight filtration},
\item a $\BC\dbl t-\theta\dbr$-lattice $\Fq\subset H\otimes_Q \BC\dpl t-\theta\dpr$ of full rank, called the \emph{Hodge-Pink lattice},
\end{itemize}
which satisfies a certain semi-stability condition; see Definition~\ref{Def1.5}. The Hodge-Pink lattice induces an exhaustive and separated decreasing \emph{Hodge-Pink filtration} $F^i H_\BC\subset H_\BC:=H\otimes_{Q,\,t\mapsto\theta}\BC$ for $i\in\BZ$ by setting $F^i H_\BC:=\bigl(\Fp\cap(t-\theta)^i\Fq\bigr)\big/\bigl((t-\theta)\Fp\cap(t-\theta)^i\Fq\bigr)$, where $\Fp:=H\otimes_Q\BC\dbl t-\theta\dbr$.

\bigskip

The mixed Hodge-Pink structures with the fiber functor $(H,W_\bullet H,\Fq)\mapsto H$ form a neutral Tannakian category over $Q$; see Theorem~\ref{ThmPinkTannaka}. It was Pink's insight that for this result the Hodge-Pink filtration does not suffice, but one needs the finer information present in the Hodge-Pink lattice. Any Hodge-Pink structure $\ulH$ generates a neutral Tannakian subcategory, and the algebraic group $\Gamma_\ulH$ obtained from Tannakian duality is called the \emph{Hodge-Pink group of $\ulH$}; see Section~\ref{SectTannaka}.

Hodge-Pink structures may arise from \emph{Drinfeld-modules} or more generally from \emph{uniformizable abelian Anderson $A$-modules} $\ulE=(E,\phi)$ over $\BC$, where $E\cong\BG_{a,\BC}^d$ and $\phi\colon A\to\End_\BC(E)$ such that $(\phi_t-\theta)^d$ annihilates the tangent space $\Lie E$ to $E$ at $0$ for some integer $d$; see Definitions~\ref{DefAndersonAModule} and \ref{DefAbelianAMod}. Namely, $\ulE$ possesses an exponential function $\exp_\ulE\colon\Lie E\to E(\BC)$ and if this function is surjective, $\ulE$ is \emph{uniformizable}. In this case the finite (locally) free $A$-module $\Lambda(\ulE):=\ker(\exp_\ulE)$ sits in an exact sequence 
\[
\xymatrix @R=0.1pc {
0 \ar[r] & \Fq \ar[r] & \Lambda(\ulE)\otimes_A \BC\dbl t-\theta\dbr \ar[r]^{\DS\gamma} & \Lie E \ar[r] & 0 \\
& & \lambda\otimes\sum_i b_i(t-\theta)^i \ar@{|->}[r] & \sum_i b_i\cdot(\Lie\phi_t-\theta)^i(\lambda)\,;
}
\]
see \eqref{Eq1.1}. If $\ulE$ is \emph{mixed} (Definition~\ref{DefMixedAModule}) the $Q$-vector space $H:=\Koh_{1,\Betti}(\ulE):=\Lambda(\ulE)\otimes_AQ$ inherits an increasing weight filtration $W_\bullet H$ and we define the \emph{mixed Hodge-Pink structures of $\ulE$} as $\ulHodge_1(\ulE):=(H,W_\bullet H,\Fq)$ and $\ulHodge^1(\ulE):=\ulHodge_1(\ulE)\dual$; see Corollary~\ref{CorHPofEandM}.

Similarly to the classical situation, one can also associate with $\ulE$ a pure (or mixed) $A$-motive $\ulM:=\Hom_\BC(E,\BG_{a,\BC})$; see Definition~\ref{DefAbelianAMod}. By an \emph{$A$-motive of rank $r$} we mean a pair $\ulM=(M,\tau_M)$ where $M$ is a (locally) free $\BC[t]$-module of rank $r$ and $\tau_M\colon\sigma^*M[\tfrac{1}{t-\theta}]\isoto M[\tfrac{1}{t-\theta}]$ is an isomorphism of $\BC[t][\tfrac{1}{t-\theta}]$-modules; see Definition~\ref{DefAMotive}. Here $\sigma^*M:=\Frob_{p,\BC}^*M=M\otimes_{\BC[t],\sigma^*}\BC[t]$ for the endomorphism $\sigma^*$ of $\BC[t]$ sending $t$ to $t$ and $b\in\BC$ to $b^p$. 
For an $A$-motive we define its \emph{$\tau$-invariants} over $\BC\langle t\rangle:=\;\bigl\{\,\sum\limits_{i=0}^\infty b_i t^i\colon b_i\in\BC,\, \lim\limits_{i\to\infty}|b_i|= 0\,\bigr\}$ as
\begin{equation}\label{EqIntroTauInv}
\Lambda(\ulM) \;:=\; \bigl(M\otimes_{\BC[t]}\BC\langle t\rangle\bigr)^\tau\;:=\; \bigl\{\,m\in M\otimes_{\BC[t]}\BC\langle t\rangle:\es \tau_M(\Frob_{p,\BC}^\ast m)=m\,\bigr\}\,.
\end{equation}
An $A$-motive of rank $r$ is \emph{uniformizable} if its $\tau$-invariants form a (locally) free $A$-module of rank $r$; see Definition~\ref{DefLambda} and Lemma~\ref{LemmaUniformizable}. We explain the results of Papanikolas~\cite{Papanikolas} and Taelman~\cite{Taelman} that the category $\AUMotCatIsog$ of uniformizable $A$-motives up to isogeny together with the fiber functor $\ulM\mapsto\Lambda(\ulM)\otimes_AQ$ is a neutral Tannakian category over $Q$; see Theorems~\ref{TheoremAMotTannakian} and \ref{TheoremDualAMotTannakian}. Considering the Tannakian subcategory $\llangle\ulM\rrangle$ generated by $\ulM$, the algebraic group $\Gamma_\ulM$ associated by Tannakian duality, is called the \emph{motivic Galois group of $\ulM$}. 

In unpublished work, the following function field analog of the classical Hodge conjecture was formulated by Pink and proved by him for pure uniformizable $A$-motives and by Pink and the first author for uniformizable mixed $A$-motives. Pink's proof was worked out for dual $A$-motives (see below) by the second author in her Diploma thesis~\cite{JuschkaDipl}. There is a realization functor $\ulHodge^1$ from uniformizable mixed $A$-motives $\ulM$ to mixed Hodge-Pink structures as follows. The $Q$-vector space $H:=\Koh^1_\Betti(\ulM,Q):=\Lambda(\ulM)\otimes_AQ$ inherits an increasing weight filtration $W_\bullet H$ and admits a canonical isomorphism $h\colon H\otimes_Q\BC\dbl t-\theta\dbr\isoto(\sigma^*M)\otimes_{\BC[t]}\BC\dbl t-\theta\dbr$; see Proposition~\ref{PropLambdaConvRadius}. We set $\Fq:=h^{-1}\circ\tau_M^{-1}(M\otimes_{\BC[t]}\BC\dbl t-\theta\dbr)\subset H\otimes_Q\BC\dpl t-\theta\dpr$ and define the \emph{mixed Hodge-Pink structures of $\ulM$} as $\ulHodge^1(\ulM):=(H,W_\bullet H,\Fq)$ and $\ulHodge_1(\ulM):=\ulHodge^1(\ulM)\dual$; see Definition~\ref{Def2.6}. The functor $\ulHodge^1$ restricts to an exact tensor functor from the Tannakian subcategory $\llangle\ulM\rrangle$ of uniformizable mixed $A$-motives generated by $\ulM$ to the Tannakian subcategory $\llangle\ulHodge^1(\ulM)\rrangle$ of mixed Hodge-Pink structures generated by $\ulHodge^1(\ulM)$. This induces a morphism from the Hodge-Pink group $\Gamma_{\ulHodge^1(\ulM)}$ of $\ulHodge^1(\ulM)$ to the motivic Galois group $\Gamma_\ulM$ of $\ulM$.

\bigskip

\noindent
{\bfseries Theorems~\ref{ThmHodgeConjecture} and \ref{ThmHodgeConjectureE} (The Hodge Conjecture over Function Fields).}\es 
{\it The morphism $\Gamma_{\ulHodge^1(\ulM)}\longto\Gamma_\ulM$ is an isomorphism of algebraic groups. Equivalently, $\ulHodge^1\colon\llangle\ulM\rrangle\longto\llangle\ulHodge^1(\ulM)\rrangle$ is an exact tensor equivalence.
}

\bigskip

The crucial part in the proof of this theorem is to show that each Hodge-Pink sub-structure $\ulH'\subset\ulHodge^1(\ulM)$ is isomorphic to $\ulHodge^1(\ulM')$ for an $A$-sub-motive $\ulM'\subset\ulM$. This is achieved by associating with $\ulH'$ a \emph{$\sigma$-bundle} over the punctured open unit disk. The theory of $\sigma$-bundles was developed in \cite{HartlPink1} and is explained in detail in Section~\ref{Sect4}, where we also show how to associate a pair of $\sigma$-bundles with a mixed Hodge-Pink structure, respectively with a uniformizable $A$-motive (or dual $A$-motive; see below). Using the classification \cite[Theorem~11.1]{HartlPink1} of $\sigma$-bundles and the rigid analytic GAGA-principle, one defines an $A$-motive $\ulM'$ such that $\ulHodge^1(\ulM')=\ulH'$.

Large parts of this article are not original but a survey of the existing literature, which tries to be largely self-contained. In Section~\ref{Sect1} we review Pink's theory of mixed Hodge-Pink structures. In Section~\ref{SectMixedAMotives} we define pure and mixed $A$-motives and slightly generalize Anderson's~\cite[\S\,2]{Anderson86} theory of uniformization of $t$-motives to $A$-motives. Also we define and study the mixed Hodge-Pink structure $\ulHodge^1(\ulM)$ of a uniformizable mixed $A$-motive $\ulM$ and its Betti, de Rham and $\ell$-adic cohomology realization, as well as the comparison isomorphisms between them. Actually the $\ell$-adic realization is called ``$v$-adic'' by us where $v\subset A$ is a place taking on the role of the prime $\ell\in\BZ$ and $\Koh^1_v$ is our analog of $\Koh^1_\et(\,.\,,\BZ_\ell)$. 

For applications to transcendence questions like in \cite{ABP,Papanikolas,ChangYu07,ChangPapaYu10,ChangPapaThakurYu,ChangPapa11,ChangPapaYu11,ChangPapa12}, it turns out that \emph{dual $A$-motives} are even more useful than $A$-motives; see the article of Chang~\cite{Chang12} in this volume for an introduction. A \emph{dual $A$-motive of rank $r$} is a pair $\uldM=(\dM,\sdtau_\dM)$ where $\dM$ is a (locally) free $\BC[t]$-module of rank $r$ and $\sdtau_\dM\colon(\sdsigma^*\dM)[\tfrac{1}{t-\theta}]\isoto \dM[\tfrac{1}{t-\theta}]$ is an isomorphism of $\BC[t][\tfrac{1}{t-\theta}]$-modules for $\sdsigma^*=(\sigma^*)^{-1}$. (Beware that a dual $A$-motive is something different then the dual $\ulM\dual$ of an $A$-motive $\ulM$). A dual $A$-motive of rank $r$ is \emph{uniformizable} if its $\sdtau$-invariants $\Lambda(\uldM):=\bigl(\dM\otimes_{\BC[t]}\BC\langle t\rangle\bigr)^\sdtau$, which are defined analogously to \eqref{EqIntroTauInv}, form a (locally) free $A$-module of rank $r$; see Definition~\ref{DefDualUnifGlobal} and Lemma~\ref{LemmaDualUniformizable}. Also the category of uniformizable dual $A$-motives with the fiber functor $\uldM\mapsto\Lambda(\uldM)\otimes_AQ$ is a neutral Tannakian category; see Theorem~\ref{TheoremDualAMotTannakian}. Actually this is the category studied by Papanikolas~\cite{Papanikolas}. If $\uldM$ is uniformizable and mixed, the $Q$-vector space $H:=\Koh_{1,\Betti}(\uldM,Q):=\Lambda(\uldM)\otimes_AQ$ inherits an increasing weight filtration $W_\bullet H$ and admits a canonical isomorphism $h_\uldM\colon H\otimes_Q\BC\dbl t-\theta\dbr\isoto\dM\otimes_{\BC[t]}\BC\dbl t-\theta\dbr$; see Proposition~\ref{Prop4.23}. We set $\Fq:=h_\uldM^{-1}\circ\sdtau_\dM(\sdsigma^*\dM\otimes_{\BC[t]}\BC\dbl t-\theta\dbr)\subset H\otimes_Q\BC\dpl t-\theta\dpr$ and define the \emph{mixed Hodge-Pink structures of $\uldM$} as $\ulHodge_1(\uldM):=(H,W_\bullet H,\Fq)$ and $\ulHodge^1(\uldM):=\ulHodge_1(\uldM)\dual$; see Definition~\ref{DualDef2.6}. This theory of pure and mixed dual $A$-motives, their theory of uniformization, their associated mixed Hodge-Pink structures, and their Betti, de Rham and $v$-adic cohomology realizations, as well as the comparison isomorphisms between them are explained in Section~\ref{SectMixedDualAMotives}.

In the longest Section~\ref{SectAndersonAModules} we recall the theory of \emph{abelian Anderson $A$-modules}, which generalize Anderson's~\cite{Anderson86} \emph{abelian $t$-modules}, and their associated $A$-motives including uniformizability, \emph{scattering matrices} (Remark~\ref{RemScattering}) and \emph{Anderson generating functions} (Corollary~\ref{CorDivisionTower}, Example~\ref{exmp:scatteringmatrix}). Moreover, in Sections~\ref{SectRelDualAMotives}, \ref{SectAnalytThyAFinite} and \ref{SectAModUniformizability} we reproduce from unpublished work of Anderson~\cite{ABP_Rohrlich} the theory of \emph{$A$-finite} Anderson $A$-modules $\ulE$ including uniformization and the description of torsion points. These are the ones for which the $\BC[t]$-module $\uldM(\ulE):=\Hom_\BC(\BG_{a,\BC},E)$ is finitely generated, and hence a dual $A$-motive. As described above, we associate a mixed Hodge-Pink structure with a uniformizable mixed abelian, respectively $A$-finite, Anderson $A$-module and $v$-adic, Betti and de Rham cohomology realizations. The latter go back to Deligne, Anderson, Gekeler, Yu, Goss, Brownawell and Papanikolas. We generalize the approach of these authors in Section~\ref{SectCohAMod} and prove comparison isomorphisms between these cohomology realizations. We also explain in Theorem~\ref{ThmPeriodIsomForE} how to recover Gekeler's comparison isomorphism \cite[\S\,2]{Gekeler89} between Betti and de Rham cohomology from ours. 

Finally, in Section~\ref{Sect3} we briefly report on applications to Galois representations and transcendence questions due to Anderson, Brownawell, Chang, Papanikolas, Pink, Thakur, Yu and others.

Although this article is mainly a review of (un)published work, we nevertheless establish the following new results: the theory of mixed Anderson $A$-modules (Section~\ref{SectAModPurity}) and the construction that associates with a uniformizable mixed (dual) $A$-motive a mixed Hodge-Pink structure (Sections~\ref{SectAMotHPStr}, \ref{SectDualAMotHPStr}). Also we clarify the relation between a uniformizable mixed $A$-motive $\ulM=(M,\tau_M)$ and the associated dual $A$-motive $\uldM(\ulM):=\bigl(\Hom_{\BC[t]}(\sigma^*M,\Omega^1_{\BC[t]/\BC})\,,\,\tau_M\dual\bigr)$ in Propositions~\ref{PropDualizing}, \ref{PropDualMixed}, \ref{PropDualizingUnif}, \ref{PropSameGroup} and Theorem~\ref{ThmHofMandDualM} and most importantly in the following 

\bigskip

\noindent
{\bfseries Theorem~\ref{ThmMandDMofE}.}\es 
{\it 
Let $\ulE$ be an Anderson $A$-module over $\BC$ which is both abelian and $A$-finite, and let $\ulM=(M,\tau_M)=\ulM(\ulE)$ and $\uldM=(\dM,\sdtau_\dM)=\uldM(\ulE)$ be its associated (dual) $A$-motive. Let $\uldM(\ulM)=\bigl(\Hom_{\BC[t]}(\sigma^*M,\Omega^1_{\BC[t]/\BC})\,,\,\tau_M\dual\bigr)$ be the dual $A$-motive associated with $\ulM$. Then there is a canonical isomorphism of dual $A$-motives $\Xi\colon\uldM(\ulM)\; \isoto \; \uldM(\ulE)$.
}

\bigskip

We illustrate the general theory with various examples, most notably Examples~\ref{exampleDModDMotdDmot} and \ref{exmp:scatteringmatrix} which for Drinfeld-modules explain Theorem~\ref{ThmMandDMofE} in concrete terms and relate it to scattering matrices. Moreover, we prove the compatibility of the cohomology realizations and comparison isomorphisms of $A$-motives, dual $A$-motives and abelian, respectively $A$-finite Anderson $A$-modules in Theorems~\ref{ThmCompIsomBettiDRAMotive}, \ref{ThmCompIsomBettiDRDualAMotive}, \ref{ThmPeriodIsomForE}, \ref{ThmPeriodIsomForAFiniteE} and Propositions~\ref{PropCohAMot}, \ref{PropCompTateModEandM}, \ref{PropCompTateModEandDualM}, and we prove the compatibility with a change of the ring $A$ in Remark~\ref{RemChangeOfA}, and with Gekeler's comparison isomorphism \cite[\S\,2]{Gekeler89} in Theorem~\ref{ThmPeriodIsomForE}. In particular, we prove the following theorems.

\bigskip

\noindent
{\bfseries Theorem~\ref{ThmHPofEandDualM}.}\es 
{\it Let $\ulE$ be a uniformizable mixed $A$-finite Anderson $A$-module over $\BC$ and let $\uldM=\uldM(\ulE)$ be its associated mixed dual $A$-motive. Then the mixed Hodge-Pink structures $\ulHodge_1(\ulE)$ and $\ulHodge_1(\uldM)$ are canonically isomorphic.
}

\bigskip

\noindent
{\bfseries Theorem~\ref{ThmHPofEandM}.}\es 
{\it Let $\ulE$ be a uniformizable mixed abelian Anderson $A$-module over $\BC$ and let $\ulM=\ulM(\ulE)$ be its associated mixed $A$-motive. Consider the Hodge-Pink structure $\ul\Omega=(H,W_\bullet H,\Fq)$ which is pure of weight $0$ and given by $H=\Omega^1_{Q/\BF_q}=Q\,dt$ and $\Fq=\BC\dbl t-\theta\dbr dt$. Then the mixed Hodge-Pink structures $\ulHodge_1(\ulE)$ and $\ulHodge_1(\ulM)\otimes\ul\Omega$ are canonically isomorphic.
}

\bigskip

\noindent
{\bfseries Theorem~\ref{ThmCompatXi}.}\es 
{\it Let $\ulE$ be a uniformizable mixed Anderson $A$-module over $\BC$ which is both abelian and $A$-finite, and let $\ulM=\ulM(\ulE)$ and $\uldM(\ulE)$ be the associated (dual) $A$-motive. Then the isomorphisms above are also compatible with the isomorphisms from Theorems~\ref{ThmHofMandDualM}, \ref{ThmHPofEandDualM} and \ref{ThmHPofEandM} and the isomorphism $\Xi\colon\uldM(\ulM)\isoto\uldM(\ulE)$ from Theorem~\ref{ThmMandDMofE}, in the sense that the following diagram commutes
\begin{equation*}
\xymatrix @C=9pc { 
\ulHodge_1\bigl(\uldM(\ulM)\bigr)  \ar[r]_{\TS\cong}^{\TS\ulHodge_1(\Xi)} & \ulHodge_1\bigl(\uldM(\ulE)\bigr) \ar[d]_{\TS\cong}^{\TS\rm Theorem~\ref{ThmHPofEandDualM}\es} \\
\ulHodge_1(\ulM)\otimes\ul\Omega \ar[r]_{\TS\cong}^{\TS\rm\es Theorem~\ref{ThmHPofEandM}} \ar[u]_{\TS\cong}^{\TS\rm Theorem~\ref{ThmHofMandDualM}\es} &  \ulHodge_1(\ulE)
}
\end{equation*}
}

\bigskip

Finally, we give a criterion in Theorem~\ref{ThmImageOfH} which characterizes those mixed Hodge-Pink structures that arise from uniformizable mixed $A$-motives.

\bigskip

Various categories of motives over $\BC$ play a part in this article. To give the reader an overview we list them in the following table. Note that the set of morphisms $\Hom_{\AMotCat}(\ulM,\ulN)$ between two $A$-motives $\ulM$ and $\ulN$ is a finitely generated $A$-module; see Remark~\ref{Rem3.3}(c). The same is true for dual $A$-motives; see Remark~\ref{RemQuillenDualAMot}(d).

\begin{tabular}{lp{19.8em}ll}
Category & Description & Def. & Properties \\
\hline
$\AMotCat$ & $A$-motives over $\BC$ & \ref{DefAMotive} &  \\
$\AMMotCat$ & mixed $A$-motives  & \ref{Def2.1} &  \raisebox{1.5ex}[-1.5ex]{$\left\}\;\raisebox{1.38ex}[-1.38ex]{\parbox[t]{9em}{exact (Rem.\newline \ref{RemQuillenAMot}(b) and \ref{RemQuillenAMMot})}}\right.$}\\[2mm]
$\AMotCatIsog$ & $A$-motives up to isogeny, that is with \newline $\Hom_{\AMotCatIsog}(\ulM,\ulN):=\Hom_{\AMotCat}(\ulM,\ulN)\otimes_AQ$ & \ref{DefAMotive} &   \\
$\AMMotCatIsog$ & mixed $A$-motives  up to isogeny & \ref{Def2.1} & \raisebox{3.4ex}[-3.4ex]{$\left\}\;\raisebox{2.38ex}[-2.38ex]{\parbox[t]{9em}{non-neutral \newline Tannakian \newline (Prop.~\ref{PropAMotTannakianCateg} and \ref{PropPure})}}\right.$} \\[2mm]
$\AUMotCatIsog$ & uniformizable $A$-motives  up to isogeny & \ref{DefUnifGlobal} &   \\
$\AMUMotCatIsog$ & uniformizable mixed $A$-motives  up to isogeny & \ref{DefUnifGlobal} & \raisebox{1.5ex}[-1.5ex]{$\left\}\;\raisebox{1.38ex}[-1.38ex]{\parbox[t]{9em}{neutral Tannakian \newline (Thm.~\ref{TheoremAMotTannakian})}}\right.$} \\
\hline
$\dualAMotCat$ & dual $A$-motives  & \ref{DefDualAMotive} & \\
$\dualAMMotCat$ & mixed dual $A$-motives  & \ref{DualDef2.1} & \raisebox{1.5ex}[-1.5ex]{$\left\}\;\raisebox{1.38ex}[-1.38ex]{\parbox[t]{9em}{exact (Rem.\newline \ref{RemQuillenDualAMot}(b) and \ref{RemQuillenDualAMMot})}}\right.$}\\[2mm]
$\dualAMotCatIsog$ & dual $A$-motives  up to isogeny, that is with \newline $\Hom_{\dualAMotCatIsog}(\uldM,\uldN):=\Hom_{\dualAMotCat}(\uldM,\uldN)\otimes_AQ$ & \ref{DefDualAMotive} & \\
$\dualAMMotCatIsog$ & mixed dual $A$-motives  up to isogeny & \ref{DualDef2.1} & \raisebox{3.4ex}[-3.4ex]{$\left\}\;\raisebox{2.38ex}[-2.38ex]{\parbox[t]{9em}{non-neutral \newline Tannakian \newline (Prop.~\ref{PropDualizing} and \ref{PropDualPure})}}\right.$} \\[2mm]
$\dualAUMotCatIsog$ & uniformizable dual $A$-motives  up to isogeny & \ref{DefDualUnifGlobal} &  \\
$\dualAMUMotCatIsog$ & uniformizable mixed dual $A$-motives  up to isogeny & \ref{DefDualUnifGlobal} & \raisebox{1.5ex}[-1.5ex]{$\left\}\;\raisebox{1.38ex}[-1.38ex]{\parbox[t]{9em}{neutral Tannakian \newline (Thm.~\ref{TheoremDualAMotTannakian})}}\right.$}
\end{tabular}

\bigskip
\noindent
{\bfseries Acknowledgements.} We are grateful to Richard Pink for teaching us his beautiful theories of mixed Hodge-Pink structures and $\sigma$-bundles and to Greg Anderson, Dale Brownawell and Matthew Papanikolas for sharing their unpublished manuscript \cite{ABP_Rohrlich} with us and allowing us in the present article to reproduce the theory of dual $A$-motives from it. We also profited much from discussions with Dale Brownawell, Matthew Papanikolas, Lenny Taelman. These notes grew out of two lecture series given by the first author in the fall of 2009 at the conference ``$t$-motives: Hodge structures, transcendence and other motivic aspects'' at the Banff International Research Station BIRS in Banff, Canada and at the Centre de Recerca Mathem\`atica CRM in Barcelona, Spain in the spring of 2010. The first author is grateful to BIRS and CRM for their hospitality. He also acknowledges support of the DFG (German Research Foundation) in form of SFB 878 and Germany's Excellence Strategy EXC 2044--390685587 ``Mathematics M\"unster: Dynamics--Geometry--Structure''.

\subsection{Preliminaries}\label{SectPreliminaries}

Throughout this article we will denote by

\begin{tabbing}
$Q_\infty\hookrightarrow\BC\dbl z-\zeta\dbr$ \es\=\kill
$\BF_q$ \> a finite field with $q$ elements and characteristic $p$, \\[1mm]
$C$ \> a smooth projective geometrically irreducible curve over $\BF_q$, \\[1mm]
$\infty\in C(\BF_q)$ \>\parbox[t]{0.8\textwidth}{a fixed closed point, (To simplify the exposition in this article $\infty$ is supposed to be $\BF_q$-rational. The main results we present here hold, and are in fact proved in \cite{PinkHodge,HartlPink2}, without this assumption.)}\\[1mm]
$\dotC=C\setminus\{\infty\}$ \> the associated affine curve,\\[1mm]
$A=\Gamma(\dotC,\CO_{\dotC})$ \> the ring of regular functions on $\dotC$ (the function field analog of $\BZ$), \\[1mm]
$Q =\BF_q(C)$ \> the function field of $C$, viz.\ the field of fractions of $A$ (the analog of ~$\BQ$),\\[1mm]
$z\in Q$\> a uniformizing parameter at $\infty$,\\[1mm]
$Q_\infty=\BF_q\dpl z\dpr$ \> the completion of $Q$ at $\infty$ (the analog of $\BR$),\\[1mm]
$A_\infty=\BF_q\dbl z\dbr$ \> the ring of integers in $Q_\infty$,\\[1mm]
$\BC\supset Q_\infty$\>\parbox[t]{0.8\textwidth}{an algebraically closed, complete, rank one valued extension, for example the completion of an algebraic closure of $Q_\infty$ (the analog of the usual field of complex numbers),}\\[1mm]
$\charmorph\colon Q\to\BC$\> the natural inclusion,\\[1mm]
$\zeta=\charmorph(z)$\> the image of $z$ in $\BC$, which satisfies $0<|\zeta|<1$,\\[1mm]
$A_{\BC}=A\otimes_{\BF_q}\BC$\> the base extension of $A$, \\[1mm]
$Q_\BC=Q\otimes_{\BF_q}\BC$\> \parbox[t]{0.8\textwidth}{the base extension of $Q$, distinguishing between $z$ and $\zeta$ allows us to abbreviate the element $z\otimes1$ of $Q_\BC$ by $z$ and the element $1\otimes\charmorph(z)$ by $\zeta$,}\\[1mm]
$C_\BC=C\times_{\Spec\BF_q}\Spec\BC$ \quad the resulting irreducible curve over~$\BC$,\\[1mm]
$J\subset A_{\BC}$ \>the (maximal) ideal generated by $a\otimes 1-1\otimes \charmorph(a)$ for all $a\in A$,\\[1mm]
$\GlobMotRing$\> \parbox[t]{0.8\textwidth}{the ring of global sections on the open affine subscheme $\Spec A_{\BC}\setminus\Var(J)$ of $C_\BC$,}\\[1mm]
$\BC\dbl z-\zeta\dbr$ \>\parbox[t]{0.8\textwidth}{the formal power series ring in the ``variable'' $z-\zeta$. It is canonically isomorphic to the completion of the local ring of  $C_\BC$ at $\Var(J)$, see Lemma~\ref{LemmaZ-Zeta}, and replaces the ring $\BC\dbl t-\theta\dbr$ from the introduction,}\\[1mm]
$\BC\dpl z-\zeta\dpr$ \>the fraction field of $\BC\dbl z-\zeta\dbr$,\\[1mm]
$Q_\infty\hookrightarrow\BC\dbl z-\zeta\dbr$\>\parbox[t]{0.8\textwidth}{ the natural $\BF_q$-algebra homomorphism satisfying $z\mapsto z=\zeta+(z-\zeta)$ and given by $\sum\limits_i a_iz^i\mapsto \sum\limits_{j=0}^\infty\bigl(\sum\limits_{i=j}^\infty \left(\begin{smallmatrix}i\\j\end{smallmatrix}\right)a_i\zeta^{i-j}\bigr)(z-\zeta)^j$,}\\[1mm]
$\sigma\colon C_\BC\to C_\BC$ \> \parbox[t]{0.8\textwidth}{the product of the identity on $C$ with the $q$-th power Frobenius on $\Spec\BC$, which acts on points and on the coordinates of $C$ as the identity, and on the elements $b\in\BC$ as $b\mapsto b^q$,}\\[1mm]
$\sigma^\ast\colon A_{\BC}\to A_{\BC}$ \> the corresponding endomorphism $a\otimes b\mapsto a\otimes b^q$ for $a\in A$ and $b\in \BC$,\\[1mm]
$\sigma^{i\ast}:=(\sigma^\ast)^i$ \> for a non-negative integer $i\in\BN_0$,\\[1mm]
$\sigma^{i\ast} M$\> the pullback $\sigma^{i\ast} M:=M\otimes_{A_{\BC},\sigma^{i\ast}}A_{\BC}$ of an $A_{\BC}$-module $M$ under $\sigma$,\\[1mm]
$\sigma^\ast(m):=m\otimes1$\> the canonical image of $m\in M$ in $\sigma^\ast M:=M\otimes_{A_{\BC},\sigma^\ast}A_{\BC}$,\\[1mm]
$\sdsigma^\ast:=(\ssigma^\ast)^{-1}$\> \parbox[t]{0.8\textwidth}{the endomorphism of $A_{\BC}$ inverse to $\ssigma^\ast$ sending $a\otimes b$ to $a\otimes \sqrt[q]{b}$ for $a\in A$ and $b\in \BC$ which exists because $\BC$ is perfect,}\\[1mm]
$\sdsigma^{i\ast}:=(\sdsigma^\ast)^i$ \> for a non-negative integer $i\in\BN_0$,\\[1mm]
$\sdsigma^{i\ast} M$\> the tensor product $\sdsigma^{i\ast} M:=M\otimes_{A_{\BC},\sdsigma^{i\ast}}A_{\BC}$ for an $A_{\BC}$-module $M$,\\[1mm]
$\sdsigma^\ast(m):=m\otimes1$\> the canonical image of $m\in M$ in $\sdsigma^\ast M:=M\otimes_{A_{\BC},\sdsigma^\ast}A_{\BC}$, \\[1mm]
$\BC\{\tau\}:=$ \> \parbox[t]{0.8\textwidth}{$\bigl\{\,\sum\limits_{i=0}^nb_i\tau^i\colon n\in\BN_0,b_i\in\BC\,\bigr\}$ the skew polynomial ring in the variable $\tau$ with the commutation rule $\tau b=b^q\tau$ for $b\in\BC$,}\\[1mm]
$\BC\{\sdtau\}:=$ \> \parbox[t]{0.8\textwidth}{$\bigl\{\,\sum\limits_{i=0}^nb_i\sdtau^i\colon n\in\BN_0,b_i\in\BC\,\bigr\}$ the skew polynomial ring in the variable $\sdtau$ with the commutation rule $\sdtau b=b^{1/q}\sdtau$ for $b\in\BC$.}
\end{tabbing}

\noindent For any module  $M$  over an integral domain  $R$  and any non-zero element  $x\in R$  we let $R[\frac{1}{x}]$ and $M[\frac{1}{x}]:= M\otimes_R R[\frac{1}{x}]$ denote the localizations obtained by inverting~$x$. Any homomorphism of $R$-modules $M\to N$ induces a homomorphism of $R[\frac{1}{x}]$-modules $M[\frac{1}{x}] \to N[\frac{1}{x}]$ denoted again by the same letter.

\begin{remark}\label{RemSigmaIsFlat}
The ring homomorphisms $\ssigma^\ast\colon A_{\BC}\to A_{\BC}$ and $\sdsigma^\ast\colon A_{\BC}\to A_{\BC}$ are flat because they arise by base change from the flat homomorphisms $\BC\to\BC,\,b\mapsto b^q$, respectively $\BC\to\BC,\,b\mapsto \sqrt[q]{b}$.
\end{remark}

For later reference we record the following two lemmas.

\begin{lemma}\label{LemmaTSep}
\begin{enumerate}
\item \label{LemmaTSep_A}
If $t\in Q$ is a uniformizing parameter at a closed point $P$ of $C$ then $Q$ is a finite separable field extension of $\BF_q(t)$.
\item \label{LemmaTSep_B}
There exists an element $t\in A$ such that $Q$ is a finite separable field extension of $\BF_q(t)$. For every maximal ideal $v\subset A$ one may even find such a $t\in A$ such that the radical ideal $\sqrt{A\cdot t}$ of $A\cdot t$ is $v$.
\end{enumerate}
\end{lemma}

\begin{proof}
\ref{LemmaTSep_A}
The point $P\in C$ is unramified under the map $C\to\BP^1_{\BF_q}$ corresponding to the inclusion $\BF_q(t)\subset Q$. Since all ramification indices are divisible by the inseparability degree, the latter has to be one. 

\medskip\noindent
\ref{LemmaTSep_B}
Choose some $a\in A\setminus\BF_q$. Then $\BF_q[a]\into A$ is a finite flat ring extension and so $Q/\BF_q(a)$ is a finite field extension. If it is not separable, let $p^e$ be its inseparability degree. Then $\BF_q(a)$ is contained in $Q^{p^e}:=\{x^{p^e}\colon x\in Q\}$ by \cite[Proof of Corollary~II.2.12]{Silverman}. So there is a $t\in Q$ with $a=t\:\!^{p^e}$. We even have $t\in A$ because $A$ is integrally closed in $Q$. By considering the inseparability degree in the tower $\BF_q(a)\subset\BF_q(t)\subset Q$ we see that $Q/\BF_q(t)$ is separable.

If a maximal ideal $v\subset A$ is given, there is a positive integer $n$ such that $v^n=A\cdot a$ is a principal ideal. Continuing as above we obtain an element $t\in A$ with $\sqrt{A\cdot t}=\sqrt{A\cdot a}=v$. 
\end{proof}

\begin{lemma}\label{LemmaZ-Zeta}
Let $K$ be a field and let $\charmorph\colon A\into K$ be an injective ring homomorphism. Let $z\in Q\setminus\BF_q$ be an element such that $Q$ is a finite separable extension of $\BF_q(z)$, and let $\zeta=\charmorph(z)$. Then the power series ring $K\dbl z-\zeta\dbr$ over $K$ in the ``variable'' $z-\zeta$ is canonically isomorphic to the completion of the local ring of $C_K$ at the closed point $\Var(J)$ defined by the ideal $J:=(a\otimes 1-1\otimes \charmorph(a)\colon a\in A)\subset A_K$.
\end{lemma}

\begin{proof}
The completion of the local ring of $C_K$ at $\Var(J)$ is $\invlim A_K/J^n$. Since this is a complete discrete valuation ring with residue field $K$ we only need to show that $z-\zeta$ is a uniformizing parameter. Clearly, $z-\zeta$ is contained in the maximal ideal. To prove the converse, let $a\in Q$. Let $f\in\BF_q(z)[X]$ be the minimal polynomial of $a$ over $\BF_q(z)$ and multiply it with the common denominator to obtain the polynomial $F(X,z)\in\BF_q[X,z]$. The two-variable Taylor expansion of $F$ at $(\charmorph(a),\zeta)\in K^2$ is
\[
F(X,z)\equiv F(\charmorph(a),\zeta)+\frac{\partial F}{\partial X}(\charmorph(a),\zeta)\cdot(X-\charmorph(a))+\frac{\partial F}{\partial z}(\charmorph(a),\zeta)\cdot(z-\zeta)\mod J^2
\]
Plugging in $a$ for $X$ yields $F(a,z)=0$ and $\frac{\partial F}{\partial X}(a,z)\ne0$ by the separability of $Q/\BF_q(z)$. Under the injective homomorphism $\charmorph\colon Q\into K$ we get $F(\charmorph(a),\zeta)=0$ and $\frac{\partial F}{\partial X}(\charmorph(a),\zeta)\ne0$. This shows that the element $a-\charmorph(a)\in J/J^2$ is a multiple of $z-\zeta$, and so $z-\zeta$ generates the $A_K$-module $J/J^2$. By Nakayama's Lemma~\cite[Corollary~4.7]{Eisenbud} there is an element $f\in 1+J$ that annihilates the $A_K$-module $J/(z-\zeta)$. Since $f$ is invertible in $\invlim A_K/J^n$ we have proved that $z-\zeta$ generates the maximal ideal of $\invlim A_K/J^n$.
\end{proof}

\subsection{Tannakian theory}\label{SectTannaka}

As already alluded to in the introduction, a good framework to discuss Hodge structures is the theory of Tannakian categories. Also Pink's results which we explain in this article use this language. Therefore, we briefly recall the definition and some facts about Tannakian categories from the articles of Deligne and Milne~\cite{DM82,DeligneCatTann,Milne92}.

\begin{definition}[{\cite[(A.7.1) and (A.7.2), page~222]{Milne92}}]
\label{defn:tannakiancategory}
Let $K$ be a field. A $K$-linear abelian tensor category $\mathscr{C}$ with unit object $\BOne$ is a \emph{Tannakian category over $K$} if
\begin{enumerate}
\item \label{defn:tannakiancategory_A}
for every object $X$ of $\mathscr{C}$ there exists an object $X\dual$ of $\mathscr{C}$, called the \emph{dual} of $X$, and morphisms ${\rm ev}\colon X\otimes X\dual\to\BOne$ and $\delta\colon\BOne\to X\dual\otimes X$ such that
\begin{eqnarray*}\label{EqEVX}
({\rm ev}\otimes \id_X)\circ(\id_X\otimes\delta) \es = & \id_X\;\;\colon & X\xrightarrow{\;\id_X\otimes\delta\,}X\otimes X\dual\otimes X\xrightarrow{\;{\rm ev}\otimes \id_X\,}X\qquad\text{and}\\[1mm]
\nonumber (\id_{X\dual}\otimes{\rm ev})\circ(\delta\otimes \id_{X\dual}) \es= & \id_{X\dual}\colon & X\dual\xrightarrow{\;\delta\otimes \id_{X\dual}\,}X\dual\otimes X\otimes X\dual\xrightarrow{\;\id_{X\dual}\otimes{\rm ev}\,}X\dual,
\end{eqnarray*}
\item \label{defn:tannakiancategory_B}
and for some non-zero $K$-algebra $L$ there is an exact faithful $K$-linear tensor functor $\omega$ from $\mathscr{C}$ to the category of finitely generated $L$-modules. Any such functor $\omega$ is called an \emph{$L$-rational fiber functor for $\mathscr{C}$}.
\end{enumerate}
A $K$-rational fiber functor for $\mathscr{C}$ is called \emph{neutral}. If $\mathscr{C}$ has a neutral fiber functor it is called a \emph{neutral Tannakian category over $K$}.
\end{definition}

\begin{remark}\label{RemDefTannakian}
(a) According to \cite[\S\,1]{DM82} being a \emph{tensor category} means that there is a ``tensor product'' functor $\mathscr{C}\times\mathscr{C}\to\mathscr{C},\,(X,Y)\mapsto X\otimes Y$ which is associative and commutative, such that $\mathscr{C}$ has a \emph{unit object}. The latter is an object $\BOne\in\mathscr{C}$ together with an isomorphism $\BOne\isoto\BOne\otimes\BOne$ such that $\mathscr{C}\to\mathscr{C},\,X\mapsto\BOne\otimes X$ is an equivalence of categories. A unit object is unique up to unique isomorphism; see \cite[Proposition~1.3]{DM82}. One sets  $X^{\otimes0}:=\BOne$ and $X^{\otimes n}:=X\otimes X^{\otimes n-1}$ for $n\in\BN_{>0}$.

\medskip\noindent
(b) Being \emph{$K$-linear} means that $\Hom_{\mathscr{C}}(X,Y)$ is a $K$-vector space for all $X,Y\in\mathscr{C}$.

\medskip\noindent
(c) Being \emph{abelian} means that $\mathscr{C}$ is an abelian category. Then automatically $\otimes$ is a bi-additive functor and is exact in each factor; see \cite[Proposition~1.16]{DM82}. 

\medskip\noindent
(d) By \cite[\S\S\,2.1--2.5]{DeligneCatTann} the conditions of Definition~\ref{defn:tannakiancategory} imply that $\End_{\mathscr{C}}(\BOne)=K$ and that the tensor product is $K$-bilinear and exact in each variable. It further implies that $\CHom(X,Y):= X\dual\otimes Y$ is an \emph{internal hom} in $\mathcal{C}$, that is an object which represents the functor $\mathscr{C}\open\to\KVec,\,T\mapsto\Hom_{\mathscr{C}}(T\otimes X,Y)$. This means that $\Hom_{\mathscr{C}}(T\otimes X,Y)=\Hom_{\mathscr{C}}(T,\CHom(X,Y))$. Then $\mathscr{C}$ is a \emph{rigid} abelian $K$-linear tensor category in the sense of \cite[Definition~2.19]{DM82}. This further means that the natural morphisms $X\to(X\dual)\dual$ are isomorphisms and that $\bigotimes_{i=1}^n\CHom(X_i,Y_i)=\CHom(\bigotimes_i X_i,\bigotimes_i Y_i)$ for all $X_i,Y_i\in\mathscr{C}$. The definition of a neutral Tannakian category over $K$ in \cite[Definition~2.19]{DM82} as a rigid abelian $K$-linear tensor category possessing a neutral fiber functor is equivalent to Definition~\ref{defn:tannakiancategory}.

\medskip\noindent
(e) A functor $F\colon\mathscr{C}\to\mathscr{C}'$ between rigid abelian $K$-linear tensor categories is a \emph{tensor functor} if $F(\BOne)$ is a unit object in $\mathscr{C}'$ and there are fixed isomorphisms $F(X\otimes Y)\cong F(X)\otimes F(Y)$ compatible with the associativity and commutativity laws. A tensor functor automatically satisfies $F(X\dual)=F(X)\dual$ and $F\bigl(\CHom(X,Y)\bigr)=\CHom\bigl(F(X),F(Y)\bigr)$; see \cite[Proposition~1.9]{DM82}. In particular, for an $L$-rational fiber functor $\omega$ this means $\omega(\BOne)\cong L$.
\end{remark}

If $G$ is an affine group scheme over $K$, let $\Rep_K(G)$ be the category of finite-dimensional $K$-rational representations of $G$, that is $K$-homomorphisms of $K$-group schemes $\rho\colon G\to\GL_K(V)\cong\GL_{\dim_K\!V,\,K}$ for varying finite dimensional $K$-vector spaces $V$. Together with the forgetful functor $\omega^G\colon(V,\rho)\mapsto V$ it is a neutral Tannakian category over $K$; see \cite[Example~1.24]{DM82}. Tannakian duality says that every neutral Tannakian category over $K$ is of this form:

\begin{theorem}[{Tannakian duality \cite[Theorem~2.11]{DM82}}] 
\label{theorem:theorem-deligne-211}
Let $\mathscr{C}$ be a neutral Tannakian category over $K$ with neutral fiber functor $\omega,$ and let $\Aut^\otimes(\omega)$ be the set of automorphisms of tensor functors of $\omega$; see \cite[p.~116]{DM82}. 
\begin{enumerate}
\item There is an affine group scheme $G$ over $K$ that represents the functor $\underline{\Aut}^\otimes(\omega)$ on $K$-algebras 
given by 
\[ \underline{\Aut}^\otimes(\omega)(R) := \Aut^\otimes(\psi_R \circ \omega)\quad\text{for all }K\text{-algebras }R,\]
where $\psi_R\colon  \KVec \to \RMod,\ V\mapsto V \otimes_K R,$  is the canonical tensor functor.
\item \label{theorem:theorem-deligne-211_B}
The fiber functor $\omega$ defines an equivalence of tensor categories $\mathscr{C}\isoto \Rep_K(G)$.
\end{enumerate}
\end{theorem} 

\begin{definition}\label{DefSubcategory}
A subcategory $\mathscr{C}'$ of a category $\mathscr{C}$ is \emph{strictly full} if it is full and contains with every $X\in\mathscr{C}'$ also all objects of $\mathscr{C}$ isomorphic to $X$. 

A strictly full subcategory $\mathscr{C}'$ of a rigid tensor category $\mathscr{C}$ is a \emph{rigid tensor subcategory} if $\BOne\in\mathscr{C}'$ and $X\otimes Y, X\dual\in\mathscr{C}'$ for all $X,Y\in\mathscr{C}'$.

If $\mathscr{C}$ is a neutral Tannakian category over $K$ and $X\in\mathscr{C}$, the rigid tensor subcategory of $\mathscr{C}$ containing as objects all subquotients of all $\bigoplus_{i=1}^r X^{\otimes n_i}\otimes(X\dual)^{\otimes m_i}$ for all $r,n_i,m_i\in\BN_0$ is called the \emph{Tannakian subcategory generated by $X$} and is denoted $\llangle X\rrangle$. It is a neutral Tannakian category over $K$.
\end{definition}

\begin{lemma}[{\cite[Proposition~2.20]{DM82}}]
\label{lem:prop-deligne-220}
An affine $K$-group scheme $G$ is (linear) algebraic, that is a closed subscheme of some $\GL_{n,K}$, if and only if there exists an object $X$ in $\Rep_K(G)$ with $\Rep_K(G)=\llangle X\rrangle$. In this case $G=\ul\Aut^\otimes(\omega^G)\into\GL_K\bigl(\omega^G(X)\bigr)$ is a closed immersion, which factors through the centralizer of $\End(X)$ inside $\GL\bigl(\omega^G(X)\bigr)$.
\end{lemma}

\begin{proof}
This was proved in \cite[Proposition~2.20]{DM82} except for the statement about the centralizer, which follows from the fact that $G$ is the automorphism group of the forgetful fiber functor $\omega^G$. 
\end{proof}

A homomorphism $f\colon G\to G'$ of affine $K$-group schemes induces a functor $\omega^f\colon \Rep_K(G')\to\Rep_K(G),\ \rho\,\mapsto\,\rho\circ f$, such that $\omega^G\circ\omega^f=\omega^{G'}$. The same holds in the other direction:

\begin{lemma}[{\cite[Corollary~2.9]{DM82}}]
\label{cor:defnmu}
Let $G$ and $G'$ be affine group schemes over $K$ and consider a tensor functor $F\colon\Rep_K(G')\to\Rep_K(G)$ such that $\omega^G\circ F=\omega^{G'}$. Then there is a unique homomorphism $f\colon G\to G'$ of affine $K$-group schemes such that $F\cong\omega^f$.
\end{lemma}

Under this correspondence various properties of group homomorphisms are reflected on the associated tensor functor.

\begin{proposition}[{\cite[Proposition~2.21]{DM82}}]
\label{prop:prop-deligne-221}
 Let $f\colon G\to G'$ be a homomorphism of affine $K$-group schemes and let $\omega^f\colon \Rep_K(G')\to\Rep_K(G)$ be defined as above.
 \begin{enumerate}
 \item $f$ is faithfully flat if and only if $\omega^f$ is fully faithful and for every object $X'$ in $\Rep_K(G')$ each subobject of $\omega^f(X')$ is isomorphic to the image of a subobject of $X'$.
 \item $f$ is a closed immersion if and only if for every object $X$ of $\Rep_K(G)$ there exists an object $X'$ in $\Rep_K(G')$ such that $X$ is isomorphic to a subquotient of $\omega^f(X')$.
 \end{enumerate}
\end{proposition}

 
\section{Hodge-Pink structures}\label{Sect1}
\setcounter{equation}{0}

In this section we present Pink's definition \cite{PinkHodge} of the Tannakian category of mixed $Q$-Hodge structures. Pink first defines pre-Hodge structures which form an additive tensor category. This category is not abelian, so he introduces a semistability condition for pre-Hodge structures. The semistable ones form a neutral Tannakian category and will be called Hodge structures. Compared to the classical theories of the rational mixed Hodge-structures of Deligne~\cite{DeligneHodge2} and the $p$-adic Hodge theory of Fontaine~\cite{Fontaine82} there is one important difference in Pink's theory. In the classical theories, Hodge structures consist of a vector space over one field (with additional structures like weight filtration or Frobenius endomorphism) and a decreasing Hodge filtration defined over a \emph{separable} extension of this field. In the function field setting $\BC/Q$ is \emph{not} separable and hence a semistability condition solely based on the Hodge filtration cannot be preserved under tensor products. This is Pink's crucial observation and the reason why he replaces Hodge filtrations by finer structures and why we call all these structures \emph{Hodge-Pink structures}.

\begin{definition}\label{DefQFiltr}
An \emph{exhaustive and separated increasing $\BQ$-filtration} $W_\bullet H$ on a finite dimensional $Q$-vector space $H$ is a collection of $Q$-subspaces $W_\mu H\subset H$ for $\mu\in\BQ$ with $W_{\mu'}H\subset W_\mu H$ whenever $\mu'<\mu$, such that the associated $\BQ$-graded vector space
\[
\Gr^W H\;:=\;\bigoplus_{\mu\in\BQ}\Gr_\mu^W H\;:=\;\bigoplus_{\mu\in\BQ}\TS\Bigl(W_\mu H/\bigcup_{\mu'<\mu}W_{\mu'}H\Bigr)
\]
has the same dimension as $H$. 
\end{definition}

\begin{remark}\label{RemQFiltr}
The \emph{jumps} of such a filtration are those real numbers $\mu$ for which 
\[
\TS\bigcup_{\mu'<\mu}W_{\mu'}H \;\subsetneq\;\bigcap_{\tilde\mu>\mu}W_{\tilde\mu}H\,.
\]
The condition $\dim_Q\Gr^WH=\dim_Q H$ is equivalent to the conditions that all jumps lie in $\BQ$, that $W_\mu H \;=\;\bigcap_{\tilde\mu>\mu}W_{\tilde\mu}H$ for all $\mu\in\BQ$, that $W_\mu H=(0)$ for $\mu\ll0$, and that $W_\mu H=H$ for $\mu\gg0$.
\end{remark}

\begin{definition}[{Pink~\cite[Definition~9.1]{PinkHodge}}]\label{Def1.1}
A \emph{(mixed) $Q$-pre Hodge-Pink structure} (\emph{at $\infty$}) is a triple $\ulH=(H,W_\bullet H,\Fq)$ with
\begin{itemize}
\item $H$ a finite dimensional $Q$-vector space,
\item $W_\bullet H$ an exhaustive and separated increasing $\BQ$-filtration,
\item a $\BC\dbl z-\zeta\dbr$-lattice $\Fq\subset H\otimes_Q \BC\dpl z-\zeta\dpr$ of full rank.
\end{itemize}
The filtration $W_\bullet H$ is called the \emph{weight filtration}, $\Fq$ is called the \emph{Hodge-Pink lattice}, and $\rk\ulH:=\dim_Q H$ is called the \emph{rank of $\ulH$}.
The jumps of the weight filtration are called the \emph{weights} of $\ulH$. If $\Gr^W_\mu H=H$, then $\ulH$ is called \emph{pure of weight $\mu$}.

A \emph{morphism} $f\colon (H,W_\bullet H,\Fq)\to(H',W_\bullet H',\Fq')$ of $Q$-pre Hodge-Pink structures consists of a morphism $f\colon H\to H'$ of $Q$-vector spaces satisfying $f(W_\mu H)\subset W_\mu H'$ for all $\mu$ and $(f\otimes\id)(\Fq)\subset\Fq'$. The morphism $f$ is called \emph{strict} if $f(W_\mu H)=f(H)\cap W_\mu H'$ for all $\mu$ and $(f\otimes\id)(\Fq)=\Fq'\cap\bigl(f(H)\otimes_Q\BC\dpl z-\zeta\dpr\bigr)$.
\end{definition}

\begin{remark}\label{Rem1.4}
The Hodge-Pink lattice of a mixed $Q$-pre Hodge-Pink structure $\ulH=(H,W_\bullet H,\Fq)$ induces an exhaustive and separated decreasing $\BZ$-filtration as follows. Define the tautological lattice $\Fp:=H\otimes_Q\BC\dbl z-\zeta\dbr$ inside $H\otimes_Q\BC\dpl z-\zeta\dpr$ and consider the natural projection
\[
\Fp\;\onto\;\Fp/(z -\zeta)\Fp\;=\;H\otimes_{Q,\charmorph}\BC\;=:\;H_\BC\,.
\]
The \emph{Hodge-Pink filtration} $F^\bullet H_\BC=(F^i H_\BC)_{i\in\BZ}$ of $H_\BC$ is defined by letting $F^i H_\BC$ be the image of $\Fp\cap(z-\zeta)^i\Fq$ in $H_\BC$ for all $i\in\BZ$; that is, $F^i H_\BC=\bigl(\Fp\cap(z-\zeta)^i\Fq\bigr)\big/\bigl((z-\zeta)\Fp\cap(z-\zeta)^i\Fq\bigr)$.
One finds that any morphism is also compatible with the Hodge-Pink filtrations, but a strict morphism is not necessarily strictly compatible with the Hodge-Pink filtrations.

The \emph{Hodge-Pink weights} $(\omega_1,\ldots,\omega_{\rk\ulH})$ of $\ulH$ are the jumps of the Hodge-Pink filtration. They are integers. Equivalently they are the elementary divisors of $\Fq$ relative to $\Fp$; that is, they satisfy $\Fq/(z-\zeta)^e\Fp\cong\bigoplus_{i=1}^{\rk\ulH}\BC\dbl z-\zeta\dbr/(z-\zeta)^{e+\omega_i}$ and $\Fp/(z-\zeta)^e\Fq\cong\bigoplus_{i=1}^{\rk\ulH}\BC\dbl z-\zeta\dbr/(z-\zeta)^{e-\omega_i}$ for all $e\gg0$. We usually assume that they are ordered $\omega_1\le\ldots\le\omega_{\rk\ulH}$.
\end{remark}

A main source for Hodge-Pink structures are Drinfeld $A$-modules or more generally uniformizable mixed abelian Anderson $A$-modules (see Section~\ref{SectAndersonAModules}). 

\begin{bigexample}\label{Example1.2}
(a) \es Let $\phi\colon A\to\End_{\BC}(\BG_{a,\BC})$ be a Drinfeld $A$-module~\cite{Drinfeld} of rank $r$ over $\BC$ where $\BG_{a,\BC}$ is the additive group scheme. We set $E=\BG_{a,\BC}$ and $\ulE=(E,\phi)$, and we write $\Lie E$ for the tangent space to $E$ at $0$.
Consider the exponential exact sequence of $A$-modules
\begin{equation}\label{Eq1.A}
0\;\longto\;\Lambda(\ulE)\;\longto\;\Lie E\;\xrightarrow{\;\exp_\ulE}\;E(\BC)\;\longto\;0\,,
\end{equation}
where $\Lambda:=\Lambda(\ulE):=\Lambda(\phi):=\ker(\exp_\ulE)$; see Section~\ref{SectDefAModules} or the survey of Brownawell and Papanikolas \cite[\S\,2.4]{BrownawellPapanikolas16} in this volume. $\Lambda(\ulE)\subset\Lie E=\BC$ is a discrete $A$-submodule of rank $r$. Clearly, $\Lambda(\ulE)$ generates the one dimensional $\BC$-vector space $\Lie E$. Through the identification $\BC\dbl z-\zeta\dbr/(z-\zeta)=\BC$ we make $\Lie E$ into a $\BC\dbl z-\zeta\dbr$-module. We obtain a $\BC\dbl z-\zeta\dbr$-epimorphism on the right in the sequence
\begin{equation}\label{Eq1.0}
\xymatrix @R=0pc {
0 \ar[r] & \Fq \ar[r] & \Lambda\otimes_A \BC\dbl z-\zeta\dbr \ar[r] & \Lie E \ar[r] & 0 \\
& & \lambda\otimes\sum_i b_i(z-\zeta)^i \ar@{|->}[r] & b_0\cdot\lambda
}
\end{equation}
and we let $\Fq$ be its kernel. By sequence \eqref{Eq1.0} the pair $(\Lambda,\Fq)$ determines the $\BC$-vector space $\Lie E$ with the $A$-action on it, and the $A$-lattice $\Lambda$ inside $\Lie E$ as the image of the $A$-homomorphism $\Lambda\into\Lambda\otimes_A \BC\dbl z-\zeta\dbr\onto\Lie E$. Therefore, the pair $(\Lambda,\Fq)$ also determines the Drinfeld $A$-module $\phi$ by sequence \eqref{Eq1.A}. We further set 
\[H\;:=\;\Hodge_1(\ulE)\;:=\;\Lambda(\ulE)\otimes_A Q\qquad \text{and}\qquad 
W_\mu H = \left\{\begin{array}{ll}(0)&\text{if }\mu<-\frac{1}{r}\,,\\[2mm]H&\text{if }\mu\ge-\frac{1}{r}\,.\end{array}\right.
\]
Then $\ulHodge_1(\ulE):=(H,W_\bullet H,\Fq)$ is a pure $Q$-pre Hodge-Pink structure of weight $-\frac{1}{r}$. It satisfies $(z-\zeta)\Fp\subset\Fq\subset\Fp$ and hence $F^{-1}H_\BC=H_\BC\supset F^0H_\BC\supset F^1H_\BC=(0)$. Since $\dim_Q H=r$ and $\dim_\BC(\Fp/\Fq)=\dim_\BC\Lie E=1$ we have $\dim_\BC F^0H_\BC=r-1$. As we will explain in Section~\ref{SectCohAMod} below, $F^0H_\BC\subset H_\BC=\Koh_{1,\Betti}(\ulE,\BC)$ is the Hodge filtration studied by Gekeler \cite[(2.13)]{Gekeler89} using the de Rham isomorphism $\Koh^1_\Betti(\ulE,\BC)\cong\Koh^1_\dR(\ulE,\BC)$. See Example~\ref{Ex1.2Cont} for a continuation of this example. Also in Section~\ref{SectAModHPStr} we will generalize the present construction to Anderson's abelian $t$-modules \cite{Anderson86}. Note that this parallels the case of complex abelian varieties $X$, whose Hodge structure $\Koh_{1,\Betti}(X,\BQ)$ is pure of weight $-\tfrac{1}{2}$.

\medskip
\noindent
(b) \es More specifically, if $C=\BP^1_{\BF_q},A=\BF_q[t],\theta:=\charmorph(t)\in\BC$ and $\ulE$ is the \emph{Carlitz-module} \cite[\S\,2.2]{BrownawellPapanikolas16} with $\phi_t=\theta+\Frob_{q,\BG_a}$, where $\Frob_{q,\BG_a}\colon x\mapsto x^q$ is the relative $q$-Frobenius of $\BG_{a,\BC}=\Spec\BC[x]$ over $\BC$, then $r=1$ and 
\[
H:=\Hodge_1(\ulE)=Q\,,\es\Gr_{-1}^W H=H\,,\es \Fq=(z-\zeta)\cdot\Fp\,,\es F^0H_\BC=(0)\,.
\]

\medskip
\noindent
(c) \es In (a) and (b) the subspace $F^0 H_\BC$ determines $\Fq$ uniquely as its preimage under the surjection $H\otimes_Q\BC\dbl z-\zeta\dbr\onto H_\BC$ because $(z-\zeta)\cdot\Fp\subset\Fq$.

However, note that in general $\Fq$ is \emph{not} determined by $F^\bullet H_\BC$. For example let $H=Q^{\oplus2}$ and $\Fq=(z-\zeta)^2\Fp+\BC\dbl z-\zeta\dbr\cdot\bigl(v_0+(z-\zeta)v_1\bigr)$ for $v_i\in H_\BC$ with $v_0\ne0$. Then 
\[
F^{-2}H_\BC=H_\BC\supset F^{-1}H_\BC=\BC\cdot v_0=F^0H_\BC\supset F^1H_\BC=(0)\,. 
\]
So the information about $v_1$ is not preserved by the Hodge-Pink filtration.
\end{bigexample}

\medskip

To continue with the general theory let $\ulH=(H,W_\bullet H,\Fq)$ be a $Q$-pre Hodge-Pink structure. A \emph{subobject} in the category of $Q$-pre Hodge-Pink structures is a morphism $\ulH'\to\ulH$ whose underlying homomorphism of $Q$-vector spaces is the inclusion $H'\into H$ of a subspace. It is called a \emph{strict subobject} if $\ulH'\to\ulH$ is strict. Likewise a \emph{quotient object} is a morphism $\ulH\to\ulH''$ whose underlying homomorphism of $Q$-vector spaces is the projection $H\onto H''$ onto a quotient space. It is called a \emph{strict quotient object} if $\ulH\to\ulH''$ is strict. 

For any $Q$-subspace $H'\subset H$ one can endow $H'$ with a unique structure of strict subobject $\ulH'$ and $H'':=H/H'$ with a unique structure of strict quotient object $\ulH''$. The sequence $0\to\ulH'\to\ulH\to\ulH''\to0$ and any sequence isomorphic to it is called a \emph{strict exact sequence}.

With these definitions the category of $Q$-pre Hodge-Pink structures is a $Q$-linear additive category. Pink makes a suitable subcategory of it into a Tannakian category. In order to do this, he defines tensor products, internal hom and duals.

\begin{definition}\label{Def1.3}
Let $\ulH_1=(H_1,W_\bullet H_1,\Fq_1)$ and $\ulH_2=(H_2,W_\bullet H_2,\Fq_2)$ be two $Q$-pre Hodge-Pink structures.
\begin{enumerate}
\item \label{Def1.3_A}
The \emph{tensor product} $\ulH_1\otimes\ulH_2$ is the $Q$-pre Hodge-Pink structure consisting of the tensor product $H_1\otimes_Q H_2$ of $Q$-vector spaces, the induced weight filtration $W_\mu(H_1\otimes_Q H_2):=\sum_{\mu_1+\mu_2=\mu}W_{\mu_1}H_1\otimes_Q W_{\mu_2}H_2$ and the lattice $\Fq_1\otimes_{\BC\dbl z-\zeta\dbr}\Fq_2$. One defines for $n\ge1$ the symmetric power $\Sym^n\ulH$ and the alternating power $\wedge^n\ulH$ as the induced strict quotient objects of $\ulH^{\otimes n}$.
\item\label{Def1.3_B}
The \emph{internal hom} $\ulTH=\CHom(\ulH_1,\ulH_2)$ consists of the $Q$-vector space $\wt H:=\Hom_Q(H_1,H_2)$, the induced weight filtration $W_\mu\wt H:=\{\,h\in\wt H:h(W_{\mu_1} H_1)\subset W_{\mu+\mu_1}H_2\es\forall\,\mu_1\,\}$, and the lattice $\tilde\Fq:=\Hom_{\BC\dbl z-\zeta\dbr}(\Fq_1,\Fq_2)$. The latter is a $\BC\dbl z-\zeta\dbr$-lattice in $\wt H\otimes_Q\BC\dpl z-\zeta\dpr$ via the inclusion
\begin{eqnarray*}
\tilde\Fq\ & \longinto & \tilde\Fq\otimes_{\BC\dbl z-\zeta\dbr}\BC\dpl z-\zeta\dpr \\[2mm]
& \isoto & \Hom_{\BC\dpl z-\zeta\dpr}\bigl(\Fq_1\otimes_{\BC\dbl z-\zeta\dbr}\BC\dpl z-\zeta\dpr,\Fq_2\otimes_{\BC\dbl z-\zeta\dbr}\BC\dpl z-\zeta\dpr\bigr)\\[2mm]
& \isoto & \Hom_{\BC\dpl z-\zeta\dpr}\bigl(H_1\otimes_Q\BC\dpl z-\zeta\dpr,H_2\otimes_Q\BC\dpl z-\zeta\dpr\bigr)\\[2mm]
& \isoto & \wt H\otimes_Q\BC\dpl z-\zeta\dpr
\end{eqnarray*}
obtained by applying \cite[Proposition~2.10]{Eisenbud}.
\item \label{Def1.3_C}
The \emph{unit object} $\UOne$ consists of the vector space $Q$ itself together with the lattice $\Fq:=\Fp$ and is pure of weight $0$. The \emph{dual} $\ulH\dual$ of a $Q$-pre Hodge-Pink structure $\ulH$ is then $\CHom(\ulH,\UOne)$.
\end{enumerate}
\end{definition}

The category of $Q$-pre Hodge-Pink structures is an additive tensor category but it is not abelian because not all subobjects and quotient objects are strict. Indeed, the category theoretical image (respectively coimage) of a subobject $\ulH'\into\ulH$ (respectively quotient object $\ulH\onto\ulH'$) is the strict subobject (respectively strict quotient object) with same underlying $Q$-vector space as $\ulH'$ (respectively $\ulH''$). In order to remedy this, Pink defines semistability as follows.

\begin{definition}\label{Def1.5}
Let $\ulH=(H,W_\bullet H,\Fq)$ be a $Q$-pre Hodge-Pink structure.
\begin{enumerate}
\item 
for any $Q_\infty$-subspace $H'_\infty\subset H_\infty:=H\otimes_Q Q_\infty$ consider the \emph{induced strict $Q_\infty$-subobject}
\[
\ulH'_\infty\;:=\;\Bigl(H'_\infty\,,\, W_\mu H'_\infty:=H'_\infty\cap(W_\mu H\otimes_Q Q_\infty)\,,\, \Fq':=\Fq\cap \bigl(H'_\infty\otimes_{Q_\infty}\BC\dpl z-\zeta\dpr\bigr)\Bigr)
\]
and (using the induced Hodge-Pink filtration $F^\bullet H_\BC$ from Remark~\ref{Rem1.4}) set
\begin{eqnarray*}
\deg_\Fq\ulH'_\infty& := &\deg_F H'_\BC\es :=\es \sum_{i\in\BZ} i\cdot\dim_\BC\Gr^i_F H'_\BC\es =\es \dim_\BC\frac{\Fq'}{\Fp'\cap\Fq'}\, -\, \dim_\BC\frac{\Fp'}{\Fp'\cap\Fq'} \\[2mm]
\deg^W\ulH'_\infty&:=&\sum_{\mu\in\BQ} \mu\cdot\dim_{Q_\infty}\Gr_\mu^W H'_\infty
\end{eqnarray*}
\item
$\ulH$ is called \emph{locally semistable} or a \emph{(mixed) $Q$-Hodge-Pink structure} (\emph{at $\infty$}) if for any $Q_\infty$-subspace $H'_\infty\subset H_\infty$ one has $\deg_\Fq\ulH'_\infty\le \deg^W\ulH'_\infty$ with equality for $H'_\infty=(W_\mu H)_\infty$ for all $\mu$.
\item 
We denote by $\QHodgeCat$ the full subcategory of all mixed $Q$-Hodge-Pink structures.
\end{enumerate}
\end{definition}

\begin{remark}\label{RemPolygons}
(a) Alternatively $\deg_\Fq\ulH'_\infty$ can be computed as $\dim_\BC\Fq'/\Fr-\dim_\BC\Fp'/\Fr$ for any $\BC\dbl z-\zeta\dbr$-lattice $\Fr$ which is contained in both $\Fq'$ and $\Fp'$. In particular, if $(z-\zeta)^d\Fp'\subset\Fq'$ for some $d\in\BZ$ with $d\ge0$, then $\deg_\Fq\ulH'_\infty=\dim_\BC\Fq'/(z-\zeta)^d\Fp'-d\dim_{Q_\infty}H'_\infty$, because $\dim_\BC\Fp'/(z-\zeta)^d\Fp'=d\cdot\dim_{Q_\infty}H'_\infty$.

\medskip\noindent
(b) The piecewise linear function on $[0,\rk\ulH]$ whose slope on $[i-1,i]$ is the $i$-th smallest Hodge-Pink weight is called the \emph{Hodge polygon of $\ulH$} and is denoted $HP(\ulH)$. Analogously one defines the \emph{weight polygon $WP(\ulH)$ of $\ulH$} using the weights of $\ulH$. A $Q$-pre Hodge-Pink structure is locally semistable if and only if for every strict $Q_\infty$-subobject $\ulH'_\infty$ the weight polygon lies above the Hodge polygon, and both have the same endpoint whenever $\ulH'_\infty=(W_{\mu\,}\ulH)_\infty$; see \cite[Proposition~6.7]{PinkHodge}.
\end{remark}

\begin{example}\label{Ex1.2Cont} We continue with Example~\ref{Example1.2}.

\medskip
\noindent
(a) \es If $\ulE$ is the Carlitz-module over $A=\BF_q[t]$ then $\ulHodge_1(\ulE)$ is a pure $Q$-Hodge-Pink structure of weight $-1$, because $\deg^W\ulHodge_1(\ulE)=-1=\deg_\Fq\ulHodge_1(\ulE)$ and there are no non-trivial $Q_\infty$-subspaces of $\Hodge_1(\ulE)\otimes_QQ_\infty=Q_\infty$.

\medskip
\noindent
(b) \es The same is true for a Drinfeld $A$-module $\phi$. Indeed, assume that $\ulHodge_1(\ulE)$ is not locally semistable. Then there is a non-trivial $Q_\infty$-subspace $H'_\infty\subset H_\infty$ with $\deg_\Fq\ulH'_\infty>\deg^W\ulH'_\infty$. Since $\ulHodge_1(\ulE)$ is pure of weight $-\frac{1}{r}$ we find $\deg^W\ulH'_\infty=-\frac{1}{r}\cdot\dim_{Q_\infty}H'_\infty>-1$ and $\deg_\Fq\ulH'_\infty\ge0$. Since $(z-\zeta)\Fp\subset\Fq\subset\Fp$ the same is true for $\Fq'=\Fq\cap \bigl(H'_\infty\otimes_{Q_\infty}\BC\dpl z-\zeta\dpr\bigr)$ and $\Fp'=H'_\infty\otimes_{Q_\infty}\BC\dbl z-\zeta\dbr$. So $\deg_\Fq\ulH'_\infty$ can only be non-negative if $\Fp'=\Fp'\cap\Fq'$; that is, $\Fp'=\Fq'$. This implies
\[
H_\infty'\;\subset\;H'_\BC\;=\;\Fp'/(z-\zeta)\Fp'\;=\;\Fq'/(z-\zeta)\Fp'\;\subset\;\Fq/(z-\zeta)\Fp\;=\;\ker(H_\BC\to\Lie E)\,.
\]
But $\Lambda(\phi)\subset\Lie E$ is discrete, which by definition means that the natural morphism $H_\infty=\Lambda(\phi)\otimes_A Q_\infty\to\Lie E$ is injective. Therefore, also $H'_\infty\to\Lie E$ must be injective and we obtain a contradiction.
\end{example}

One of the main results of Pink~\cite{PinkHodge} is the following

\begin{theorem}[{\cite[Theorem~9.3]{PinkHodge}}] \label{ThmPinkTannaka}
The category $\QHodgeCat$ together with the $Q$-rational fiber functor $\omega_0\colon \QHodgeCat\to \QVec$, $(H,W_\bullet H,\Fq)\mapsto H$, is a neutral Tannakian category over $Q$.
\end{theorem}

See Section~\ref{SectTannaka} for some explanations.

\begin{remark}\label{RemQ-HPTannakian}
(a) The assertion that $\QHodgeCat$ is abelian rests on the relatively easy fact that in $\QHodgeCat$ any subobject and quotient object is strict.

\medskip
\noindent
(b) The difficult part of the proof is to show that the condition of local semistability is closed under tensor products. For this it is essential to work with Hodge-Pink lattices instead of Hodge-Pink filtrations. Indeed, if one works with triples $(H,W_\bullet H,F^\bullet H_\BC)$ consisting of $Q$-vector spaces $H$ with weight filtration $W_\bullet H$ and decreasing Hodge-Pink filtrations $F^\bullet H_\BC$ and defines local semistability analogous to Definition~\ref{Def1.5}, then this local semistability would not be closed under tensor products due to the inseparability of the field extension $\BC/Q$; see \cite[Example~5.16]{PinkHodge}. This is Pink's ingenious insight.
\end{remark}

This theorem allows to associate with each $Q$-Hodge-Pink structure $\ulH=(H,W_\bullet H,\Fq)$ an algebraic group $\Gamma_\ulH$ over $Q$ as follows. Consider the Tannakian subcategory $\llangle\ulH\rrangle$ of $\QHodgeCat$ generated by $\ulH$. By \cite[Theorem~2.11 and Proposition~2.20]{DM82} the category $\llangle\ulH\rrangle$ is tensor equivalent to the category of $Q$-rational representations of a linear algebraic group scheme $\Gamma_\ulH$ over $Q$ which is a closed subgroup of $\GL_Q(H)$.

\begin{definition}\label{DefHPGp}
The linear algebraic $Q$-group scheme $\Gamma_\ulH$ associated with $\ulH$ is called the \emph{Hodge-Pink group of $\ulH$}.
\end{definition}

Pink proves that $\Gamma_\ulH$ is connected and reduced and that any connected semisimple group over $Q$ can occur as $\Gamma_\ulH$ for a $Q$-Hodge-Pink structure \cite[Propositions~9.4 and 9.12]{PinkHodge}. Note however, that in general $\Gamma_\ulH$ does not even need to be reductive.

If the Hodge-Pink structure $\ulH$ comes from a pure (or mixed) uniformizable abelian $t$-module $\ulE$, Pink (respectively Pink and the first author) also proved in unpublished work, that $\Gamma_\ulH$ equals the motivic Galois group of $\ulE$ as considered by Papanikolas~\cite{Papanikolas} and Taelman~\cite{Taelman}; see Remark~\ref{RemPapanikolasGaloisGp}. If $\ulH$ comes from a pure dual $A$-motive, Pink's proof was worked out by the second author in her Diploma thesis~\cite{JuschkaDipl}. We will explain these proofs in Theorems~\ref{ThmHodgeConjecture} and \ref{ThmDualHodgeConjecture} below. In the special case when $\ulE$ is a Drinfeld module, there are further results of Pink on the structure of $\Gamma_\ulH$; see Section~\ref{Sect3}.

 
\section{Mixed \texorpdfstring{$A$}{A}-motives}\label{SectMixedAMotives}
\setcounter{equation}{0}

The functor $\ulE\mapsto\ulHodge_1(\ulE)$ from Drinfeld $A$-modules to $Q$-Hodge-Pink structures from Examples~\ref{Example1.2} and \ref{Ex1.2Cont} extends to the uniformizable abelian $t$-modules of Anderson~\cite{Anderson86}, the higher dimensional generalizations of Drinfeld-modules. We will define the functor in Section~\ref{SectAModHPStr} below. In order to prove that $\ulHodge_1(\ulE)$ is a pure $Q$-Hodge-Pink structure when $\ulE$ is a pure uniformizable abelian $t$-module, we need to review Anderson's theory of $t$-motives \cite{Anderson86} or more generally $A$-motives. We do this first because it also allows to define mixed abelian $t$-modules and their associated mixed $Q$-Hodge-Pink structures.

\subsection{\texorpdfstring{$A$}{A}-motives}\label{SectAMotives}

Recall that we denote the natural inclusion $Q\into\BC$ by $\charmorph$ and consider the maximal ideal $J:=(a\otimes 1-1\otimes \charmorph(a):a\in A)\subset A_{\BC}:=A\otimes_{\BF_q}\BC$. The open subscheme $\Spec A_{\BC}\setminus{\rm V}(J)$ of $C_\BC$ is affine. We denote its ring of global sections by $\GlobMotRing$. For example if $C=\BP^1_{\BF_q}$ and $A=\BF_q[t]$ then $J=(t-\theta)$ for $\theta:=\charmorph(t)$. In this case $\GlobMotRing=\BC[t][\tfrac{1}{t-\theta}]$.

\begin{definition}\label{DefAMotive}
\begin{enumerate}
\item 
An \emph{$A$-motive} over $\BC$ of characteristic $\charmorph$ is a pair $\ulM=(M,\stau_M)$ consisting of a finite projective $A_{\BC}$-module $M$ and an isomorphism of $\GlobMotRing$-modules
\[
\stau_M\colon \ssigma^\ast M[J^{-1}]\isoto M[J^{-1}]\,.
\]
where we set $\ssigma^\ast M[J^{-1}]:=(\ssigma^\ast M)\otimes_{A_{\BC}}\GlobMotRing$ and $M[J^{-1}]:=M\otimes_{A_{\BC}}\GlobMotRing$. A \emph{morphism} of $A$-motives $f\colon \ulM\to\ulN$ is a homomorphism of the underlying $A_{\BC}$-modules $f\colon M\to N$ that satisfies $f\circ\stau_M=\stau_N\circ\ssigma^\ast f$. 
The category of $A$-motives over $\BC$ is denoted $\AMotCat$.
\item\label{DefAMotive_b}
The rank of the $A_\BC$-module $M$ is called the \emph{rank} of $\ulM$ and is denoted by $\rk\ulM$. The \emph{virtual dimension} $\dim\ulM$ of $\ulM$ is defined as
\[
\dim\ulM\;:=\;\dim_\BC \,M\big/(M\cap\tau_M(\sigma^\ast M))\;-\;\dim_\BC \,\tau_M(\sigma^\ast M)\big/(M\cap\tau_M(\sigma^\ast M))\,.
\]
\item 
An $A$-motive $(M,\stau_M)$ is called \emph{effective} if $\stau_M$ comes from an $A_{\BC}$-homomorphism $\ssigma^\ast M\to M$. An effective $A$-motive has virtual dimension $\ge0$.
\item For two $A$-motives $\ulM$ and $\ulN$ over $\BC$ we call $\QHom(\ulM,\ulN):=\Hom_{\AMotCat}(\ulM,\ulN)\otimes_A Q$ the set of \emph{quasi-morphisms} from $\ulM$ to $\ulN$.
\item 
The category with all $A$-motives as objects and the $\QHom(\ulM,\ulN)$ as $\Hom$-sets is called the \emph{category of $A$-motives over $\BC$ up to isogeny}. It is denoted $\AMotCatIsog$.
\end{enumerate}
\end{definition}

\begin{remark}\label{Rem3.2}
(a) If $C=\BP^1_{\BF_q}$, $A=\BF_q[t]$ and $A_{\BC}=\BC[t]$, we set $\theta:=\charmorph(t)$ and then $J=(t-\theta)$. In this case, our effective $A$-motives are a slight generalization of Anderson's \emph{$t$-motives} \cite{Anderson86}, which are called \emph{abelian $t$-motives} in \cite[\S4.1]{BrownawellPapanikolas16}. Namely, Anderson required in addition, that $M$ is finitely generated over the skew-polynomial ring $\BC\{\tau\}$, where $\tau$ acts on $M$ through $m\mapsto \tau_M(\sigma^\ast m)$. 

\medskip\noindent
(b) We will explain in Remark~\ref{Rem3.3}(c) below, that the set of morphisms $\Hom_{\AMotCat}(\ulM,\ulN)$ between $A$-motives $\ulM$ and $\ulN$ is a finite projective $A$-module of rank at most $(\rk\ulM)\cdot(\rk\ulN)$.

\medskip\noindent
(c) By definition, for every quasi-morphism $f\in\QHom(\ulM,\ulN)$ there is an element $a\in A\setminus\{0\}$ such that $a\cdot f \in\Hom_{\AMotCat}(\ulM,\ulN)$ is a morphism of $A$-Motives. Moreover,
\[
\QHom(\ulM,\ulN)\;=\;\bigl\{\,f\colon M\otimes_{A_\BC}\Quot(A_\BC) \to N\otimes_{A_\BC}\Quot(A_\BC)\text{ such that }f\circ\stau_M=\stau_N\circ\ssigma^\ast f\,\bigr\}\,,
\]
where $\Quot(A_\BC)$ denotes the fraction field of $A_\BC$ and $f$ is a homomorphism of $\Quot(A_\BC)$-vector spaces. Indeed, the inclusion $\subset$ is obvious and the equality was proved in \cite[Corollary~5.4]{BH1} and also follows from \cite[Proposition~3.4.5]{Papanikolas} and \cite[Proposition~3.1.2]{Taelman}. Note that this is \emph{not} equivalent to the inclusion $f(J^n\cdot M)\subset N$ for $n\gg0$, as can be seen from $f=\id_\ulM\otimes\frac{1}{a}\in\QEnd(\ulM)$ for $a\in A\setminus\BF_q$.

\medskip\noindent
(d) The name for the category $\AMotCatIsog$ stems from the fact that a morphism $f\colon \ulM\to\ulN$ in $\AMotCat$ is an \emph{isogeny}, that is injective with torsion cokernel, if and only if it becomes an isomorphism in $\AMotCatIsog$; see for example \cite[Theorem~5.12]{HartlIsog} or \cite[Proposition~3.1.2]{Taelman}.
\end{remark}

The \emph{tensor product} of two $A$-motives $\ulM$ and $\ulN$ is the $A$-motive $\ulM\otimes\ulN$ consisting of the $A_{\BC}$-module $M\otimes_{A_{\BC}}N$ and the isomorphism $\stau_M\otimes\stau_N$. The $A$-motive $\UOne(0)$ with underlying $A_{\BC}$-module $A_{\BC}$ and $\stau=\id_{A_{\BC}}$ is a \emph{unit object} for the tensor product in $\AMotCat$ and $\AMotCatIsog$. Both categories possess finite direct sums in the obvious way. We also define the \emph{tensor powers} of an $A$-motive $\ulM$ as $\ulM^{\otimes0}=\UOne(0)$ and as $\ulM^{\otimes n}:=\ulM^{\otimes n-1}\otimes\ulM$ for $n>0$. 
%
%
The \emph{dual} of an $A$-motive $\ulM$ is the $A$-motive $\ulM\dual=(M\dual,\tau_{M\dual})$ consisting of the $A_{\BC}$-module $M\dual := \Hom_{A_{\BC}}(M,A_{\BC})$ and the isomorphism 
\[
\stau_M\dual\colon \ssigma^*M\dual[J^{-1}]\;=\;\Hom_{A_\BC}(\ssigma^*M,A_\BC)[J^{-1}]\;\isoto\; M\dual[J^{-1}],\quad h\mapsto h\circ\tau_M^{-1}.
\] 
If $\ulM=(M,\tau_M)$ and $\ulN=(N,\tau_N)$ are $A$-motives the \emph{internal hom} $\CHom(\ulM,\ulN)$ is the $A$-motive with underlying $A_\BC$-module $H:=\Hom_{A_\BC}(M,N)$ and $\tau_H\colon\sigma^*H[J^{-1}]\isoto H[J^{-1}],\, h\mapsto \tau_N\circ h\circ \tau_M^{-1}$. In particular, $\ulM\dual=\CHom(\ulM,\UOne(0))$. Moreover, there is a canonical isomorphism of $A$-motives $\ulM\dual\otimes\ulN\cong\CHom(\ulM,\ulN)$ sending $\sum_i m_i\dual\otimes n_i\in M\dual\otimes_{A_\BC}N$ to $[m\mapsto\sum_i m_i\dual(m)\cdot n_i]\in \Hom_{A_\BC}(M,N)$. Indeed, this is an isomorphism on the underlying finite locally free $A_\BC$-modules, and it is obviously compatible with the isomorphisms $\tau$. This implies that there are morphisms in $\AMotCat$
\begin{eqnarray}\label{EqEV1}
& {\rm ev}\colon & \TS\ulM\otimes\ulM\dual\;\longto\;\UOne(0)\,,\quad\sum\limits_i m_i\otimes m_i\dual\;\longmapsto\;\sum\limits_i m_i\dual(m_i)\quad\text{and}\\[2mm]
& \delta\colon & \UOne(0)\;\longto\;\ulM\dual\otimes\ulM \;=\;\CHom(\ulM,\ulM)\,,\quad a\;\longmapsto\;a\cdot\id_\ulM\,,\label{EqEV2}
\end{eqnarray}
which satisfy the conditions of Definition~\ref{defn:tannakiancategory}\ref{defn:tannakiancategory_A}. We also note the following formulas for the rank and the virtual dimension 
\begin{align}\label{EqRkDim}
\rk\UOne(0) & = 1\,, & \dim \UOne(0) & = 0\,,\nonumber\\[2mm]
\rk\CHom(\ulM,\ulN) & = (\rk\ulM)\cdot(\rk\ulN)\,, & \dim\CHom(\ulM,\ulN) & = (\rk\ulM)\cdot(\dim\ulN)-(\rk\ulN)\cdot(\dim\ulM)\,,\nonumber\\[2mm]
\rk\ulM\dual & = \rk\ulM\,, & \dim\ulM\dual & = - \dim\ulM\,,\\[2mm]
\rk\ulM\otimes\ulN & =(\rk\ulM)\cdot(\rk\ulN)\,, & \dim\ulM\otimes\ulN & = (\rk\ulN)\cdot(\dim\ulM)+(\rk\ulM)\cdot(\dim\ulN)\,,\nonumber\\[2mm]
\rk\ulM\oplus\ulN & = (\rk\ulM)+(\rk\ulN)\,, & \dim\ulM\oplus\ulN & = (\dim\ulM)+(\dim\ulN)\,,\nonumber
\end{align}
which follow easily from the elementary divisor theorem.

\begin{proposition}\label{PropImageAMotive}
Let $f\colon\ulM\to\ulM'$ be a morphism of $A$-motives.
\begin{enumerate}
\item \label{PropImageAMotive_A}
Then
\[ 
\ul\ker\,f \; := \; \bigl(\ker f,\,\tau_M|_{(\sigma^*\ker f)[J^{-1}]}\bigr)\qquad\text{and}\qquad \ul\im\,f \; := \; \bigl(\im f,\,\tau_{M'}|_{(\sigma^*\im f)[J^{-1}]}\bigr)
\]
are $A$-motives, which are called the \emph{kernel}, respectively \emph{image} $A$-motive of $f$.
\item \label{PropImageAMotive_B}
Let $N=M'/f(M)$ and let $N_\tors\subset N$ be the $A_\BC$-torsion submodule. Then $\tau_{M'}$ induces an isomorphism $\tau_{N/N_\tors}\colon \ssigma^*(N/N_\tors)[J^{-1}]\isoto N/N_\tors[J^{-1}]$ and
\[
\ul\coker\,f \; := \; (N/N_\tors,\,\tau_{N/N_\tors})\qquad\text{and}\qquad \ul\coim\,f \; := \; \ul\ker(\ulM'\to\ul\coker\, f)
\]
are $A$-motives, which are called the \emph{cokernel}, respectively \emph{coimage} $A$-motive of $f$. The $A$-motive $\ul\coim\,f$ equals the saturation $\{m'\in M'\colon \exists\,h\in A_\BC,h\ne0\text{ with }h\cdot m'\in f(M)\}$ of $\ul\im\,f$ and the natural inclusion $\ul\im\,f\into\ul\coim\,f$ is an isogeny, and hence an isomorphism in $\AMotCatIsog$. In particular, $\rk(\ul\im\,f)=\rk(\ul\coim\,f)$.
\end{enumerate}
\end{proposition}

\begin{proof}
Since $A_\BC$ is a Dedekind domain, the kernel and image of $f$ and $N/N_\tors$ are again finite, locally free $A_\BC$-modules and therefore $A$-motives with the inherited isomorphism $\tau$. That $\ul\coim\,f$ is the saturation of $\ul\im\,f$ follows from the definition of $N_\tors$. Therefore, the inclusion $\ul\im\,f\into\ul\coim\,f$ is injective with torsion cokernel, hence an isogeny and an isomorphism in $\AMotCatIsog$ by \cite[Theorem~5.12]{HartlIsog} or \cite[Proposition~3.1.2]{Taelman}.
\end{proof}

\begin{proposition}\label{PropAMotTannakianCateg}
The category $\AMotCatIsog$ is a $Q$-linear (non-neutral) Tannakian category, and in particular, a rigid abelian tensor category.
\end{proposition}

\begin{proof}
Since the $\sigma$-invariants in $A_\BC$ equal $A$, we have $\End_{\AMotCat}\bigl(\UOne(0)\bigr)=A$ and $\End_{\AMotCatIsog}\bigl(\UOne(0)\bigr)=Q$. In particular, $\AMotCatIsog$ is a $Q$-linear tensor category. If $f\colon \ulM\to\ulN$ is a morphism in $\AMotCatIsog$ we may multiply $f$ with an element of $Q$ and assume that $f$ is a morphism in $\AMotCat$. Therefore, it follows from Proposition~\ref{PropImageAMotive} that $\AMotCatIsog$ is abelian.

To show that $\AMotCatIsog$ is Tannakian we use the morphisms \eqref{EqEV1} and \eqref{EqEV2}. In addition, we have to exhibit an exact faithful $Q$-linear fiber functor over some non-zero $Q$-algebra. For example, we can take the quotient field $\Quot(A_\BC)$ of $A_\BC$ and the functor $\ulM=(M,\tau_M)\longmapsto M\otimes_{A_\BC}\Quot(A_\BC)$. This functor is faithful, because $M\subset M\otimes_{A_\BC}\Quot(A_\BC)$. Moreover, it is exact, because a sequence $0\longto\ulM'\xrightarrow{\es f\;}\ulM\xrightarrow{\es g\;}\ulM''\longto0$ in $\AMotCatIsog$ is exact if and only if $f$ is injective, $\ul\im\, f\cong\ul\ker\, g$, and $\ul\im\, g\cong\ulM''$ in $\AMotCatIsog$. By the definition of morphisms in $\AMotCatIsog$ as quasi-morphisms, these isomorphisms are in general not isomorphisms of the underlying $A_\BC$-modules, but they provide isomorphisms of the associated $\Quot(A_\BC)$-vector spaces.
\end{proof}

\begin{remark}\label{RemQuillenAMot}
(a) Fiber functors over $\BC$, respectively $Q_v$, are also provided by the \emph{de Rham cohomology realization} $\Koh^1_\dR(\ulM,\BC)$, respectively the \emph{$v$-adic cohomology realization} $\Koh^1_v(\ulM,Q_v)$; see Section~\ref{CohAMot}. A neutral fiber functor only exists on the full subcategory of \emph{uniformizable} $A$-motives; see Theorem~\ref{TheoremAMotTannakian}

\medskip\noindent
(b) The category $\AMotCat$ is an exact category in the sense of Quillen~\cite[\S2]{Quillen} if one defines the class $E$ of \emph{short exact sequences} to be those sequences $0\longto\ulM'\xrightarrow{\es f\;}\ulM\xrightarrow{\es g\;}\ulM''\longto0$ of $A$-motives whose underlying sequence of $A_\BC$-modules is exact. Then $f$ (respectively $g)$ is called an \emph{admissible monomorphism} (respectively \emph{admissible epimorphism}).

Indeed, this means that $f$ is the kernel of $g$ and $g$ is the cokernel of $f$ in $\AMotCat$, that every canonical split sequence $0\to\ulM'\to\ulM'\oplus\ulM''\to\ulM''\to0$ lies in $E$, that $E$ is closed under isomorphisms, pullbacks via morphisms $\ulN''\to\ulM''$ and pushout via morphisms $\ulM'\to\ulN'$, and that the composition of admissible monomorphisms is an admissible monomorphism and the composition of admissible epimorphisms is an admissible epimorphism. All this is straight forward to prove.

Moreover, with the analogous definition of $E$, also the subcategories of $\AMotCat$ consisting of $A$-motives which are effective, respectively effective and finitely generated over $\BC\{\tau\}$, are exact.
\end{remark}

\begin{example}\label{ExampleCHMotive}
An effective $A$-motive of rank $1$ with $\stau_M(\ssigma^\ast M)=J\cdot M$ is called a \emph{Carlitz-Hayes $A$-motive}. It has virtual dimension $1$.
Carlitz-Hayes $A$-motives can be constructed as follows. Let $P\in C_\BC$ be a ($\BC$-valued) point whose projection onto $C$ is the point $\infty\in C$. (Under our assumption $\infty\in C(\BF_q)$ there is a unique such point $P$.) The divisor $({\rm V}(J))-(P)$ on $C_\BC$ has degree zero and induces a line bundle $\CO\bigl(({\rm V}(J))-(P)\bigr)$. Since the endomorphism $\id-\Frob_\sq$ of the abelian variety $\Pic^0_{C/\BF_\sq}$ is surjective, there is a line bundle $\CL$ of degree zero on $C_\BC$ with $\CO\bigl(({\rm V}(J))-(P)\bigr)=(\id-\Frob_\sq)(\CL)=\CL\otimes\ssigma^\ast\CL\dual$ in $\Pic^0_{C/\BF_\sq}(\BC)$. The $A_{\BC}$-module $M:=\Gamma(\Spec A_{\BC},\CL)$ is locally free of rank one and the isomorphism $\ssigma^\ast\CL\cong\CL\otimes\CO\bigl((P)-({\rm V}(J))\bigr)$ of line bundles yields an isomorphism $\stau_M\colon \ssigma^\ast M[J^{-1}]\isoto M[J^{-1}]$ with $\stau_M(\ssigma^\ast M)=J\cdot M$. So $\ulM=(M,\stau_M)$ is a Carlitz-Hayes $A$-motive.

If $\ulM$ is a Carlitz-Hayes $A$-motive and $\ulM'$ is any $A$-motive of rank $1$, then $\stau_{M'}(\ssigma^\ast M')=J^d\cdot M'$ for a uniquely determined integer $d$ by the elementary divisor theorem. Under our assumption that $\infty$ is $\BF_q$-rational, we claim that $\ulM'$ is isogenous to $\ulM^{\otimes d}$. So in particular all Carlitz-Hayes $A$-motives are isomorphic in the category $\AMotCatIsog$. Namely, consider the $A$-motive $\ulN:=\ulM'\otimes(\ulM^{\otimes d})\dual$ of rank one. It satisfies $\stau_N\colon \ssigma^\ast N\isoto N$ and its $\stau$-invariants $N_0:=\{f\in N: \stau_N(\ssigma^\ast f)=f\}$ form a locally free $A$-module of rank one with $\ulN\cong N_0\otimes_A\UOne(0)$. Indeed, one can extend $N$ to a locally free sheaf $\olN$ on $C_\BC$ of degree zero and, by reasons of degree, $\tau_N$ will extend to an isomorphism $\tau_\olN\colon\sigma^*\olN\isoto\olN$ at the one missing point $\infty_\BC=C_\BC\setminus\Spec A_\BC$. This means that the element $\olN\in\Pic^0_{C/\BF_q}(\BC)$ arises from an $\BF_q$-rational point $\olN_0$ of $\Pic^0_{C/\BF_q}$. It follows that $\ulN\cong N_0\otimes_A\UOne(0)$ for $N_0:=\Gamma(\Spec A,\olN_0)$ as claimed. Now the $A$-module $N_0$ is isomorphic to an ideal of $A$ which we again denote by $N_0$. Tensoring the inclusion $N_0\into A$ with $\ulM^{\otimes d}$ yields the desired isogeny $\ulM'\cong N_0\otimes_A\ulM^{\otimes d}\into \ulM^{\otimes d}$. 

We may therefore denote any Carlitz-Hayes $A$-motive by $\UOne(1)$. We also define $\UOne(n):=\UOne(1)^{\otimes n}$ for $n\ge0$ and $\UOne(n)=\UOne(-n)\dual$ for $n\le0$. Then $\dim\UOne(n)=n$. In the special case where $C=\BP^1_{\BF_q},A=\BF_q[t]$ and $\theta:=\charmorph(t)\in\BC$, all Carlitz-Hayes $A$-motives are already in $\AMotCat$ isomorphic to the \emph{Carlitz $t$-motive} $\ulM=(M,\tau_M)$ with $M=\BC[t]$ and $\tau_M=t-\theta$, because in this case the $A$-module $N_0$ is free and isomorphic to $A$.
\end{example}

\begin{remark}\label{Rem3.3}
(a) Every $A$-motive is isomorphic to the tensor product of an effective $A$-motive and a power of a Carlitz-Hayes $A$-motive.
In fact, if $\ulM$ is an $A$-motive with $\stau_{M}(\ssigma^\ast M)\subset J^{-d}\cdot M$ Then $\ulM':=\ulM\otimes\UOne(1)^{\otimes d}$ satisfies $\stau_{M'}(\ssigma^\ast M')\subset M'$; hence, $\ulM'$ is effective and $\ulM\cong\ulM'\otimes\UOne(1)^{\otimes -d}$. Note that $\rk\ulM'=\rk\ulM$ and $\dim\ulM'=\dim\ulM+d\cdot\rk\ulM$.

\medskip
\noindent
(b) This implies that for $A=\BF_\sq[t]$ the category $\AMotCat$ is equivalent to Taelman's  category $t\CM_\BC$ of $t$-motives \cite[Def.~2.3.2]{Taelman} and $\AMotCatIsog$ is equivalent to Taelman's category $t\CM^{\rm o}_\BC$ of $t$-motives up to isogeny \cite[\S3]{Taelman}. Indeed, Taelman defines $t\CM_\BC$ as the category of effective $A$-motives with the formally adjoined inverse of a Carlitz-Hayes $A$-motive. 

\medskip
\noindent
(c) Let us explain why the set of morphisms $\Hom_{\AMotCat}(\ulM,\ulN)$ between $A$-motives $\ulM$ and $\ulN$ is a finite projective $A$-module of rank at most $(\rk\ulM)\cdot(\rk\ulN)$. By (a) we may write $\ulM\cong\ulM'\otimes\UOne(1)^{\otimes -d}$ and $\ulN\cong\ulN'\otimes\UOne(1)^{\otimes -d}$ for effective $A$-motives $\ulM'$ and $\ulN'$. Then $\Hom_{\AMotCat}(\ulM,\ulN)\cong\Hom_{\AMotCat}(\ulM',\ulN')$ and for the latter the statement was proved by Anderson~\cite[Corollary~1.7.2]{Anderson86}.
\end{remark}

\subsection{Purity and mixedness}\label{SectAMotPurity}

We fix a uniformizing parameter $z\in Q=\BF_q(C)$ of $C$ at $\infty$. For simplicity of the exposition we assume that $\infty\in C(\BF_q)$. The main results we present here hold, and are in fact proved in \cite{PinkHodge,HartlPink2}, without this assumption. The assumption implies that there is a unique point on $C_\BC$ above $\infty\in C$, which we call $\infty_{\SSC\BC}$. The completion of the local ring of $C_\BC$ at $\infty_{\SSC\BC}$ is canonically isomorphic to $\BC\dbl z\dbr$.

\begin{definition}\label{Def2.1}
\begin{enumerate}
\item \label{Def2.1_a}
An $A$-motive $\ulM=(M,\tau_M)$ is called \emph{pure} if $M\otimes_{A_{\BC}}\BC\dpl z\dpr$ contains a $\BC\dbl z\dbr$-lattice $M_\infty$ such that for some integers $d,r$ with $r>0$ the map 
\[
\tau_M^r\;:=\;\tau_M\circ\sigma^*(\tau_M)\circ\ldots\circ\sigma^{r-1*}(\tau_M)\colon\; \sigma^{r*} M\otimes_{A_\BC}\Quot(A_\BC)\;\isoto\; M\otimes_{A_\BC}\Quot(A_\BC)
\]
induces an isomorphism $z^d\tau_M^r\colon \sigma^{r*}M_\infty\isoto M_\infty$. Then the \emph{weight} of $\ulM$ is defined as $\weight\ulM=\frac{d}{r}$. 
\item
An $A$-motive $\ulM$ is called \emph{mixed} if it possesses an increasing \emph{weight filtration} by saturated $A$-sub-motives $W_{\mu\,}\ulM$ for $\mu\in\BQ$ (i.e. $W_\mu M\subset M$ is a saturated $A_{\BC}$-submodule) such that all graded pieces $\Gr_\mu^W\ulM:=W_{\mu\,}\ulM/\bigcup_{\mu'<\mu}W_{\mu'}\ulM$ are pure $A$-motives of weight $\mu$ and $\sum_{\mu\in\BQ}\rk\Gr_\mu^W\ulM=\rk\ulM$. 
\item The full subcategory of $\AMotCat$ consisting of mixed $A$-motives is denoted $\AMMotCat$. The full subcategory of $\AMotCatIsog$ consisting of mixed $A$-motives is denoted $\AMMotCatIsog$. 
\end{enumerate}
\end{definition}

\begin{example}\label{ExCarlitz}
For $A=\BF_q[t]$ the Carlitz $t$-motive $\ulM=(\BC[t],\stau_M=t-\theta)$ is pure of weight $1$ with $M_\infty=\BC\dbl z\dbr$ on which $z\stau_M=1-\theta z$ is an isomorphism, where $z=\tfrac{1}{t}$. For general $A$, any Carlitz-Hayes motive (Example~\ref{ExampleCHMotive}) is pure of weight $1$ by Proposition~\ref{PropWeights}\ref{PropWeights_c} below, because $z\tau_M\colon\sigma^*M_\infty\to M_\infty$ is an isomorphism for the lattice $M_\infty:=\CL\otimes_{\CO_{C_\BC}}\BC\dbl z\dbr$.
\end{example}

\begin{remark}\label{RemAMotWts}
(a) The \emph{weights of $\ulM$} are the jumps of the weight filtration; that is, those real numbers $\mu$ for which 
\[
\TS\bigcup_{\mu'<\mu}W_{\mu'}\ulM \;\subsetneq\;\bigcap_{\tilde\mu>\mu}W_{\tilde\mu\,}\ulM\,.
\]
The condition $\sum_{\mu\in\BQ}\rk\Gr_\mu^W\ulM=\rk\ulM$ is equivalent to the conditions that all weights lie in $\BQ$, that $W_{\mu\,}\ulM \;=\;\bigcap_{\tilde\mu>\mu}W_{\tilde\mu\,}\ulM$ for all $\mu\in\BQ$, that $W_{\mu\,}\ulM=(0)$ for $\mu\ll0$, and that $W_{\mu\,}\ulM=\ulM$ for $\mu\gg0$; compare Remark~\ref{RemQFiltr}.

\medskip\noindent
(b) Every pure $A$-motive of weight $\mu$ is also mixed with $W_{\mu'}\ulM=(0)$ for $\mu'<\mu$, and $W_{\mu'}\ulM=\ulM$ for $\mu'\ge\mu$, and $\Gr_\mu^W\ulM=\ulM$.
\end{remark}

\medskip

To explain this definition we use the notion of \emph{$z$-isocrystals} over $\BC$; see \cite[Definition~5.1]{HartlKim}. These are defined to be pairs $\wh\ulM=(\wh{M},\tau_{\wh{M}})$ consisting of a finite dimensional $\BC\dpl z\dpr$-vector space $\wh{M}$ together with a $\BC\dpl z\dpr$-isomorphism $\tau_{\wh{M}}\colon\ssigma^*\wh{M}\isoto\wh{M}$. They are also called \emph{Dieudonn\'e-$\BF_q\dpl z\dpr$-modules} in \cite[\S\,2.4]{Laumon} and \emph{local isoshtukas} in \cite[\S\,8]{BH1}. Some of the following results were proved by Taelman~\cite{Taelman, Taelman09}.

\begin{proposition}\label{PropPure}
Let $\ulM=(M,\tau_M)$ be an $A$-motive and consider the \emph{$z$-isocrystal} $\ulHM\,:=\,\ulM\otimes_{A_{\BC}}\BC\dpl z\dpr\,=\,\bigl(M\otimes_{A_{\BC}}\BC\dpl z\dpr,\tau_M\otimes\id\bigr)$. Then $\ulHM$ is isomorphic to $\bigoplus_i\ulHM_{d_i,r_i}$ where for $d,r\in\BZ,r>0,(d,r)=1$ and $m:=\lceil\frac{d}{r}\rceil$ we set 
\begin{equation}\label{EqStandardZIsocrystal}
\ulHM_{d,r}\;:=\;\Bigl(\BC\dpl z\dpr^{\oplus r},\tau=\tau_{d,r}:=\left( \raisebox{6.2ex}{$
\xymatrix @C=0pc @R=0.3pc {
0 \ar@{.}[drdrdr] & & z^{-m}\ar@{.}[dr] & \\
& & & z^{-m} \\
z^{1-m} \ar@{.}[dr] & & & \\
 & z^{1-m}  & & 0\\
}$}
\right)\Bigr)
\end{equation}
and where in the matrix the term $z^{1-m}$ occurs exactly $mr-d$ times. In particular,
\begin{enumerate}
\item \label{PropPure_a}
$\ulM$ is pure of weight $\mu$ if and only if $\frac{d_i}{r_i}=\mu$ for all $i$.
\item \label{PropPure_b}
$\ulM$ is mixed if and only if the filtration $\DS W_{\mu\,}\ulHM:=\bigoplus_{\frac{d_i}{r_i}\le\mu}\ulHM_{d_i,r_i}$ comes from a filtration of $\ulM$ by saturated $A$-sub-motives $\wt W_{\mu\,}\ulM\subset\ulM$ with $W_{\mu\,}\ulHM=(\wt W_{\mu\,}\ulM)\otimes_{A_{\BC}}\BC\dpl z\dpr$. In this case the filtration $\wt W_{\mu\,}\ulM$ equals the weight filtration $W_{\mu\,}\ulM$ of $\ulM$ and the $\tfrac{d_i}{r_i}$ are the weights of $\ulM$. In particular, the weight filtration of a mixed $A$-motive $\ulM$ is uniquely determined by $\ulM$.
\item \label{PropPure_c}
Any $A$-sub-motive $\ulM'\into\ulM$ and $A$-quotient motive $f\colon \ulM\onto\ulM''$ of a pure (mixed) $A$-motive $\ulM$ is itself pure (mixed) of the same weight(s), (by letting $W_{\mu\,}\ulM':=\ulM'\cap W_{\mu\,}\ulM$, and letting $W_{\mu\,}\ulM''$ be the saturation of $f(W_{\mu\,}\ulM)$ inside $\ulM''$, if $\ulM$ is mixed).
\item \label{PropPure_f}
Any $A$-motive which is isomorphic in $\AMotCatIsog$ to a pure (mixed) $A$-motive is itself pure (mixed).
\item\label{PropPure_h}
The weight of a pure $A$-motive $\ulM$ is $\weight\ulM=(\dim\ulM)/(\rk\ulM)$. The tensor product of two pure $A$-motives $\ulM$ and $\ulN$ is again pure of weight $(\weight\ulM)+(\weight\ulN)$. 
\item \label{PropPure_g}
The category $\AMMotCatIsog$ is a full $Q$-linear (non-neutral) Tannakian subcategory of $\AMotCatIsog$, and in particular, a rigid abelian tensor category.
\item \label{PropPure_d}
Any morphism $f\colon \ulM'\to\ulM$ between mixed $A$-motives satisfies $f(W_{\mu\,}\ulM')\subset W_{\mu\,}\ulM$. More precisely, the saturation of $f(W_{\mu\,}\ulM')$ inside $f(\ulM')$ equals $f(\ulM')\cap W_{\mu\,}\ulM$.
\end{enumerate}
\end{proposition}

\noindent
{\it Remark.} We do not know whether in {\it \ref{PropPure_d}} the submodule $f(W_{\mu\,}\ulM')\subset f(\ulM')$ is always saturated, that is, whether the equality $f(W_{\mu\,}\ulM')=f(\ulM')\cap W_{\mu\,}\ulM$ always holds.

\begin{proof}
The fact that over the algebraically closed field $\BC$ any $z$-isocrystal is isomorphic to a direct sum of standard ones is proved in \cite[Theorem~2.4.5]{Laumon}. It is analogous to the Dieudonn\'e-Manin classification of $F$-isocrystals over an algebraically closed field of positive characteristic \cite{Manin}. That the standard $z$-isocrystals in \cite{Laumon} are isomorphic to our standard ones $\ulHM_{d,r}$ follows by an elementary computation. 

\medskip\noindent
\ref{PropPure_a} If $\ulHM\cong\ulHM_{d,r}^{\oplus(\rk\ulM)/r}$ with $\mu=\frac{d}{r}$ and $(d,r)=1$, we can take for $M_\infty$ the tautological $\BC\dbl z\dbr$-lattice $\BC\dbl z\dbr^{\oplus\rk\ulM}$ inside $\ulHM_{d,r}^{\oplus(\rk\ulM)/r}$ to see that $\ulM$ is pure. Conversely, if there is an $i$ with $\mu=\frac{d}{r}\ne\frac{d_i}{r_i}$, then $z^{r_id}\tau_M^{r_ir}=z^{r_id}z^{-d_ir}$ cannot be an isomorphism for any $\BC\dbl z\dbr$-lattice $M_\infty$ in $\ulHM$. So $\ulM$ is not pure of weight $\mu$. (Compare \cite[Proposition~5.1.4]{Taelman}.)

\medskip\noindent
\ref{PropPure_b} If $\ulM$ has a weight filtration $W_{\mu\,}\ulM\subset\ulM$ with respect to which it is mixed, then \ref{PropPure_a} implies that $(\Gr_\mu^W\ulM)\otimes_{A_{\BC}}\BC\dpl z\dpr\cong\ulHM_{d,r}^{\oplus(\rk\Gr_\mu^W\ulM)/r}$ for $\mu=\frac{d}{r}$ with $(d,r)=1$. Since the category of $z$-isocrystals is semi-simple by \cite[Theorem~2.4.5]{Laumon} the sequences
\[
\xymatrix {
0\ar[r] & (\bigcup_{\mu'<\mu}W_{\mu'}\ulM)\otimes_{A_{\BC}}\BC\dpl z\dpr \ar[r] & (W_{\mu\,}\ulM)\otimes_{A_{\BC}}\BC\dpl z\dpr \ar[r] & (\Gr^W_{\mu\,}\ulM)\otimes_{A_{\BC}}\BC\dpl z\dpr \ar[r] & 0
}
\] 
split canonically for all $\mu$. This inductively yields $\ulHM=\bigoplus_\mu(\Gr^W_{\mu\,}\ulM)\otimes_{A_{\BC}}\BC\dpl z\dpr$ and $(W_{\mu\,}\ulM)\otimes_{A_{\BC}}\BC\dpl z\dpr=\bigoplus_{\mu'\le\mu}(\Gr^W_{\mu'}\ulM)\otimes_{A_{\BC}}\BC\dpl z\dpr=W_{\mu\,}\ulHM$. So the filtration $W_{\mu\,}\ulHM$ comes from $W_{\mu\,}\ulM$.

Conversely, if there is a filtration $\wt W_{\mu\,}\ulM\subset\ulM$ satisfying $W_{\mu\,}\ulHM=(\wt W_{\mu\,}\ulM)\otimes_{A_{\BC}}\BC\dpl z\dpr$ then $(\Gr_\mu^{\wt W}\ulM)\otimes_{A_{\BC}}\BC\dpl z\dpr\cong\bigoplus_{\frac{d_i}{r_i}=\mu}\ulHM_{d_i,r_i}$. So $\Gr_\mu^{\wt W}\ulM$ is pure of weight $\mu$ by \ref{PropPure_a} and $\ulM$ is mixed with $\wt W_{\mu\,}\ulM$ as a weight filtration. 
In this case $\wt W_{\mu\,}\ulM\subset\ulM\cap W_{\mu\,}\ulHM=:\ulN$ are two saturated $A$-sub-motives of $\ulM$. Since 
\[
\wt W_{\mu\,}\ulM\otimes_{A_{\BC}}\BC\dpl z\dpr\;\subset\;\ulN\otimes_{A_{\BC}}\BC\dpl z\dpr\;\subset\;W_{\mu\,}\ulHM\;\subset\;\wt W_{\mu\,}\ulM\otimes_{A_{\BC}}\BC\dpl z\dpr
\]
they have the same rank. This implies $\wt W_{\mu\,}\ulM=\ulM\cap W_{\mu\,}\ulHM$. Thus, the weight filtration $W_{\mu\,}\ulM$ is uniquely determined by $\ulM$ if $\ulM$ is mixed.

\medskip\noindent
\ref{PropPure_c} If $\ulM$ is pure of weight $\mu=\frac{d}{r}$ with $(d,r)=1$, we see that $\ulHM\cong\ulHM_{d,r}^{\oplus(\rk\ulM)/r}$ and $\ulHM'\subset\ulHM$ and $\ulHM\onto\ulHM''$. By \cite[Proposition 1.2.11]{HartlPSp} also $\ulHM'\cong\ulHM_{d,r}^{\oplus(\rk\ulM')/r}$ and $\ulHM''\cong\ulHM_{d,r}^{\oplus(\rk\ulM'')/r}$. So $\ulM'$ and $\ulM''$ are likewise pure of weight $\mu$ by \ref{PropPure_a}. If $\ulM$ is mixed, we set $W_{\mu\,}\ulM':=\ulM'\cap W_{\mu\,}\ulM$ and we let $W_{\mu\,}\ulM''$ be the saturation of $f(W_{\mu\,}\ulM)$ inside $\ulM''$. Then $W_{\mu\,}\ulM'\subset\ulM'$ is a saturated $A$-sub-motive with $\Gr^W_{\mu\,}\ulM'\subset\Gr^W_{\mu\,}\ulM$. Thus the graded piece $\Gr^W_{\mu\,}\ulM'$ is pure by the above and $\ulM'$ is mixed. Also $\Gr^W_\mu\ulM\to \Gr^W_\mu\ulM''$ has torsion cokernel because $W_\mu\ulM\to W_\mu\ulM''$ has, and hence, the $z$-isocrystal $\wh{\Gr^W_\mu\ulM''}$ is a quotient of $\wh{\Gr^W_\mu\ulM}$ and pure by \cite[Proposition~1.2.11]{HartlPSp}. Therefore, $\Gr^W_{\mu\,}\ulM''$ is pure by \ref{PropPure_a} and $\ulM''$ is mixed.

\medskip\noindent
\ref{PropPure_f} If $f\colon \ulM'\to\ulM$ is an isomorphism in $\AMotCatIsog$ and $\ulM$ is pure (mixed), we can multiply $f$ by a non-zero element of $A$ and assume that $f$ is a morphism in $\AMotCat$. Then $f$ realizes $\ulM'$ as an $A$-sub-motive of $\ulM$ and $\ulM'$ is pure (mixed) by \ref{PropPure_c}.

\medskip\noindent
\ref{PropPure_h} If $M_\infty\subset\wh\ulM$ and $N_\infty\subset\wh\ulN$ are $\BC\dbl z\dbr$-lattices on which the maps $z^d\tau_M^r\colon\sigma^{r*}M_\infty\isoto M_\infty$ and \mbox{$z^{d'}\tau_N^{r'}\colon\sigma^{r'*}N_\infty\isoto N_\infty$} are isomorphisms, then $M_\infty\otimes_{\BC\dbl z\dbr}N_\infty\subset\wh\ulM\otimes_{\BC\dpl z\dpr}\wh\ulN$ is a $\BC\dbl z\dbr$-lattice with 
\begin{eqnarray*}
z^{dr'+rd'}\,\tau_{M\otimes N}^{rr'}\bigl(\sigma^{rr'*}(M_\infty\otimes_{\BC\dbl z\dbr} N_\infty)\bigr) & = & z^{dr'}\tau_{M}^{rr'}(\sigma^{rr'*}M_\infty)\otimes_{\BC\dbl z\dbr} z^{rd'}\tau_{N}^{rr'}(\sigma^{rr'*}N_\infty)\\[2mm]
 & = &  M_\infty\otimes_{\BC\dbl z\dbr} N_\infty\,.
\end{eqnarray*}
So $\ulM\otimes\ulN$ is pure of weight $\tfrac{dr'+rd'}{rr'}=\tfrac{d}{r}+\tfrac{d'}{r'}=(\weight\ulM)+(\weight\ulN)$. The formula $\weight\ulM=(\dim\ulM)/(\rk\ulM)$ follows from \cite[Lemma~1.10.1]{Anderson86} if $\ulM$ is effective and from Remark~\ref{Rem3.3} in the general case.

\medskip\noindent
\ref{PropPure_g} Clearly a direct sum $\ulM\oplus\ulN$ of mixed $A$-motives is mixed with the direct sum weight filtration $W_\mu(\ulM\oplus\ulN)=(W_{\mu\,}\ulM)\oplus(W_{\mu\,}\ulN)$. If $f$ is a morphism in $\AMMotCatIsog$, we may multiply it with a non-zero element $a\in A$ to obtain a morphism in $\AMMotCat$. Then its kernel, cokernel, image and coimage in $\AMotCat$ from Proposition~\ref{PropImageAMotive} again belong to $\AMMotCat$ by \ref{PropPure_c}. This shows that $\AMMotCatIsog$ is an abelian subcategory of $\AMotCatIsog$. Moreover, $\AMMotCatIsog$ is strictly full by \ref{PropPure_f} and contains $\UOne(0)$, which is pure of weight $0$. Also the tensor product of two mixed $A$-motives $\ulM$ and $\ulN$, equipped with the weight filtration $W_\lambda(\ulM\otimes\ulN):=\sum_{\mu+\nu=\lambda}W_\mu\ulM\otimes W_\nu\ulN$ is again mixed, because $\Gr^W_\lambda(\ulM\otimes\ulN)$ is a quotient of the pure $A$-motive $\bigoplus_{\mu+\nu=\lambda}\Gr^W_\mu\ulM\otimes \Gr^W_\nu\ulN$ of weight $\lambda$ and therefore is itself pure of weight $\lambda$ by \ref{PropPure_c}. Furthermore, the dual $\ulM\dual$ of a mixed $A$-motive $\ulM$, equipped with the weight filtration $W_\mu\ulM\dual:=\{\, m\dual\in M\dual\colon m\dual(W_\lambda\ulM)=0 \es\forall\,\lambda<-\mu\,\}$ is mixed. Indeed, one easily computes that $\Gr^W_\lambda(\ulM\dual)=(\Gr^W_{-\lambda}\ulM)\dual$ and the latter is pure of weight $\lambda$ by \ref{PropPure_a} and the fact that $(\ulHM_{d,r})\dual\cong\ulHM_{-d,r}$ in the category of $z$-isocrystals. So also the internal hom $\CHom(\ulM,\ulN)\cong\ulN\otimes\ulM\dual$ of two mixed (pure) $A$-motives $\ulM$ and $\ulN$ is mixed (pure).

\medskip\noindent
\ref{PropPure_d} By \ref{PropPure_c} the image $A$-motive $f(\ulM')\subset\ulM$ is mixed both as a sub-motive of $\ulM$ with $W_\mu f(\ulM')=f(\ulM')\cap W_{\mu\,}\ulM$ and as a quotient motive of $\ulM$ with $W_\mu f(\ulM')$ being the saturation of $f(W_{\mu\,}\ulM')$. By \ref{PropPure_b} both filtrations coincide, so $f(W_{\mu\,}\ulM')\subset W_{\mu\,}\ulM$.
\end{proof}

\begin{remark}\label{RemQuillenAMMot}
The category $\AMMotCat$ is an exact subcategory of $\AMotCat$ in the sense of Quillen~\cite[\S2]{Quillen} if one takes as class $E$ of \emph{exact} sequences $0\longto\ulM'\xrightarrow{\es f\;}\ulM\xrightarrow{\es g\;}\ulM''\longto0$ of mixed $A$-motives those which are exact in $\AMotCat$, that is whose underlying sequence of $A_\BC$-modules is exact; see Remark~\ref{RemQuillenAMot}(b). Indeed, the only statement that does not directly follow from Remark~\ref{RemQuillenAMot}(b) is that $E$ is closed under pullbacks via morphisms $h''\colon\ulN''\to\ulM''$ and pushout via morphisms $h'\colon\ulM'\to\ulN'$. For this one has to show that the pullback $(h'')^*\ulM=\ker\bigl((g,-h'')\colon\ulM\oplus\ulN''\to\ulM''\bigr)$ and pushout $h'_*(\ulM)=\coker\bigl((f,h')\colon\ulM'\to\ulM\oplus\ulN'\bigr)$ in the category $\AMotCat$ are mixed. This follows from Proposition~\ref{PropPure}\ref{PropPure_c} and \ref{PropPure_g}, because $(h'')^*\ulM$ is a submotive of $\ulM\oplus\ulN''$, and $h'_*(\ulM)$ is a quotient of $\ulM\oplus\ulN'$.
\end{remark}

\begin{example}\label{ExNotMixed}
Not every $A$-motive is mixed. For example, let $A=\BF_q[t]$, $z=\frac{1}{t}$, $\theta=\charmorph(t)\in\BC$, and $M=A_{\BC}^{\oplus2}$ with $\tau_M=\Phi:=\left( \begin{array}{cc}(t-\theta)^2&1\\0&t-\theta \end{array}\right)$. Then there is an exact sequence of $A$-motives $0\to\ulM'\to\ulM\to\ulM''\to0$ with $\ulM'=(A_{\BC},\tau_{M'}=(t-\theta)^2)$ and $\ulM''=(A_{\BC},\tau_{M''}=(t-\theta))$. Since $\ulHM'\cong\ulHM_{2,1}$ and $\ulHM''\cong\ulHM_{1,1}$ and any sequence of $z$-isocrystals splits, we see that $\ulHM\cong\ulHM_{1,1}\oplus\ulHM_{2,1}$. Hence, if $\ulM$ were mixed, its weights must be $1$ and $2$. But $\ulM$ contains no pure $A$-sub-motive of weight one. Indeed such a sub-motive would be isomorphic to $\ulM''$ and generated by a non-zero vector $\left(\begin{smallmatrix}u\\ v\end{smallmatrix}\right)\in A_{\BC}^2$ with $\Phi\cdot\left(\begin{smallmatrix}\sigma^\ast(u)\\\sigma^\ast(v)\end{smallmatrix}\right)=(t-\theta)\cdot\left(\begin{smallmatrix}u\\v\end{smallmatrix}\right)$. This amounts to $(t-\theta)^2\sigma^\ast(u)+\sigma^\ast(v)=(t-\theta)u$ and $\sigma^\ast(v)=v$, whence $v\in A$. Since $A\to\BC$, $t\mapsto\theta$, is injective $(t-\theta)\nmid v=\sigma^\ast(v)$ in $A_{\BC}$ if $v\ne0$. We conclude that $v=0$ and $(t-\theta)\sigma^\ast(u)=u$, which is impossible by reasons of degree. Hence, $\ulM$ is not mixed. The reason is of course that $\ulM$ is an extension of two pure $A$-motives but ``in the wrong direction''.
\end{example}

\begin{proposition}\label{PropWeights}
Let $\ulM=(M,\tau_M)$ be an $A$-motive whose $z$-isocrystal $\ulHM\,:=\,\ulM\otimes_{A_{\BC}}\BC\dpl z\dpr$ is isomorphic to $\bigoplus_i\ulHM_{d_i,r_i}$ for $d_i,r_i\in\BZ$ with $r_i>0$ and $(d_i,r_i)=1$ for all $i$; see Proposition~\ref{PropPure}.
\begin{enumerate}
\item \label{PropWeights_a} If $\ulM$ is effective and $M$ is a finitely generated module over the skew-polynomial ring $\BC\{\tau\}$, where $\tau$ acts on $M$ through $m\mapsto \tau_M(\sigma^\ast m)$, then $d_i>0$ for all $i$.
\item \label{PropWeights_b}
If $d_i/r_i\le n$ for all $i$ and $\tau_M(\sigma^*M)\subset J^{-d}M$ for $d\in\BZ$, then $\ulM$ extends to a locally free sheaf $\olM$ on $C_\BC$ with $\tau\colon \sigma^\ast\olM\to\olM\bigl(n\cdot\infty_{\SSC\BC}+d\cdot\Var(J)\bigr)$, where the notation $\bigl(n\cdot\infty_{\SSC\BC}+d\cdot\Var(J)\bigr)$ means that we allow poles at $\infty_{\SSC\BC}$ of order less than or equal to $n$ and at $\Var(J)$ of order less than or equal to $d$. 
\item \label{PropWeights_c}
If $\ulM$ is  not necessarily effective, then $\ulM$ is pure of weight $\mu=\tfrac{d}{r}$ with $(d,r)=1$ if and only if $\ulM$ extends to a locally free sheaf $\olM$ on $C_\BC$ such that $z^d\tau_M^r$ is an isomorphism $\sigma^{r\ast}\olM_\infty\isoto\olM_\infty$ on the stalks at $\infty_{\SSC\BC}$.
\end{enumerate}
\end{proposition}

\begin{proof}
\ref{PropWeights_a} (compare \cite[Proposition~8]{Taelman09}) We may assume that $\tfrac{d_i}{r_i}\ge\tfrac{d_{i+1}}{r_{i+1}}$ for all $i$. By the explicit description of $\ulHM_{d_i,r_i}$ in \eqref{EqStandardZIsocrystal} there is a $\BC\dpl z\dpr$-basis $\CB$ of $\ulHM$ and an integer $s>0$ such that $(\tau_M\otimes\id)^s$ is a diagonal matrix with entries $z^{-sd_i/r_i}$ with respect to $\CB$. Assume that $d_1\le0$. Since $M$ is finitely generated as a $\BC\{\tau\}$-module there are finitely many elements $m_i\in M$ such that $M=\sum_{i,j\ge0}\BC\cdot\tau_M^{sj}(\sigma^{sj*}m_i)$. By definition of $\ulHM=\ulM\otimes_{A_\BC}\BC\dpl z\dpr$ the set $M[z]=M\otimes_{A_\BC}A_\BC[z]=\sum_{i,j,k\ge0}\BC\cdot z^k\tau_M^{sj}(\sigma^{sj*}m_i)$ is $z$-adically dense in $\ulHM$. We write $m_i$ with respect to the basis $\CB$ as a vector $(m_{i,1},\ldots,m_{i,r})^T\in\BC\dpl z\dpr^{\oplus r}$. Then the first coordinates of the elements of $M[z]$ have the form $\sum_{i,j,k\ge0}b_{i,j,k} z^k\tau_M^{sj}(\sigma^{sj*}m_{i,1})\;=\;\sum_{i,j,k\ge0}b_{i,j,k} z^{k-sjd_1/r_1}\sigma^{sj*}m_{i,1}$ for $b_{i,j,k}\in\BC$. Since $k-sjd_1/r_1\ge0$ for all $j$ and $k$, all these terms lie in $z^N\BC\dbl z\dbr$ for a suitable $N\in\BZ$. In particular, elements of $\ulHM$ with first coordinate outside $z^N\BC\dbl z\dbr$ can not belong to the $z$-adic closure of $M[z]$. This contradiction shows that our assumption was false and $d_1>0$.

\medskip\noindent
\ref{PropWeights_c} If the described extension $\olM$ of $\ulM$ exists, then $\olM\otimes_{\CO_{C_\BC}}\BC\dbl z\dbr$ is a $\BC\dbl z\dbr$-lattice inside $\ulHM$ on which $z^d\tau_M^r$ is an isomorphism and $\ulM$ is pure of weight $\tfrac{d}{r}$ by Definition~\ref{Def2.1}\ref{Def2.1_a}.

\medskip\noindent
We prove \ref{PropWeights_b} and the remaining implication of \ref{PropWeights_c}. Since $m_i:=\lceil\frac{d_i}{r_i}\rceil\le n$ for all $i$, we can define $\olM$ by requiring that $\olM\otimes_{\CO_{C_\BC}}\BC\dbl z\dbr$ is equal to the sum of the tautological $\BC\dbl z\dbr$-lattices $\BC\dbl z\dbr^{\oplus r_i}$ inside $\ulHM_{d_i,r_i}$ in \eqref{EqStandardZIsocrystal}. Then $\olM$ has the desired properties. 
\end{proof}

\begin{proposition}\label{PropTau_MInvZDense}
If $d_i>0$ for all $i$, then the set $\bigcup_{n\in\BN_{>0}}\sdsigma^{n*}\tau_M^{-n}(M)$ is $z$-adically dense in $\wh M\,:=\,M\otimes_{A_{\BC}}\BC\dpl z\dpr$.
\end{proposition}

\begin{proof}
We choose a finite flat inclusion $\BF_q[t]\into A$ and set $\tilde z:=\tfrac{1}{t}$. Then $M$ is a finite (locally) free $\BC[t]$-module, say of rank $r$ and $\wh M=M\otimes_{\BC[t]}\BC\dpl\tilde z\dpr$. We choose a $\BC[t]$-basis $\CB$ of $M$. By Proposition~\ref{PropPure} there is a $\BC\dpl z\dpr$-isomorphism $\ulHM\cong\bigoplus_i\ulHM_{d_i,r_i}$. We let $\whCB$ be the $\BC\dpl\tilde z\dpr$-basis of $\wh M$ obtained from the standard basis of the $\ulHM_{d_i,r_i}$ given by \eqref{EqStandardZIsocrystal} and from the choice of a $\BC\dbl\tilde z\dbr$-basis of $\BC\dbl z\dbr$. The base change between $\whCB$ and $\CB$ is given by a matrix $U\in\GL_r\bigl(\BC\dpl\tilde z\dpr\bigr)$. There is an integer $N\ge0$ such that $U,U^{-1}\in\tilde z^{-N}\BC\dbl\tilde z\dbr^{r\times r}$. By our assumption $d_i>0$ and the explicit form of the $\ulHM_{d_i,r_i}$ from \eqref{EqStandardZIsocrystal} there is a positive integer $s$ such that the matrix $T$ representing $\tau_M^{-s}$ with respect to the basis $\whCB$ lies in $\tilde z^{2N+1}\BC\dbl\tilde z\dbr^{r\times r}$. Therefore, the matrix $\sigma^{s*}(U)TU^{-1}$ representing $\tau_M^{-s}$ with respect to the basis $\CB$ lies in $\tilde z\,\BC\dbl\tilde z\dbr^{r\times r}$. Now the proposition is a consequence of the following
\end{proof}

\begin{lemma}\label{LemmaTau_MInvZDense}
If $\tau_M^{-s}\in\tilde z\,\BC\dbl\tilde z\dbr^{r\times r}$. Then for all $x\in\BC\dpl\tilde z\dpr^{r}$ and for all $n\in\BN_0$ there exists a $y\in\BC[t]^{r}$ such that $x-\sdsigma^{sn*}\tau_M^{-sn}(y)\in \tilde z^{n+1}\BC\dbl\tilde z\dbr^r$.
\end{lemma}

\begin{proof}
Write $\tau_M^{sn}(\sigma^{sn*}(x))=\sum_i b_i\tilde z^i$ with $b_i\in\BC^r$ and set $y:=\sum_{i\le0}b_i t^{-i}\in\BC[t]^r$. Then $x-\sdsigma^{sn*}\tau_M^{-sn}(y)=\sdsigma^{sn*}\tau_M^{-sn}\bigl(\sum_{i>0}b_i\tilde z^i\bigr)\in \tilde z^{n+1}\BC\dbl\tilde z\dbr^r$ because $\sum_{i>0}b_i\tilde z^i\in \tilde z\,\BC\dbl\tilde z\dbr^r$.
\end{proof}

\medskip

\subsection{Uniformizability}\label{SectAMotUniformizability}

In order to define the notion of uniformizability (also called rigid analytic triviality) for $A$-motives we have to introduce some notation of rigid analytic geometry as in \cite{HartlPink1,HartlPink2}. See \cite{Bosch} or \cite{BGR} for a general introduction to rigid analytic geometry. With the curve $C_\BC$ and its open affine part $\dotC_\BC$ one can associate by \cite[\S9.3.4]{BGR} rigid analytic spaces $\FC_\BC:=(C_\BC)^\rig$ and $\dotFC_\BC:=(\dotC_\BC)^\rig=\FC_\BC\setminus\{\infty_{\SSC\BC}\}$ where, using our convention that the point $\infty\in C$ is $\BF_q$-rational, $\infty_{\SSC\BC}\in C_\BC$ is the unique point above $\infty\in C$. By construction, the underlying sets of $\FC_\BC$ and $\dotFC_\BC$ are the sets of $\BC$-valued points of $C_\BC$ and $\dotC_\BC$, respectively.
For any open rigid analytic subspace $\FU\subset\FC_\BC$ we let $\CO(\FU) := \Gamma(\FU,\CO_\FU)$ denote the ring of regular functions on~$\FU$. The endomorphism $\ssigma$ of $C_\BC$ induces endomorphisms of $\FC_\BC$ and $\dotFC_\BC$ which we denote by the same symbol $\ssigma$.

Let $\CO_\BC$ be the valuation ring of $\BC$ and let $\kappa_\BC$ be its residue field. By the valuative criterion of properness every point of $\FC_\BC=C_\BC(\BC)=C(\BC)$ extends uniquely to an $\CO_\BC$-valued point of $C$ and in the reduction gives rise to a $\kappa_\BC$-valued point of $C$. This yields a reduction map $red\colon \FC_\BC\to C(\kappa_\BC)$. The curve $C_{\kappa_\BC}$ is non-singular and, due to our convention $\infty\in C(\BF_q)$, the subscheme $\{\infty\}\times_{\Spec\BF_q}\Spec\kappa_\BC\subset C_{\kappa_\BC}$ consists of a single point which we call $\infty_{\kappa_\BC}$. So \cite[Proposition~2.2]{BL85} implies that the preimage $\Disc$ of $\infty_{\kappa_\BC}$ under $red$ is an open rigid analytic unit disc in $\FC_\BC$ around $\infty_{\SSC\BC}$. Without the convention $\infty\in C(\BF_q)$ the subscheme $\{\infty\}\times_{\Spec\BF_q}\Spec\kappa_\BC\subset C_{\kappa_\BC}$ decomposes into finitely many points and there is a corresponding disc for each one of them; see \cite[\S\,11]{HartlPink2}. Let further $\PDisc=\Disc\setminus\{\infty_{\SSC\BC}\}$ be the punctured open unit disc around $\infty_{\SSC\BC}$ in $\FC_\BC$. By \cite[Proposition~2.2]{BL85} both discs have $z$ as a coordinate. By Lemma~\ref{LemmaZ-Zeta} the power series ring $\BC\dbl z-\zeta\dbr$ is also canonically isomorphic to the completion of the local ring of $\FC_\BC$ at the closed point $\Var(J)$, respectively of $\Disc$ and $\PDisc$ at the point $\{z=\zeta\}\in\Disc$. The complement $\FC_\BC\setminus\Disc$ of $\Disc$ in $\FC_\BC$ equals the preimage of the open affine curve $\dotC_{\kappa_\BC}$ under the reduction map $red$ and is hence affinoid. 

For example, if $C=\BP^1_{\BF_q}$ and $A=\BF_q[t]$ we can give the following explicit description 
\begin{equation}\label{EqC<t>}
\CO(\FC_\BC\setminus\Disc)\;=\;\BC\langle t\rangle\;:=\;\bigl\{\,\sum_{i=0}^\infty b_i t^i\colon b_i\in\BC,\, \lim\limits_{i\to\infty}|b_i|= 0\,\bigr\}
\end{equation}
and $\FC_\BC\setminus\Disc=\Spm\BC\langle t\rangle$ is the closed unit disc inside $C(\BC)\setminus\{\infty_{\SSC\BC}\}=\BC$ on which the coordinate $t$ has absolute value less or equal to $1$. Also we can take $z=\tfrac{1}{t}$ as the coordinate on the disc $\Disc$. For general $C$ we may choose an element $a\in A\setminus\BF_q$ and consider the finite flat morphisms $\BF_q[t]\to A$ and $\BC[t]\to A_\BC$ which send $t$ to $a$. Then $\CO(\FC_\BC\setminus\Disc)=A_\BC\otimes_{\BC[t]}\BC\langle t\rangle$ and $\FC_\BC\setminus\Disc=\Spm(A_\BC\otimes_{\BC[t]}\BC\langle t\rangle)$.

The spaces $\dotFC_\BC$, $\Disc$ and $\PDisc$ are quasi-Stein spaces in the sense of Kiehl~\cite[\S2]{KiehlAB}. In particular, the global section functors are equivalences between the categories of locally free coherent sheaves on these spaces and the categories of finitely generated projective modules over their rings of global sections; see Gruson~\cite[Chapter~V, Theorem~1 and Remark on p.~85]{Gruson68}.

\begin{definition}\label{DefLambda}
For an $A$-motive $\ulM$, we define the \emph{$\tau$-invariants}
\[
\Lambda(\ulM) \;:=\; \bigl(M\otimes_{A_{\BC}}\CO(\FC_\BC\setminus\Disc)\bigr)^\tau\;:=\; \bigl\{\,m\in M\otimes_{A_{\BC}}\CO(\FC_\BC\setminus\Disc):\es \tau_M(\sigma^\ast m)=m\,\bigr\}\,.
\]
We also set $\Hodge^1(\ulM):=\Lambda(\ulM)\otimes_A Q$.
\end{definition}

Since the ring of $\ssigma^*$-invariants in $\CO(\FC_\BC\setminus\Disc)$ is equal to $A$, the set $\Lambda(\ulM)$ is an $A$-module. By \cite[Lemma~4.2(b)]{BoeckleHartl}, it is finite projective of rank at most equal to $\rank\ulM$. Therefore, also $\Hodge^1(\ulM)$ is a finite dimensional $Q$-vector space. 

\begin{definition}\label{DefUnifGlobal}
An $A$-motive $\ulM$ is called \emph{uniformizable} (or \emph{rigid analytically trivial}) if the natural homomorphism
\[
\TS h_\ulM\colon \Lambda(\ulM) \otimes_A \CO(\FC_\BC\setminus\Disc) \longto M\otimes_{A_{\BC}}\CO(\FC_\BC\setminus\Disc)\,,\quad \lambda\otimes f\mapsto f\cdot\lambda,
\]
is an isomorphism. The full subcategory of $\AMotCatIsog$ consisting of all uniformizable $A$-motives is denoted $\AUMotCatIsog$. The full subcategory of $\AMMotCatIsog$ consisting of all uniformizable mixed $A$-motives is denoted $\AMUMotCatIsog$.
\end{definition}

\begin{remark}\label{RemTaelman2}
\rm If $A=\BF_\sq[t]$, then the category $\AUMotCatIsog$ is canonically equivalent to the category $t\CM_{\rm a.t.}\open$ of Taelman~\cite[Def.~3.2.8]{Taelman} in view of Remark~\ref{Rem3.3}. 
\end{remark}

\begin{example}\label{ExCarlitzMotiveUnif}
(a) $\UOne(0)=(A_\BC,\tau=\id_{A_\BC})$ is uniformizable, because $\Lambda\bigl(\UOne(0)\bigr)=A$ and $A\otimes_A\CO(\FC_\BC\setminus\Disc)=A_\BC\otimes_{A_\BC}\CO(\FC_\BC\setminus\Disc)$.

\medskip\noindent
(b) Let $C=\BP^1_{\BF_q}$, $A=\BF_q[t]$, $z=\tfrac{1}{t}$, $\theta:=\charmorph(t)=\frac{1}{\zeta}\in\BC$. The \emph{Carlitz $t$-motive} $\ulM=(A_\BC,\tau_M=t-\theta)$ is uniformizable. Namely, we set $\ell_\zeta^{\SSC -}:=\prod_{i=0}^\infty(1-\zeta^{q^i}t)\in \CO(\dotFC_\BC)$ and choose an element $\eta\in\BC$ with $\eta^{q-1}=-\zeta$. Then $\eta\ell_\zeta^{\SSC -}\in\Lambda(\ulM)\setminus\{0\}$. Since $\eta\ell_\zeta^{\SSC -}$ has no zeroes outside $\Disc$ it generates the $\CO(\FC_\BC\setminus\Disc)$-module $M\otimes_{A_{\BC}}\CO(\FC_\BC\setminus\Disc)=\CO(\FC_\BC\setminus\Disc)$ and so $h_\ulM$ is an isomorphism and $\ulM$ is uniformizable.
\end{example}

The following criterion for uniformizability is well known.

\begin{lemma}\label{LemmaUniformizable}
Let $\ulM$ be an $A$-motive of rank $r$. 
\begin{enumerate}
 \item \label{LemmaUniformizableA} The homomorphism $h_\ulM$ is injective and satisfies $h_\ulM\circ(\id_{\Lambda(\ulM)}\otimes\id)=(\stau_M\otimes\id)\circ \ssigma^\ast h_\ulM$.
 \item \label{LemmaUniformizableB} $\ulM$ is uniformizable if and only if $\rank_A\Lambda(\ulM)=r$.
\end{enumerate}
\end{lemma}

\begin{proof}
Assertion \ref{LemmaUniformizableA} follows for example from \cite[Lemma~4.2(b)]{BoeckleHartl}, and assertion \ref{LemmaUniformizableB} from \cite[Lemma~4.2(c)]{BoeckleHartl}.
\end{proof}

\begin{lemma} \label{LemmaUniformizableBF_q[t]} 
Let $C=\BP^1_{\BF_q}$, $A=\BF_q[t]$, $A_\BC=\BC[t]$ and $\theta=\charmorph(t)$. Then $\CO(\FC_\BC\setminus\Disc\bigr)=\BC\langle t\rangle$; see \eqref{EqC<t>}. Let $\Phi=(\Phi_{ij})_{ij}\in\GL_r\bigl(\BC[t][\tfrac{1}{t-\theta}]\bigr)$ represent $\stau_M$ with respect to a $\BC[t]$-basis $\CB=(m_1,\ldots,m_r)$ of $M$, that is $\stau_M(\sigma^*m_j)=\sum_{i=1}^r\Phi_{ij}\,m_i$. Then $\ulM$ is uniformizable if and only if there is a matrix $\Psi\in\GL_{r}(\BC\langle t\rangle)$ such that
\[
\ssigma^\ast\Psi^T\;=\;\Psi^T\cdot\Phi\,,
\]
In that case, $\Psi$ is called a \emph{rigid analytic trivialization of $\Phi$}. It is uniquely determined up to multiplication on the right with a matrix in $\GL_r(\BF_q[t])$. The columns of $(\Psi^T)^{-1}$ are the coordinate vectors with respect to $\CB$ of an $\BF_q[t]$-basis $\CC$ of $\Lambda(\ulM)$. Moreover, with respect to the bases $\CC$ and $\CB$ the isomorphism $h_\ulM$ is represented by $(\Psi^T)^{-1}$.
\end{lemma}

\begin{remark}
Here $(\ldots)^T$ denotes the transpose matrix. The matrix $\Psi$ will turn out to be Anderson's \emph{scattering matrix} and this is the reason why we work with $\Psi^T$ here; see Remark~\ref{RemScattering} below.
\end{remark}

\begin{proof}[Proof of Lemma~\ref{LemmaUniformizableBF_q[t]}]
Assume that $\ulM$ is rigid analytically trivial and choose an $\BF_q[t]$-basis $\CC$ of $\Lambda(\ulM)$. Let $(\Psi^T)^{-1}$ be the matrix representing the isomorphism $h_\ulM\colon\Lambda(\ulM)\otimes_{\BF_q[t]}\BC\langle t\rangle\isoto M\otimes_{\BC[t]}\BC\langle t\rangle$ with respect to the bases $\CC$ and $\CB$. Then $\Phi\cdot\ssigma^*(\Psi^T)^{-1}=(\Psi^T)^{-1}$ and $\Psi\in\GL_{r}(\BC\langle t\rangle)$ is a rigid analytic trivialization. Conversely, if there is a rigid analytic trivialization $\Psi$, then the columns of $(\Psi^T)^{-1}$ provide a $\BC\langle t\rangle$-basis of $M\otimes_{\BC[t]}\BC\langle t\rangle$, with respect to which $\tau_M$ is represented by the identity matrix $\Psi^T\,\Phi\,\sigma^*(\Psi^T)^{-1}=\Id_r$. Therefore, the columns of $(\Psi^T)^{-1}$ form an $\BF_q[t]$-basis $\CC$ of $\Lambda(\ulM)$ and $h_\ulM$ is represented with respect to the bases $\CC$ and $\CB$ by $(\Psi^T)^{-1}$. Therefore, $h_\ulM$ is an isomorphism and $\ulM$ is uniformizable.
\end{proof}

Before we can conclude that $\AUMotCatIsog$ and $\AMUMotCatIsog$ are neutral Tannakian categories over $Q$ with fiber functor $\ulM\mapsto\Hodge^1(\ulM)$, we need to state the following
\begin{proposition}\label{PropUnifAMotive}
\begin{enumerate}
\item \label{PropUnifAMotiveA}
Every $A$-motive which is isomorphic to a uniformizable $A$-motive in $\AMotCatIsog$ is itself uniformizable.
\item \label{PropUnifAMotiveB}
Every $A$-motive of rank $1$ is uniformizable.
\item \label{PropUnifAMotiveC}
If $\ulM$ and $\ulN$ are uniformizable $A$-motives, then also $\ulM\otimes\ulN$ and $\CHom(\ulM,\ulN)$ and $\ulM\dual$ are uniformizable with 
\begin{eqnarray*}
\Lambda(\ulM\otimes\ulN) & \cong & \Lambda(\ulM)\otimes_A\Lambda(\ulN) \qquad\text{and}\\[2mm]
\Lambda\bigl(\CHom(\ulM,\ulN)\bigr) & \cong & \Hom_A\bigl(\Lambda(\ulM),\Lambda(\ulN)\bigr) \qquad\text{and}\\[2mm]
\Lambda(\ulM\dual) & \cong & \Hom_A(\Lambda(\ulM),A)\,.
\end{eqnarray*}
\item \label{PropUnifAMotiveD}
If $\ulM$ and $\ulN$ are uniformizable, the natural map $\QHom(\ulM,\ulN) \to \Hom_{Q}(\Hodge^1(\ulM),\Hodge^1(\ulN))$,
\[
f\otimes a\;\longmapsto\; \Hodge^1(f\otimes a)\;:=\;a\cdot\bigl(h_\ulN^{-1}\circ (f\otimes\id)\circ h_\ulM|_{\Hodge^1(\ulM)}\bigr)
\]
for $f\in\Hom_{\AMotCat}(\ulM,\ulN)$ and $a\in Q$, is injective.
\end{enumerate}
\end{proposition}

\begin{proof}
\ref{PropUnifAMotiveA} Let $\ulM$ be uniformizable and let $f\colon\ulM\isoto\ulN$ be an isomorphism of $A$-motives in $\AMotCatIsog$. By multiplying $f$ with an element of $A$ we can assume that $f\colon\ulM\into\ulN$ is an $A$-sub-motive in $\AMotCat$. Then $f\colon\Lambda(\ulM)\into\Lambda(\ulN)$ and $\rk\ulM=\rk_A\Lambda(\ulM)\le\rk_A\Lambda(\ulN)\le\rk\ulN=\rk\ulM$. So $\ulN$ is uniformizable by Lemma~\ref{LemmaUniformizable}\ref{LemmaUniformizableB}.

\medskip\noindent
\ref{PropUnifAMotiveB} is proved in \cite[Propositions 12.3(b) and 12.5]{HartlPink2}. In the special case where $C=\BP^1_{\BF_q}$, $A=\BF_q[t]$ and $\theta=\charmorph(t)\in\BC$, assertion~\ref{PropUnifAMotiveB} follows from \ref{PropUnifAMotiveC} and from Examples~\ref{ExampleCHMotive} and \ref{ExCarlitzMotiveUnif}, because all $t$-motives of rank $1$ are tensor powers of the Carlitz $t$-motive $(\BC[t],t-\theta)$.

\medskip\noindent
\ref{PropUnifAMotiveC}  If $\ulM$ and $\ulN$ are uniformizable, then $h_\ulM$ and $h_\ulN$ induce an isomorphism
\[
\TS\Lambda(\ulM)\otimes_A\Lambda(\ulN)\otimes_A \CO(\FC_\BC\setminus\Disc) \xrightarrow{\enspace h_\ulM\otimes h_\ulN\;} M\otimes_{A_{\BC}}N\otimes_{A_{\BC}}\CO(\FC_\BC\setminus\Disc)
\]
satisfying $(h_\ulM\otimes h_\ulN)\circ(\id_{\Lambda(\ulM)}\otimes\id_{\Lambda(\ulN)}\otimes\id)=(\stau_M\otimes\stau_N\otimes\id)\circ \ssigma^\ast(h_\ulM\otimes h_\ulN)$. Therefore, the $\stau$-invariants are 
\[
\TS\Lambda(\ulM\otimes\ulN)=\bigl(M\otimes_{A_{\BC}}N\otimes_{A_{\BC}}\CO(\FC_\BC\setminus\Disc)\bigr)^\stau \cong \Lambda(\ulM)\otimes_A\Lambda(\ulN)\,.
\]
Likewise, by applying \cite[Proposition~2.10]{Eisenbud}, the uniformizability of $\ulM$ yields an isomorphism
\[
\TS\Hom_A(\Lambda(\ulM),A)\otimes_A \CO(\FC_\BC\setminus\Disc) \xrightarrow{\enspace (h_\ulM\dual)^{-1}\;} M\dual\otimes_{A_{\BC}}\CO(\FC_\BC\setminus\Disc)
\]
satisfying
$(h_\ulM\dual)^{-1}\circ(\id_{\Hom_A(\Lambda(\ulM),A)}\otimes\id)=(\stau_{M\dual}\otimes\id)\circ \ssigma^\ast(h_\ulM\dual)^{-1}$. Therefore, the $\stau$-invariants are 
\[
\TS\Lambda(\ulM\dual)=\bigl(M\dual\otimes_{A_{\BC}}\CO(\FC_\BC\setminus\Disc)\bigr)^\stau \cong \Hom_A(\Lambda(\ulM),A)\,.
\]
From this also the statement about $\CHom(\ulM,\ulN)\cong\ulN\otimes\ulM\dual$ follows.

\medskip\noindent
\ref{PropUnifAMotiveD} Since $h_\ulM$ and $h_\ulN$ are isomorphisms, $f\otimes a$ can be recovered from $\Hodge^1(f\otimes a)$.
\end{proof}

\begin{lemma}\label{Lemma2.3}
Let $0\to\ulM'\to\ulM\to\ulM''\to0$ be a short exact sequence of $A$-motives. Then $\ulM$ is uniformizable if and only if both $\ulM'$ and $\ulM''$ are. In this case the induced sequence of $A$-modules $0\to\Lambda(\ulM')\to\Lambda(\ulM)\to\Lambda(\ulM'')\to0$ is exact.
\end{lemma}

\begin{proof}
The first assertion follows from Anderson~\cite[Lemma~2.7.2 and 2.10.4]{Anderson86}. For the second assertion observe that $\Lambda(\ulM')=\ker\bigl(\id-\tau\colon M'\otimes_{A_{\BC}}\CO(\FC_\BC\setminus\Disc)\to M'\otimes_{A_{\BC}}\CO(\FC_\BC\setminus\Disc)\bigr)$. Since the map $\id-\tau$ is surjective by \cite[Proposition~6.1]{BoeckleHartl} the snake lemma proves the exactness of the sequence $0\to\Lambda(\ulM')\to\Lambda(\ulM)\to\Lambda(\ulM'')\to0$.
\end{proof}

\begin{remark}
If a mixed $A$-motive $\ulM$ is uniformizable, then all filtration steps $W_\mu\ulM$ and factors $\Gr^W_\mu\ulM$ of the weight filtration of $\ulM$ are uniformizable by Lemma~\ref{Lemma2.3}. Therefore, $\ulM$ could equivalently be called a \emph{uniformizable mixed $A$-motive} or a \emph{mixed uniformizable $A$-motive}.
\end{remark}

\begin{theorem}\label{TheoremAMotTannakian}
The category $\AUMotCatIsog$ of uniformizable $A$-motives up to isogeny and its rigid tensor subcategory $\AMUMotCatIsog$ of uniformizable mixed $A$-motives up to isogeny are neutral Tannakian categories over $Q$ with fiber functor $\ulM\mapsto \Hodge^1(\ulM)$.
\end{theorem}

\begin{proof} 
By Propositions~\ref{PropUnifAMotive} and \ref{PropPure}\ref{PropPure_g}, $\AUMotCatIsog$ and $\AMUMotCatIsog$ are closed under taking tensor products, internal homs and duals, contain the unit object $\UOne(0)$ for the tensor product, and $\ulM\mapsto \Hodge^1(\ulM)$ is a faithful $Q$-linear tensor functor, which is exact by Lemma~\ref{Lemma2.3}. Moreover, $\Hodge^1(\ulM)$ is finite-dimensional for any uniformizable $A$-motive $\ulM$ by Lemma~\ref{LemmaUniformizable}\ref{LemmaUniformizableB}. As strictly full subcategories of the $Q$-linear abelian category $\AMotCatIsog$ also $\AUMotCatIsog$ and $\AMUMotCatIsog$ are $Q$-linear. Let $f\colon \ulM\to\ulN$ be a morphism in $\AUMotCatIsog$. Then the kernel, cokernel, image and coimage of $f$ in $\AMotCatIsog$ are uniformizable by Lemma~\ref{Lemma2.3} and belong to $\AUMotCatIsog$. Therefore, $\AUMotCatIsog$ and $\AMUMotCatIsog$ are abelian.
\end{proof}

This theorem allows to associate with each (mixed) uniformizable $A$-motive $\ulM$ an algebraic group $\Gamma_\ulM$ over $Q$ as follows. Consider the Tannakian subcategory $\llangle\ulM\rrangle$ of $\AUMotCatIsog$, respectively $\AMUMotCatIsog$ generated by $\ulM$. By Tannakian duality \cite[Theorem~2.11 and Proposition~2.20]{DM82}, the category $\llangle\ulM\rrangle$ is tensor equivalent to the category of $Q$-rational representations of a linear algebraic group scheme $\Gamma_\ulM$ over $Q$ which is a closed subgroup of $\GL_Q(\Hodge^1(\ulM))$.

\begin{definition}\label{DefMotGp}
The linear algebraic $Q$-group scheme $\Gamma_\ulM$ associated with $\ulM$ is called the \emph{(motivic) Galois group of $\ulM$}.
\end{definition}

\begin{example}\label{ExGalGpOfCarlitz}
The trivial $A$-motive $\UOne(0)$ has trivial motivic Galois group $\Gamma_{\UOne(0)}=(1)$. 

For any $A$-motive $\UOne(n)$ of rank $1$ with $n\ne0$ (see Example~\ref{ExampleCHMotive}) the motivic Galois group equals $\Gamma_{\UOne(n)}=\BG_{m,Q}$. Indeed, since $\Hodge^1\bigl(\UOne(n)\bigr)\cong Q$, the group $\Gamma_{\UOne(n)}$ is a subgroup of $\GL_Q\bigl(\Hodge^1(\UOne(n))\bigr)=\BG_{m,Q}$. If it were a finite group, it would be annihilated by some positive integer $d$. This implies that it operates trivially on $\UOne(n)^{\otimes d}\cong\UOne(dn)\in\llangle\UOne(n)\rrangle$. Therefore, $\UOne(dn)$ must be a direct sum of the trivial object $\UOne(0)$, that is $\UOne(dn)\cong\UOne(0)$, which is a contradiction.
\end{example}

\subsection{The associated Hodge-Pink structure}\label{SectAMotHPStr}

We associate a mixed $Q$-Hodge-Pink structure with every uniformizable mixed $A$-motive. Note that (a variant of) this is used by Taelman~\cite{Taelman16} in this volume to study $1$-$t$-motives.

For $i\in\BN_0$ we consider the pullbacks $\ssigma^{i*}J=(a\otimes1-1\otimes\charmorph(a)^{q^i}\colon a\in A)\subset A_\BC$ and the points $\Var(\ssigma^{i*}J)$ of $\dotC_\BC$ and $\dotFC_\BC$. They correspond to the points $\Var(z-\zeta^{q^i})\in\dotFD_\BC$ and have $\infty_{\SSC\BC}$ as accumulation point. Therefore, $\dotFC_\BC\setminus\bigcup_{i\in\BN_0}\Var(\sigma^{i\ast}J)$ is an admissible open rigid analytic subspace of $\dotFC_\BC$.

\begin{proposition}\label{PropLambdaConvRadius}
\label{PropMaphM}
Let $\ulM$ be a uniformizable $A$-motive over $\BC$. 
\begin{enumerate}
\item \label{PropMaphMA}
Then $\Lambda(\ulM)$ equals $\bigl\{\,m\in M\otimes_{A_{\BC}}\CO\bigl(\dotFC_\BC\setminus\bigcup_{i\in\BN_0}\Var(\sigma^{i\ast}J)\bigr):\es \tau_M(\sigma^\ast m)=m\,\bigr\}$ and the isomorphisms $h_\ulM$ and $\ssigma^*h_\ulM$ extend to isomorphisms of locally free sheaves
\begin{eqnarray*}
h_\ulM\colon \Lambda(\ulM)\otimes_A \CO_{\dotFC_\BC\setminus\bigcup_{i\in\BN_0}\Var(\sigma^{i\ast}J)} & \isoto & M\otimes_{A_{\BC}}\CO_{\dotFC_\BC\setminus\bigcup_{i\in\BN_0}\Var(\sigma^{i\ast}J)}\,,\\
 \lambda\otimes f & \longmapsto & f\cdot\lambda\,,\\[2mm]
\ssigma^*h_\ulM\colon \Lambda(\ulM)\otimes_A \CO_{\dotFC_\BC\setminus\bigcup_{i\in\BN_{>0}}\Var(\sigma^{i\ast}J)} & \isoto & \ssigma^\ast M\otimes_{A_{\BC}}\CO_{\dotFC_\BC\setminus\bigcup_{i\in\BN_{>0}}\Var(\sigma^{i\ast}J)}\,,\\
 \lambda\otimes f & \longmapsto & f\cdot\sigma^*\lambda\,,
\end{eqnarray*}
satisfying $h_\ulM\circ(\id_{\Lambda(\ulM)}\otimes\id)=(\stau_M\otimes\id)\circ \ssigma^\ast h_\ulM$.
\item \label{PropMaphMB}
If moreover $\ulM$ is effective, then $\Lambda(\ulM)$ equals $\bigl\{\,m\in M\otimes_{A_{\BC}}\CO(\dotFC_\BC):\es \tau_M(\sigma^\ast m)=m\,\bigr\}$ and the isomorphism $h_\ulM$ extends to an injective homomorphism
\[
h_\ulM\colon \Lambda(\ulM)\otimes_A \CO_{\dotFC_\BC}\es\longto\es M\otimes_{A_{\BC}}\CO_{\dotFC_\BC}\,,\quad \lambda\otimes f \mapsto f\cdot\lambda,
\]
with $h_\ulM\circ(\id_{\Lambda(\ulM)}\otimes\id)=(\stau_M\otimes\id)\circ \ssigma^\ast h_\ulM$.
At the point $\Var(J)$ its cokernel satisfies $\coker h_\ulM\otimes\BC\dbl z-\zeta\dbr=M/\stau_M(\ssigma^*M)$. 
\end{enumerate} 
\end{proposition}

\begin{proof}
\ref{PropMaphMB} If $\ulM$ is effective, the claimed equality for $\Lambda(\ulM)$ was proved in \cite[Proposition~3.4]{BoeckleHartl}. This allows to extend $h_\ulM$ to a homomorphism
\[
h_\ulM\colon \Lambda(\ulM)\otimes_A \CO(\dotFC_\BC)\;\longto\;M\otimes_{A_{\BC}}\CO(\dotFC_\BC)\,,\quad \lambda\otimes f \mapsto f\cdot\lambda\,.
\]
which satisfies $h_\ulM\circ(\id_{\Lambda(\ulM)}\otimes\id)=(\stau_M\otimes\id)\circ \ssigma^\ast h_\ulM$ and is injective because $\CO(\dotFC_\BC)\subset\CO(\FC_\BC\setminus\Disc)$. Let $D:=\coker h_\ulM$ and consider the following diagram, in which the first row is exact because of the flatness of $\sigma^\ast$; see Remark~\ref{RemSigmaIsFlat}.
\[
\xymatrix @R+1pc @C+2pc {
0\ar[r] & \Lambda(\ulM)\otimes_A \sigma^\ast\CO(\dotFC_\BC)\ar[r]^{\DS\sigma^\ast h_\ulM}\ar[d]_{\DS\id_{\Lambda(\ulM)}\otimes\id}^\cong & \sigma^\ast M\otimes_{A_{\BC}}\CO(\dotFC_\BC)\ar[r]\ar@{^{ (}->}[d]^{\DS\tau_M\otimes\id} & \sigma^\ast D \ar[r]\ar[d]^{\DS\tau_D} & 0\\
0\ar[r] & \Lambda(\ulM)\otimes_A \CO(\dotFC_\BC)\ar[r]^{\DS h_\ulM} &  M\otimes_{A_{\BC}}\CO(\dotFC_\BC)\ar[r] & D \ar[r] & 0
}
\]
By the snake lemma, $\tau_D$ is injective and $\coker\tau_M\cong\coker\tau_D$. The support of $D$ is contained in $\Disc$. So we now look at the points in $\Disc$ and use $z$ as a coordinate on $\Disc$. Let $\alpha\ne\zeta$ and consider the point $\{z=\alpha\}$ in $\Disc$. Since $\{z=\alpha\}\ne\Var(J)$ and $\coker\tau_M$ is supported at $\Var(J)$, we find $\sigma^\ast\bigl(D\otimes\BC\dbl z-\alpha^{q^{-1}}\dbr\bigr)=(\sigma^\ast D)\otimes\BC\dbl z-\alpha\dbr\cong D\otimes\BC\dbl z-\alpha\dbr$. Since the support of $D$ is discrete on $\dotFC_\BC$ it cannot have a limit point on the affinoid space $\FC_\BC\setminus\{P\in\Disc:|z(P)|<|\zeta|\}$. This implies $D\otimes\BC\dbl z-\alpha\dbr=(0)$ for all $\alpha\notin\bigcup_{i\in\BN_0}\{\zeta^{q^i}\}$ and proves that $h_\ulM$ is an isomorphism outside $\bigcup_{i\in\BN_0}\Var(\sigma^{i\ast}J)$. Moreover, $(\sigma^\ast D)\otimes\BC\dbl z-\zeta\dbr=(0)$ and $\coker\tau_M=\coker\tau_M\otimes\BC\dbl z-\zeta\dbr\cong D\otimes\BC\dbl z-\zeta\dbr$, and so $\ssigma^*h_\ulM$ is an isomorphism outside $\bigcup_{i\in\BN_{>0}}\Var(\sigma^{i\ast}J)$.

\medskip\noindent	
\ref{PropMaphMA} If $\ulM$ is not effective, then $\ulM$ is isomorphic to $\ulN\otimes\UOne(-n)$ by Remark~\ref{Rem3.3} for an effective $A$-motive $\ulN$ and some positive integer $n$. By Proposition~\ref{PropUnifAMotive} the $A$-motives $\UOne(n)$ and $\ulN\cong\ulM\otimes\UOne(n)$ are uniformizable. Since $\ulN$ and $\UOne(n)$ are effective, our proof of \ref{PropMaphMB} yields isomorphisms
\begin{eqnarray*}
&&h_\ulN\colon \Lambda(\ulN)\otimes_A \CO_{\dotFC_\BC\setminus\bigcup_{i\in\BN_0}\Var(\sigma^{i\ast}J)}\;\isoto\;N\otimes_{A_{\BC}}\CO_{\dotFC_\BC\setminus\bigcup_{i\in\BN_0}\Var(\sigma^{i\ast}J)}\qquad\text{and}\\[2mm]
&&h_{\UOne(n)}\colon \Lambda({\UOne(n)})\otimes_A \CO_{\dotFC_\BC\setminus\bigcup_{i\in\BN_0}\Var(\sigma^{i\ast}J)}\;\isoto\;\BOne(n)\otimes_{A_{\BC}}\CO_{\dotFC_\BC\setminus\bigcup_{i\in\BN_0}\Var(\sigma^{i\ast}J)}\,.
\end{eqnarray*}
Dualizing and inverting the second isomorphism and tensoring with the first yields the isomorphism
\begin{eqnarray*}
h_\ulN\otimes(h_{\UOne(n)}\dual)^{-1}\colon \Lambda(\ulN)\otimes_A\Lambda(\UOne(n))\dual\otimes_A \CO_{\dotFC_\BC\setminus\bigcup_{i\in\BN_0}\Var(\sigma^{i\ast}J)}\;\isoto\hspace{3cm}\\[2mm]
N\otimes_{A_{\BC}}\BOne(n)\dual\otimes_{A_{\BC}}\CO_{\dotFC_\BC\setminus\bigcup_{i\in\BN_0}\Var(\sigma^{i\ast}J)}
\end{eqnarray*}
which satisfies 
\[
h_\ulN\otimes(h_{\UOne(n)}\dual)^{-1}\circ(\id_{\Lambda(\ulN)}\otimes\id_{\Lambda(\UOne(n))}\otimes\id)\;=\;(\stau_N\otimes(\stau_{\BOne(n)}\dual)^{-1}\otimes\id)\circ\sigma^\ast(h_\ulN\otimes(h_{\UOne(n)}\dual)^{-1})\,.
\]
Combined with the isomorphisms $\ulN\otimes_{A_{\BC}}\UOne(n)\dual\cong\ulM$ and $\Lambda(\ulM)\cong\Lambda(\ulN)\otimes_A\Lambda(\UOne(n))\dual$, this yields the desired extension of $h_\ulM$
\[
\Lambda(\ulM)\otimes_A \CO_{\dotFC_\BC\setminus\bigcup_{i\in\BN_0}\Var(\sigma^{i\ast}J)}\;\isoto\;M\otimes_{A_\BC}\CO_{\dotFC_\BC\setminus\bigcup_{i\in\BN_0}\Var(\sigma^{i\ast}J)}
\]
and proves $\Lambda(\ulM)=\bigl\{\,m\in M\otimes_{A_{\BC}}\CO\bigl(\dotFC_\BC\setminus\bigcup_{i\in\BN_0}\Var(\sigma^{i\ast}J)\bigr):\es \tau_M(\sigma^\ast m)=m\,\bigr\}$.
\end{proof}

\begin{corollary}\label{CorLambdaConvRadius}
In the situation of Lemma~\ref{LemmaUniformizableBF_q[t]} let $\Psi\in\GL_r(\BC\langle t\rangle)$ be a rigid analytic trivialization of $\Phi$. Then the entries of $\Psi$ and $\Psi^{-1}$ converge for all $t\in\BC$ with $|t|<|\theta|$. If $\ulM$ is effective, then the entries of $\Psi^{-1}$ even converge for all $t\in\BC$.
\end{corollary}

\begin{proof}
In view of $J=(t-\theta)$ this follows from the fact that $h_\ulM$ is represented by the matrix $(\Psi^T)^{-1}$.
\end{proof}

Proposition~\ref{PropLambdaConvRadius} implies that $\sigma^\ast h_\ulM$ is an isomorphism locally at $V(J)=\{z=\zeta\}\subset\dotFD_\BC$. This allows us to associate a $Q$-pre Hodge-Pink structure with any uniformizable mixed $A$-motive as follows. Namely, $ h_\ulM$ induces isomorphisms
\begin{equation}\label{EqhM}
\xymatrix @R+1pc @C+7pc {
\Lambda(\ulM)\otimes_A\BC\dpl z-\zeta\dpr \ar[r]^{\sigma^\ast h_\ulM\otimes\id_{\BC\dpl z-\zeta\dpr}}_\cong \ar[d]_{\id_{\Lambda(\ulM)}\otimes\id_{\BC\dpl z-\zeta\dpr}}^\cong & \sigma^* M\otimes_{A_{\BC}}\BC\dpl z-\zeta\dpr \ar[d]_\cong^{\tau_M\otimes\id_{\BC\dpl z-\zeta\dpr}}\\
\Lambda(\ulM)\otimes_A\BC\dpl z-\zeta\dpr \ar[r]^{h_\ulM\otimes\id_{\BC\dpl z-\zeta\dpr}}_\cong &  M\otimes_{A_{\BC}}\BC\dpl z-\zeta\dpr\,.
}
\end{equation}
Here $h_\ulM\otimes\id_{\BC\dpl z-\zeta\dpr}$ is an isomorphism because the three others are.
Therefore, the preimage $\Fq\,:=\,( h_\ulM\otimes\id_{\BC\dpl z-\zeta\dpr})^{-1}\bigl(M\otimes_{A_{\BC}}\BC\dbl z-\zeta\dbr\bigr)$ is a $\BC\dbl z-\zeta\dbr$-lattice in $\Lambda(\ulM)\otimes_A\BC\dpl z-\zeta\dpr$. The tautological lattice is $\Fp:=\Lambda(\ulM)\otimes_A\BC\dbl z-\zeta\dbr=(\sigma^\ast h_\ulM\otimes\id_{\BC\dbl z-\zeta\dbr})^{-1}\bigl(\sigma^* M\otimes_{A_{\BC}}\BC\dbl z-\zeta\dbr\bigr)$.

\begin{definition}\label{Def2.6}
Let $\ulM$ be a uniformizable mixed $A$-motive with weight filtration $W_{\mu\,}\ulM$.
We set $\ulHodge^1(\ulM):=(H,W_\bullet H,\Fq)$ with
\begin{itemize}
\item $H\;:=\;\Hodge^1(\ulM)\;:=\;\Lambda(\ulM)\otimes_A Q$,
\item $W_\mu H\;:=\;\Hodge^1(W_{\mu\,}\ulM)\;=\;\Lambda(W_{\mu\,}\ulM)\otimes_A Q\;\subset\;H$ for each $\mu\in\BQ$,
\item $\Fq\,:=\,( h_\ulM\otimes\id_{\BC\dpl z-\zeta\dpr})^{-1}\bigl(M\otimes_{A_{\BC}}\BC\dbl z-\zeta\dbr\bigr)$.
\end{itemize}
We call $\ulHodge^1(\ulM)$ the \emph{$Q$-Hodge-Pink structure associated with $\ulM$}. (This name is justified by Theorem~\ref{ThmHodgeConjecture} below.) We also set $\ulHodge_1(\ulM):=\ulHodge^1(\ulM)\dual$ in $\QHodgeCat$. The functor $\ulHodge^1$ is covariant and $\ulHodge_1$ is contravariant in $\ulM$.
\end{definition}

\begin{remark}\label{Rem2.7}
(a) If $\ulM=\ulM(\ulE)$ is the $A$-motive associated with a Drinfeld $A$-module $\ulE$, then $\ulHodge^1(\ulM)\cong\ulHodge_1(\ulE)\dual=:\ulHodge^1(\ulE)$. We will prove this more generally for a uniformizable pure (or mixed) abelian Anderson $A$-module $\ulE$ in Theorem~\ref{ThmHPofEandM} below.

\medskip\noindent
(b) We draw some conclusions from the description of $\Fq$ and $\Fp:=\Lambda(\ulM)\otimes_A\BC\dbl z-\zeta\dbr$ given before the definition: If $J^m\cdot\tau_M(\ssigma^*M)\subset M\subset J^n\cdot\tau_M(\ssigma^*M)$ for some integers $n\le m$, then $(z-\zeta)^m\Fp\subset\Fq\subset(z-\zeta)^n\Fp$. For example, if $\ulM$ is effective, that is $\tau_M(\ssigma^*M)\subset M$, then $\Fp\subset\Fq$ and there is an exact sequence of $\BC\dbl z-\zeta\dbr$-modules
\[
0\longto\Fp\longto\Fq\xrightarrow{\es h_\ulM\otimes\id_{\BC\dpl z-\zeta\dpr}\,} M/\stau_M(\ssigma^*M)\longto0\,.
\]
Note that $M/\stau_M(\ssigma^*M)$ is a $\BC\dbl z-\zeta\dbr$-module because it is annihilated by some power of $z-\zeta$.

\medskip\noindent
(c) In terms of Definition~\ref{Def1.5} the virtual dimension of $\ulM$ is $\dim\ulM=\deg_\Fq\ulHodge^1(\ulM)$.
\end{remark}

The following theorem is the main theorem of \cite{HartlPink2}.

\begin{theorem}\label{ThmHodgeConjecture} 
Consider a uniformizable mixed $A$-motive $\ulM$.
\begin{enumerate}
\item \label{ThmHodgeConjectureA}
$\ulHodge^1(\ulM)$ is locally semistable and hence indeed a $Q$-Hodge-Pink structure.
\item \label{ThmHodgeConjectureB}
The functor $\ulHodge^1\colon \ulM\to\ulHodge^1(\ulM)$ is a $Q$-linear exact fully faithful tensor functor from the category $\AMUMotCatIsog$ to the category $\QHodgeCat$.
\item\label{ThmHodgeConjectureC}
The essential image of $\ulHodge^1$ is closed under the formation of subquotients; that is, if $\ulH'\subset\ulHodge^1(\ulM)$ is a $Q$-Hodge-Pink sub-structure, then there exists a uniformizable mixed $A$-sub-motive $\ulM'\subset\ulM$ in $\AMUMotCatIsog$ with $\ulHodge^1(\ulM')=\ulH'$.
\item\label{ThmHodgeConjectureD}
The functor $\ulHodge^1$ defines an exact tensor equivalence between the Tannakian subcategory $\llangle\ulM\rrangle\subset\AMUMotCatIsog$ generated by $\ulM$ and the Tannakian subcategory $\llangle\ulHodge^1(\ulM)\rrangle\subset\QHodgeCat$ generated by its $Q$-Hodge-Pink structure $\ulHodge^1(\ulM)$.
\end{enumerate}
\end{theorem}

Assertions~\ref{ThmHodgeConjectureC} and \ref{ThmHodgeConjectureD} are the function field analog of the Hodge Conjecture \cite{Hodge52,GrothendieckHodge,Deligne06}. We will prove Theorem~\ref{ThmHodgeConjecture} in Section~\ref{Sect4} and discuss its consequences for the Hodge-Pink group $\Gamma_{\ulHodge^1(\ulM)}$ in Section~\ref{Sect3}.

\begin{example}\label{Ex2.9}
Let $C=\BP^1_{\BF_q}$, $A=\BF_q[t]$, $z=\frac{1}{t}$, $\theta=\charmorph(t)=\frac{1}{\zeta}\in\BC$. Let $M=A_{\BC}^{\oplus2}$ with $\tau_M=\Phi:=\left( \begin{array}{cc}t-\theta&b\\0&(t-\theta)^3 \end{array}\right)$. Then $\ulM=(M,\tau_M)$ is mixed with $\Gr^W_1\ulM=W_1\ulM\cong(A_{\BC},\tau=(t-\theta))$ and $\Gr^W_3\ulM\cong(A_{\BC},\tau=(t-\theta)^3)$. So $\ulM$ has weights $1$ and $3$. Moreover, $\ulM$ is uniformizable by Lemma~\ref{Lemma2.3} and Proposition~\ref{PropUnifAMotive}\ref{PropUnifAMotiveB}.

We set $\ell_\zeta^{\SSC -}:=\prod_{i=0}^\infty(1-\zeta^{q^i}t)\in \CO(\dotFC_\BC)$ and choose an element $\eta\in\BC$ with $\eta^{q-1}=-\zeta$. Then 
\begin{eqnarray*}
\Lambda(W_1\ulM)&=&\{\lambda\in \CO(\dotFC_\BC):(t-\theta)\sigma^\ast(\lambda)\;=\;\lambda\}\es=\es\eta\ell_\zeta^{\SSC -}\cdot\BF_q[t],\\[2mm]
\Lambda(\Gr^W_3\ulM)&=&(\eta\ell_\zeta^{\SSC -})^3\cdot\BF_q[t],\qquad\text{and}\\[2mm]
\Lambda(\ulM)&=&\left(\begin{smallmatrix}\eta\ell_\zeta^{\SSC -}\\0\end{smallmatrix}\right)\cdot\BF_q[t]\oplus\left(\begin{smallmatrix}f\\(\eta\ell_\zeta^{\SSC -})^3\end{smallmatrix}\right)\cdot\BF_q[t]
\end{eqnarray*}
for an $f\in \CO(\dotFC_\BC)$ with $(t-\theta)\sigma^\ast(f)+b\cdot\eta^{3q}\sigma^\ast(\ell_\zeta^{\SSC -})^3=f$. Putting $\lambda_1:=\left(\begin{smallmatrix}\eta\ell_\zeta^{\SSC -}\\0\end{smallmatrix}\right)$ and $\lambda_2:=\left(\begin{smallmatrix}f\\(\eta\ell_\zeta^{\SSC -})^3\end{smallmatrix}\right)$, we get $H(\ulM)=\lambda_1\cdot Q\oplus\lambda_2\cdot Q$ and $W_1H(\ulM)=\lambda_1\cdot Q$. 

With respect to the bases $(\left(\begin{smallmatrix}1\\0\end{smallmatrix}\right),\left(\begin{smallmatrix}0\\1\end{smallmatrix}\right))$ of $\ulM$ and $(\lambda_1,\lambda_2)$ of $\Lambda(\ulM)$ the isomorphism $h_\ulM$ is given by the matrix $(\Psi^T)^{-1}:=\left(\begin{array}{cc}\eta{\ell}_\zeta^{\SSC -}&f\\0&(\eta{\ell}_\zeta^{\SSC -})^3 \end{array}\right)$. Therefore, the Hodge-Pink lattice is described by 
\[
\Fq\;=\;\left(\begin{array}{cc}\eta\ell_\zeta^{\SSC -}& f\\0&(\eta\ell_\zeta^{\SSC -})^3\end{array}\right)^{\!\!-1}\hspace{-0.5em}\cdot\Fp\;=\;\left( \begin{array}{cc}t-\theta&b\\0&(t-\theta)^3 \end{array}\right)^{\!\!-1}\hspace{-0.5em}\cdot\Fp.
\]
Since $\ell_\zeta^{\SSC -}$ has a simple zero at $z=\zeta$, one sees that $\Fq/\Fp$ (which is also isomorphic to $\coker\tau_M$) is isomorphic to $\BC\dbl z-\zeta\dbr/(z-\zeta)\oplus\BC\dbl z-\zeta\dbr/(z-\zeta)^3$ if $(t-\theta)|f$ (equivalently, if $(t-\theta)|b$) and isomorphic to $\BC\dbl z-\zeta\dbr/(z-\zeta)^4$ if $(t-\theta)\nmid f$ (equivalently, if $(t-\theta)\nmid b$). So the Hodge-Pink weights of $\ulHodge^1(\ulM)$ are $(1,3)$ or $(0,4)$, and the weight polygon lies above the Hodge polygon with the same endpoint $WP(\ulM)\ge HP(\ulM)$ in accordance with Theorem~\ref{ThmHodgeConjecture}\ref{ThmHodgeConjectureA} and Remark~\ref{RemPolygons}. 

In particular, if $b=(t-\theta)\cdot b'$ then the equation defining $f$ shows that $f$ vanishes at $t=\theta^{q^i}$ for all $i\in\BN_0$, whence $f=\eta\ell_\zeta^{\SSC -}\tilde f$ for an $\tilde f\in\CO(\dotFC_\BC)$ satisfying $\sigma^*(\tilde f)+b'\cdot\eta^{2q}\sigma^*(\ell_\zeta^{\SSC -})^2=\tilde f$.
\end{example}

\subsection{Cohomology Realizations}\label{CohAMot}

Let $\ulM=(M,\stau_{M})$ be an $A$-motive of rank $r$ over $\BC$.
Anderson defined the \emph{Betti cohomology realization} of $\ulM$ by setting
\[\Koh_\Betti^1(\ulM,B):=\Lambda(\ulM)\otimes_A B
\quad\text{and}\quad \Koh_{1,\Betti}(\ulM,B):=\Hom_A(\Lambda(\ulM),B)
\]
for any $A$-algebra $B$; see \cite[\S\,2.5]{Goss94}. This is most useful when $\ulM$ is uniformizable, in which case both are locally free $B$-modules of rank equal to $\rk\ulM$ and $\Hodge^1(\ulM)=\Koh^1_\Betti(\ulM,Q)$; see Lemma~\ref{LemmaUniformizable}. By Theorem~\ref{TheoremAMotTannakian} this realization provides for $B=Q$ an exact faithful neutral fiber functor on $\AUMotCatIsog$.

\medskip

Moreover, the \emph{de Rham cohomology realization} of $\ulM$ is defined to be
\[
\Koh^1_{\dR}(\ulM,\BC):=\ssigma^\ast M/J\cdot\ssigma^\ast M
 \quad\text{and}\quad \Koh_{1,\dR}(\ulM,\BC):=\Hom_\BC(\ssigma^\ast M/J\cdot\ssigma^\ast M,\,\BC).
\]
We define a decreasing filtration of $\Koh^1_{\dR}(\ulM,\BC)$ by $\BC$-subspaces
\[
F^{i}\Koh^1_{\dR}(\ulM,\BC):=\text{image of }\bigl(\ssigma^\ast M \cap J^i\cdot\tau_M^{-1}(M)\bigr)\quad\text{in }\Koh^1_{\dR}(\ulM,\BC)\quad\text{for all }i\in\BZ,
\]
which we call the \emph{Hodge-Pink filtration of $\ulM$}; see \cite[\S\,2.6]{Goss94}. 

If $\ulM$ satisfies $J\cdot M\subset\tau_M(\sigma^*M)\subset M$ then 
\[
F^0\;=\;\Koh^1_{\dR}(\ulM,\BC) \;\supset\; F^1\;=\;\tau_M^{-1}(J\cdot M)/J\cdot\sigma^*M \;\supset\; F^2\;=\;(0).
\]
For example, this is the case if $\ulM$ is the $A$-motive associated with a Drinfeld $A$-module. In this case the Hodge-Pink filtration coincides with the Hodge filtration studied by Gekeler, see Proposition~\ref{PropCompTateModEandM}\ref{PropCompTateModEandM_C} and Lemma~\ref{LemmaGekeler}.

As noted in Remark~\ref{RemQ-HPTannakian} and Example~\ref{Example1.2}(c), more useful than the Hodge-Pink filtration is actually the Hodge-Pink lattice $\Fq$, and the latter cannot be recovered from the Hodge-Pink filtration in general. We therefore propose to lift the de Rham cohomology to $\BC\dpl z-\zeta\dpr$ and define the \emph{generalized de Rham cohomology realization} of $\ulM$ by
\[
\begin{array}{llll}
\Koh^1_\dR(\ulM,\BC\dbl z-\zeta\dbr) & := & \ssigma^*M\otimes_{A_\BC}\BC\dbl z-\zeta\dbr
 & \quad\text{and}\\[2mm]
\Koh^1_\dR\bigl(\ulM,\BC\dpl z-\zeta\dpr\bigr) & := & \ssigma^*M\otimes_{A_\BC}\BC\dpl z-\zeta\dpr & \quad\text{and}\\[2mm]
\Koh_{1,\dR}(\ulM,\BC\dbl z-\zeta\dbr) & := & \Hom_{A_\BC}(\ssigma^*M,\,\BC\dbl z-\zeta\dbr)
 & \quad\text{and}\\[2mm]
\Koh_{1,\dR}\bigl(\ulM,\BC\dpl z-\zeta\dpr\bigr) & := & \Hom_{A_\BC}\bigl(\ssigma^*M,\,\BC\dpl z-\zeta\dpr\bigr)\,.
\end{array}
\]
In particular by tensoring with the morphism $\BC\dbl z-\zeta\dbr\onto\BC,\, z-\zeta\mapsto0$ we get back $\Koh^1_\dR(\ulM,\BC)=\Koh^1_\dR\bigl(\ulM,\BC\dbl z-\zeta\dbr\bigr)\otimes_{\BC\dbl z-\zeta\dbr}\BC$ and $\Koh_{1,\dR}(\ulM,\BC)=\Koh_{1,\dR}\bigl(\ulM,\BC\dbl z-\zeta\dbr\bigr)\otimes_{\BC\dbl z-\zeta\dbr}\BC$.
We define the \emph{Hodge-Pink lattices} of $\ulM$ as the $\BC\dbl z-\zeta\dbr$-submodules
\[
\begin{array}{ccccl}
\Fq^\ulM & := & \tau_M^{-1}(M)\otimes_{A_\BC}\BC\dbl z-\zeta\dbr & \subset & \Koh^1_{\dR}\bigl(\ulM,\BC\dpl z-\zeta\dpr\bigr)\quad\text{and}\\[2mm]
\Fq_\ulM & := & (\tau_M\dual\otimes\id_{\BC\dpl z-\zeta\dpr})\bigl(\Hom_{A_\BC}(M,\,\BC\dbl z-\zeta\dbr)\bigr) & \subset & \Koh_{1,\dR}\bigl(\ulM,\BC\dpl z-\zeta\dpr\bigr)\,.
\end{array}
\]
Then the Hodge-Pink filtrations $F^i \Koh^1_{\dR}(\ulM,\BC)$ and $F^i \Koh_{1,\dR}(\ulM,\BC)$ of $\ulM$ are recovered as the images of $\Koh^1_{\dR}\bigl(\ulM,\BC\dbl z-\zeta\dbr\bigr)\cap(z-\zeta)^i\Fq^\ulM$ in $\Koh^1_{\dR}(\ulM,\BC)$, respectively of $\Koh_{1,\dR}\bigl(\ulM,\BC\dbl z-\zeta\dbr\bigr)\cap(z-\zeta)^i\Fq_\ulM$ in $\Koh_{1,\dR}(\ulM,\BC)$ like in Remark~\ref{Rem1.4}. All these structures are compatible with the natural duality between $H^1_\dR$ and $H_{1,\dR}$. The de Rham realization provides (covariant) exact faithful tensor functors
\begin{align}\label{EqDeRhamFiberFunctor}
& \Koh_\dR^1(\,.\,,\BC)\colon & \hspace{-1.1cm}\AMotCatIsog & \longto \es {\tt Vect}_{\BC}\,, & & \ulM \;\longmapsto\; \Koh_\dR^1(\ulM,\BC)\qquad\text{and}\\[2mm]
& \Koh_\dR^1(\,.\,,\BC\dbl z-\zeta\dbr)\colon & \hspace{-1.1cm}\AMotCatIsog & \longto \es {\tt Mod}_{\BC\dbl z-\zeta\dbr}\,, & & \ulM \;\longmapsto\; \Koh_\dR^1(\ulM,\BC\dbl z-\zeta\dbr)\,. \nonumber
\end{align}
This is clear for $\Koh_\dR^1(\,.\,,\BC\dbl z-\zeta\dbr)$ and for $\Koh_\dR^1(\,.\,,\BC)$ exactness follows from the snake lemma applied to multiplication with $z-\zeta$ on $\Koh_\dR^1(\,.\,,\BC\dbl z-\zeta\dbr)$. To prove faithfulness for $\Koh_\dR^1(\,.\,,\BC)$ note that every morphism $f\colon\ulM'\to\ulM$ can in $\AMotCatIsog$ be factored into $\ulM'\onto\im(f)\isoto\coim(f)\into\ulM$. If $\Koh_\dR^1(f,\BC)$ is the zero map the exactness of $\Koh_\dR^1(\,.\,,\BC)$ shows that $\Koh_\dR^1(\im(f),\BC)=(0)$. Since $\dim_\BC\Koh_\dR^1(\ulM,\BC)=\rk\ulM$ it follows that the $A$-motive $\im(f)$ has rank zero and therefore $\im(f)=(0)$ and $f=0$.

\medskip

Finally, let $v\in\dotC$ be a closed point. We say that $v$ is a \emph{finite place} of $C$. Let $A_v$ be the $v$-adic completion of $A$, and let $Q_v$ be the fraction field of $A_v$. Consider the $v$-adic completions $A_{\BC,v}:=\invlim A_\BC/v^nA_\BC$ of $A_\BC$ and $M_v:= \varprojlim M/v^n M$ of $M$. Note that $\tau\colon m\mapsto\stau_M(\ssigma^\ast m)$ for $m\in M$ induces a $\ssigma^\ast\wh\otimes\id_{A_v}$-linear map $\stau\colon M_v\to M_v$. We let the \emph{$\stau$-invariants of $M_v$} be the $A_v$-module 
\[M_v^\stau :=\{m\in M_v\;|\;\stau(m)=m\}.\]
It is isomorphic to $A_v^{\oplus\rk\ulM}$ and the inclusion $M_v^\tau\subset M_v$  induces a canonical $\stau$-equivariant isomorphism $M_v^\stau\otimes_{A_v}A_{\BC,v}\isoto M_v$ by \cite[Proposition~6.1]{TW}. The \emph{$v$-adic cohomology realizations of $\ulM$} are given by 
\[
\begin{array}{lllllll}
\Koh_v^1(\ulM,A_v) & := & M_v^\stau & \quad\text{and}\quad & \Koh_v^1(\ulM,Q_v) & := & M_v^\stau\otimes_{A_v}Q_v\qquad\text{and}\\[2mm]
\Koh_{1,v}(\ulM,A_v) & := & \Hom_{A_v}(M_v^\stau,A_v) & \quad\text{and}\quad & \Koh_{1,v}(\ulM,Q_v) & := & \Hom_{A_v}(M_v^\stau,Q_v)\,;
\end{array}
\]
see \cite[\S\,2.3]{Goss94}. If $\ulM$ is defined over a subfield $L$ of $\BC$ (with $L=\BC$ allowed) then they carry a continuous action of $\Gal(L^\sep/L)$ and the $v$-adic realization provides (covariant) exact faithful tensor functors
\begin{eqnarray}\label{EqVAdicFiberFunctor}
\Koh_v^1(\,.\,,A_v)\colon & \AMotCat & \longto \es {\tt Mod}_{A_v[\Gal(L^\sep/L)]}\,,\quad \ulM \;\longmapsto\; \Koh_v^1(\ulM,A_v)\qquad\text{and}\\[2mm]
\Koh_v^1(\,.\,,Q_v)\colon & \AMotCatIsog & \longto \es {\tt Mod}_{Q_v[\Gal(L^\sep/L)]}\,,\quad \ulM \;\longmapsto\; \Koh_v^1(\ulM,Q_v)\,. \nonumber
\end{eqnarray}
This follows from the isomorphism $\Koh_v^1(\ulM,A_v)\otimes_{A_v}A_{\BC,v}\isoto M_v$ because $A_v\subset A_{\BC,v}$ is faithfully flat. Moreover, if $L$ is a \emph{finitely generated} field then Taguchi~\cite{Taguchi95b} and Tamagawa~\cite[\S\,2]{Tamagawa} proved that
\begin{equation}\label{EqTateConjAMotives}
\Koh_v^1(\,.\,,A_v)\colon \; \Hom(\ulM,\ulM')\otimes_A A_v \;\isoto\; \Hom_{A_v[\Gal(L^\sep/L)]}\bigl(\Koh_v^1(\ulM,A_v),\Koh_v^1(\ulM',A_v)\bigr)
\end{equation}
is an isomorphism for $A$-motives $\ulM$ and $\ulM'$. This is the analog of the \emph{Tate conjecture} for $A$-motives. 

\begin{proposition}\label{PropWeightsTateModule}
Let $\ulM$ be a pure or mixed $A$-motive, which is defined over a \emph{finite field extension} $L$ of $Q$. Let $\CP$ be a finite place of $L$, not lying above $\infty$ or $v$, where $\ulM$ has good reduction, and let $\BF_\CP$ be its residue field. Then the geometric Frobenius $\Frob_\CP$ of $\CP$ has a well defined action on $\Koh_v^1(\ulM,A_v)$ and each of its eigenvalues lies in the algebraic closure of $Q$ in $\BC$ and has absolute value $(\#\BF_\CP)^\mu$ for a weight $\mu$ of $\ulM$. These eigenvalues are independent of $v$.
\end{proposition}

\noindent
{\it Remark.} The \emph{geometric Frobenius} $\Frob_\CP$ of $\CP$ is the inverse of the \emph{arithmetic Frobenius} $\Frob_\CP^{-1}$, which satisfies $\Frob_\CP^{-1}(x) \equiv x^{\#\BF_\CP}\mod\CP$ for $x\in\CO_L$.

\begin{proof}
Let $\rho_{v,\ulM}\colon\Gal(L^\sep/L)\to\Aut_{A_v}\Koh_v^1(\ulM,A_v)$ be the associated Galois representation. By Gardeyn's criterion \cite[Theorem~1.1]{Gardeyn2} for good reduction, the inertia group of $\Gal(L^\sep/L)$ at $\CP$ acts trivially on $\Koh_v^1(\ulM,A_v)$ for every $v\ne\infty$ not lying below $\CP$, and therefore the Frobenius $\Frob_\CP$ of $\CP$ has a well defined action $\rho_{v,\ulM}(\Frob_\CP)$ on $\Koh_v^1(\ulM,A_v)$. Let $\ulM_\CP$ be the reduction of $\ulM$ at $\CP$. Then there is a canonical isomorphism $\Koh_v^1(\ulM,A_v)\isoto\Koh_v^1(\ulM_\CP,A_v)$ under which the action of $\Frob_\CP$ corresponds to the action of the Frobenius endomorphism 
%
\[
\tau_{M_\CP}^{[\BF_\CP:\BF_q]}\;:=\;\tau_{M_\CP}\circ\sigma^*\tau_{M_\CP}\circ\ldots\circ\sigma^{([\BF_\CP:\BF_q]-1)*}\tau_{M_\CP}\colon \ulM_\CP[J^{-1}]\;=\;\sigma^{[\BF_\CP:\BF_q]*}\ulM_\CP[J^{-1}]\isoto\ulM_\CP[J^{-1}]\,.
\]
Indeed, the action of $\Frob_\CP^{-1}=\sigma^{[\BF_\CP:\BF_q]*}$ on $\Koh_v^1(\ulM_\CP,A_v)$ is computed as $\rho_{v,\ulM}(\Frob_\CP^{-1}):=h^{-1}\circ(\Frob_\CP^{-1})^*h$ via the vertical isomorphisms $h$ in the following commutative diagram
\[
\xymatrix @C+2pc {
(\ulM_\CP)_v\otimes_{A_{\BF_\CP,v}}A_{\BF_\CP^\alg,v} \ar@{=}[r] & \sigma^{[\BF_\CP:\BF_q]*}(\ulM_\CP)_v\otimes_{A_{\BF_\CP,v}}A_{\BF_\CP^\alg,v} \ar[r]^{\qquad\tau_{M_\CP}^{[\BF_\CP:\BF_q]}}_\cong & (\ulM_\CP)_v\otimes_{A_{\BF_\CP,v}}A_{\BF_\CP^\alg,v} \\
(\ulM_\CP)_v^\tau\otimes_{A_v}A_{\BF_\CP^\alg,v} \ar[u]^h_\cong & (\ulM_\CP)_v^\tau\otimes_{A_v}A_{\BF_\CP^\alg,v} \ar[u]^{(\Frob_\CP^{-1})^*h}_\cong \ar[l]_{\rho_{v,\ulM}(\Frob_\CP^{-1})\otimes\id}^\cong \ar@{=}[r] & (\ulM_\CP)_v^\tau\otimes_{A_v}A_{\BF_\CP^\alg,v}\,. \ar[u]^h_\cong
}
\]
In particular $h\circ\rho_{v,\ulM}(\Frob_\CP)=\tau_{M_\CP}^{[\BF_\CP:\BF_q]}\circ h$ on $\Koh_v^1(\ulM_\CP,A_v)$. Since $Q\otimes_A\End_{\BF_\CP}(\ulM_\CP)$ is a finite dimensional $Q$-algebra, $\tau_{M_\CP}^{[\BF_\CP:\BF_q]}$ satisfies a polynomial equation with coefficients in $Q$ and its eigenvalues on $\Koh_v^1(\ulM_\CP,A_v)$ satisfy the same equation. In particular, these eigenvalues are independent of the place $v\ne\infty$ not lying below $\CP$. Now our formula for the absolute values of the eigenvalues was proved for pure $\ulM$ by Goss~\cite[Theorem~5.6.10]{Goss} and follows for mixed $\ulM$, because the eigenvalues of $\Frob_\CP$ coincide with the eigenvalues on the graded pieces $\Gr^W_\mu\ulM$ of $\ulM$ by considerations of triangular matrices. This motivates our convention that the weights of an effective $A$-motive are non-negative; see Proposition~\ref{PropPure}\ref{PropPure_h}.
\end{proof}

\bigskip

The morphism $h_\ulM$ from Proposition~\ref{PropMaphM} induces comparison isomorphisms between the Betti and the $v$-adic, respectively the de Rham realizations as follow.

\begin{theorem}\label{ThmCompIsomBettiDRAMotive}
If $\ulM$ is a uniformizable $A$-motive there are canonical \emph{comparison isomorphisms}, sometimes also called \emph{period isomorphisms}
\[
h_{\Betti,\,v}\colon\Koh^1_\Betti(\ulM,A_v)\;=\;\Lambda(\ulM)\otimes_A A_v\;\isoto\;\Koh^1_v(\ulM,A_v)\,,\quad\lambda\otimes f\longmapsto (f\cdot\lambda \mod v^n)_{n\in\BN}
\]
and
\[
\begin{array}[b]{rcc@{\hspace{-0em}}cccl}
h_{\Betti,\,\dR} & := & \sigma^\ast h_\ulM\otimes\id_{\BC\dbl z-\zeta\dbr} & \colon &\Koh^1_\Betti\bigl(\ulM,\BC\dbl z-\zeta\dbr\bigr) & \isoto & \Koh^1_{\dR}\bigl(\ulM,\BC\dbl z-\zeta\dbr\bigr)\,,\\[2mm]
h_{\Betti,\,\dR} & := & \sigma^\ast h_\ulM\mod J & \colon &\Koh^1_\Betti(\ulM,\BC) & \isoto & \Koh^1_{\dR}(\ulM,\BC)\,.
\end{array}
\]
The latter are compatible with the Hodge-Pink lattices, respectively the Hodge-Pink filtration provided on the Betti realization $\Koh^1_\Betti(\ulM,Q)=\Hodge^1(\ulM)$ via the associated Hodge-Pink structure $\ulHodge^1(\ulM)$.
\end{theorem}

\begin{proof}
Since $v\ne\infty$ the points in the closed subscheme $\{v\}\times_{\BF_q}\BC\subset C_\BC$ do not specialize to $\infty_{\kappa_\BC}\in C_{\kappa_\BC}$ and so this closed subscheme lies in the rigid analytic space $\FC_\BC\setminus\FD_\BC$. This yields isomorphisms $\CO(\FC_\BC\setminus\FD_\BC)/v^n\CO(\FC_\BC\setminus\FD_\BC)\isoto A_\BC/v^nA_\BC$ for all $n$. The isomorphism $h_\ulM$ induces a $\tau$-equivariant isomorphism
\[
\Lambda(\ulM)\otimes_A \invlim\CO(\FC_\BC\setminus\FD_\BC)/v^n\CO(\FC_\BC\setminus\FD_\BC) \;\isoto\; M\otimes_{A_\BC}\invlim A_\BC/v^n A_\BC\;=\;M_v\,.
\]
Taking $\tau$-invariants on both sides and observing 
\[
\bigl(\invlim\CO(\FC_\BC\setminus\FD_\BC)/v^n\CO(\FC_\BC\setminus\FD_\BC)\bigr)^{\ssigma=\id}\;=\;\bigl(\invlim A_\BC/v^nA_\BC\bigr)^{\ssigma=\id}=A_v
\]
provides $h_{\Betti,\,v}$.

The compatibility of the Betti--de Rham comparison isomorphism with the Hodge-Pink lattice and the Hodge-Pink filtration follows from diagram~\eqref{EqhM}.
\end{proof}

\begin{remark}
(a) If $\ulM=\ulM(\ulE)$ is the $A$-motive associated with a Drinfeld $A$-module $\ulE$, the isomorphism $h_{\Betti,\,\dR}$ coincides with the period isomorphism studied by Gekeler \cite[Theorem~5.14]{Gekeler89}; see Section~\ref{SectCohAMod}, in particular Theorem~\ref{ThmPeriodIsomForE} and Proposition~\ref{PropCompTateModEandM}.

\medskip\noindent
(b) Note that there are no $A$-homomorphisms between $A_v$ and $\BC$ and therefore no comparison isomorphism between $\Koh^1_v(\ulM,A_v)$ and $\Koh^1_{\dR}(\ulM,\BC)$ or $\Koh^1_{\dR}\bigl(\ulM,\BC\dbl z-\zeta\dbr\bigr)$. However, if one considers $A$-motives $\ulM$ over an algebraically closed, complete extension $K$ of the $v$-adic completion $Q_v$ instead of over $\BC$, there is a comparison isomorphism  between $\Koh^1_v(\ulM,A_v)$ and $\Koh^1_{\dR}\bigl(\ulM,K\dpl z-\zeta\dpr\bigr)$; see \cite[Remark~4.16]{HartlKim}.
\end{remark}

\begin{example}\label{ExCarlitzPeriod}
Let $C=\BP^1_{\BF_q}$, $A=\BF_q[t]$, $z=\frac{1}{t}$, $\theta=\charmorph(t)=\frac{1}{\zeta}\in\BC$, and let $\ulM=(\BC[t],\tau_M=t-\theta)$ be the Carlitz $t$-motive from Example~\ref{ExampleCHMotive}. As in Example~\ref{ExCarlitzMotiveUnif}(b) we obtain 
\[
\Lambda(\ulM)\;=\;\{\lambda\in \CO(\dotFC_\BC):(t-\theta)\sigma^\ast(\lambda)\;=\;\lambda\}\;=\;\eta\ell_\zeta^{\SSC -}\cdot\BF_q[t]
\]
for $\ell_\zeta^{\SSC -}:=\prod_{i=0}^\infty(1-\zeta^{q^i}t)\in \CO(\dotFC_\BC)$ and $\eta\in\BC$ with $\eta^{q-1}=-\zeta$. The comparison isomorphism $h_{\Betti,\,\dR}=\sigma^*h_\ulM\otimes\id_{\BC\dbl z-\zeta\dbr}$ sends the basis $\eta\ell_\zeta^{\SSC -}$ of $\Koh^1_\Betti(\ulM,\BF_q[t])=\Lambda(\ulM)$ to the element $\sigma^*(\eta\ell_\zeta^{\SSC -})=-\zeta\eta\sigma^*(\ell_\zeta^{\SSC -})\in\Koh^1_\dR(\ulM,\BC\dbl z-\zeta\dbr)=\BC\dbl z-\zeta\dbr$, respectively to the element $-\zeta\eta\sigma^*(\ell_\zeta^{\SSC -})|_{t=\theta}=-\zeta\eta\prod_{i=1}^\infty(1-\zeta^{q^i-1})\in\Koh^1_\dR(\ulM,\BC)=\BC$. The latter is the function field analog of the complex number $(2i\pi)^{-1}$, the inverse of the period of the multiplicative group $\BG_{m,\BQ}$. It is transcendental over $\BF_q(\theta)$ by a result of Wade~\cite{Wade41}. See Example~\ref{ExCarlitzModulePeriod} for more explanations.
\end{example}


\section{Mixed dual \texorpdfstring{$A$}{A}-motives}\label{SectMixedDualAMotives}
\setcounter{equation}{0}

For applications to transcendence questions like in \cite{ABP,Papanikolas,ChangYu07,ChangPapaYu10,ChangPapaThakurYu,ChangPapa11,ChangPapaYu11,ChangPapa12}, it turns out that \emph{dual $A$-motives} are even more useful than $A$-motives; see the article of Chang~\cite{Chang12} in this volume for an introduction. Beware that a dual $A$-motive is something different then the dual $\ulM\dual$ of an $A$-motive $\ulM$. We clarify the relation between dual $A$-motives and $A$-motives, also in view of purity, mixedness and uniformizability in this section.

\subsection{Dual \texorpdfstring{$A$}{A}-motives} \label{SectDualAMotives}

We continue with the conventions made in Section~\ref{SectAMotives}. In particular, we denote the natural inclusion $Q\into\BC$ by $\charmorph$ and consider the maximal ideal $J:=(a\otimes 1-1\otimes \charmorph(a):a\in A)\subset A_{\BC}:=A\otimes_{\BF_q}\BC$. The open subscheme $\Spec A_{\BC}\setminus{\rm V}(J)$ of $C_\BC$ is affine. We denote its ring of global sections by $\GlobMotRing$.

\begin{definition}\label{DefDualAMotive}
\begin{enumerate}
\item 
A \emph{dual $A$-motive} over $\BC$ of characteristic $\charmorph$ is a pair $\uldM=(\dM,\sdtau_\dM)$ consisting of a finite projective $A_{\BC}$-module $\dM$ and an isomorphism of $\GlobMotRing$-modules
\[
\sdtau_\dM\colon \sdsigma^\ast \dM[J^{-1}]\isoto \dM[J^{-1}]
\]
where we set $\sdsigma^\ast \dM[J^{-1}]:=(\sdsigma^\ast \dM)\otimes_{A_{\BC}}\GlobMotRing$ and $\dM[J^{-1}]:=\dM\otimes_{A_{\BC}}\GlobMotRing$. A \emph{morphism} of dual $A$-motives $\df\colon \uldM\to\uldN$ is a homomorphism of the underlying $A_{\BC}$-modules $\df\colon \dM\to \dN$ that satisfies $\df\circ\sdtau_\dM=\sdtau_\dN\circ\sdsigma^\ast \df$. 
The category of dual $A$-motives over $\BC$ is denoted $\dualAMotCat$.
\item
The rank of the $A_\BC$-module $\dM$ is called the \emph{rank} of $\uldM$ and is denoted by $\rk\uldM$. The \emph{virtual dimension} $\dim\uldM$ of $\uldM$ is defined as
\[
\dim\uldM\;:=\;\dim_\BC \,\dM\big/(\dM\cap\sdtau_\dM(\sdsigma^\ast \dM))\;-\;\dim_\BC \,\sdtau_\dM(\sdsigma^\ast \dM)\big/(\dM\cap\sdtau_\dM(\sdsigma^\ast \dM))\,.
\]
\item 
A dual $A$-motive $(\dM,\sdtau_\dM)$ is called \emph{effective} if $\sdtau_\dM$ comes from an $A_{\BC}$-homomorphism $\sdsigma^\ast \dM\to \dM$. An effective $A$-motive has virtual dimension $\ge0$.
\item For two dual $A$-motives $\uldM$ and $\uldN$ over $\BC$ we call $\QHom(\ul\dM,\ul\dN):=\Hom_{\dualAMotCat}(\uldM,\uldN)\otimes_A Q$ the set of \emph{quasi-morphisms} from $\uldM$ to $\uldN$.
\item 
The category with all dual $A$-motives as objects and the $\QHom(\ulM,\ulN)$ as $\Hom$-sets is called the \emph{category of dual $A$-motives over $\BC$ up to isogeny}. It is denoted $\dualAMotCatIsog$.
\end{enumerate}
\end{definition}

Again, if $C=\BP^1_{\BF_q}$ and $A=\BF_q[t]$, our effective dual $A$-motives are a slight generalization of the \emph{abelian dual $t$-motives} in \cite[\S4.4]{BrownawellPapanikolas16}, who in addition require that $\dM$ is finitely generated over $\BC\{\sdtau\}$ where $\sdtau$ acts on $\dM$ through $\dm\mapsto \sdtau_\dM(\sdsigma^\ast \dm)$. 

\medskip

The \emph{tensor product} of two dual $A$-motives $\uldM$ and $\uldN$ is the dual $A$-motive $\uldM\otimes\uldN$ consisting of the $A_{\BC}$-module $\dM\otimes_{A_{\BC}}\dN$ and the isomorphism $\sdtau_\dM\otimes\sdtau_\dN$. 
The dual $A$-motive $\dUOne(0)$ with underlying $A_{\BC}$-module $A_{\BC}$ and $\sdtau=\id_{A_{\BC}}$ is a unit object for the \emph{tensor product} in $\dualAMotCat$ and $\dualAMotCatIsog$. Both categories possess finite direct sums in the obvious way. We also define the \emph{tensor powers} of a dual $A$-motive $\uldM$ as $\uldM^{\otimes0}=\dUOne(0)$ and as $\uldM^{\otimes n}:=\uldM^{\otimes n-1}\otimes\uldM$ for $n>0$. If $\uldM=(\dM,\sdtau_\dM)$ and $\uldN=(\dN,\sdtau_\dN)$ are dual $A$-motives the \emph{internal hom} $\CHom(\uldM,\uldN)$ is the dual $A$-motive with underlying $A_\BC$-module $\dH:=\Hom_{A_\BC}(\dM,\dN)$ and $\sdtau_\dH\colon\sdsigma^*\dH[J^{-1}]\isoto \dH[J^{-1}],\, \dh\mapsto \sdtau_\dN\circ \dh\circ \sdtau_\dM^{-1}$. The \emph{dual} of a dual $A$-motive $\uldM$ is the dual $A$-motive $\uldM\dual:=\CHom(\uldM,\dUOne(0))$ consisting of the $A_{\BC}$-module $\dM\dual := \Hom_{A_{\BC}}(\dM,A_{\BC})$ and the isomorphism $(\sdtau_\dM\dual)^{-1}$.

\begin{remark}\label{RemImportant}
The reader should be careful not to confuse dual $A$-motives $\uldM$ with the duals $\ulM\dual$ of $A$-motives $\ulM$, which are again $A$-motives. In fact, the relation between $A$-motives and dual $A$-motives is the following. Let $\Omega^1_{A/\BF_q}$ be the $A$-module of K\"ahler differentials. Then $\Omega^1_{A_\BC/\BC}=\Omega^1_{A/\BF_q}\otimes_{\BF_q}\BC=\sigma^*\Omega^1_{A_\BC/\BC}=\sdsigma^*\Omega^1_{A_\BC/\BC}$ under the $\BF_q$-isomorphism $\Frob_{q,\BC}\colon \BC\isoto\BC$.
\end{remark}

\begin{proposition}\label{PropDualizing}
Every $A$-motive $\ulM=(M,\tau_M)$ induces a dual $A$-motive $\uldM(\ulM):=(\dM,\sdtau_\dM)$ where 
\begin{eqnarray*}
\dM & := & \Hom_{A_\BC}(\sigma^\ast M,\,\Omega^1_{A_\BC/\BC}), \quad\text{hence,}\quad \sdsigma^*\dM \es = \es \Hom_{A_\BC}(M,\,\Omega^1_{A_\BC/\BC}), \quad\text{and}\\
\sdtau_\dM & := & (\tau_M)\dual\es:=\es\Hom_{A_\BC}(\tau_M,\,\Omega^1_{A_\BC/\BC})\colon\es (\sdsigma^*\dM)\otimes_{A_{\BC}}\GlobMotRing\isoto\dM\otimes_{A_{\BC}}\GlobMotRing\,,\\
& & \qquad\qquad\qquad\qquad\qquad\qquad\qquad\qquad\qquad\qquad\,\, \sdsigma^*\dm\quad\longmapsto\quad \sdsigma^*\dm\circ\tau_M\,.
\end{eqnarray*}
Every morphism $f\colon\ulM\to\ulN$ of $A$-motives induces a morphism $\df:=\Hom_{A_\BC}(\sigma^\ast f,\,\Omega^1_{A_\BC/\BC})\colon$ $\uldM(\ulN)\to\uldM(\ulM)$ of the associated dual $A$-motives. 

Conversely, every dual $A$-motive $\uldM=(\dM,\sdtau_\dM)$ induces an $A$-motive $\ulM(\uldM):=(M,\tau_M)$ where
\begin{eqnarray*}
M & := & \Hom_{A_\BC}(\sdsigma^\ast \dM,\,\Omega^1_{A_\BC/\BC}), \quad\text{hence,}\quad \sigma^*M \es = \es \Hom_{A_\BC}(\dM,\,\Omega^1_{A_\BC/\BC}), \qquad\text{and}\\
\tau_M & := & (\sdtau_\dM)\dual\es:=\es\Hom_{A_\BC}(\sdtau_\dM,\,\Omega^1_{A_\BC/\BC})\colon\es (\sigma^*M)\otimes_{A_{\BC}}\GlobMotRing\isoto M\otimes_{A_{\BC}}\GlobMotRing\bigr)\,,\\
& & \qquad\qquad\qquad\qquad\qquad\qquad\qquad\qquad\qquad\qquad\,\, \sigma^*m\quad\longmapsto\quad \sigma^*m\circ\sdtau_\dM\,.
\end{eqnarray*}
Every morphism $\df\colon\uldM\to\uldN$ of dual $A$-motives induces a morphism $f:=\Hom_{A_\BC}(\sdsigma^*\df,\,\Omega^1_{A_\BC/\BC})\colon$ $\ulM(\uldN)\to\ulM(\uldM)$ of the associated $A$-motives. 

These mutually inverse functors induce exact tensor-anti-equivalences of categories $\AMotCat\longleftrightarrow \dualAMotCat$ and $\AMotCatIsog\longleftrightarrow \dualAMotCatIsog$. They map effective $A$-motives to effective dual $A$-motives and vice versa. In particular, the category $\dualAMotCatIsog$ is a $Q$-linear (non-neutral) Tannakian category, and hence a rigid abelian tensor category.
\end{proposition}

The motivation to throw in the Kähler differentials is given by Theorem~\ref{ThmMandDMofE} below.

\begin{proof}[Proof of Proposition~\ref{PropDualizing}]
Since $\sigma^\ast$ and $\sdsigma^\ast$ are flat by Remark~\ref{RemSigmaIsFlat} and $M$ and $\check M$ are locally free, it follows from \cite[Proposition~2.10]{Eisenbud} that $\sdsigma^*\Hom_{A_\BC}(\sigma^\ast M,\,\Omega^1_{A_\BC/\BC})=\Hom_{A_\BC}(M,\,\Omega^1_{A_\BC/\BC})$ and $\sigma^*\Hom_{A_\BC}(\sdsigma^\ast \dM,\,\Omega^1_{A_\BC/\BC})=\Hom_{A_\BC}(\dM,\,\Omega^1_{A_\BC/\BC})$. With this observation the proposition is straight forward to prove, and the final statements about the category $\dualAMotCatIsog$ follow from Proposition~\ref{PropAMotTannakianCateg}. 
\end{proof}

\begin{remark}\label{RemQuillenDualAMot}
(a) A neutral fiber functor only exists on the full subcategory of \emph{uniformizable} dual $A$-motives; see Theorem~\ref{TheoremDualAMotTannakian}

\medskip\noindent
(b) The category $\dualAMotCat$ is an exact category in the sense of Quillen~\cite[\S2]{Quillen} if one calls a sequence of dual $A$-motives \emph{exact} when its underlying sequence of $A_\BC$-modules is exact; compare Remark~\ref{RemQuillenAMot}(b). The same is true for the subcategories of dual $A$-motives which are effective, respectively effective and finitely generated over $\BC\{\sdtau\}$.

\medskip\noindent
(c) For the rank and (virtual) dimension of dual $A$-motives the formulas \eqref{EqRkDim} hold correspondingly and $\rk\uldM(\ulM)=\rk\ulM$ and $\dim\uldM(\ulM)=\dim\ulM$.

\medskip\noindent
(d) It can be proved directly, but also follows from Proposition~\ref{PropDualizing} and Remark~\ref{Rem3.3}(c) that the set of morphisms $\Hom_{\dualAMotCat}(\uldM,\uldN)$ between two dual $A$-motives $\uldM$ and $\uldN$ is a finite projective $A$-module of rank at most $(\rk\uldM)\cdot(\rk\uldN)$.
\end{remark}

\begin{example}\label{ExDualCHMotive}
An effective dual $A$-motive of rank $1$ with $\sdtau_\dM(\sdsigma^\ast \dM)=J\cdot \dM$ is called a \emph{dual Carlitz-Hayes $A$-motive}. Clearly, $\uldM(\UOne(1))$ is a dual Carlitz-Hayes $A$-motive for any (non-dual) Carlitz-Hayes $A$-motive $\UOne(1)$.
Therefore, Example~\ref{ExampleCHMotive} proves the existence of dual Carlitz-Hayes $A$-motives and that they are all isomorphic in $\dualAMotCatIsog$. So we may denote any one of them by $\check\UOne(1)$. We also define $\check\UOne(n):=\check\UOne(1)^{\otimes n}$ for $n>0$ and $\check\UOne(n):=\check\UOne(-n)\dual$ for $n<0$. 

If $C=\BP^1_{\BF_q}$, $A=\BF_q[t]$ and $\theta=\charmorph(t)\in\BC$, again all dual Carlitz-Hayes $A$-motives are already in $\dualAMotCat$ isomorphic to the \emph{dual Carlitz $t$-motive} with $\dM=\BC[t]$ and $\sdtau_\dM=t-\theta$. The latter is obtained via the functor $\uldM(\,.\,)$ from the Carlitz $t$-motive $\ulM=(\BC[t],\tau_M=t-\theta)$ from Example~\ref{ExampleCHMotive}.

Every dual $A$-motive is isomorphic to the tensor product of an effective dual $A$-motive and a power of a dual Carlitz-Hayes $A$-motive.
In fact, if $\uldM$ is a dual $A$-motive with $\sdtau_{\dM}(\sdsigma^\ast \dM)\subset J^{-d}\cdot \dM$, then $\uldN:=\uldM\otimes\check\UOne(1)^{\otimes d}$ satisfies $\sdtau_\dN(\sdsigma^\ast \dN)\subset \dN$; hence, $\uldN$ is effective and $\uldM\cong\uldN\otimes\check\UOne(1)^{\otimes -d}$. Note that $\rk\uldN=\rk\uldM$ and $\dim\uldN=\dim\uldM+d\cdot\rk\uldM$.
\end{example}

\subsection{Purity and mixedness}\label{SectDualAMotPurity}

As in Section~\ref{SectAMotPurity} we fix a uniformizing parameter $z\in Q=\BF_q(C)$ of $C$ at $\infty$ and assume that $\infty\in C(\BF_q)$. We denote the unique point on $C_\BC$ above $\infty\in C$ by $\infty_{\SSC\BC}$. The completion of the local ring of $C_\BC$ at $\infty_{\SSC\BC}$ is canonically isomorphic to $\BC\dbl z\dbr$.

\begin{definition}\label{DualDef2.1}
\begin{enumerate}
\item 
A dual $A$-motive $\uldM=(\dM,\sdtau_\dM)$ is called \emph{pure} if $\dM\otimes_{A_\BC}\BC\dpl z\dpr$ contains a $\BC\dbl z\dbr$-lattice $\dM_\infty$ such that for some integers $d,r$ with $r>0$ the map
\[
\sdtau_\dM^r\;:=\;\sdtau_\dM\circ\sdsigma^*(\sdtau_\dM)\circ\ldots\circ\sdsigma^{r-1*}(\sdtau_\dM)\colon\; \sdsigma^{r*} \dM\otimes_{A_\BC}\Quot(A_\BC)\;\isoto\; \dM\otimes_{A_\BC}\Quot(A_\BC)
\]
induces an isomorphism $z^d\sdtau_\dM^r\colon \sdsigma^{r*}\dM_\infty\isoto \dM_\infty$. Then the \emph{weight} of $\uldM$ is defined as $\weight\uldM=-\frac{d}{r}$.
\item
A dual $A$-motive $\uldM$ is called \emph{mixed} if it possesses an increasing \emph{weight filtration} by saturated dual $A$-sub-motives $W_{\mu\,}\uldM$ for $\mu\in\BQ$ (i.e. $W_\mu \dM\subset \dM$ is a saturated $A_\BC$-submodule) such that all graded pieces $\Gr_\mu^W\uldM:=W_{\mu\,}\uldM/\bigcup_{\mu'<\mu}W_{\mu'}\uldM$ are pure dual $A$-motives of weight $\mu$ and $\sum_{\mu\in\BQ}\rk\Gr_\mu^W\uldM=\rk\uldM$. 
\item The full subcategory of $\dualAMotCat$ consisting of mixed dual $A$-motives is denoted $\dualAMMotCat$. The full subcategory of $\dualAMotCatIsog$ consisting of mixed dual $A$-motives is denoted $\dualAMMotCatIsog$. 
\end{enumerate}
\end{definition}

\begin{example}\label{ExDualCarlitz}
For example the dual Carlitz $t$-motive $\uldM=(\BC[t],\sdtau_\dM=t-\theta)$ is pure of weight $-1$ with $\dM_\infty=\BC\dbl z\dbr$ on which $z\sdtau_\dM=1-\theta z$ is an isomorphism, where $z=\tfrac{1}{t}$.
\end{example}

\begin{remark}\label{RemDualAMotWts}
(a) The \emph{weights of $\uldM$} are the jumps of the weight filtration; that is, those real numbers $\mu$ for which 
\[
\TS\bigcup_{\mu'<\mu}W_{\mu'}\uldM \;\subsetneq\;\bigcap_{\tilde\mu>\mu}W_{\tilde\mu\,}\uldM\,.
\]
The condition $\sum_{\mu\in\BQ}\rk\Gr_\mu^W\uldM=\rk\uldM$ is equivalent to the conditions that all weights lie in $\BQ$, that $W_{\mu\,}\uldM \;=\;\bigcap_{\tilde\mu>\mu}W_{\tilde\mu\,}\uldM$ for all $\mu\in\BQ$, that $W_{\mu\,}\uldM=(0)$ for $\mu\ll0$, and that $W_{\mu\,}\uldM=\uldM$ for $\mu\gg0$; compare Remark~\ref{RemQFiltr}.

\medskip\noindent
(b) Every pure dual $A$-motive of weight $\mu$ is also mixed with $W_{\mu'}\uldM=(0)$ for $\mu'<\mu$, and $W_{\mu'}\uldM=\uldM$ for $\mu'\ge\mu$, and $\Gr_\mu^W\uldM=\uldM$.
\end{remark}

Proposition~\ref{PropDualizing} extends to mixed (dual) $A$-motives as follows.

\begin{proposition}\label{PropDualMixed}
 A dual $A$-motive $\uldM$ is mixed (pure) if and only if the corresponding $A$-motive $\ulM(\uldM)$ from Proposition~\ref{PropDualizing} is mixed (pure). In that case the weights of $\ulM(\uldM)$ are the negatives of the weights of $\uldM$. More precisely, if $\uldM$ is mixed with weights $\mu_1<\ldots<\mu_n$ then the weight filtration on $\ulM=(M,\tau_M)=\ulM(\uldM):=\bigl(\Hom_{A_\BC}(\sdsigma^\ast \dM,\,\Omega^1_{A_\BC/\BC}),\sdtau_\dM\dual\bigr)$ is given by 
\begin{equation}\label{EqWeightOfDualM}
W_{-\mu\,}\ulM \; := \; \bigl\{\,m\in M=\Hom_{A_\BC}(\sdsigma^*\dM,\,\Omega^1_{A_\BC/\BC})\colon m(\sdsigma^*W_{\mu'}\dM)=0\es\text{for all }\mu'<\mu\,\bigr\}\,,
\end{equation}
that is, by $W_{-\mu}\ulM=(0)$ for all $-\mu<-\mu_n$, by $W_{-\mu}\ulM=\ker\bigl(\ulM\onto\ulM(W_{\mu_i}\uldM)\bigr)$ for all $-\mu_{i+1}\le-\mu<-\mu_i$, and by $W_{-\mu_1}\ulM=\ulM$. In particular the functors $\uldM\mapsto\ulM(\uldM)$ and $\ulM\mapsto\uldM(\ulM)$ from Proposition~\ref{PropDualizing} induce exact tensor-anti-equivalences of categories $\dualAMMotCat\longleftrightarrow \AMMotCat$ and $\dualAMMotCatIsog\longleftrightarrow \AMMotCatIsog$.
\end{proposition}

\begin{proof} 
First assume that $\uldM$ is pure of weight $\mu=-\frac{d}{r}$. This means that there is a $\BC\dbl z\dbr$-lattice $\dM_\infty\subset\dM\otimes_{A_\BC}\BC\dpl z\dpr$ such that $z^d\sdtau_\dM^r$ is an isomorphism $\sdsigma^{r\ast}\dM_\infty\isoto\dM_\infty$. Then $M_\infty:=\Hom_{\BC\dbl z\dbr}(\sdsigma^*\dM_\infty,\BC\dbl z\dbr dz)$ is a $\BC\dbl z\dbr$-lattice in 
\[
M(\uldM)\otimes_{A_\BC}\BC\dpl z\dpr\;=\;\Hom_{A_\BC}(\sdsigma^\ast \dM,\,\Omega^1_{A_\BC/\BC})\otimes_{A_\BC}\BC\dpl z\dpr\;=\;\Hom_{\BC\dpl z\dpr}(\sdsigma^\ast \dM\otimes_{A_\BC}\BC\dpl z\dpr,\,\BC\dpl z\dpr dz)
\]
such that $\ssigma^{r-1*}(z^d\sdtau_{\dM}^r)\dual=z^d\stau_{M(\uldM)}^r$ defines an isomorphism $\ssigma^{r\ast} M_\infty\isoto M_\infty$. Therefore, $\ulM(\uldM)$ is pure of weight $-\mu=\frac{d}{r}$. 

Conversely, a $\BC\dbl z\dbr$-lattice $M_\infty\subset M(\uldM)\otimes_{A_\BC}\BC\dpl z\dpr$ with $z^d\stau_{M(\uldM)}^r\colon\ssigma^{r\ast} M_\infty\isoto M_\infty$ induces the lattice $\dM_\infty:=\Hom_{\BC\dbl z\dbr}(\ssigma^*M_\infty,\BC\dbl z\dbr dz)\subset\dM\otimes_{A_\BC}\BC\dpl z\dpr$ with $\sdsigma^{r-1*}(z^d\stau_{M(\uldM)}^r)\dual=z^d\sdtau_{\dM}^r\colon\sdsigma^{r\ast}\dM_\infty\isoto\dM_\infty$. This proves that $\uldM$ is pure of weight $\mu$ if and only if $\ulM(\uldM)$ is pure of weight $-\mu$.

Now we consider a mixed dual $A$-motive $\uldM$. Applying the exact contravariant functor $\uldM\mapsto\ulM(\uldM)$ gives for all $\mu$ exact sequences
\[
0\longrightarrow\ulM(\Gr^W_{\mu\,}\uldM)\longrightarrow\ulM(W_{\mu\,}\uldM)\longrightarrow\ulM\Bigl(\TS\bigcup\limits_{\mu'<\mu}W_{\mu'}\uldM\Bigr)\longrightarrow 0
\] 
Thus we can define an increasing filtration $W_\bullet\ulM$ of $\ulM$ by saturated $A$-sub-motives by letting
\begin{eqnarray*}
\TS W_{-\mu\,}\ulM & := & \ker\Bigl(\ulM\onto\ulM\bigl(\TS\bigcup\limits_{\mu'<\mu}W_{\mu'}\uldM\bigr)\Bigr) \\
& = & \bigl\{\,m\in M=\Hom_{A_\BC}(\sdsigma^*\dM,\,\Omega^1_{A_\BC/\BC})\colon m(\sdsigma^*W_{\mu'}\dM)=0\es\text{for all }\mu'<\mu\,\bigr\}\,. 
\end{eqnarray*}
More explicitly, if $\mu_1<\ldots<\mu_n$ are the jumps of the weight filtration $W_\bullet\uldM$, we set in addition $\mu_0:=-\infty, \mu_{n+1}:=+\infty$, and $W_{\mu_0\,}\uldM=(0)$. Then $W_{\mu_i\,}\uldM=W_{\mu'}\uldM\subsetneq W_{\mu_{i+1}\,}\uldM$ for all $\mu_i\le\mu'<\mu_{i+1}$ and hence, for any $\mu$ with $\mu_i<\mu\le\mu_{i+1}$ we have the equalities $\bigcup\limits_{\mu'<\mu}W_{\mu'}\uldM=W_{\mu_i\,}\uldM$ and $W_{-\mu\,}\ulM=\ker\bigl(\ulM\onto\ulM(W_{\mu_i\,}\uldM)\bigr)$. In particular, if $\mu_i\le\mu<\tilde\mu\le\mu_{i+1}$, then 
\begin{equation}\label{EqWtsEandM}
W_{-\tilde\mu\,}\ulM\;=\;\ker\bigl(\ulM\onto\ulM(W_{\mu_i\,}\uldM)\bigr)\;=\;\ker\bigl(\ulM\onto\ulM(W_{\mu\,}\uldM)\bigr)\,.
\end{equation}
This yields the following diagram with exact rows
\[
\xymatrix {
0 \ar[r] & *!U(0.5) \objectbox{\TS\bigcup\limits_{-\tilde\mu<-\mu}W_{-\tilde\mu\,}\ulM} \ar[r] \ar@{^{ (}->}[d] & \ulM \ar[r] \ar@{=}[d] & \ulM(W_{\mu\,}\uldM) \ar[r] \ar@{->>}[d] & 0 \\
0 \ar[r] & W_{-\mu\,}\ulM \ar[r] & \ulM \ar[r] & **{!U(0.3) =<6.7pc,2pc>} \objectbox{\DS\ulM\bigl(\TS\bigcup\limits_{\mu'<\mu}W_{\mu'}\uldM\bigr)} \ar[r] & 0 \,.
}
\]
By the snake lemma $\Gr_{-\mu}^W\ulM\cong\ker\Bigl(\ulM(W_{\mu\,}\uldM)\onto\ulM(\bigcup\limits_{\mu'<\mu}W_{\mu'}\uldM)\Bigr)\cong\ulM(\Gr_\mu^W\uldM)$ is pure of weight $-\mu$ and $\ulM(\uldM)$ is mixed with weight filtration $W_\bullet\ulM$ which jumps at $-\mu_n<\ldots<-\mu_1$. A similar argument for the inverse functor $\ulM\mapsto\uldM(\ulM)$ shows that conversely $\uldM$ is mixed provided $\ulM(\uldM)$ is mixed. This proves the proposition.
\end{proof}

The following proposition follows from Propositions~\ref{PropDualMixed}, \ref{PropPure} and \ref{PropWeights}, but can also be proved directly along the lines of Propositions~\ref{PropPure} and \ref{PropWeights} without using (non-dual) $A$-motives.

\begin{proposition}\label{PropDualPure}
Let $\uldM$ be a dual $A$-motive and consider the \emph{$z$-isocrystal} 
\[
\uldHM\;:=\;(\wh M,\wh\tau)\;:=\;\bigl(\dM\otimes_{A_\BC}\BC\dpl z\dpr\,,\,\ssigma^\ast\sdtau_\dM^{-1}\otimes\id\colon \ssigma^\ast\wh M\isoto\wh M\bigr)\,.
\]
Then $\uldHM$ is isomorphic to $\bigoplus_i\ulHM_{d_i,r_i}$ where for $d,r\in\BZ,r>0,(d,r)=1,m:=\lceil\frac{d}{r}\rceil$ we set 
\[
\ulHM_{d,r}\;:=\;\Bigl(\BC\dpl z\dpr^{\oplus r},\wh\stau=\left( \raisebox{6.2ex}{$
\xymatrix @C=0pc @R=0.3pc {
0 \ar@{.}[drdrdr] & & z^{-m}\ar@{.}[dr] & \\
& & & z^{-m} \\
z^{1-m} \ar@{.}[dr] & & & \\
 & z^{1-m}  & & 0\\
}$}
\right)\Bigr)
\]
and where in the matrix the term $z^{1-m}$ occurs exactly $mr-d$ times. In particular,
\begin{enumerate}
\item \label{PropDualPure_a}
$\uldM$ is pure of weight $\mu$ if and only if $\mu=\frac{d_i}{r_i}$ for all $i$.
\item \label{PropDualPure_b}
$\uldM$ is mixed if and only if the filtration $\DS W_{\mu\,}\ulHM:=\bigoplus_{\frac{d_i}{r_i}\le\mu}\ulHM_{d_i,r_i}$ comes from a filtration of $\uldM$ by saturated dual $A$-sub-motives $\wt W_{\mu\,}\uldM\subset\uldM$ with $W_{\mu\,}\ulHM=\wh{\wt W_{\mu\,}\uldM}$. In this case the filtration $\wt W_{\mu\,}\uldM$ equals the weight filtration $W_{\mu\,}\uldM$ of $\uldM$ and the $\tfrac{d_i}{r_i}$ are the weights of $\uldM$. In particular, the weight filtration of a mixed dual $A$-motive $\uldM$ is uniquely determined by $\uldM$.
\item \label{PropDualPure_c}
Any dual $A$-sub-motive $\uldM'\into\uldM$ and dual $A$-quotient motive $\df\colon \uldM\onto\uldM''$ of a pure (mixed) dual $A$-motive $\uldM$ is itself pure (mixed) of the same weight(s), (by letting $W_{\mu\,}\uldM':=\uldM'\cap W_{\mu\,}\uldM$, and $W_{\mu\,}\uldM''$ be the saturation of $\df(W_{\mu\,}\uldM)$ inside $\uldM''$, if $\uldM$ is mixed).
\item \label{PropDualPure_f}
Any dual $A$-motive which is isomorphic in $\dualAMotCatIsog$ to a pure (mixed) dual $A$-motive is itself pure (mixed).
\item\label{PropDualPure_h}
The weight of a pure dual $A$-motive $\uldM$ is $\weight\uldM=-(\dim\uldM)/(\rk\uldM)$. The tensor product of two pure dual $A$-motives $\uldM$ and $\uldN$ is again pure of weight $(\weight\uldM)+(\weight\uldN)$. 
\item \label{PropDualPure_g}
The category $\dualAMMotCatIsog$ is a full $Q$-linear (non-neutral) Tannakian subcategory of $\dualAMotCatIsog$, and in particular, a rigid abelian tensor category.
\item \label{PropDualPure_d}
Any morphism $\df\colon \uldM'\to\uldM$ between mixed dual $A$-motives satisfies $\df(W_{\mu\,}\uldM')\subset W_{\mu\,}\uldM$. More precisely, the saturation of $\df(W_{\mu\,}\uldM')$ inside $\df(\uldM')$ equals $\df(\uldM')\cap W_{\mu\,}\uldM$.
\item \label{PropDualPure_i} If $\uldM$ is effective and $\dM$ is a finitely generated module over the skew-polynomial ring $\BC\{\sdtau\}$, where $\sdtau$ acts on $\dM$ through $\dm\mapsto \sdtau_M(\sdsigma^\ast\dm)$, then $d_i<0$ for all $i$.
\item \label{PropDualPure_e}
If $\uldM$ is effective and $d_i/r_i\ge-n$ for all $i$, then $\uldM$ extends to a locally free sheaf $\oldM$ on $C_\BC$ with $\sdtau\colon \sdsigma^\ast\oldM\to\oldM(\mbox{$n\cdot\infty_{\SSC\BC}$})$, where the notation $(\mbox{$n\cdot\infty_{\SSC\BC}$})$ means that we allow poles at $\infty_{\SSC\BC}$ of order less than or equal to $n$. Moreover $\uldM$ is pure of weight $\mu=-\frac{d}{r}$ with $(d,r)=1$ if and only if there is an $\oldM$ such that in addition, $z^d\sdtau_\dM^r$ is an isomorphism $\sdsigma^{r\ast}\oldM_\infty\isoto\oldM_\infty$ on the stalks at $\infty_{\SSC\BC}$.
\end{enumerate}
\end{proposition}

\begin{remark}\label{RemQuillenDualAMMot}
The category $\dualAMMotCat$ is an exact category in the sense of Quillen~\cite[\S2]{Quillen} if one calls a sequence of mixed dual $A$-motives \emph{exact} when its underlying sequence of $A_\BC$-modules is exact; compare Remarks~\ref{RemQuillenDualAMot}(b) and \ref{RemQuillenAMMot}.
\end{remark}

We have seen in Example \ref{ExNotMixed} that not every $A$-motive is mixed; thus the same is true for dual $A$-motives.

\begin{example}\label{ExDualMixed2.9}
Let $C=\BP^1_{\BF_q}$, $A=\BF_q[t]$, $z=\frac{1}{t}$, $\theta=\charmorph(t)=\frac{1}{\zeta}\in\BC$, and recall the mixed $A$-motive $\ulM$ from Example~\ref{Ex2.9} with $\Gr^W_1 \ulM=W_1 \ulM=(A_\BC,t-\theta)$ and $\Gr^W_3 \ulM=(A_\BC,(t-\theta)^3)$. Via an identification $\Omega^1_{A_\BC/\BC}=\BC[t]dt\cong\BC[t]$ its corresponding dual $A$-motive is isomorphic to 
\[
\uldM=\Bigl(A_{\BC}^{\oplus 2},\check\Phi:=\left(\begin{array}{cc}t-\theta&0\\b&(t-\theta)^3 \end{array}\right)\Bigr).
\]
As in the previous proposition, we set
\[
 \TS W_{-1\,}\uldM\;:=\;\ker\left(\uldM\onto\uldM\left(\bigcup_{\mu'<1}W_{\mu'}\ulM\right)\right)\;
 =\;\uldM
\]
and
\[
 \TS W_{-3\,}\uldM\;:
 =\;\ker\left(\uldM\stackrel{(1,0)}{\onto}\uldM(W_{1}\ulM)\right)
 =\;\begin{pmatrix} 0 \\ 1\end{pmatrix}\cdot\left(A_\BC, (t-\theta)^3\right),
\]
such that $\Gr^W_{-1} \uldM\cong\left(A_\BC,(t-\theta)\right)$ and $\Gr^W_{-3} \uldM\cong\left(A_\BC,(t-\theta)^3\right)$. Thus $\uldM$ has weights -3 and -1.
\end{example}

\medskip

\subsection{Uniformizability}\label{SectDualAMotUniformizability}

Recall the notation introduced at the beginning of Section \ref{SectAMotUniformizability}.

\begin{definition}\label{DualDefLambda}
For a dual $A$-motive $\uldM$, we define the \emph{$\sdtau$-invariants}
\[
\Lambda(\uldM) \;:=\; \bigl(\dM\otimes_{A_\BC}\CO(\FC_\BC\setminus\Disc)\bigr)^\sdtau\;:=\; \bigl\{\,\dm \in \dM\otimes_{A_\BC}\CO(\FC_\BC\setminus\Disc):\es \sdtau_\dM(\sdsigma^\ast \dm )=\dm \,\bigr\}\,,
\]
and let $\Hodge_1(\uldM):=\Lambda(\uldM)\otimes_A Q$.
\end{definition}

Since the ring of $\sdsigma$-invariants in $\CO(\FC_\BC\setminus\Disc)$ equals $A$, the set $\Lambda(\uldM)$ is an $A$-module. It is finite projective of rank at most equal to $\rank\uldM$ by the analog for dual $A$-motives of \cite[Lemma~4.2(b)]{BoeckleHartl}. Therefore, also $H(\uldM)$ is a finite dimensional $Q$-vector space. 

\begin{definition}\label{DefDualUnifGlobal}
A dual $A$-motive $\uldM$ is called \emph{uniformizable} (or \emph{rigid analytically trivial}) if the natural homomorphism
\[
\TS h_\uldM\colon \Lambda(\uldM) \otimes_A \CO(\FC_\BC\setminus\Disc) \longto \dM\otimes_{A_\BC}\CO(\FC_\BC\setminus\Disc)\,,\quad\lambda\otimes f\mapsto f\cdot\lambda,
\]
is an isomorphism. The full subcategory of $\dualAMotCatIsog$ consisting of all uniformizable dual $A$-motives is denoted $\dualAUMotCatIsog$. The full subcategory of $\dualAMotCatIsog$ consisting of all uniformizable mixed dual $A$-motives is denoted $\dualAMUMotCatIsog$.
\end{definition}

\begin{remark}\label{RemPapanikolasCateg}
In \cite[3.4.10]{Papanikolas} Papanikolas defines a neutral Tannakian category $\CT$ over $\BF_\sq(t)$ which is equivalent to $\dualAUMotCatIsog$ if $A=\BF_\sq[t]$. This can be seen as follows. Let $\uldM$ be an object of $\dualAUMotCatIsog$. Then $\uldM\otimes_{A_\BC}\Quot(A_\BC)$ is a rigid analytically trivial (dual) pre-$t$-motive in the language of \cite[\S3.3.1]{Papanikolas}. The latter form a neutral Tannakian category $\CR$ over $Q$ by \cite[Theorem~3.3.15]{Papanikolas} and $\uldM\mapsto\uldM\otimes_{A_\BC}\Quot(A_\BC)$ is a fully faithful functor $\dualAUMotCatIsog\to\CR$. Papanikolas defines the category $\CT$ as the Tannakian subcategory of $\CR$ generated by the effective dual $A$-motives in $\dualAUMotCatIsog$. It thus follows from Example~\ref{ExDualCHMotive} and Proposition~\ref{PropUnifDualAMotive} below that $\CT$ coincides with the image of $\dualAUMotCatIsog$ in $\CR$. By Proposition~\ref{PropDualizingUnif} below, the category $\dualAUMotCatIsog$ is also anti-equivalent to $\AUMotCatIsog$ and hence also to Taelman's category $t\CM_{\rm a.t.}\open$ by Remark~\ref{RemTaelman2}.
\end{remark}

\begin{lemma}\label{LemmaDualUniformizable}
Let $\uldM$ be a dual $A$-motive of rank $r$. 
\begin{enumerate}
\item \label{LemmaDualUniformizableA} 
The homomorphism $h_\uldM$ is injective.
\item \label{LemmaDualUniformizableB}
$\uldM$ is uniformizable if and only if $\rank_A\Lambda(\uldM)=r$.
\item \label{LemmaDualUniformizableC}
 If $\uldM$ is uniformizable then the following sequence of $A$-modules is exact
\[
\xymatrix @C+1pc @R=0pc {
0 \ar[r] & \Lambda(\uldM) \ar[r]^{h_\uldM\qquad\quad} & \dM\otimes_{A_\BC}\CO(\FC_\BC\setminus\Disc) \ar[r] & \dM\otimes_{A_\BC}\CO(\FC_\BC\setminus\Disc) \ar[r] & 0\,. \\
& & \dm \ar@{|->}[r] & \sdtau_\dM(\sdsigma^*\dm)-\dm
}
\]
\end{enumerate}
\end{lemma}
\begin{proof}
Assertions \ref{LemmaDualUniformizableA} and \ref{LemmaDualUniformizableB} can be proved for general $A$ as in \cite[Lemma~3.3.7 and Proposition~3.3.8]{Papanikolas}; see also the proof of Lemma~\ref{LemmaUniformizable}.

To establish \ref{LemmaDualUniformizableC} we must prove exactness on the right. We choose a finite flat ring homomorphism $\BF_q[t]\into A$ of degree $d$. Then $\CO(\FC_\BC\setminus\Disc)=A\otimes_{\BF_q[t]}\BC\langle t\rangle$; see \eqref{EqC<t>}. We view $\dM$ as a (locally) free $\BC[t]$-module of rank $dr$ and $\Lambda(\uldM)$ as a (locally) free $\BF_q[t]$-module of rank $dr$. With respect to a basis of the latter we identify $\dM\otimes_{A_\BC}\CO(\FC_\BC\setminus\Disc)\cong\Lambda(\uldM)\otimes_A\CO(\FC_\BC\setminus\Disc)\cong\BC\langle t\rangle^{\oplus dr}$. In this basis $\sdtau_\dM$ is given by the identity matrix. Now let $\dm\in\dM\otimes_{A_\BC}\CO(\FC_\BC\setminus\Disc)$ be given as $\sum_i b_it^i$ with $b_i=(b_{i,1},\ldots,b_{i,dr})^T\in\BC^{dr}$. Since $\BC$ is algebraically closed there is for every $i$ and $j$ a $c_{i,j}\in\BC$ with $c_{i,j}^q-c_{i,j}=b_{i,j}$. If $|b_{i,j}|<1$ we may even take $c_{i,j}=-\sum_{n=0}^\infty b_{i,j}^{q^n}$, whence $|c_{i,j}|=|b_{i,j}|$. With this choice $\dm':=\sum_{i=0}^\infty(c_{i,1},\ldots,c_{i,dr})^Tt^i\in\BC\langle t\rangle^{\oplus dr}$ satisfies $\sdtau_\dM(\sdsigma^*\dm')-\dm'=\dm$. This proves \ref{LemmaDualUniformizableC}.
\end{proof}

\begin{proposition}\label{PropDualizingUnif}
 A dual $A$-motive $\uldM$ is uniformizable if and only if the corresponding $A$-motive $\ulM:=\ulM(\uldM)$ from Proposition~\ref{PropDualizing} is uniformizable. Moreover, 
\[
h_\uldM\;=\;(\ssigma^\ast h_\ulM\dual)^{-1}\;:=\;\Hom_{\CO(\FC_\BC\setminus\Disc)}\bigl(\ssigma^\ast h_\ulM,\,\Omega^1_{A_\BC/\BC}\otimes_{A_\BC}\CO(\FC_\BC\setminus\Disc)\bigr)^{-1}
\]
and $\Lambda(\uldM)\cong\Hom_A(\Lambda(\ulM),\Omega^1_{A/\BF_q})$ under the perfect pairing
\[
\Lambda(\uldM)\times\Lambda(\ulM)\,\longto\, \Omega^1_{A/\BF_q},\quad(\check\lambda,\lambda)\,\longmapsto\, h_\uldM(\check\lambda)\bigl(\sigma^*h_\ulM(\lambda)\bigr)
\]
obtained from $\dM=\Hom_{A_\BC}(\sigma^*M,\,\Omega^1_{A_\BC/\BC})$. In particular, the functors $\uldM\mapsto\ulM(\uldM)$ and $\ulM\mapsto\uldM(\ulM)$ from Proposition~\ref{PropDualizing} restrict to exact tensor-anti-equivalences $\dualAUMotCatIsog\longleftrightarrow\AUMotCatIsog$ and $\dualAMUMotCatIsog\longleftrightarrow\AMUMotCatIsog$.
\end{proposition}

\begin{proof}
 We assume $\ulM$ is uniformizable, that is
\[
 \TS h_\ulM\colon \Lambda(\ulM) \otimes_A \CO(\FC_\BC\setminus\Disc) \isoto M\otimes_{A_\BC}\CO(\FC_\BC\setminus\Disc)\,,\quad\lambda\otimes f\mapsto f\cdot\lambda,
\]
is an isomorphism. Applying $\ssigma^\ast$ and $\Hom_{\CO(\FC_\BC\setminus\Disc)}\bigl(\ .\ ,\,\Omega^1_{A_\BC/\BC}\otimes_{A_\BC}\CO(\FC_\BC\setminus\Disc)\bigr)$ yields an isomorphism
\[
 \TS \ssigma^\ast h_\ulM\dual\colon  \dM\otimes_{A_\BC}\CO(\FC_\BC\setminus\Disc)\isoto\Hom_A(\Lambda(\ulM),\Omega^1_{A/\BF_q}) \otimes_A \CO(\FC_\BC\setminus\Disc)\,.
\]
Since the $\sdtau$-invariants of $\uldM=\uldM(\ulM)$ are
\[
\Lambda\bigl(\uldM(\ulM)\bigr) = \bigl(\dM \otimes_{A_\BC}\CO(\FC_\BC\setminus\Disc)\bigr)^\sdtau \cong \Hom_A(\Lambda(\ulM),\Omega^1_{A/\BF_q}),
\]
$h_\uldM:=(\ssigma^\ast h_\ulM\dual)^{-1}$ provides a rigid analytic trivialization for $\uldM$. 

The converse assertion follows similarly and the statement about the exact tensor-anti-equi\-valence follows from Propositions~\ref{PropDualizing} and \ref{PropDualMixed}.
\end{proof}

\begin{lemma}\label{LemmaDualUniformizableBF_q[t]}
Let $C=\BP^1_{\BF_q}$, $A=\BF_q[t]$, $A_\BC=\BC[t]$ and $\theta=\charmorph(t)$. Then $\CO(\FC_\BC\setminus\Disc\bigr)=\BC\langle t\rangle$; see \eqref{EqC<t>}. Let $\check\Phi=(\check\Phi_{ij})_{ij}\in\GL_r\bigl(\BC[t][\tfrac{1}{t-\theta}]\bigr)$ represent $\sdtau_\dM$ with respect to a $\BC[t]$-basis $\check\CB=(\dm_1,\ldots,\dm_r)$ of $\dM$, that is $\sdtau_\dM(\sdsigma^*\dm_j)=\sum_{i=1}^r\check\Phi_{ij}\,\dm_i$. Then $\uldM$ is uniformizable if and only if there is a matrix $\check\Psi\in\GL_{r}(\BC\langle t\rangle)$ such that
\[
\sdsigma^\ast\check\Psi=\check\Psi\cdot\check\Phi\,.
\]
In that case, $\check\Psi$ is called a \emph{rigid analytic trivialization of $\check\Phi$}. It is uniquely determined up to multiplication on the left with a matrix in $\GL_r(\BF_q[t])$. The columns of $\check\Psi^{-1}$ are the coordinate vectors with respect to $\check\CB$ of an $\BF_q[t]$-basis $\check\CC$ of $\Lambda(\uldM)$. Moreover, with respect to the bases $\check\CC$ and $\check\CB$ the isomorphism $h_\uldM$ is represented by $\check\Psi^{-1}$.

If $\ulM(\uldM)=(M,\tau_M)$ and $\Phi\in\GL_r\bigl(\BC[t][\tfrac{1}{t-\theta}]\bigr)$ is the matrix representing $\tau_M$ with respect to the basis $\CB$ of $M$ which is dual to $\check\CB$, then $\Phi=\check\Phi^T$ and $\Psi:=(\sdsigma^*\check\Psi)^{-1}$ is a rigid analytic trivialization of $\Phi$.
\end{lemma}

\begin{proof}
This was proved by Papanikolas~\cite[Proposition~3.3.9]{Papanikolas} (in terms of row vectors, whereas we use column vectors); see also the proof of Lemma~\ref{LemmaUniformizableBF_q[t]}. The formula $\Psi=(\sdsigma^*\check\Psi)^{-1}$ follows from an elementary calculation.
\end{proof}

\begin{example}\label{ExDualizingCarlitz}
Let $C=\BP^1_{\BF_q}$, $A=\BF_q[t]$, $z=\frac{1}{t}$, $\theta=\charmorph(t)=\frac{1}{\zeta}\in\BC$. Via the identification $\Omega^1_{A_\BC/\BC}=\BC[t]dt\cong\BC[t],\,dt\mapsto 1$, the Carlitz $t$-motive $\ulM=(\BC[t],\tau_M=t-\theta)$ and the dual Carlitz $t$-motive $\uldM=(\BC[t],\sdtau_\dM=t-\theta)$ from Examples~\ref{ExCarlitz} and \ref{ExDualCarlitz} satisfy $\uldM\cong\uldM(\ulM)$ and $\ulM\cong\ulM(\uldM)$. Furthermore, $\ulM$ is pure of weight $1$ and $\uldM$ is pure of weight $-1$. In Example~\ref{ExCarlitzPeriod} we saw that $\ulM$ is uniformizable with $\Lambda(\ulM)=\eta\ell_\zeta^{\SSC -}\cdot\BF_q[t]$ for $\ell_\zeta^{\SSC -}:=\prod_{i=0}^\infty(1-\zeta^{q^i}t)\in \CO(\dotFC_\BC)$ and an element $\eta\in\BC$ with $\eta^{q-1}=-\zeta$. It follows that
\[
\Lambda(\uldM)\;=\;\{\check\lambda\in \CO(\FC_\BC\setminus\Disc):(t-\theta)\sdsigma^\ast(\check\lambda)\;=\;\check\lambda\}\;=\;\sigma^*(\eta\ell_\zeta^{\SSC -})^{-1}\cdot\BF_q[t]
\]
The pairing $\Lambda(\uldM)\times\Lambda(\ulM)\to \Omega^1_{A/\BF_q}=\BF_q[t]dt$, $(\check\lambda,\lambda)\mapsto h_\uldM(\check\lambda)\bigl(\sigma^*h_\ulM(\lambda)\bigr)$ sends $\bigl(\sigma^*(\eta\ell_\zeta^{\SSC -})^{-1},\eta\ell_\zeta^{\SSC -}\bigr)$ to $dt$, because $h_\uldM\bigl(\sigma^*(\eta\ell_\zeta^{\SSC -})^{-1}\bigr)=\sigma^*(\eta\ell_\zeta^{\SSC -})^{-1}$ and $\sigma^*h_\ulM(\eta\ell_\zeta^{\SSC -})=\sigma^*(\eta\ell_\zeta^{\SSC -})$.
\end{example}

Before we conclude that $\dualAUMotCatIsog$ and $\dualAMUMotCatIsog$ are Tannakian categories over $Q$ with fiber functors $\uldM\mapsto\Hodge_1(\uldM)$, we note that Proposition~\ref{PropDualizingUnif} together with Proposition~\ref{PropUnifAMotive} and Lemma~\ref{Lemma2.3} implies the following

\begin{proposition}\label{PropUnifDualAMotive}
\begin{enumerate}
\item Every dual $A$-motive which in $\dualAMotCatIsog$ is isomorphic to a uniformizable dual $A$-motive is itself uniformizable.
\item Every dual $A$-motive of rank $1$ is uniformizable.
\item 
If $\uldM$ and $\uldN$ are uniformizable dual $A$-motives, then also $\uldM\otimes\uldN$ and $\CHom(\uldM,\uldN)$ and $\uldM\dual$ are uniformizable with 
\begin{eqnarray*}
\Lambda(\uldM\otimes\uldN) & \cong & \Lambda(\uldM)\otimes_A\Lambda(\uldN) \qquad\text{and}\\[2mm]
\Lambda\bigl(\CHom(\uldM,\uldN)\bigr) & \cong & \Hom_A\bigl(\Lambda(\uldM),\Lambda(\uldN)\bigr) \qquad\text{and}\\[2mm]
\Lambda(\uldM\dual) & \cong & \Hom_A(\Lambda(\uldM),A)\,.
\end{eqnarray*}
\item
If $\uldM$ and $\uldN$ are uniformizable, the natural map $\QHom(\uldM,\uldN) \to \Hom_{Q}(\Hodge_1(\uldM),\Hodge_1(\uldN))$,
\[
\df\otimes a\;\longmapsto\; \Hodge_1(\df\otimes a)\;:=\;a\cdot\bigl(h_\uldN^{-1}\circ (\df\otimes\id)\circ h_\uldM|_{\Hodge_1(\uldM)}\bigr)
\]
for $\df\in\Hom_{\dualAMotCat}(\uldM,\uldN)$ and $a\in Q$, is injective.\qed
\end{enumerate}
\end{proposition}

\begin{lemma}\label{LemmaDual2.3}
Let $0\to\uldM'\to\uldM\to\uldM''\to0$ be a short exact sequence of dual $A$-motives. Then $\uldM$ is uniformizable if and only if both $\uldM'$ and $\uldM''$ are. In this case the induced sequence of $A$-modules $0\to\Lambda(\uldM')\to\Lambda(\uldM)\to\Lambda(\uldM'')\to0$ is exact.\qed
\end{lemma}

\begin{remark}
If a mixed dual $A$-motive $\uldM$ is uniformizable, then all filtration steps $W_\mu\uldM$ and factors $\Gr^W_\mu\uldM$ of the weight filtration of $\uldM$ are uniformizable by Lemma~\ref{LemmaDual2.3}. Therefore, $\uldM$ could equivalently be called a \emph{uniformizable mixed dual $A$-motive} or a \emph{mixed uniformizable dual $A$-motive}.
\end{remark}

Summarizing these properties of $\dualAUMotCatIsog$, we obtain the analog of Theorem~\ref{TheoremAMotTannakian}.

\begin{theorem}\label{TheoremDualAMotTannakian}
The categories $\dualAUMotCatIsog$ and $\dualAMUMotCatIsog$ of (mixed) uniformizable dual $A$-motives up to isogeny are neutral Tannakian categories over $Q$ with fiber functor $\uldM\mapsto \Hodge_1(\uldM)$.\qed
\end{theorem}

This theorem allows to associate with each (mixed) uniformizable dual $A$-motive $\uldM$ an algebraic group $\Gamma_\uldM$ over $Q$ as follows. Consider the Tannakian subcategory $\llangle\uldM\rrangle$ of $\dualAUMotCatIsog$, respectively $\dualAMUMotCatIsog$ generated by $\uldM$. By Tannakian duality \cite[Theorem~2.11 and Proposition~2.20]{DM82}, the category $\llangle\uldM\rrangle$ is tensor equivalent to the category of $Q$-rational representations of a linear algebraic group scheme $\Gamma_\uldM$ over $Q$ which is a closed subgroup of $\GL_Q(\Hodge_1(\uldM))$.

\begin{definition}\label{DefMotGpDual}
The linear algebraic $Q$-group scheme $\Gamma_\uldM$ associated with $\uldM$ is called the \emph{(motivic) Galois group of $\uldM$}.
\end{definition}

\begin{proposition}\label{PropSameGroup}
If $\uldM$ is a uniformizable mixed dual $A$-motive and $\ulM:=\ulM(\uldM)$ is the associated uniformizable mixed $A$-motive, then the motivic Galois groups $\Gamma_\uldM$ and $\Gamma_\ulM$ are canonically isomorphic.
\end{proposition}

\begin{proof}
This follows from the anti-equivalence of the categories $\dualAMUMotCatIsog\longleftrightarrow\AMUMotCatIsog$ and the compatibility of the fiber functors $\Hodge_1(\uldM)=\Hodge_1\bigl(\ulM(\uldM)\bigr)\otimes_A\Omega^1_{A/\BF_q}$ from Proposition~\ref{PropDualizingUnif}.
\end{proof}

\begin{remark}\label{RemPapanikolasGaloisGp}
Let $\uldM$ be a uniformizable dual $A$-motive. By Remark~\ref{RemPapanikolasCateg} the Tannakian subcategories $\llangle\uldM\rrangle$ of $\dualAUMotCatIsog$, respectively of Papanikolas's category $\CT$, generated by $\uldM$ are canonically equivalent. Therefore, our motivic Galois group $\Gamma_\uldM$ is canonically isomorphic to the one defined by Papanikolas~\cite[\S\,3.5.2]{Papanikolas}.
\end{remark}

\subsection{The associated Hodge-Pink structure}\label{SectDualAMotHPStr}

We keep the notation introduced at the beginning of Section \ref{SectAMotHPStr}, where we associated a mixed Hodge-Pink structure $\ulHodge^1(\ulM)$ with a uniformizable mixed effective $A$-motive $\ulM$.

\begin{proposition}\label{Prop4.23}
\label{PropMaphdM}
Let $\uldM$ be a uniformizable dual $A$-motive over $\BC$.
\begin{enumerate}
\item \label{PropMaphdMA}
Then $\Lambda(\uldM)$ equals $\bigl\{\,\dm \in \dM\otimes_{A_\BC}\CO\bigl(\dotFC_\BC\setminus\bigcup_{i\in\BN_{>0}}\Var(\ssigma^{i\ast}J)\bigr):\es \sdtau_\dM(\sdsigma^\ast \dm )=\dm \,\bigr\}$ and the isomorphisms $h_\uldM$ and $\sdsigma^*h_\uldM$ extend to isomorphisms of locally free sheaves
\begin{eqnarray*}
h_\uldM\colon\Lambda(\uldM)\otimes_A\CO_{\dotFC_\BC\setminus\bigcup_{i\in\BN_{>0}}\Var(\ssigma^{i\ast}J)} & \isoto & \dM\otimes_{A_\BC} \CO_{\dotFC_\BC\setminus\bigcup_{i\in\BN_{>0}}\Var(\ssigma^{i\ast}J)}\,,\\[2mm]
\sdsigma^*h_\uldM\colon\Lambda(\uldM)\otimes_A\CO_{\dotFC_\BC\setminus\bigcup_{i\in\BN_0}\Var(\ssigma^{i\ast}J)} & \isoto & \sdsigma^*\dM\otimes_{A_\BC} \CO_{\dotFC_\BC\setminus\bigcup_{i\in\BN_0}\Var(\ssigma^{i\ast}J)}\,,
\end{eqnarray*}
satisfying $(\sdtau_\dM\otimes\id)\circ\sdsigma^\ast h_\uldM=h_\uldM\circ(\id_{\Lambda(\uldM)}\otimes\id)$.
\item \label{PropMaphdMB}
If moreover $\uldM$ is effective, then the isomorphism $(h_\uldM)^{-1}$ extends to an injective homomorphism
\[
h_\uldM^{-1}\colon\dM\otimes_{A_\BC} \CO_{\dotFC_\BC} \es \longto \es \Lambda(\uldM)\otimes_A\CO_{\dotFC_\BC}\,,
\]
with $h_\uldM^{-1}\circ(\sdtau_\dM\otimes\id) =(\id_{\Lambda(\uldM)}\otimes\id)\circ\sdsigma^\ast h_\uldM^{-1}$ and $\coker\sdsigma^\ast h_\uldM^{-1}\otimes\BC\dbl z-\zeta\dbr=\dM/\sdtau_\dM(\sdsigma^*\dM)$. 
\end{enumerate} 
\end{proposition}

\begin{proof}
It would be possible to adapt the proof of Proposition~\ref{PropMaphM} to the dual setting.

Instead we use the associated $A$-motive $\ulM=\ulM(\uldM)$ and recall from Proposition~\ref{PropDualizingUnif} that $h_\uldM=(\ssigma^*h_\ulM\dual)^{-1}$. We deduce from Proposition~\ref{PropMaphM} that $h_\uldM$ is an isomorphism outside the discrete set $\bigcup_{i\in\BN_{>0}}\Var(\ssigma^{i\ast}J)$ and that $\sdsigma^*h_\uldM=(h_\ulM\dual)^{-1}$ is an isomorphism outside $\bigcup_{i\in\BN_0}\Var(\ssigma^{i\ast}J)$. By dualizing the equation $h_\ulM\circ(\id_{\Lambda(\ulM)}\otimes\id)=(\tau_M\otimes\id)\circ\ssigma^*h_\ulM$ and observing $\sdtau_\dM=(\stau_M)\dual$ we obtain $h_\uldM\circ(\id_{\Lambda(\uldM)}\otimes\id)=(\sdtau_\dM\otimes\id)\circ \sdsigma^\ast h_\uldM$. The given description of $\Lambda(\uldM)$ follows from that.

Moreover, if $\uldM$ is effective, then also $\ulM$ is effective. So \ref{PropMaphdMB} follows from Proposition~\ref{PropMaphM}\ref{PropMaphMB}. 
\end{proof}

\begin{remark}
Note that if $\uldM$ is effective then it is in general not true that $\Lambda(\uldM)$ is equal to the $A$-module $\bigl\{\,\dm \in \dM\otimes_{A_\BC}\CO(\dotFC_\BC):\es \sdtau_\dM(\sdsigma^\ast \dm )=\dm \,\bigr\}$. Namely, this is true if and only if $\sdtau_\dM(\sdsigma^*\dM)\supset\dM$, and hence equivalent to $\sdtau_\dM(\sdsigma^*\dM)=\dM$ by the effectivity of $\uldM$.
\end{remark}

\begin{corollary}\label{CorLambdaDualConvRadius}
In the situation of Lemma~\ref{LemmaDualUniformizableBF_q[t]} let $\check\Psi\in\GL_r(\BC\langle t\rangle)$ be a rigid analytic trivialization of $\check\Phi$. Then the entries of $\check\Psi$ and $\check\Psi^{-1}$ converge for all $t\in\BC$ with $|t|<|\theta|^{1/q}$. If $\uldM$ is effective, then the entries of $\check\Psi$ even converge for all $t\in\BC$.
\end{corollary}

\begin{proof}
In view of $J=(t-\theta)$ this follows from the fact that $(h_\uldM)^{-1}$ is represented by the matrix~$\check\Psi$.
\end{proof}

In order to encode the relative position of $\sdsigma^*\dM$ and $\dM$ under $\sdtau_\dM$ at the point $\Var(J)$, we make the following

\begin{definition}\label{DualDef2.6}
Let $\uldM$ be a uniformizable mixed dual $A$-motive with weight filtration $W_{\mu\,}\uldM$.
We set $\ulHodge_1(\uldM):=(H,W_\bullet H,\Fq)$ with
\begin{itemize}
\item $H\;:=\;\Hodge_1(\uldM)\;:=\;\Lambda(\uldM)\otimes_A Q$,
\item $W_\mu H\;:=\;\Hodge_1(W_{\mu\,}\uldM)\;=\;\Lambda(W_{\mu\,}\uldM)\otimes_A Q\;\subset\;H$ for each $\mu\in\BQ$,
\item $\Fq\,:=\,(\sdsigma^\ast h_\uldM\otimes\id_{\BC\dpl z-\zeta\dpr})^{-1}\bigl(\sdsigma^\ast \dM\otimes_{A_{\BC}}\BC\dbl z-\zeta\dbr\bigr)$.
\end{itemize}
We call $\ulHodge_1(\uldM)$ the \emph{$Q$-Hodge-Pink structure associated with $\uldM$}. (This name is justified by Theorem~\ref{ThmDualHodgeConjecture} below.) We also set $\ulHodge^1(\uldM):=\ulHodge_1(\uldM)\dual$ in $\QHodgeCat$. The functor $\ulHodge_1$ is covariant and $\ulHodge^1$ is contravariant in $\uldM$.
\end{definition}

\begin{remark}\label{DualRem2.7} 
(a) If $\uldM=\uldM(\ulE)$ is the dual $A$-motive associated with a Drinfeld $A$-module $\ulE$ then $\ulHodge_1(\uldM)\cong\ulHodge_1(\ulE)$. We will prove this more generally for a uniformizable pure (or mixed) $A$-finite Anderson $A$-module $\ulE$ in Theorem~\ref{ThmHPofEandDualM} below.

\medskip\noindent
(b) If $\uldM$ is effective, that is $\sdtau_\dM(\sdsigma^*\dM)\subset \dM$, then $\Fq\subset\Fp:=\Lambda(\uldM)\otimes_A\BC\dbl z-\zeta\dbr$. More generally, if $J^m\cdot\sdtau_\dM(\sdsigma^*\dM)\subset\dM\subset J^n\cdot\sdtau_\dM(\sdsigma^*\dM)$ for integers $n\le m$, then $(z-\zeta)^{-n}\Fp\subset\Fq\subset(z-\zeta)^{-m}\Fp$. Indeed, from Proposition~\ref{PropMaphdM} we obtain a commutative diagram of isomorphisms
\begin{equation}\label{EqhdM}
\xymatrix @R+1pc @C+7pc {
\Lambda(\uldM)\otimes_A\BC\dpl z-\zeta\dpr \ar[r]^{\sdsigma^\ast h_\uldM\otimes\id_{\BC\dpl z-\zeta\dpr}}_\cong \ar[d]_{\id_{\Lambda(\uldM)}\otimes\id_{\BC\dpl z-\zeta\dpr}}^\cong & \sdsigma^* \dM\otimes_{A_{\BC}}\BC\dpl z-\zeta\dpr \ar[d]_\cong^{\sdtau_\dM\otimes\id_{\BC\dpl z-\zeta\dpr}}\\
\Lambda(\uldM)\otimes_A\BC\dpl z-\zeta\dpr \ar[r]^{h_\uldM\otimes\id_{\BC\dpl z-\zeta\dpr}}_\cong &  \dM\otimes_{A_{\BC}}\BC\dpl z-\zeta\dpr\,.
}
\end{equation}
Here $\sdsigma^*h_\uldM\otimes\id_{\BC\dpl z-\zeta\dpr}$ is an isomorphism because the three others are. This implies 
\begin{eqnarray*}
\Fq & = & (h_\uldM\otimes\id_{\BC\dpl z-\zeta\dpr})^{-1}\circ(\sdtau_\dM\otimes\id_{\BC\dpl z-\zeta\dpr})(\sdsigma^* \dM\otimes_{A_{\BC}}\BC\dbl z-\zeta\dbr)\qquad\text{and}\\[2mm]
\Fp & = & (h_\uldM\otimes\id_{\BC\dpl z-\zeta\dpr})^{-1}(\dM\otimes_{A_{\BC}}\BC\dbl z-\zeta\dbr)\,.
\end{eqnarray*}

\medskip\noindent
(c) In terms of Definition~\ref{Def1.5} the virtual dimension of $\uldM$ is $\dim\uldM=-\deg_\Fq\ulHodge_1(\uldM)=\deg_\Fq\ulHodge^1(\uldM)$.
\end{remark}

\begin{theorem}\label{ThmHofMandDualM}
Let $\uldM$ be a uniformizable mixed dual $A$-motive and let $\ulM=\ulM(\uldM)$ be its associated uniformizable mixed $A$-motive from Proposition~\ref{PropDualizing}. Consider the $Q$-Hodge-Pink structure $\ul\Omega=(H,W_\bullet H,\Fq)$ which is pure of weight $0$ and given by $H=\Omega^1_{Q/\BF_q}=Q\,dz$ and $\Fq=\BC\dbl z-\zeta\dbr dz$. Then there are canonical isomorphisms in $\QHodgeCat$
\begin{eqnarray*}
\ulHodge^1(\ulM)\es=&\CHom(\ulHodge_1(\uldM),\ul\Omega) & =\es\ulHodge^1(\uldM)\otimes\ul\Omega \qquad\text{and}\\[2mm] 
\ulHodge_1(\uldM)\es=&\CHom(\ulHodge^1(\ulM),\ul\Omega) & =\es\ulHodge_1(\ulM)\otimes\ul\Omega\,.
\end{eqnarray*}
\end{theorem}

\begin{proof}
By Proposition~\ref{PropDualizingUnif} there is a canonical identification $\Lambda(\ulM)=\Hom_A(\Lambda(\uldM),\Omega^1_{A/\BF_q})$ which gives rise to $\Hodge^1(\ulM)=\Hom_Q(\Hodge_1(\uldM),\Omega^1_{Q/\BF_q})=\Hodge^1(\uldM)\otimes_Q\;\Omega^1_{Q/\BF_q}$. By Definition~\ref{Def1.3}\ref{Def1.3_B} the weight filtration of $\wt H:=\Hom_Q(\Hodge_1(\uldM),\Omega^1_{Q/\BF_q})$ is given by 
\[
W_{-\mu}\wt H\;=\;\bigl\{\,\lambda\in\wt H\colon\;\lambda\bigl(W_{\mu'}\Hodge_1(\uldM)\bigr)=0\es\text{for all }\mu'<\mu\,\bigr\}\,.
\]
On the other hand the weight filtration on $\Hodge^1(\ulM)$ is given by $W_{-\mu}\Hodge^1(\ulM)=\Lambda(W_{-\mu}\ulM)\otimes_AQ$. From \eqref{EqWeightOfDualM} in Proposition~\ref{PropDualMixed} we know that
\[
W_{-\mu\,}\ulM \; := \; \bigl\{\,m\in M=\Hom_{A_\BC}(\sdsigma^*\dM,\,\Omega^1_{A_\BC/\BC})\colon m(\sdsigma^*W_{\mu'}\dM)=0\es\text{for all }\mu'<\mu\,\bigr\}\,.
\]
Tensoring this with $\CO(\FC_\BC\setminus\Disc)$ over $A_\BC$ it follows from the commutative diagram
\[
\xymatrix @R-1pc @C+1pc {
\Lambda(W_{\mu'}\dM)\otimes_A \CO(\FC_\BC\setminus\Disc) \ar[r]^{\cong} \ar@{^{ (}->}[d] & \sdsigma^*W_{\mu'}\dM\otimes_{A_\BC}\CO(\FC_\BC\setminus\Disc) \ar@{^{ (}->}[d]\\
\Lambda(\dM)\otimes_A \CO(\FC_\BC\setminus\Disc) \ar[r]^{\sdsigma^*h_\uldM}_{\cong} & \sdsigma^*\dM\otimes_{A_\BC}\CO(\FC_\BC\setminus\Disc)
}
\]
that $W_{-\mu}\Hodge^1(\ulM)=\Lambda(W_{-\mu}\ulM)\otimes_AQ=W_{-\mu}\Hom_Q(\Hodge_1(\uldM),\Omega^1_{Q/\BF_q})$ for all $-\mu\in\BQ$.

Finally, since $\check\Fq=(\sdsigma^\ast h_\uldM\otimes\id_{\BC\dpl z-\zeta\dpr})^{-1}\bigl(\sdsigma^\ast \dM\otimes_{A_{\BC}}\BC\dbl z-\zeta\dbr\bigr)$ is the Hodge-Pink lattice of $\ulHodge_1(\uldM)$, the Hodge-Pink lattice $\Hom_{\BC\dbl z-\zeta\dbr}(\check\Fq,\BC\dbl z-\zeta\dbr dz)$ of $\CHom(\ulHodge_1(\uldM),\ul\Omega)$ equals by Definition~\ref{Def1.3}\ref{Def1.3_B} the image in $\Hom_Q(\Hodge_1(\uldM),\Omega^1_{Q/\BF_q})\otimes_Q\BC\dpl z-\zeta\dpr=\Hodge^1(\ulM)\otimes_Q\BC\dpl z-\zeta\dpr dz$ of the map
\begin{eqnarray*}
&\sdsigma^*h_\uldM\dual\otimes\id_{\BC\dpl z-\zeta\dpr dz}\colon & \Hom_{\BC\dbl z-\zeta\dbr}(\sdsigma^\ast \dM\otimes_{A_{\BC}}\BC\dbl z-\zeta\dbr,\BC\dbl z-\zeta\dbr dz)\\[2mm]
& & \qquad\qquad\qquad \into \es\Hom_{\BC\dpl z-\zeta\dpr}\bigl(\Hodge_1(\uldM)\otimes_Q\BC\dpl z-\zeta\dpr,\BC\dpl z-\zeta\dpr dz\bigr)\\[2mm]
& & \qquad\qquad\qquad = \es \Hodge^1(\ulM)\otimes_Q\BC\dpl z-\zeta\dpr dz\,.
\end{eqnarray*}
Since $\sdsigma^*h_\uldM\dual=h_\ulM^{-1}$ and $M=\Hom_{A_\BC}(\sdsigma^*\dM,\,\Omega^1_{A_\BC/\BC})$ by Proposition~\ref{PropDualizingUnif}, we conclude that 
\[
\Hom_{\BC\dbl z-\zeta\dbr}(\check\Fq,\BC\dbl z-\zeta\dbr dz)=(h_\ulM\otimes\id_{\BC\dpl z-\zeta\dpr})^{-1}(M\otimes_{A_\BC}\BC\dbl z-\zeta\dbr)
\]
equals the Hodge-Pink lattice of $\ulHodge^1(\ulM)$ as desired.
\end{proof}

The main theorem of \cite{HartlPink2} also holds for uniformizable mixed dual $A$-motives:

\begin{theorem}\label{ThmDualHodgeConjecture}
Consider a uniformizable mixed dual $A$-motive $\uldM$.
\begin{enumerate}
\item \label{ThmDualHodgeConjectureA}
$\ulHodge_1(\uldM)$ is locally semistable and hence indeed a $Q$-Hodge-Pink structure.
\item \label{ThmDualHodgeConjectureB}
The functor $\ulHodge_1\colon \uldM\to\ulHodge_1(\uldM)$ is a $Q$-linear exact fully faithful tensor functor from the category $\dualAMUMotCatIsog$ to the category $\QHodgeCat$.
\item\label{ThmDualHodgeConjectureC}
The essential image of $\ulHodge_1$ is closed under the formation of subquotients; that is, if $\ulH'\subset\ulHodge_1(\uldM)$ is a $Q$-Hodge-Pink sub-structure, then there exists a uniformizable mixed dual $A$-sub-motive $\uldM'\subset\uldM$ in $\dualAMUMotCatIsog$ with $\ulHodge_1(\uldM')=\ulH'$.
\item\label{ThmDualHodgeConjectureD}
The functor $\ulHodge_1$ defines an exact tensor equivalence between the Tannakian subcategory $\llangle\uldM\rrangle\subset\dualAMUMotCatIsog$ generated by $\uldM$ and the Tannakian subcategory $\llangle\ulHodge_1(\uldM)\rrangle\subset\QHodgeCat$ generated by its $Q$-Hodge-Pink structure $\ulHodge_1(\uldM)$.
\end{enumerate}
\end{theorem}

\begin{proof}
By Theorem~\ref{ThmHofMandDualM} the functor $\ulHodge_1\colon\uldM\mapsto\ulHodge_1(\uldM)$ is naturally isomorphic to the composition of functors $\uldM\mapsto\ulM(\uldM)\mapsto\CHom\bigl(\ulHodge^1\bigl(\ulM(\uldM)\bigr),\ul\Omega\bigr)$, the first of which is an exact tensor-anti-equivalence by Proposition~\ref{PropDualizingUnif}. Thus the theorem follows from Theorems~\ref{ThmPinkTannaka} and \ref{ThmHodgeConjecture}.
\end{proof}

Assertions~\ref{ThmDualHodgeConjectureC} and \ref{ThmDualHodgeConjectureD} are the function field analog of the Hodge Conjecture \cite{Hodge52,GrothendieckHodge,Deligne06}. We will discuss its consequences for the Hodge-Pink group $\Gamma_{\ulHodge_1(\uldM)}$ in Section~\ref{Sect3}.

\begin{example}\label{ExDual2.9}
To continue with Example~\ref{ExDualMixed2.9}, we let $A=\BF_q[t]$, $z=\tfrac{1}{t}$, $\theta:=\charmorph(t)=\frac{1}{\zeta}\in\BC$, and $\dM=A_\BC^{\oplus2}$ with $\sdtau_\dM=\check\Phi:=\left(\begin{array}{cc}t-\theta&0\\b&(t-\theta)^3 \end{array}\right)$. Thus $\Gr^W_{-1} \uldM\cong\left(A_\BC,(t-\theta)\right)$ and $\Gr^W_{-3} \uldM\cong\left(A_\BC,(t-\theta)^3\right)$, and $\uldM$ has weights -3 and -1.

Similarly as in Example~\ref{Ex2.9}, we set $\check{\ell}_\zeta^{\SSC -}:=\prod_{i=1}^\infty(1-\zeta^{q^i}t)=\sigma^*(\ell_\zeta^{\SSC -})\in \CO(\dotFC_\BC)$ and choose an $\eta\in\BC$ with $\eta^{q-1}=-\zeta$. Then 
\begin{eqnarray*}
\Lambda(\Gr^W_{-1}\uldM)&=&\{\lambda\in \CO(\dotFC_\BC):(t-\theta)\sdsigma^\ast(\lambda)\;=\;\lambda\}\es=\es(\eta^q\check{\ell}_\zeta^{\SSC -})^{-1}\cdot\BF_q[t],\\[2mm]
\Lambda(\Gr^W_{-3}\uldM)&=&(\eta^q\check{\ell}_\zeta^{\SSC -})^{-3}\cdot\BF_q[t],\qquad\text{and}\\[2mm]
\Lambda(\uldM)&=&\left(\begin{smallmatrix}(\eta^q\check{\ell}_\zeta^{\SSC -})^{-1}\\ (\eta^q\check{\ell}_\zeta^{\SSC -})^{-4}g\end{smallmatrix}\right)\cdot\BF_q[t]\oplus\left(\begin{smallmatrix}0\\(\eta^q\check{\ell}_\zeta^{\SSC -})^{-3}\end{smallmatrix}\right)\cdot\BF_q[t]
\end{eqnarray*}
for $g\in \CO(\dotFC_\BC)$ with $b\cdot(\eta^q\check{\ell}_\zeta^{\SSC -})^3+\sdsigma^\ast(g)=(t-\theta)\cdot g$. Note that $g=-\sigma^*(f)$ for the $f$ from Example~\ref{Ex2.9}. Putting $\lambda_1:=\left(\begin{smallmatrix}(\eta^q\check{\ell}_\zeta^{\SSC -})^{-1}\\ (\eta^q\check{\ell}_\zeta^{\SSC -})^{-4}g\end{smallmatrix}\right)$ and $\lambda_2:=\left(\begin{smallmatrix}0\\(\eta^q\check{\ell}_\zeta^{\SSC -})^{-3}\end{smallmatrix}\right)$, we get $H(\uldM)=\lambda_1\cdot Q\oplus\lambda_2\cdot Q$ and $W_{-3}H(\uldM)=\lambda_2\cdot Q$. 

Thus $\check\Psi=\left(\begin{array}{cc}(\eta^q\check{\ell}_\zeta^{\SSC -})^{-1}&0\\(\eta^q\check{\ell}_\zeta^{\SSC -})^{-4}g&(\eta^q\check{\ell}_\zeta^{\SSC -})^{-3} \end{array}\right)^{\!\!-1}\hspace{-0.5em}=\left(\begin{array}{cc}\eta^q\check{\ell}_\zeta^{\SSC -}&0\\-g&(\eta^q\check{\ell}_\zeta^{\SSC -})^3 \end{array}\right)\in \CO(\dotFC_\BC)^{2\times 2}$ gives the rigid analytic trivialization of $\check\Phi$, which represents $h_\uldM^{-1}$. According to Lemma~\ref{LemmaDualUniformizableBF_q[t]} we have $\check\Psi=(\sigma^*\Psi)^{-1}$ for the matrix $\Psi$ from Example~\ref{Ex2.9}. Now the Hodge-Pink lattice of $\ulHodge_1(\uldM)$ is described by 
\[
\Fq\;=\;\sdsigma^\ast\check\Psi\cdot\Fp\;=\;\left(\begin{array}{cc}\eta{\ell}_\zeta^{\SSC -}&0\\ -\sdsigma^\ast g&(\eta{\ell}_\zeta^{\SSC -})^{3}\end{array}\right)\cdot\Fp. 
\]
Since ${\ell}_\zeta^{\SSC -}$ has a simple zero at $z=\zeta$, one sees that $\Fp/\Fq$ (which is also isomorphic to $\coker\sdtau_\dM$) is isomorphic to $\BC\dbl z-\zeta\dbr/(z-\zeta)\oplus\BC\dbl z-\zeta\dbr/(z-\zeta)^3$ if $(t-\theta)|\sdsigma^\ast g$ (equivalently, if $(t-\theta)|b$) and isomorphic to $\BC\dbl z-\zeta\dbr/(z-\zeta)^{4}$ if $(t-\theta)\nmid \sdsigma^\ast g$ (equivalently, if $(t-\theta)\nmid b$). So the Hodge-Pink weights of $\ulHodge_1(\uldM)$ are $(-1,-3)$ or $(-4,0)$, and the weight polygon lies above the Hodge polygon with the same endpoint $WP(\uldM)\ge HP(\uldM)$ in accordance with Theorem~\ref{ThmDualHodgeConjecture}\ref{ThmDualHodgeConjectureA} and Remark~\ref{RemPolygons}.
\end{example}

\subsection{Cohomology Realizations}\label{CohdualAMot}

Let $\uldM=(\dM,\sdtau_{\dM})$ be a dual $A$-motive over $\BC$. Similarly as in Section~\ref{CohAMot}, the \emph{Betti cohomology realization} of $\uldM$ is given by
\[
\Koh_{1,\Betti}(\uldM,B):=\Lambda(\uldM)\otimes_A B
\quad\text{and}\quad \Koh_\Betti^1(\uldM,B):=\Hom_A(\Lambda(\uldM),B)
\]
for any $A$-algebra $B$. This is most useful when $\uldM$ is uniformizable, in which case both are locally free $B$-modules of rank equal to $\rk\uldM$ and $\Hodge_1(\uldM)=\Koh_{1,\Betti}(\uldM,Q)$. By Theorem~\ref{TheoremDualAMotTannakian} this realization provides for $B=Q$ an exact faithful neutral fiber functor on $\dualAUMotCatIsog$.

\medskip

Moreover, the \emph{de Rham cohomology realization} of $\uldM$ is defined to be
\[
\Koh_{1,\dR}(\uldM,\BC):=\dM/J\cdot\dM
\quad\text{and}\quad \Koh_{\dR}^1(\uldM,\BC):=\Hom_\BC(\dM/J\cdot\dM,\,\BC).
\]
We define a decreasing filtration of $\Koh_{1,\dR}(\uldM,\BC)$ by $\BC$-subspaces
\[
F^{i}\Koh_{1,\dR}(\uldM,\BC):=\text{image of }\dM \cap J^i\cdot\sdtau_\dM(\sdsigma^\ast\dM)\quad\text{in }\Koh_{1,\dR}(\uldM,\BC)\quad\text{for all }i\in\BZ\,,
\]
which we call the \emph{Hodge-Pink filtration of $\uldM$}. 

If $\uldM$ satisfies $J\cdot\dM\subset\sdtau_\dM(\sdsigma^*\dM)\subset \dM$, for example if $\uldM$ is the dual $A$-motive associated with a Drinfeld $A$-module, then 
\[
F^{-1}\;=\;\Koh_{1,\dR}(\uldM,\BC) \;\supset\; F^0\;=\;\sdtau_\dM(\sdsigma^*\dM)/J\cdot\dM \;\supset\; F^1\;=\;(0).
\]

As noted in Remark~\ref{RemQ-HPTannakian} and Example~\ref{Example1.2}(c), more useful than the Hodge-Pink filtration is actually the Hodge-Pink lattice $\Fq$, and the latter cannot be recovered from the Hodge-Pink filtration in general. We therefore propose to lift the de Rham cohomology to $\BC\dbl z-\zeta\dbr$ and define the \emph{generalized de Rham cohomology realization} of $\ulM$ by
\[
\begin{array}{llll}
\Koh_{1,\dR}(\uldM,\BC\dbl z-\zeta\dbr) & := & \dM\otimes_{A_\BC}\BC\dbl z-\zeta\dbr
 & \quad\text{and}\\[2mm]
\Koh_{1,\dR}\bigl(\uldM,\BC\dpl z-\zeta\dpr\bigr) & := & \dM\otimes_{A_\BC}\BC\dpl z-\zeta\dpr & \quad\text{and}\\[2mm]
\Koh^1_{\dR}(\uldM,\BC\dbl z-\zeta\dbr) & := & \Hom_{A_\BC}(\dM,\,\BC\dbl z-\zeta\dbr)
 & \quad\text{and}\\[2mm]
\Koh^1_{\dR}\bigl(\uldM,\BC\dpl z-\zeta\dpr\bigr) & := & \Hom_{A_\BC}\bigl(\dM,\,\BC\dpl z-\zeta\dpr\bigr)\,.
\end{array}
\]
In particular by tensoring with the morphism $\BC\dbl z-\zeta\dbr\onto\BC,\, z-\zeta\mapsto0$ we get back $\Koh_{1,\dR}(\uldM,\BC)=\Koh_{1,\dR}\bigl(\uldM,\BC\dbl z-\zeta\dbr\bigr)\otimes_{\BC\dbl z-\zeta\dbr}\BC$ and $\Koh^1_\dR(\uldM,\BC)=\Koh^1_\dR\bigl(\uldM,\BC\dbl z-\zeta\dbr\bigr)\otimes_{\BC\dbl z-\zeta\dbr}\BC$.
We define the \emph{Hodge-Pink lattices} of $\uldM$ as the $\BC\dbl z-\zeta\dbr$-submodules
\[
\begin{array}{ccccl}
\Fq^\uldM & := & (\sdtau_\dM\dual)^{-1}\bigl(\Hom_{A_\BC}(\sdsigma^*\dM,\,\BC\dbl z-\zeta\dbr)\bigr) & \subset & \Koh^1_\dR\bigl(\uldM,\BC\dpl z-\zeta\dpr\bigr)\quad\text{and}\\[2mm]
\Fq_\uldM & := & \sdtau_\dM(\sdsigma^*\dM)\otimes_{A_\BC}\BC\dbl z-\zeta\dbr & \subset & \Koh_{1,\dR}\bigl(\uldM,\BC\dpl z-\zeta\dpr\bigr)\,.
\end{array}
\]
Then the Hodge-Pink filtrations $F^i \Koh^1_\dR(\uldM,\BC)$ and $F^i \Koh_{1,\dR}(\uldM,\BC)$ of $\uldM$ are recovered as the images of $\Koh^1_\dR\bigl(\uldM,\BC\dbl z-\zeta\dbr\bigr)\cap(z-\zeta)^i\Fq^\uldM$ in $\Koh^1_\dR(\uldM,\BC)$ and of $\Koh_{1,\dR}\bigl(\uldM,\BC\dbl z-\zeta\dbr\bigr)\cap(z-\zeta)^i\Fq_\uldM$ in $\Koh_{1,\dR}(\uldM,\BC)$ like in Remark~\ref{Rem1.4}. All these structures are compatible with the natural duality between $H^1_\dR$ and $H_{1,\dR}$. The de Rham realization provides (covariant) exact faithful tensor functors
\begin{align}\label{EqDualDeRhamFiberFunctor}
& \Koh_{1,\dR}(\,.\,,\BC)\colon & \hspace{-0.9cm}\dualAMotCatIsog & \longto \es {\tt Vect}_{\BC}\,, & & \uldM \;\longmapsto\; \Koh_{1,\dR}(\uldM,\BC)\qquad\text{and}\\[2mm]
& \Koh_{1,\dR}(\,.\,,\BC\dbl z-\zeta\dbr)\colon & \hspace{-0.9cm}\dualAMotCatIsog & \longto \es {\tt Mod}_{\BC\dbl z-\zeta\dbr}\,, & & \uldM \;\longmapsto\; \Koh_{1,\dR}(\uldM,\BC\dbl z-\zeta\dbr)\,. \nonumber
\end{align}
This is clear for $\Koh_{1,\dR}(\,.\,,\BC\dbl z-\zeta\dbr)$ and for $\Koh_{1,\dR}(\,.\,,\BC)$ exactness follows from the snake lemma applied to multiplication with $z-\zeta$ on $\Koh_{1,\dR}(\,.\,,\BC\dbl z-\zeta\dbr)$. To prove faithfulness for $\Koh_{1,\dR}(\,.\,,\BC)$ note that every morphism $f\colon\uldM'\to\uldM$ can in $\dualAMotCatIsog$ be factored into $\uldM'\onto\im(f)\isoto\coim(f)\into\uldM$. If $\Koh_{1,\dR}(f,\BC)$ is the zero map the exactness of $\Koh_{1,\dR}(\,.\,,\BC)$ shows that $\Koh_{1,\dR}(\im(f),\BC)=(0)$. Since $\dim_\BC\Koh_{1,\dR}(\uldM,\BC)=\rk\uldM$ it follows that the dual $A$-motive $\im(f)$ has rank zero and therefore $\im(f)=(0)$ and $f=0$.

\medskip

Finally, let $v\in\dotC$ be a closed point. We say that $v$ is a \emph{finite place} of $C$. Let $A_v$ be the $v$-adic completion of $A$, and let $Q_v$ be the fraction field of $A_v$. Consider the $v$-adic completions $A_{\BC,v}:=\invlim A_\BC/v^nA_\BC$ of $A_\BC$ and $\dM_v:= \varprojlim \dM/v^n \dM=\dM\otimes_{A_\BC}A_{\BC,v}$ of $\dM$. Note that $\sdtau\colon \dm \mapsto\sdtau_\dM(\sdsigma^*\dm )$ for $\dm \in\dM$ induces a $\sdsigma^\ast\wh\otimes\id_{A_{v}}$-linear map $\sdtau\colon \dM_v\to\dM_v$. We let the \emph{$\sdtau$-invariants of $\dM_v$} be the $A_v$-module 
\[
\dM_v^\sdtau :=\{\dm \in\dM_v\;|\;\sdtau(\dm )=\dm \}.
\]
It is isomorphic to $A_v^{\oplus\rk\uldM}$ and the inclusion $\dM_v^\sdtau\subset\dM_v$ induces a canonical $\sdtau$-equivariant isomorphism $\dM_v^\sdtau\otimes_{A_v}A_{\BC,v}\isoto\dM_v$ by an argument similar to \cite[Proposition~6.1]{TW}. Then the \emph{$v$-adic cohomology realization of $\uldM$} is given by 
\[
\begin{array}{lllllll}
\Koh_{1,v}(\uldM,A_v) & := & \dM_v^\sdtau & \quad\text{and}\quad & \Koh_{1,v}(\uldM,Q_v) & := & \dM_v^\sdtau\otimes_{A_v}Q_v\qquad\text{and}\\[2mm]
\Koh_v^1(\uldM,A_v) & := & \Hom_{A_v}(\dM_v^\sdtau,A_v) & \quad\text{and}\quad & \Koh_v^1(\uldM,Q_v) & := & \Hom_{A_v}(\dM_v^\sdtau,Q_v).
\end{array}
\]
If $\uldM$ is defined over a subfield $L$ of $\BC$ then they carry a continuous action of $\Gal(L^\sep/L)$ and the $v$-adic realization provides (covariant) exact faithful tensor functors
\begin{eqnarray}\label{EqVAdicFiberFunctorDual}
\Koh_{1,v}(\,.\,,A_v)\colon & \dualAMotCat & \longto \es {\tt Mod}_{A_v[\Gal(L^\sep/L)]}\,,\quad \uldM \;\longmapsto\; \Koh_{1,v}(\uldM,A_v)\qquad\text{and}\\[2mm]
\Koh_{1,v}(\,.\,,Q_v)\colon & \dualAMotCatIsog & \longto \es {\tt Mod}_{Q_v[\Gal(L^\sep/L)]}\,,\quad \uldM \;\longmapsto\; \Koh_{1,v}(\uldM,Q_v)\,. \nonumber
\end{eqnarray}
This follows from the isomorphism $\Koh_{1,v}(\uldM,A_v)\otimes_{A_v}A_{\BC,v}\isoto \dM_v$ because $A_v\subset A_{\BC,v}$ is faithfully flat. Moreover, if $L$ is a \emph{finitely generated} field then
\begin{equation}\label{EqTateConjDualAMotives}
\Koh_{1,v}(\,.\,,A_v)\colon \; \Hom(\uldM,\uldM')\otimes_A A_v \;\isoto\; \Hom_{A_v[\Gal(L^\sep/L)]}\bigl(\Koh_{1,v}(\uldM,A_v),\Koh_{1,v}(\uldM',A_v)\bigr)
\end{equation}
is an isomorphism for dual $A$-motives $\uldM$ and $\uldM'$. This is the analog of the \emph{Tate conjecture} for dual $A$-motives and follows by Proposition~\ref{PropDualizing} from the analogous result \eqref{EqTateConjAMotives} of Taguchi~\cite{Taguchi95b} and Tamagawa~\cite[\S\,2]{Tamagawa} for $A$-motives.

\begin{proposition}\label{PropWeightsTateModuleDual}
Let $\uldM$ be a pure or mixed dual $A$-motive, which is defined over a \emph{finite field extension} $L$ of $Q$. Let $\CP$ be a finite place of $L$, not lying above $\infty$ or $v$, where $\uldM$ has good reduction, and let $\BF_\CP$ be its residue field. Then the geometric Frobenius $\Frob_\CP$ of $\CP$ has a well defined action on $\Koh_{1,v}(\uldM,A_v)$ and each of its eigenvalues lies in the algebraic closure of $Q$ in $\BC$ and has absolute value $(\#\BF_\CP)^\mu$ for a weight $\mu$ of $\uldM$. Dually every eigenvalue of $\Frob_\CP$ on $\Koh^1_v(\uldM,A_v)$ has absolute value $(\#\BF_\CP)^{-\mu}$ for a weight $\mu$ of $\uldM$. These eigenvalues are independent of $v$.
\end{proposition}

\noindent
{\it Remark.} The \emph{geometric Frobenius} $\Frob_\CP$ of $\CP$ is the inverse of the \emph{arithmetic Frobenius} $\Frob_\CP^{-1}$, which satisfies $\Frob_\CP^{-1}(x) \equiv x^{\#\BF_\CP}\mod\CP$ for $x\in\CO_L$.

\begin{proof}
This follows by Proposition~\ref{PropDualMixed} from the corresponding fact for $\ulM(\uldM)$ proved in Proposition~\ref{PropWeightsTateModule}.
\end{proof}

\bigskip

The morphism $h_\uldM$ from Proposition~\ref{PropMaphdM} induces comparison isomorphisms between the Betti and the $v$-adic, respectively the de Rham realizations similarly to Theorem~\ref{ThmCompIsomBettiDRAMotive}.

\begin{theorem}\label{ThmCompIsomBettiDRDualAMotive}
If $\uldM$ is a uniformizable dual $A$-motive there are canonical \emph{comparison isomorphisms}, sometimes also called \emph{period isomorphisms}
\[
h_{\Betti,\,v}\colon\Koh_{1,\Betti}(\uldM,A_v)\;=\;\Lambda(\uldM)\otimes_A A_v\;\isoto\;\Koh_{1,v}(\uldM,A_v)\,,\quad\check\lambda\otimes f\longmapsto (f\cdot\check\lambda \mod v^n)_{n\in\BN}
\]
and
\[
\begin{array}[b]{rcc@{\hspace{-0em}}cccl}
h_{\Betti,\,\dR} & := & h_\uldM\otimes\id_{\BC\dbl z-\zeta\dbr} & \colon &\Koh_{1,\Betti}\bigl(\uldM,\BC\dbl z-\zeta\dbr\bigr) & \isoto & \Koh_{1,\dR}\bigl(\uldM,\BC\dbl z-\zeta\dbr\bigr)\,,\\[2mm]
h_{\Betti,\,\dR} & := & h_\uldM\mod J & \colon &\Koh_{1,\Betti}(\uldM,\BC) & \isoto & \Koh_{1,\dR}(\uldM,\BC)\,.
\end{array}
\]
By diagram \eqref{EqhdM} the latter are compatible with the Hodge-Pink lattices, respectively the Hodge-Pink filtration provided on the Betti realization $\Koh_{1,\Betti}(\uldM,Q)=\Hodge_1(\uldM)$ via the associated Hodge-Pink structure $\ulHodge_1(\uldM)$.\qed
\end{theorem}

\begin{example}\label{ExDualCarlitzPeriod}
Let $C=\BP^1_{\BF_q}$, $A=\BF_q[t]$, $z=\frac{1}{t}$, $\theta=\charmorph(t)=\frac{1}{\zeta}\in\BC$, and let $\uldM=(\BC[t],\sdtau_\dM=t-\theta)$ be the dual Carlitz $t$-motive from Example~\ref{ExDualCHMotive}. As in Example~\ref{ExDualizingCarlitz} we obtain $\Lambda(\uldM)=(\eta^q\check\ell_\zeta^{\SSC -})^{-1}\cdot\BF_q[t]$ for $\check\ell_\zeta^{\SSC -}:=\prod_{i=1}^\infty(1-\zeta^{q^i}t)=\sigma^*(\ell_\zeta^{\SSC -})\in \CO(\dotFC_\BC)$ and an $\eta\in\BC$ with $\eta^{q-1}=-\zeta$. The comparison isomorphism $h_{\Betti,\,\dR}=h_\uldM\otimes\id_{\BC\dbl z-\zeta\dbr}$ sends the basis $(\eta^q\check\ell_\zeta^{\SSC -})^{-1}$ of $\Koh_{1,\Betti}(\uldM,\BF_q[t])=\Lambda(\uldM)$ to the element $(\eta^q\check\ell_\zeta^{\SSC -})^{-1}=\sigma^*(\eta\ell_\zeta^{\SSC -})^{-1}\in\Koh_{1,\dR}(\uldM,\BC\dbl z-\zeta\dbr)=\BC\dbl z-\zeta\dbr$, respectively to the \emph{Carlitz period} $(\eta^q\check\ell_\zeta^{\SSC -})^{-1}|_{t=\theta}=\bigl(-\zeta\eta\prod_{i=1}^\infty(1-\zeta^{q^i-1})\bigr)^{-1}\in\Koh_{1,\dR}(\ulM,\BC)=\BC$. The latter is the function field analog of the complex number $2i\pi$, the period of the multiplicative group $\BG_{m,\BQ}$, and is likewise transcendental over $\BF_q(\theta)$ by a result of Wade~\cite{Wade41}. See Example~\ref{ExCarlitzModulePeriod} for more explanations.
\end{example}

To formulate the next result let $\wh\Omega^1_{A_v/\BF_v}$ and $\wh\Omega^1_{A_{\BC,v}/\BC}$ and $\wh\Omega^1_{\BC\dbl z-\zeta\dbr/\BC}=\BC\dbl z-\zeta\dbr dz$ denote the modules of continuous differentials. They equal $\Omega_{A/\BF_q}\otimes_A A_v$, respectively $\Omega_{A/\BF_q}\otimes_A A_{\BC,v}$, respectively $\Omega_{A/\BF_q}\otimes_A\BC\dbl z-\zeta\dbr$. See also Remark~\ref{RemContDiff} below.

\begin{proposition}\label{PropCohAMot}
 Let $\uldM$ be a dual $A$-motive and let $\ulM=\ulM(\uldM)$ be the corresponding $A$-motive from Proposition~\ref{PropDualizing}. Then there are canonical isomorphisms 
\begin{eqnarray*}
\Koh^1_v(\ulM,A_v) & \cong & \Koh^1_v(\uldM,A_v)\otimes_{A_v}\wh\Omega^1_{A_v/\BF_v}\qquad\text{and}\\[2mm]
\Koh^1_\dR(\ulM,\BC\dbl z-\zeta\dbr) & \cong & \Koh^1_\dR(\uldM,\BC\dbl z-\zeta\dbr)\otimes_{\BC\dbl z-\zeta\dbr}\BC\dbl z-\zeta\dbr dz\,.
\end{eqnarray*}
The latter is compatible with the Hodge-Pink lattices. If $\uldM$ and $\ulM$ are uniformizable then in addition, 
\[
\Koh^1_\Betti(\ulM,A)\es\cong\es\Koh^1_\Betti(\uldM,A)\otimes_A\Omega^1_{A/\BF_q}
\]
and these isomorphisms are compatible with the period isomorphisms from Theorems~\ref{ThmCompIsomBettiDRAMotive} and \ref{ThmCompIsomBettiDRDualAMotive}.
\end{proposition}

\begin{proof}
If $\uldM$ and $\ulM$ are uniformizable, then the isomorphism between $\Lambda(\ulM)=\Koh_\Betti^1(\ulM,A)$ and $\Koh_\Betti^1(\uldM,A)\otimes_A\Omega^1_{A/\BF_q}=\Hom_A(\Lambda(\uldM),\Omega^1_{A/\BF_q})$ was established in Proposition~\ref{PropDualizingUnif}.

To establish the isomorphism for the $v$-adic realizations, note that $M=\Hom_{A_\BC}(\sdsigma^*\dM,\,\Omega^1_{A_\BC/\BC})$. By applying \cite[Proposition~2.10]{Eisenbud} this yields a chain of canonical isomorphisms
\begin{eqnarray}\label{EqPropCohAMot}
& M_v\;=\;\Hom_{A_\BC}(\sdsigma^*\dM,\,\Omega^1_{A_\BC/\BC})\otimes_{A_\BC}A_{\BC,v}\;=\;\Hom_{A_{\BC,v}}(\sdsigma^*\dM_v,\,\wh\Omega^1_{A_{\BC,v}/\BC})\;\isoto \\[2mm]
& \isoto\;\Hom_{A_{\BC,v}}(\dM_v^\sdtau\otimes_{A_v}A_{\BC,v},\,\wh\Omega^1_{A_{\BC,v}/\BC})\;=\;\Hom_{A_v}(\dM_v^\sdtau,\,\wh\Omega^1_{A_v/\BF_v})\otimes_{A_v}A_{\BC,v} \nonumber
\end{eqnarray}
under which the $\sigma^*$-linear endomorphism $m\mapsto\tau_M(\sigma^*m)$ of $M_v$ corresponds to the $\sigma^*$-linear endomorphism $\check\lambda\otimes f\mapsto\check\lambda\otimes\sigma^*(f)$ of $\Hom_{A_v}(\dM_v^\sdtau,\,\wh\Omega^1_{A_v/\BF_v})\otimes_{A_v}A_{\BC,v}$. By taking the invariants under these endomorphisms and observing that $(A_{\BC,v})^\tau=A_v$ we obtain the canonical isomorphism $\Koh^1_v(\ulM,A_v):=M_v^\tau\isoto\Hom_{A_v}(\dM_v^\sdtau,\,\wh\Omega^1_{A_v/\BF_v})=:\Koh^1_v(\uldM,A_v)\otimes_{A_v}\wh\Omega^1_{A_v/\BF_v}$. If moreover $\uldM$ and $\ulM$ are uniformizable this isomorphism is compatible with the period isomorphisms $h_{\Betti,\,v}$ because \eqref{EqPropCohAMot} is compatible with $h_\uldM$ and $h_\ulM=(\sdsigma^*h_\uldM\dual)^{-1}$; see Proposition~\ref{PropDualizingUnif}.

Finally, the equalities $M=\Hom_{A_\BC}(\sdsigma^*\dM,\,\Omega^1_{A_\BC/\BC})$ and $\tau_M=(\sdtau_\dM)\dual$ yield the isomorphism for the de Rham realization 
\begin{eqnarray*}
& \Koh^1_\dR(\ulM,\BC\dbl z-\zeta\dbr)\;:=\;\sigma^*M\otimes_{A_\BC}\BC\dbl z-\zeta\dbr\;=\;\\[2mm]
& \;=\;\Hom_{A_\BC}(\dM,\BC\dbl z-\zeta\dbr dz)=:\;\Koh^1_\dR(\uldM,\BC\dbl z-\zeta\dbr)\otimes_{\BC\dbl z-\zeta\dbr}\BC\dbl z-\zeta\dbr dz
\end{eqnarray*}
and its compatibility with the Hodge-Pink lattices.
If moreover $\uldM$ and $\ulM$ are uniformizable, its compatibility with the period isomorphisms $h_{\Betti,\,\dR}$ follows from the equation $h_\uldM=(\sigma^*h_\ulM\dual)^{-1}$ that was established in Proposition~\ref{PropDualizingUnif}.
\end{proof}

 
\section{Anderson \texorpdfstring{$A$}{A}-modules}\label{SectAndersonAModules}
\setcounter{equation}{0}

A main source from which effective $A$-motives arise are Drinfeld $A$-modules~\cite{Drinfeld} and abelian Anderson $A$-modules. For $C=\BP^1_{\BF_q}$ and $A=\BF_q[t]$ the latter were introduced by Anderson~\cite{Anderson86} under the name \emph{abelian $t$-module}; see also \cite[\S3.1]{BrownawellPapanikolas16}. In this section we review the notion of abelian Anderson $A$-modules and their associated $A$-motives. Likewise we review the notion and analytic theory of $A$-finite Anderson $A$-modules and their associated dual $A$-motives which was developed by Greg Anderson in unpublished work \cite{ABP_Rohrlich}. Also for Anderson $A$-modules which are both abelian and $A$-finite we prove the compatibility between the associated $A$-motive and dual $A$-motive.

\subsection{Definition of Anderson $A$-modules} \label{SectDefAModules}

 To recall the definition for general $A$ we need the following notation. For a smooth commutative group scheme $E$ over $\BC$ we let $\Lie E:=\Hom_\BC(e^*\Omega^1_{E/\BC},\BC)$ be its tangent space at the neutral element $e\colon\Spec\BC\to E$. It is a vector space over $\BC$. The differential $d\colon\Hom_{\BC}(E,E')\to\Hom_\BC(\Lie E,\Lie E')$ associates with each homomorphism $f\colon E\to E'$ of smooth group schemes the induced homomorphism $\Lie f\colon\Lie E\to\Lie E'$ of tangent spaces. We consider the additive group scheme $\BG_{a,\BC}=\Spec\BC[X]$ as a $\BC$-module scheme via the action of $b\in\BC$ by $\psi_b^\ast\colon\BC[X]\to\BC[X], X\mapsto bX$. Its relative $q$-Frobenius endomorphism $\Frob_{q,\BG_a}$ is given by $\Frob_{q,\BG_a}^\ast\colon\BC[X]\to\BC[X], X\mapsto X^q$. Let $\BC\{\tau\}:=\bigl\{\,\sum_{i=0}^nb_i\tau^i\colon n\in\BN_0,b_i\in\BC\,\bigr\}$ be the non-commutative polynomial ring in the variable $\tau$ with the commutation rule $\tau b=b^q\tau$ for $b\in\BC$. 

\begin{lemma}\label{LemmaLucas}
There is a natural isomorphism of $\BC$-modules between the $d'\times d$-matrix space $\BC\{\tau\}^{d'\times d}$ and the $\BC$-module $\Hom_{\BF_q,\BC}(\BG_{a,\BC}^d,\BG_{a,\BC}^{d'})$ of $\BF_q$-linear homomorphisms of group schemes over $\BC$, which sends the matrix $F=(f_{ij})_{i,j}\in\BC\{\tau\}^{d'\times d}$ to the $\BF_q$-homomorphism $f\colon\BG_{a,\BC}^d\to\BG_{a,\BC}^{d'}$ with $f^*(y_i)=\sum_j f_{ij}(x_j)$ where $\BG_{a,\BC}^d=\Spec \BC[x_1,\ldots,x_d]$ and $\BG_{a,\BC}^{d'}=\Spec \BC[y_1,\ldots,y_{d'}]$. Under this isomorphism the map $f\mapsto \Lie f$ is given by the map $\BC\{\tau\}^{d'\times d}\to \BC^{d'\times d},\,F=\sum_n F_n\tau^n\mapsto F_0$.
\end{lemma}

\begin{proof}
This is straight forward to prove using Lucas's theorem \cite{Lucas1878}, see also \cite[p.~589]{Fine47}, on congruences of binomial coefficients which states that $\left(\begin{smallmatrix}pn+t\\pm+s\end{smallmatrix}\right)\equiv\left(\begin{smallmatrix}n\\m\end{smallmatrix}\right)\left(\begin{smallmatrix}t\\s\end{smallmatrix}\right)\mod p$ for all $n,m,t,s\in\BN_0$, and implies that $\left(\begin{smallmatrix}n\\i\end{smallmatrix}\right)\equiv 0\mod p$ for all $0<i<n$ if and only if $n=p^e$ for an $e\in\BN_0$.
\end{proof}

\begin{definition}\label{DefAndersonAModule}
Let $d$ be a positive integer.
\begin{enumerate}
\item \label{DefAndersonAModule_A}
An \emph{Anderson $A$-module $\ulE=(E,\phi)$ of dimension $d$} over $\BC$ consists of a group scheme $E$ isomorphic to the $d$-th power $\BG_{a,\BC}^d$ of $\BG_{a,\BC}$ together with a ring homomorphism $\phi\colon A\to \End_{\BC}(E), a\mapsto\phi_a$ such that
\begin{equation}\label{EqDefAndersonAModuleA}
\bigl(\Lie\phi_a-\charmorph(a)\bigr)^d=0\quad\text{on}\quad\Lie E\,.
\end{equation}
Under $a\mapsto\Lie\phi_a$ the tangent space $\Lie E$ becomes an $A_\BC$-module. Note that the commutativity of $A$ implies that $\phi_a$ lies in the ring $\End_{\BF_q,\BC}(E)$ of $\BF_q$-linear endomorphisms for the structure of $\BF_q$-module scheme on $E$ provided via $\phi$. Note further that there always exists an isomorphism $E\isoto\BG_{a,\BC}^d$ of $\BF_q$-module schemes.
\item \label{DefAndersonAModule_B}
An Anderson $A$-module of dimension $1$ is called a \emph{Drinfeld $A$-module of rank $r\in\BN_{>0}$} if under such an isomorphism $E\isoto\BG_{a,\BC}$ and the induced identification $\End_{\BF_q,\BC}(E)\cong\BC\{\tau\}$ from Lemma~\ref{LemmaLucas}, the $\tau$-degree of $\phi_a$ equals $r$ times the order of pole of $a$ at $\infty$ for all $a$.
\item \label{DefAndersonAModule_C}
A \emph{morphism} of Anderson $A$-modules $f\colon(E',\phi')\to(E,\phi)$ is a homomorphism of group schemes $f\colon E'\to E$ satisfying $\phi_a\circ f=f\circ\phi'_a$ for all $a\in A$. We say that $f\colon(E',\phi')\to(E,\phi)$ or simply $(E',\phi')$ is an \emph{Anderson $A$-submodule} if $f$ is a closed immersion.
\item \label{DefAndersonAModule_D}
For $a\in A$, we define the closed subgroup scheme $\ulE[a]:=\ker(\phi_a\colon E\to E)$.
\end{enumerate}
\end{definition}

Every Anderson $A$-module over $\BC$ possesses a unique exponential function $\exp_\ulE\colon\Lie E\to E(\BC)$ satisfying $\phi_a(\exp_\ulE(x))=\exp_\ulE(\Lie\phi_a(x))$ for all $x\in\Lie E$ and $a\in A$. Under an isomorphism $\rho\colon E\isoto\BG_{a,\BC}^d$ of $\BF_q$-module schemes and the induced isomorphism $\Lie\rho\colon\Lie E\isoto\BC^d$ the exponential function $\exp_\ulE$ is given by matrices $E_i\in\BC^{d\times d}$ with $E_0=\Id_d$ such that the series $(\rho\circ\exp_\ulE\circ(\Lie\rho)^{-1})(\xi)=\sum_{i=0}^\infty E_i\,\sigma^{i*}(\xi)$ converges for all $\xi\in\BC^d$; see \cite[Theorem~3]{Anderson86} for $A=\BF_q[t]$ and \cite[\S\,8.6]{BoeckleHartl} for the passage to general $A$. In loc.\ cit.\ these facts are formulated and proved under the additional condition that $\ulE$ is abelian (see Definition~\ref{DefAbelianAMod} below), which is actually unnecessary. We also define 
\[
\Lambda(\ulE)\;:=\;\ker(\exp_\ulE)\,.
\]
It is an $A$-module via the action of $a\in A$ as $\Lie\phi_a$ on $\Lambda(\ulE)\subset\Lie E$. Moreover, the exponential map $\exp_\ulE$ and $\Lambda(\ulE)$ are covariant functorial in $\ulE$, in the sense that $f\circ\exp_{\ulE'}=\exp_\ulE\circ\Lie f$ and $\Lie f\colon\Lambda(\ulE')\to\Lambda(\ulE)$ when $f\colon\ulE'\to\ulE$ is a morphism of Anderson $A$-modules. The following lemmas are well known.

\begin{lemma}\label{LemmaConvergent}
For every $\xi\in\Lie E$ and $a\in A$ the sequence $\Lie\phi_a^{-n}(\xi)$ converges to $0$ as $n\to\infty$.
\end{lemma}

\begin{proof}
Identify $\Lie E\cong\BC^d$ and write $\Lie\phi_a=\charmorph(a)(\Id_d+N)$ with strictly upper triangular (nilpotent) $N$ having only entries $0$ and $1$. Then $\|\Lie\phi_a^{-1}(\xi)\|\le|\charmorph(a)|^{-1}\cdot\|\xi\|$ with respect to the maximum norm $\|\,.\,\|$ on $\BC^d$. Now the lemma follows from $|\charmorph(a)|>1$.
\end{proof}

\begin{lemma}\label{LemmaLog}
For every isomorphism $\rho\colon E\isoto\BG_{a,\BC}^d$ of $\BF_q$-module schemes and every norm $\|\,.\,\|$ on $\BC^d$ there exists a constant $C>0$ such that $\exp_\ulE$ maps $\{\xi\in\Lie E\colon \|\Lie\rho(\xi)\|< C\}$ isometrically onto $\{x\in E(\BC)\colon \|\rho(x)\|< C\}$. The inverse of this isometry is a rigid analytic function
\[
\log_\ulE\colon \{x\in E(\BC)\colon \|\rho(x)\|< C\} \isoto \{\xi\in\Lie E\colon \|\Lie\rho(\xi)\|\le C\}
\]
satisfying $\log_\ulE(\phi_a(x))=(\Lie\phi_a)(\log_\ulE(x))$ for all $a\in A$ and all $x\in E(\BC)$ subject to the condition $\|\rho(x)\|,\|\rho(\phi_a(x))\|< C$. It is called the \emph{logarithm of $\ulE$}.

In particular $\Lambda(\ulE)=\ker(\exp_\ulE)\subset\Lie E$ is a discrete $A$-submodule.
\end{lemma}

\begin{proof}
Since all norms on $\BC^d$ are equivalent by \cite[Theorem~13.3]{Schikhof}, we may assume that $\|\,.\,\|$ is the maximum norm on $\BC^d$ and on $\BC^{d\times d}$. If $\rho\circ\exp_\ulE\circ(\Lie\rho)^{-1}=\sum_{i=0}^\infty E_i\,\tau^i$ then the constant $C:=\sup\{\!\sqrt[q^i-1]{\|E_i\|}\colon i\ge1\,\}^{-1}$ suffices and $\log_\ulE$ equals $(\sum_{i=0}^\infty E_i\,\tau^i)^{-1}=\sum_{n=0}^\infty\bigl(-\sum_{i=1}^\infty E_i\,\tau^i\bigr)^n\in\BC\{\!\{\tau\}\!\}^{d\times d}$ where $\BC\{\!\{\tau\}\!\}:=\bigl\{\,\sum_{i=0}^\infty b_i\tau^i\colon b_i\in\BC\,\bigr\}$ is the non-commutative power series ring with $\tau b=b^q\tau$ for $b\in \BC$.
\end{proof}

With every Anderson $A$-module $\ulE=(E,\phi)$ is associated an $A_\BC$-module as follows. This construction is due to Anderson \cite{Anderson86}; see also \cite[\S\,4.1]{BrownawellPapanikolas16}. Let $M:=M(\ulE):=\Hom_{\BF_q,\BC}(E,\BG_{a,\BC})$ be the $A_{\BC}$-module of $\BF_q$-linear homomorphisms of group schemes, where $a\in A$ and $b\in\BC$ act on $m\in M$ via
\[
a\colon m\mapsto m\circ\phi_a \qquad\text{and}\qquad b\colon m\mapsto \psi_b\circ m\,.
\]
The $\sigma^*$-semi-linear endomorphism of $M$ given by $m\mapsto\Frob_{q,\BG_a}\circ m$ yields an $A_{\BC}$-linear homomorphism $\tau_M\colon \sigma^\ast M\to M$. Note that after choosing an isomorphism $E\cong\BG_{a,\BC}^d$ of $\BF_q$-module schemes we obtain $M(\ulE)\cong\BC\{\tau\}^{1\times d}$ from Lemma~\ref{LemmaLucas}, where $\sum_ib_i\tau^i\in\BC\{\tau\}^{1\times d}$ with $b_i=(b_{i,1},\ldots,b_{i,d})\in\BC^{1\times d}$ corresponds to the morphism $\BG_{a,\BC}^d\to\BG_{a,\BC}$ given by $(x_1,\ldots,x_d)^T\mapsto\sum_{i,j}b_{i,j}x_j^{q^i}$. In particular the endomorphism $m\mapsto\tau_M(\sigma^*m)=\Frob_{q,\BG_a}\circ m$ of $M$ corresponds to the endomorphism $\sum_ib_i\tau^i\mapsto\tau\cdot(\sum_ib_i\tau^i)$ of $\BC\{\tau\}^{1\times d}$ which is injective. Since $\BC$ is perfect, $\sigma^*$ is an automorphism of $A_\BC$. So $\sigma^*\colon M\to\sigma^*M$ is an isomorphism and hence, $\tau_M$ is injective.

There is a natural isomorphism of $A_{\BC}$-modules
\begin{equation}\label{EqIsomLieE}
M/\tau_M(\sigma^*M)\;\isoto\;\Hom_\BC(\Lie E,\BC),\quad m\mod\tau_M(\sigma^*M)\;\longmapsto\; \Lie m\,.
\end{equation}
see \cite[Lemma~1.3.4]{Anderson86}, where $a\in A$ acts on $\Lie E$ via $\Lie\phi_a$. Condition~\eqref{EqDefAndersonAModuleA} in Definition~\ref{DefAndersonAModule}\ref{DefAndersonAModule_A} implies that $J^d=0$ on $M/\tau_M(\sigma^*M)$, where $J:=(a\otimes1-1\otimes\charmorph(a)\colon a\in A)\subset A_\BC$. Therefore, $\tau_M$ induces an isomorphism $\tau_M\colon \sigma^\ast M\otimes_{A_{\BC}}\GlobMotRing\isoto M\otimes_{A_{\BC}}\GlobMotRing$.

\begin{definition}\label{DefAbelianAMod}
Let $\ulE$ be an Anderson $A$-module over $\BC$ and define $\ulM(\ulE):=\bigl(M(\ulE),\tau_M\bigr)$ as above. If $M(\ulE)$ is a finite locally free $A_\BC$-module then $\ulE$ is called \emph{abelian} and $\ulM(\ulE)$ is the \emph{(effective) $A$-motive associated with $\ulE$}. The rank of $\ulM(\ulE)$ is called the \emph{rank of $\ulE$} and is denoted $\rk\ulE$.
\end{definition}

For example, if $C=\BP^1_{\BF_q}$, $A=\BF_q[t]$, $\theta=\charmorph(t)\in\BC$, and $\ulE=(\BG_{a,\BC},\phi_t=\theta+\tau)$ is the \emph{Carlitz-module}, then $\ulE$ is abelian of rank $1$ and $\ulM(\ulE)=(\BC[t],\tau_M=t-\theta)$ is the \emph{Carlitz $t$-motive} from Example~\ref{ExampleCHMotive}.

\begin{remark}\label{RemRankE}
(a) By \cite[Proposition~1.8.3]{Anderson86} the rank of $E$ is characterized by the isomorphism $\ulE[a](\BC)\cong \bigl(A/(a)\bigr)^{\oplus\rk\ulE}$ for every $a\in A$.

\medskip\noindent
(b) If $\ulE$ is a Drinfeld $A$-module the rank of $\ulE$ from Definition~\ref{DefAndersonAModule}\ref{DefAndersonAModule_B} equals the rank from Definition~\ref{DefAbelianAMod} by \cite[\S\,4.5]{Goss}.
\end{remark}

Anderson~\cite[Theorem~1]{Anderson86} proved the following

\begin{theorem}\label{ThmAnderson}
The contravariant functor $\ulE\mapsto\ulM(\ulE)$ is an anti-equivalence from the category of abelian Anderson $A$-modules onto the full subcategory of $\AMotCat$ consisting of those effective $A$-motives $(M,\tau_M)$ that are finitely generated over $\BC\{\tau\}$, where $\tau$ acts on $M$ through $m\mapsto \tau_M(\sigma^\ast m)$.
\end{theorem}

\subsection{The Relation with dual $A$-motives}\label{SectRelDualAMotives}

In unpublished work \cite{ABP_Rohrlich} Greg Anderson has clarified the relation between Anderson $A$-modules and dual $A$-motives. For convenience of the reader we reproduce some of his results here (in our own words); see also \cite[\S\,4.4]{BrownawellPapanikolas16}.

Let $E$ be a group scheme over $\BC$ isomorphic to $\BG_{a,\BC}^d$, and let $\phi\colon A\to\End_\BC(E)$ be a ring homomorphism. The set $\dM:=\dM(\ulE):=\Hom_{\BF_q,\BC}(\BG_{a,\BC},E)$ of $\BF_q$-linear homomorphisms of group schemes is an $A_\BC$-module, where $a\in A$ and $b\in\BC$ act on $\dm\in\dM$ via
\[
a\colon \dm\mapsto \phi_a\circ\dm \qquad\text{and}\qquad b\colon\dm\mapsto \dm\circ \psi_b\,.
\]
There is a $\sdsigma^\ast$-semi-linear endomorphism of $\dM=\dM(\ulE)$ given by $\dm\mapsto\dm\circ \Frob_{q,\BG_a}$, which induces an $A_{\BC}$-linear homomorphism $\sdtau_\dM\colon \sdsigma^\ast \dM\to \dM$. Note that after choosing an isomorphism $E\cong\BG_{a,\BC}^d$ of $\BF_q$-module schemes we obtain $\dM(\ulE)\cong\BC\{\tau\}^d$ from Lemma~\ref{LemmaLucas}, where $\sum_ib_i\tau^i\in\BC\{\tau\}^d$ with $b_i\in\BC^d$ corresponds to the morphism $\BG_{a,\BC}\to\BG_{a,\BC}^d$ given by $x\mapsto\sum_ib_ix^{q^i}$. In particular the endomorphism $\dm\mapsto\sdtau_\dM(\sdsigma^*\dm)=\dm\circ \Frob_{q,\BG_a}$ of $\dM$ corresponds to the endomorphism $\sum_ib_i\tau^i\mapsto(\sum_ib_i\tau^i)\cdot\tau$ of $\BC\{\tau\}^d$ which is injective. Since $\BC$ is perfect, $\sdsigma$ is an automorphism of $A_\BC$. So $\sdsigma^*\colon \dM\to\sdsigma^*\dM$ is an isomorphism and hence, $\sdtau_\dM$ is injective.

There is the following alternative description of $\dM(\ulE)$. Let $\BC\{\sdtau\}$ be the non-commutative polynomial ring over $\BC$ in the variable $\sdtau$ with $\sdtau b=\sqrt[q]{b}\,\sdtau$ for $b\in\BC$. Consider the $\dagger$-operation (called $*$-operation in \cite[\S\,4.4]{BrownawellPapanikolas16}) which sends a matrix $B=\sum_i B_i\tau^i\in\BC\{\tau\}^{r\times r'}$ with $B_i\in\BC^{r\times r'}$ to the matrix $B^\dagger:=(\sum_i\sdsigma^{i*}(B_i)\sdtau^i)^T \in\BC\{\sdtau\}^{r'\times r}$. Here $(\ldots)^T$ denotes the transpose. The $\dagger$-operation satisfies $(BC)^\dagger=C^\dagger B^\dagger$ for matrices $B\in\BC\{\tau\}^{r\times r'}$ and $C\in\BC\{\tau\}^{r'\times r''}$. It induces an isomorphism of $A_\BC$-modules
\begin{equation}\label{EqDagger}
\dagger\colon\;\dM(\ulE)\;\cong\;\BC\{\tau\}^d\;\isoto\;\BC\{\sdtau\}^{1\times d},\quad\dm\;\mapsto\;\dm^\dagger,
\end{equation}
where $a\in A$ and $b\in\BC$ act on $\dm^\dagger\in\BC\{\sdtau\}^{1\times d}$ via
\[
a\colon \dm^\dagger\mapsto \dm^\dagger\cdot\Delta_a^{\,\dagger} \qquad\text{and}\qquad b\colon\dm^\dagger\mapsto b\cdot\dm^\dagger\,.
\]
Here $\Delta_a\in\BC\{\tau\}^{d\times d}=\End_{\BF_q,\BC}(\BG_{a,\BC}^d)\cong\End_{\BF_q,\BC}(E)$ is the matrix corresponding to $\phi_a$. Under this isomorphism $\dagger\colon\dM(\ulE)\isoto\BC\{\sdtau\}^{1\times d}$ the $\sdsigma^*$-semi-linear endomorphism $\dm\mapsto\dm\circ \Frob_{q,\BG_a}$ of $\dM(\ulE)$ corresponds to the $\sdsigma^*$-semi-linear endomorphism $\dm^\dagger\mapsto\sdtau\cdot\dm^\dagger$ of $\BC\{\sdtau\}^{1\times d}$. This gives $\dM(\ulE)$ the structure of a finite free left $\BC\{\sdtau\}$-module which is independent of the isomorphism $E\cong\BG_{a,\BC}^d$.

\begin{proposition}\label{prop:commutingdiagrams}
Let $E$ be a group scheme over $\BC$ isomorphic to $\BG_{a,\BC}^d$, and let $\phi\colon A\to\End_\BC(E)$ be a ring homomorphism. Set $\ulE=(E,\varphi)$ and let $\dM=\dM(\ulE)$ and $\sdtau_\dM\colon \sdsigma^\ast \dM\to \dM$ be as above. Then there is a canonical exact sequence of $A$-modules
\begin{equation}\label{EqDualMAndE1}
\xymatrix @R-2pc {
0\es \ar[r] & \es\dM \es\ar[r] & \qquad\qquad\qquad\es\dM\qquad \ar[r]^{\qquad\quad\delta_1} & \es E(\BC)\es \ar[r] & \es 0 \,,\\
& \es\dm\es \ar@{|->}[r] & \es\sdtau_\dM(\sdsigma^*\dm)-\dm\,,\quad\dm\es \ar@{|->}[r] & \es\dm(1)\es
}
\end{equation}
and a canonical exact sequence of $A_\BC$-modules
\begin{equation}\label{EqDualMAndE2}
\xymatrix @R-2pc @C+1pc{
0\es \ar[r] & \es\sdsigma^*\dM\es \ar[r]^{\es\sdtau_\dM} & \;\,\dM\;\, \ar[r]^{\delta_0\es\quad} & \quad\;\Lie E\;\quad \ar[r] & \es0\,. \\
& & \es\dm\es \ar@{|->}[r] & \es(\Lie\dm)(1)
}
\end{equation}
In particular, $\ulE=(E,\phi)$ is an Anderson $A$-module if and only if $\sdtau_\dM$ induces an isomorphism $\sdtau_\dM\colon \sdsigma^\ast \dM\otimes_{A_{\BC}}\GlobMotRing\isoto \dM\otimes_{A_{\BC}}\GlobMotRing$. In this case, $\delta_0$ factors through $\dM/J^d\dM$ and extends to an $A_\BC$-homomorphism $\delta_0\colon\dM\otimes_{A_\BC}\CO\bigl(\dotFC_\BC\setminus\bigcup_{i\in\BN_{>0}}\Var(\ssigma^{i\ast}J)\bigr)\onto \Lie E$.

Under the above identifications $E(\BC)\cong\BC^d$ and $\Lie E\cong\BC^d$ and $\dagger\colon\dM(\ulE)\isoto\BC\{\sdtau\}^{1\times d}$ these sequences take the form
\[
\xymatrix @R-2pc {
0\es \ar[r] & \es\BC\{\sdtau\}^{1\times d}\es\ar[r]^{\sdtau-1\;\qquad} & \qquad\quad\BC\{\sdtau\}^{1\times d}\qquad \ar[r]^{\quad\;\delta_1} & \qquad\; \BC^d\;\qquad \ar[r] & \es 0 \,,\\
& \quad\es\;\;\dm^\dagger\;\es\quad \ar@{|->}[r] & \es\sdtau\dm^\dagger-\dm^\dagger\,,\quad\sum\limits_i c_i\sdtau^i\es \ar@{|->}[r] & \es\sum\limits_i\sigma^{i*}(c_i)^T\es
}
\]
and
\[
\xymatrix @R-2pc @C+1pc {
0\es \ar[r] & \es\BC\{\sdtau\}^{1\times d}\es \ar[r]^{\sdtau\quad} & \quad\es\BC\{\sdtau\}^{1\times d}\es\quad \ar[r]^{\qquad\delta_0} & \es\BC^d\es \ar[r] & \es0\,. \\
& \quad\es\;\;\dm^\dagger\;\es\quad \ar@{|->}[r] & \es\sdtau\dm^\dagger\,,\quad\sum\limits_i c_i\sdtau^i\es \ar@{|->}[r] & \es c_0^T \es
}
\]
\end{proposition}

\begin{proof}
The map $\delta_1$ is $A$-linear because $a\cdot\dm=\phi_a\circ\dm\mapsto(\phi_a\circ\dm)(1)=\phi_a(\dm(1))$. The map $\delta_0$ is a homomorphism of $A_\BC$-modules because $a\cdot\dm=\phi_a\circ\dm\mapsto\Lie(\phi_a\circ\dm)(1)=\Lie\phi_a(\Lie\dm(1))$ and 
\[
b\cdot\dm \;=\; \dm\circ\psi_b \;\longmapsto\; \Lie(\dm\circ\psi_b)(1) \;=\; (\Lie\dm\circ\Lie\psi_b)(1) \;=\; b\cdot(\Lie\dm)(1)\,. 
\]
To prove that the composition of the two morphisms in \eqref{EqDualMAndE1} is zero, we compute $(\sdtau_\dM(\sdsigma^*\dm)-\dm)(1):=\dm\circ \Frob_{q,\BG_a}(1) - \dm(1) = \dm(1)-\dm(1)=0$ for all $\dm\in\dM$.  To prove that $\delta_0\circ\sdtau_\dM=0$ in \eqref{EqDualMAndE2}, note that since $\BC$ is perfect, $\sdsigma^*\colon\dM\to\sdsigma^*\dM$ is an isomorphism. Therefore, every element of $\sdtau_\dM(\sdsigma^*\dM)$ is of the form $\sdtau_\dM(\sdsigma^*\dm)=\dm\circ \Frob_{q,\BG_a}$ and satisfies $\Lie(\dm\circ \Frob_{q,\BG_a})=(\Lie\dm)\circ(\Lie\Frob_{q,\BG_a})=0$.

Furthermore, $\delta_1$ is surjective because through every point $x\in E(\BC)$ there is a morphism $\dm\colon\BG_{a,\BC}\to E$ with $\dm(1)=x$. For example if we identify the $\BF_q$-module schemes $\rho\colon E\isoto\BG_{a,\BC}^d=\Spec\BC[X_1,\ldots,X_d]$ and $\BG_{a,\BC}=\Spec\BC[Y]$ we can take $\dm\colon X_i\mapsto x_iY$ where $\rho(x)=(x_1,\ldots,x_d)^T$. This $\dm$ also satisfies $(\Lie\dm)(1)=x$ and this shows that $\delta_0$ is surjective.

To show that \eqref{EqDualMAndE1} and \eqref{EqDualMAndE2} are exact, we keep this identification and the induced isomorphism $\dM(\ulE)\cong\BC\{\tau\}^{d}$. If $\dm=\sum_i b_i\tau^i\in\BC\{\tau\}^d$ satisfies $0=\dm(1)=\sum_i b_i$, then 
\[
\dm \;=\; \sdtau_\dM(\sdsigma^*\dm')-\dm' \;=\; \dm'\tau-\dm'
\]
for $\dm'=\sum_i b_i(1+\tau+\ldots+\tau^{i-1})$. This proves that \eqref{EqDualMAndE1} is exact in the middle. Exactness on the left holds because multiplication with $\tau-1$ is injective on $\BC\{\tau\}^{d}$. Clearly, \eqref{EqDualMAndE2} is exact on the left because $\sdtau_\dM$ is injective. If $\dm=\sum_i b_i\tau^i$ satisfies $0=(\Lie\dm)(1)=b_0$ then $\dm=(\sum_ib_i\tau^{i-1})\cdot\tau=\sdtau_\dM(\sdsigma^*\sum_ib_i\tau^{i-1})\in\sdtau_\dM(\sdsigma^*\dM)$, and this proves the exactness of \eqref{EqDualMAndE2}. Moreover, under the $\dagger$-operation $\dm=\sum_i b_i\tau^i$ is sent to $\dm^\dagger=\sum_ic_i\sdtau^i$ for $c_i=\sdsigma^{i*}(b_i)^T$ and so $\delta_1(\dm)=\sum_ib_i=\sum_i\sigma^{i*}(c_i)^T$ and $\delta_0(\dm)=b_0=c_0^T$.

Finally, $\sdtau_\dM$ induces an isomorphism $\sdtau_\dM\colon \sdsigma^\ast \dM\otimes_{A_{\BC}}\GlobMotRing\isoto \dM\otimes_{A_{\BC}}\GlobMotRing$ if and only if the elements of $J$ are nilpotent on $\Lie E$. Since $\Lie E$ is a $d$-dimensional $\BC$-vector space, the latter is equivalent to condition \eqref{EqDefAndersonAModuleA} in Definition~\ref{DefAndersonAModule}\ref{DefAndersonAModule_A}. If this holds, the morphism $\delta_0$ factors through $\dM/J^d\dM$, and extends to a homomorphism $\delta_0\colon\dM\otimes_{A_\BC}\CO\bigl(\dotFC_\BC\setminus\bigcup_{i\in\BN_{>0}}\Var(\ssigma^{i\ast}J)\bigr)\onto \Lie E$ because $\CO\bigl(\dotFC_\BC\setminus\bigcup_{i\in\BN_{>0}}\Var(\ssigma^{i\ast}J)\bigr)/(J^d)\;=\;A_\BC/J^d$.
\end{proof}

\begin{definition}\label{DefDualMOfE}
Let $\ulE$ be an Anderson $A$-module over $\BC$ and define $\uldM(\ulE):=\bigl(\dM(\ulE),\sdtau_\dM\bigr)$ as above. If $\dM(\ulE)$ is a finite locally free $A_\BC$-module then $\ulE$ is called \emph{$A$-finite} and $\uldM(\ulE)$ is the \emph{(effective) dual $A$-motive associated with $\ulE$}. The rank of $\uldM(\ulE)$ is called the \emph{rank of $\ulE$} and is denoted $\rk\ulE$.
\end{definition}

\begin{remark}\label{Rem2RankE}
By the analog of \cite[Proposition~1.8.3]{Anderson86} (see Proposition~\ref{PropSwitcheroo} below) the rank of $E$ is characterized by $\ulE[a](\BC)\cong \bigl(A/(a)\bigr)^{\oplus\rk\ulE}$ for every $a\in A$, where $\ulE[a]:=\ker(\phi_a\colon E\to E)$. Together with Remark~\ref{RemRankE} this shows that for an Anderson $A$-module $\ulE$ which is both abelian and $A$-finite the Definitions~\ref{DefAbelianAMod} and \ref{DefDualMOfE} of the rank of $\ulE$ coincide.
\end{remark}

The assignment $\ulE\mapsto\uldM(\ulE)=\left(\Hom_{\BF_q,\BC}(\BG_{a,\BC},E),\sdsigma^\ast\dm \mapsto \dm\circ\Frob_{q,\BG_a}\right)$ is a covariant functor because a morphism $f\colon\ulE=(E,\phi)\to\ulE'=(E',\phi')$ between abelian Anderson $A$-modules (which satisfies $f\circ\phi_a=\phi'_a\circ f$) is sent to 
\[
\uldM(f)\colon\dM(\ulE)\longto\uldM(E'),\quad \dm\mapsto f\circ\dm\,,
\]
which satisfies $a\cdot\uldM(f)(\dm)=\phi'_a\circ (f\circ\dm)=f\circ(\phi_a\circ\dm)=\uldM(f)(a\cdot\dm)$ and $b\cdot\uldM(f)(\dm)=(f\circ\dm)\circ\psi_b=\uldM(f)(b\cdot\dm)$ and $(\sdtau_{\dM(\ulE')}\circ\sdsigma^\ast\uldM(f))(\sdsigma^*\dm)=(f\circ\dm)\circ\Frob_{q,\BG_a}=(\uldM(f)\circ\sdtau_{\dM(\ulE)})(\sdsigma^*\dm)$ for $a\in A,b\in\BC$ and $\dm\in\dM(\ulE)$. The following result is due to Anderson; see \cite[Theorem~4.4.1]{BrownawellPapanikolas16}.

\begin{theorem}
\label{theorem:equivalence}
\begin{enumerate}
\item \label{theorem:equivalence_A}
The functor $\uldM(\,.\,)\colon\ulE\mapsto\uldM(\ulE)$ from the category of Anderson $A$-modules to the category of pairs $(\dM,\sdtau_\dM)$ consisting of an $A_\BC$-module $\dM$ and an isomorphism of $\GlobMotRing$-modules $\sdtau_\dM\colon \sdsigma^\ast \dM[J^{-1}]\isoto \dM[J^{-1}]$ is fully faithful.
\item \label{theorem:equivalence_B}
The functor $\uldM(\,.\,)$ restricts to an equivalence from the category of $A$-finite Anderson $A$-modules onto the full subcategory of $\dualAMotCat$ consisting of those effective dual $A$-motives $(\dM,\sdtau_\dM)$ which are finitely generated as left $\BC\{\sdtau\}$-modules, where $\sdtau$ acts on $\dM$ through $\dm\mapsto \sdtau_\dM(\sdsigma^\ast \dm)$.
\end{enumerate}
\end{theorem}

\begin{proof}
\ref{theorem:equivalence_A} \es Let $\ulE$ and $\ulE'$ be Anderson $A$-modules and fix isomorphisms $E\cong\BG_{a,\BC}^d$ and $E'\cong\BG_{a,\BC}^{d'}$ of $\BF_q$-module schemes. Then under the identification $\Hom_{\BF_q,\BC}(E,E')\cong\BC\{\tau\}^{d'\times d}$ from Lemma~\ref{LemmaLucas}, a morphism $f\colon\ulE\to\ulE'$ corresponds to a matrix $F\in\BC\{\tau\}^{d'\times d}$ and the induced morphism $\uldM(f)\colon\BC\{\tau\}^d\cong\uldM(\ulE)\to\uldM(\ulE')\cong\BC\{\tau\}^{d'}$ corresponds to multiplication on the left with the matrix $F$. 

Conversely, let $g\colon\BC\{\tau\}^d\cong\uldM(\ulE)\to\uldM(\ulE')\cong\BC\{\tau\}^{d'}$ be a morphism, that is $\sdtau_{\dM(\ulE')}\circ\sdsigma^*g=g\circ\sdtau_{\dM(\ulE)}$. Since $\sdtau_{\dM(\ulE)}(\sdsigma^*\dm):=\dm\circ\Frob_{q,\BG_a}=\dm\cdot\tau$ in $\uldM(\ulE)\cong\BC\{\tau\}^d$, this means that the map $g\colon\BC\{\tau\}^d\to\BC\{\tau\}^{d'}$ is compatible with multiplication by $\BC\{\tau\}$ on the right. Therefore, $g$ corresponds to multiplication on the left by a matrix $G\in\BC\{\tau\}^{d'\times d}$. This means that $g$ induces a morphism of $\BF_q$-module schemes $f\colon E\to E'$ with $\uldM(f)=g$. Since $g$ commutes with the $A$-action on $\uldM(\ulE)$ and $\uldM(\ulE')$, also $f$ commutes with the $A$-action on $\ulE$ and $\ulE'$, that is $f$ is a morphism of Anderson $A$-modules. This proves the full faithfulness of $\uldM(\,.\,)$.

\medskip\noindent
\ref{theorem:equivalence_B} \es
Let $\uldM$ be a dual $A$-motive which is finitely generated over $\BC\{\sdtau\}$. Then $\dM$ is a finite free $\BC\{\sdtau\}$-module by the $\BC\{\sdtau\}$-analog of \cite[Lemma~1.4.5]{Anderson86}, because it is a torsion free $A_\BC$-module. Any $\BC\{\sdtau\}$-basis of $\dM$ provides an isomorphism $\dM\cong\Hom_{\BF_q,\BC}(\BG_{a,\BC},E)=:\dM(E)$ compatible with $\sdtau_\dM$ and $\sdtau_{\dM(E)}$, where $E:=\BG_{a,\BC}^d$ with $d:=\rk_{\BC\{\sdtau\}}\dM$. The action of $a\in A$ on $\dM$ commutes with $\sdtau_\dM$. Therefore, it is given by multiplication on $\dM\cong\BC\{\sdtau\}^{1\times d}$ on the right by a matrix $\Delta_a^\dagger=\sum_iB_i\sdtau^i\in\BC\{\sdtau\}^{d\times d}$. The map $\phi\colon A\to\BC\{\tau\}^{d\times d}=\End_{\BF_q,\BC}(E)$, $a\mapsto\Delta_a:=(\sum_i\sigma^{i*}(B_i)\tau^i)^T$ makes $E$ into an $A$-module scheme. Sequence \eqref{EqDualMAndE2} shows that $\ulE=(E,\phi)$ is an Anderson $A$-module which is $A$-finite, because $\uldM\cong\uldM(\ulE)$.
\end{proof}

Let $\ulE=(E,\phi)$ be a (not necessarily $A$-finite) Anderson $A$-module and let $\uldM=(\dM,\sdtau_\dM)=\uldM(\ulE)$ be as in Definition~\ref{DefDualMOfE}. The following crucial description of the torsion points of $\ulE$ is Anderson's ``switcheroo''; see \cite{ABP_Rohrlich} or \cite[Lemma~4.1.23]{JuschkaDipl}. 

\begin{proposition}\label{PropSwitcheroo}
Let $\dm\in\dM$ and let $x=\delta_1(\dm)=\dm(1)\in E(\BC)$. Let $a\in A\setminus\BF_q$. Then there is a canonical bijection
\begin{eqnarray}
\bigl\{\,\dm'\in \dM/a\dM\colon \sdtau_\dM(\sdsigma^*\dm')-\dm'=\dm \text{ in }\dM/a\dM\,\bigr\} & \isoto &  \bigl\{\,x'\in E(\BC)\colon\phi_a(x')=x\,\bigr\}\nonumber \\[2mm]
\dm' \qquad \qquad \qquad \qquad & \longmapsto & \delta_1\bigl(a^{-1}(\dm+\dm'-\sdtau_\dM(\sdsigma^*\dm'))\bigr)\,,\label{EqSwitcheroo}
\end{eqnarray}
where $x':=\delta_1\bigl(a^{-1}(\dm+\dm'-\sdtau_\dM(\sdsigma^*\dm'))\bigr)$ is defined by choosing any representative $\dm'\in\dM$ of $\dm'\in\dM/a\dM$, taking $\dm''\in\dM$ as the unique element with $\dm+\dm'-\sdtau_\dM(\sdsigma^*\dm')=a\dm''$, and setting $x':=\delta_1(\dm'')$.

If $\dm=0$ both sides are $A/(a)$-modules and the bijection is an isomorphism of $A/(a)$-modules
\[
(\dM/a\dM)^\sdtau \es\isoto\es \ulE[a](\BC)\,,\quad \dm'\es\longmapsto\es\delta_1\bigl(a^{-1}(\dm'-\sdtau_\dM(\sdsigma^*\dm'))\bigr)\,.
\]
\end{proposition}

\begin{proof}
First note that the map is well defined. Namely, any two representatives of $\dm'\in\dM/a\dM$ differ by $a\dn$ for an element $\dn\in\dM$. Then the corresponding elements $\dm''$ differ by $\dn-\sdtau_\dM(\sdsigma^*\dn)$ which lies in the kernel of $\delta_1$. Therefore, $x'$ is independent of the representative $\dm'\in\dM$. Moreover, $x':=\delta_1(\dm'')$ satisfies $\phi_a(x')=\phi_a(\delta_1(\dm''))=\delta_1(a\dm'')=\delta_1(\dm)=x$. If $\dm=0$, then the map clearly is an $A/(a)$-homomorphism.

If $x'\in E(\BC)$ with $\phi_a(x')=x$ is given, there is an $\dm''\in\dM$ with $\delta_1(\dm'')=x'$ by \eqref{EqDualMAndE1} in Proposition~\ref{prop:commutingdiagrams} and then $\delta_1(a\dm'')=\phi_a(\delta_1(\dm''))=\phi_a(x')=x=\delta_1(\dm)$ implies that $\dm-a\dm''=\sdtau_\dM(\sdsigma^*\dm')-\dm'$ for an element $\dm'\in\dM$. This proves the surjectivity.

To prove injectivity let $\dm'_1,\dm'_2\in\dM$ be mapped to the same element $x'\in E(\BC)$ and let $\dm''_i=a^{-1}\bigl(\dm+\dm'_i-\sdtau_\dM(\sdsigma^*\dm'_i)\bigr)$ for $i=1,2$. Then $\delta_1(\dm''_1)=x'=\delta_1(\dm''_2)$ implies by \eqref{EqDualMAndE1} in Proposition~\ref{prop:commutingdiagrams} that $\dm''_2=\dm''_1+\sdtau_\dM(\sdsigma^*\dn)-\dn$ for an element $\dn\in\dM$. From this it follows that $\sdtau_M\bigl(\sdsigma^*(\dm'_2+a\dn-\dm'_1)\bigr)-(\dm'_2+a\dn-\dm'_1)=0$ and the exactness of \eqref{EqDualMAndE1} on the left implies $\dm'_1=\dm'_2+a\dn$. 
\end{proof}

The relation between $\ulM(\ulE)$ and $\uldM(\ulE)$ of an abelian and $A$-finite Anderson $A$-module $\ulE$ is described by the following

\begin{theorem}\label{ThmMandDMofE}
Let $\ulE$ be an abelian Anderson $A$-module over $\BC$, let $\ulM=(M,\tau_M)=\ulM(\ulE)$ and $\uldM=(\dM,\sdtau_\dM)=\uldM(\ulE)$ be as in Definitions~\ref{DefAbelianAMod} and \ref{DefDualMOfE}. Let $\uldM(\ulM)=\bigl(\Hom_{A_\BC}(\sigma^*M,\Omega^1_{A_\BC/\BC})\,,\,\tau_M\dual\bigr)$ be the dual $A$-motive from Proposition~\ref{PropDualizing}. Then there is a canonical injective $A_\BC$-homomorphism
\[
\Xi\colon\Hom_{A_\BC}(\sigma^*M,\Omega^1_{A_\BC/\BC})\es \longinto \es \dM\,,\quad \eta\longmapsto \dm_\eta
\]
such that for every $m\in M$ 
\begin{equation}\label{EqThmMandDMofE}
m\circ\dm_\eta\;=\;\sum\limits_{i=0}^\infty \Bigl(\Res_\infty\eta\bigl(\sdsigma^{i*}(\tau_M^{-i-1}m)\bigr)\Bigr)^{q^i}\!\!\cdot\tau^i\;\in\;\End_{\BF_q,\BC}(\BG_{a,\BC})\;=\;\BC\{\tau\}\,.
\end{equation}
It is compatible with $\tau_M$ and $\sdtau_\dM$, that is, the following diagram commutes:
\begin{equation}\label{EqThmMandDMofECompatib}
\xymatrix @C-2pc {
\Hom_{A_\BC}(\sigma^*M,\Omega^1_{A_\BC/\BC}) \ar[r]^{\qquad\qquad\TS\Xi} & \dM \\
**{!R(0.55) =<19pc,2pc>} \objectbox{\sdsigma^*\Hom_{A_\BC}(\sigma^*M,\Omega^1_{A_\BC/\BC})\;=\; \Hom_{A_\BC}(M,\Omega^1_{A_\BC/\BC})} \ar[r]^{\qquad\qquad\TS\sdsigma^*\Xi} \ar[u]^{\TS.\circ\tau_M} & \sdsigma^*\dM \ar[u]_{\TS\sdtau_\dM}
}
\end{equation}
Moreover, $\Xi$ is an isomorphism if and only if $\ulE$ is $A$-finite. In this case $\Xi$ is an isomorphism of dual $A$-motives $\Xi\colon\uldM(\ulM)\isoto\uldM(\ulE)$. 
\end{theorem}

\begin{proof}
1. To show that the sum in \eqref{EqThmMandDMofE} belongs to $\BC\{\tau\}$ we have to show that
\[
\Res_\infty\eta\bigl(\sdsigma^{i*}(\tau_M^{-i-1}m)\bigr)\;=\;0 \quad\text{for all}\quad i\,\gg\,0\,. 
\]
By Proposition~\ref{PropPure} the $z$-isocrystal $\ulHM:=\ulM\otimes_{A_\BC}\BC\dpl z\dpr$ is isomorphic to $\bigoplus_i\ulHM_{d_i,r_i}$ with all $d_i>0$ by Proposition~\ref{PropWeights}\ref{PropWeights_a}. The explicit description of $\ulHM_{d_i,r_i}$ in \eqref{EqStandardZIsocrystal} shows that there is a $\BC\dbl z\dbr$-lattice $V$ of full rank in $\wh{M}:=M\otimes_{A_\BC}\BC\dpl z\dpr$ such that $V\subset\tau_M^j(\sigma^{j*}V)$ for all $j\ge0$, and a positive integer $s$ with $z^{-1}V\subset\tau_M^s(\sigma^{s*}V)$. This implies 
\[
\sdsigma^{(ns+j-1)*}(\tau_M^{-ns-j}V)\;\subset\; z^n\sigma^*V \quad\text{for all integers} \quad n\,\ge\,0\quad\text{and}\quad 0\,\le\, j<s\,. 
\]
We extend $\eta\in\Hom_{A_\BC}(\sigma^*M,\Omega^1_{A_\BC/\BC})$ to $\eta\in\Hom_{\BC\dpl z\dpr}(\sigma^*\wh{M},\BC\dpl z\dpr dz)$. In particular, $\eta(\sigma^*V)\subset z^{-N}\BC\dbl z\dbr dz$ for an integer $N$. For every $m\in M$, there is an integer $e$ with $m\in z^{-e}V$ so that $\eta\bigl(\sdsigma^{(ns+j-1)*}(\tau_M^{-ns-j}m)\bigr)\in z^{n-e-N}\BC\dbl z\dbr dz$. It follows that $\Res_\infty\eta\bigl(\sdsigma^{ns+j-1*}(\tau_M^{-ns-j}m)\bigr)=0$ for all $n\ge N+e$ and all $0\le j<s$.

\medskip\noindent 
2. Fix an $\eta\in\Hom_{A_\BC}(\sigma^*M,\Omega^1_{A_\BC/\BC})$. To define $\dm_\eta\in\uldM$ we choose an isomorphism $\rho\colon E\isoto\BG_{a,\BC}^d$ of $\BF_q$-module schemes, let $pr_j\colon\BG_{a,\BC}^d\to\BG_{a,\BC}$ be the projection onto the $j$-th factor, and set $m_j:=pr_j\circ\rho\in M(\ulE)=\Hom_{\BF_q,\BC}(E,\BG_{a,\BC})$ for $j=1,\ldots,d$. We define $\dm_\eta\in\dM=\Hom_{\BF_q,\BC}(\BG_{a,\BC},E)$ via
\[
\rho\circ\dm_\eta\;:=\;\left(\sum\limits_{i=0}^\infty \Bigl(\Res_\infty\eta\bigl(\sdsigma^{i*}(\tau_M^{-i-1}m_j)\bigr)\Bigr)^{q^i}\!\!\cdot\tau^i\right)_{j=1}^d\!\in\;\BC\{\tau\}^{\oplus d}\,.
\]
In particular, \eqref{EqThmMandDMofE} holds when $m=m_j$ for $j=1,\ldots,d$. To prove that \eqref{EqThmMandDMofE} holds for all $m\in M$ we use that $m_1,\ldots,m_d$ form a $\BC\{\tau\}$-basis of $M$. Thus it suffices to show that \eqref{EqThmMandDMofE} is compatible with
\begin{enumerate}
\item \label{WithAddition} addition in $M$,
\item \label{WithScalMult} scalar multiplication by elements of $\BC$, and
\item \label{WithTau} multiplication with $\tau$.
\end{enumerate}
Since both sides of \eqref{EqThmMandDMofE} are additive in $m$, \ref{WithAddition} is clear.

\medskip\noindent
\ref{WithScalMult} Let $m\in M$ and $b\in\BC$ and assume that \eqref{EqThmMandDMofE} holds for $m$. The left hand side equals $(bm)\circ\dm_\eta=b\cdot (m\circ\dm_\eta)$. On the right hand side we have $\Res_\infty\eta\bigl(\sdsigma^{i*}(\tau_M^{-i-1}(bm))\bigr)=b^{q^{-i}}\cdot\Res_\infty\eta\bigl(\sdsigma^{i*}(\tau_M^{-i-1}m)\bigr)$. Therefore, \eqref{EqThmMandDMofE} also holds for $bm$.

\medskip\noindent
\ref{WithTau} We assume that \eqref{EqThmMandDMofE} holds for some $m\in M$. The left hand side equals $(\tau m)\circ\dm_\eta=\tau\cdot (m\circ\dm_\eta)$. The right hand side for $\tau m=\tau_M(\sigma^*m)$ equals 
\begin{eqnarray*}
\sum\limits_{i=0}^\infty \Bigl(\Res_\infty\eta\bigl(\sdsigma^{i*}(\tau_M^{-i-1}\circ\tau_M(\sigma^*m))\bigr)\Bigr)^{q^i}\!\!\cdot\tau^i 
& = & \sum\limits_{i=1}^\infty \Bigl(\Res_\infty\eta\bigl(\sdsigma^{(i-1)*}(\tau_M^{-i}m)\bigr)\Bigr)^{q^{(i-1)}q}\!\!\cdot\tau^i\\
& = & \tau\cdot\bigl(\sum\limits_{i=1}^\infty \Bigl(\Res_\infty\eta\bigl(\sdsigma^{(i-1)*}(\tau_M^{-i}m)\bigr)\Bigr)^{q^{i-1}}\!\!\cdot\tau^{i-1}\bigr),
\end{eqnarray*}
because $\tau_M^{-i-1}=\sigma^{i*}\tau_M^{-1}\circ\ldots\circ\sigma^*\tau_M^{-1}\circ\tau_M^{-1}$ and in the first line the term $\Res_\infty\eta(\sigma^*m)$ for $i=0$ vanishes by \cite[Theorem~9.3.22]{VillaSalvador} as $\eta(\sigma^*m)\in\Omega^1_{A_\BC/\BC}$. Therefore, \eqref{EqThmMandDMofE} also holds for $\tau m$.

This establishes \eqref{EqThmMandDMofE} for all $m\in M$.

\medskip\noindent 
3. To prove that the assignment $\Xi\colon\eta\mapsto\dm_\eta$ defined in step 2 is $\BC$-linear, note that additivity is clear. Let $b\in\BC$. Then $b\eta$ is sent to $b\cdot\dm_\eta$ because
\begin{eqnarray*}
\rho\circ\dm_{(b\eta)} & := & \left(\sum\limits_{i=0}^\infty \Bigl(\Res_\infty(b\eta)\bigl(\sdsigma^{i*}(\tau_M^{-i-1}m_j)\bigr)\Bigr)^{q^i}\!\!\cdot\tau^i\right)_{j=1}^d\\
& = & \left(\sum\limits_{i=0}^\infty \Bigl(\Res_\infty\eta\bigl(\sdsigma^{i*}(\tau_M^{-i-1}m_j)\bigr)\Bigr)^{q^i}\!\!\cdot b^{q^i}\cdot\tau^i\right)_{j=1}^d\\[2mm]
& = & \rho\circ\dm_\eta\circ\psi_b\\[2mm]
& =: & \rho\circ(b\dm_\eta)\,.
\end{eqnarray*}

\medskip\noindent 
4. The map $\Xi$ is also $A$-linear. Indeed, let $a\in A$. Then $a\eta$ is sent to $a\cdot\dm_\eta$ because
\begin{eqnarray*}
pr_j\circ\rho\circ\dm_{(a\eta)} & := & \sum\limits_{i=0}^\infty \Bigl(\Res_\infty(a\eta)\bigl(\sdsigma^{i*}(\tau_M^{-i-1}m_j)\bigr)\Bigr)^{q^i}\!\!\cdot\tau^i\\
& := & \sum\limits_{i=0}^\infty \Bigl(\Res_\infty\eta\bigl(a\cdot\sdsigma^{i*}(\tau_M^{-i-1}m_j)\bigr)\Bigr)^{q^i}\!\!\cdot\tau^i\\
& = & \sum\limits_{i=0}^\infty \Bigl(\Res_\infty\eta\bigl(\sdsigma^{i*}(\tau_M^{-i-1}(m_j\circ\phi_a))\bigr)\Bigr)^{q^i}\!\!\cdot\tau^i\\
& = & m_j\circ\phi_a\circ\dm_\eta\\[2mm]
& =: & pr_j\circ\rho\circ(a\dm_\eta)\,.
\end{eqnarray*}

\medskip\noindent 
5. To prove that $\Xi$ is compatible with $\tau_M$ and $\sdtau_\dM$ we must show that $\Xi(\sdsigma^*\eta\circ\tau_M)=\sdtau_\dM(\sdsigma^*\dm_\eta)$. This is true because $\tau_M^{-i-1}=\sigma^{i*}\tau_M^{-1}\circ\ldots\circ\sigma^*\tau_M^{-1}\circ\tau_M^{-1}$ implies $\tau_M\circ\sdsigma^{i*}\tau_M^{-i-1}=\sdsigma^{i*}\tau_M^{-i}$, and hence,
\begin{eqnarray*}
\rho\circ\dm_{(\sdsigma^*\eta\,\circ\,\tau_M)} & := & \left(\sum\limits_{i=0}^\infty \Bigl(\Res_\infty(\sdsigma^*\eta\circ\tau_M)\bigl(\sdsigma^{i*}(\tau_M^{-i-1}m_j)\bigr)\Bigr)^{q^i}\!\!\cdot\tau^i\right)_{j=1}^d\\
& = & \left(\sum\limits_{i=0}^\infty \Bigl(\Res_\infty(\sdsigma^*\eta)\bigl(\sdsigma^{i*}(\tau_M^{-i}m_j)\bigr)\Bigr)^{q^i}\!\!\cdot\tau^i\right)_{j=1}^d\\
& = & \left(\sum\limits_{i=0}^\infty \Bigl(\Res_\infty\sdsigma^*\bigl(\eta(\sdsigma^{(i-1)*}(\tau_M^{-i}m_j))\bigr)\Bigr)^{q^i}\!\!\cdot\tau^i\right)_{j=1}^d\\
& = & \left(\sum\limits_{i=1}^\infty \Bigl(\Res_\infty\eta\bigl(\sdsigma^{(i-1)*}(\tau_M^{-i}m_j)\bigr)\Bigr)^{q^{i-1}}\!\!\cdot\tau^{i-1}\right)_{j=1}^d\cdot\tau\\[2mm]
& = & \rho\circ\dm_\eta\circ\Frob_{q,\BG_a} \\[2mm]
& =: & \rho\circ\sdtau_\dM(\sdsigma^*\dm_\eta)\,,
\end{eqnarray*}
where in the fourth line the term $\Res_\infty\eta(\sigma^*m_j)$ for $i=0$ vanishes again by \cite[Theorem~9.3.22]{VillaSalvador} as $\eta(\sigma^*m_j)\in\Omega^1_{A_\BC/\BC}$.

\medskip\noindent 
6. We prove that the $A_\BC$-homomorphism $\Xi$ is injective. If $\dm_\eta=0$, then formula \eqref{EqThmMandDMofE} implies that $\Res_\infty\eta\bigl(\sdsigma^{i*}(\tau_M^{-i-1}m)\bigr)=0$ for all $i\ge0$ and all $m\in M$. We must show that $\eta=0$. Since $\eta\;\in\;\Hom_{A_\BC}(\sigma^*M,\Omega^1_{A_\BC/\BC})\;\subset\;\Hom_{\BC\dpl z\dpr}(\sigma^*\wh{M},\BC\dpl z\dpr dz)$ is $z$-adically continuous with $\sigma^*\wh M:=\sigma^*M\otimes_{A_\BC}\BC\dpl z\dpr$, the preimage $U:=\eta^{-1}(\BC\dbl z\dbr dz)$ is a $z$-adically open neighborhood of $0$ in $\sigma^*\wh M$. By Proposition~\ref{PropTau_MInvZDense}, $\sigma^*\wh M=U+\bigcup_{i\in\BN_0}\sdsigma^{i*}\tau_M^{-i-1}(M)$. Since the $\BC$-linear map $\Res_\infty\!\circ\,\eta$ is zero on $U$ and also on the second summand, it is zero on all of $\sigma^*\wh M$. This implies that $\eta=0$.

\medskip\noindent 
7. If $\Xi$ is an isomorphism, then $\dM=\dM(\ulE)$ is locally free over $A_\BC$ of rank equal to $\rk\ulE$, because $M$ and hence $\Hom_{A_\BC}(\sigma^*M,\Omega^1_{A_\BC/\BC})$ are, as $\ulE$ is abelian. So $\ulE$ is $A$-finite.

\medskip\noindent 
8. Conversely, assume that $\ulE$ is $A$-finite, that is, $\dM$ is locally free over $A_\BC$ of rank equal to~$\rk\ulE$. Since also $M$ and hence, $\Hom_{A_\BC}(\sigma^*M,\Omega^1_{A_\BC/\BC})$ are locally free over $A_\BC$ of rank equal to~$\rk\ulE$, as $\ulE$ is abelian, an argument analogous to \cite[Proposition~3.1.2]{Taelman} shows that $\coker\Xi$ is annihilated by an element $a\in A$ (and not just by an element of $A_\BC$); see also \cite[Corollary~5.4]{BH1}. We use this to prove the surjectivity of $\Xi$ in the next step.

\medskip\noindent 
9. To prove that $\Xi$ is surjective, when $\ulE$ is $A$-finite, take for the moment an arbitrary element $a\in A\setminus\BF_q$ and let $\eta\in\Hom_{A_\BC}(\sigma^*M,\Omega^1_{A_\BC/\BC})$ be such that $\eta-(\sdsigma^*\eta\circ\tau_M)=a\,\eta'$ for some $\eta'\in\Hom_{A_\BC}(\sigma^*M,\Omega^1_{A_\BC/\BC})$, where $\sdsigma^*\eta\in\Hom_{A_\BC}(M,\Omega^1_{A_\BC/\BC})$ and $\sdsigma^*\eta\circ\tau_M\in\Hom_{A_\BC}(\sigma^*M,\Omega^1_{A_\BC/\BC})$, as $\ulM(\ulE)$ is effective. Then $\dm_\eta-\sdtau_\dM(\sdsigma^*\dm_\eta)=a\,\dm_{\eta'}$ by parts~4 and 5 above. Moreover, let $m\in M$ be such that $m-\tau_M(\sigma^*m)=a\,m'$ for some $m'\in M$. Then we have a telescoping sum
\[
\sdsigma^{i*}(\tau_M^{-i-1}m)-\sigma^*m\;=\;a\sum_{j=0}^i\sdsigma^{j*}(\tau_M^{-j-1}m') \quad\text{for all}\quad i\,\ge\,0\,. 
\]
Since $\eta'(\sigma^*m),\,\sdsigma^*\bigl(\eta(\sigma^*m')\bigr)\in\Omega^1_{A_\BC/\BC}$ we have $\Res_\infty\eta'(\sigma^*m)=\Res_\infty\sdsigma^*\bigl(\eta(\sigma^*m')\bigr)=0$ by \cite[Theorem~9.3.22]{VillaSalvador}. Finally, by part 1 above there is an integer $N$ such that $\eta\bigl(\sdsigma^{n*}(\tau_M^{-n-1}m)\bigr)$ and $\eta'\bigl(\sdsigma^{n*}(\tau_M^{-n-1}m)\bigr)$ lie in $\BC\dbl z\dbr dz$ for all $n\ge N$. Since $a^{-1}\in z\BF_q\dbl z\dbr$, also $a^{-1}\eta\bigl(\sdsigma^{N*}(\tau_M^{-N-1}m)\bigr)\,\in\,\BC\dbl z\dbr dz$. For all such $n> N$ this implies
\begin{eqnarray}
(m\circ\dm_{\eta'})(1) & = & {\TS\sum\limits_{i=0}^n} \Bigl(\Res_\infty\eta'\bigl(\sdsigma^{i*}(\tau_M^{-i-1}m)\bigr)\Bigr)^{q^i} \nonumber \\
& = & {\TS\sum\limits_{i=0}^n} \Bigl({\TS\sum\limits_{j=0}^i}\Res_\infty a\,\eta'\bigl(\sdsigma^{j*}(\tau_M^{-j-1}m')\bigr)\Bigr)^{q^i} \nonumber \\
& = & {\TS\sum\limits_{i=0}^n} \Bigl({\TS\sum\limits_{j=0}^i}\Res_\infty \bigl(\eta-(\sdsigma^*\eta\circ\tau_M)\bigr)\bigl(\sdsigma^{j*}(\tau_M^{-j-1}m')\bigr)\Bigr)^{q^i} \nonumber \\
& = & {\TS\sum\limits_{i=0}^n} \Bigl({\TS\sum\limits_{j=0}^i}\Res_\infty \eta\bigl(\sdsigma^{j*}(\tau_M^{-j-1}m')\bigr)-{\TS\sum\limits_{j=0}^i}\Res_\infty \sdsigma^*\bigl(\eta(\sdsigma^{(j-1)*}(\tau_M^{-j}m'))\bigr)\Bigr)^{q^i} \nonumber \\
& = & {\TS\sum\limits_{i=0}^n} \Bigl({\TS\sum\limits_{j=0}^i}\Res_\infty \eta\bigl(\sdsigma^{j*}(\tau_M^{-j-1}m')\bigr)\Bigr)^{q^i} - {\TS\sum\limits_{i=0}^n} \Bigl({\TS\sum\limits_{j=0}^{i-1}}\Res_\infty \eta\bigl(\sdsigma^{j*}(\tau_M^{-j-1}m')\bigr)\Bigr)^{q^{i-1}} \label{Eqm(dm'_eta)(1)}\\
& = & \Bigl({\TS\sum\limits_{j=0}^n}\Res_\infty \eta\bigl(\sdsigma^{j*}(\tau_M^{-j-1}m')\bigr)\Bigr)^{q^n}\nonumber \\
& = & \Bigl({\TS\sum\limits_{j=0}^N}\Res_\infty \eta\bigl(\sdsigma^{j*}(\tau_M^{-j-1}m')\bigr)\Bigr)^{q^n}\nonumber \\
& = & {\TS\sum\limits_{j=0}^N}\Res_\infty \eta\bigl(\sdsigma^{j*}(\tau_M^{-j-1}m')\bigr) \nonumber \\
& = & \Res_\infty a^{-1}\eta\bigl(\sdsigma^{N*}(\tau_M^{-N-1}m)\bigr) - \Res_\infty a^{-1}\eta(\sigma^*m) \nonumber \\[2mm]
& = & - \Res_\infty a^{-1}\eta(\sigma^*m)\,,\nonumber 
\end{eqnarray}
where the independence of $n\ge N$ of the expression in the seventh line implies that this expression lies in $\BF_q$. Since $(m\circ\dm_{\eta'})(1)=m\bigl(\delta_1(a^{-1}(\dm_\eta-\sdtau_\dM(\sdsigma^*\dm_\eta))\bigr)$ by definition of $\delta_1$, it follows that the diagram \eqref{EqDiagMandDMofE} described in the next corollary is commutative. In this diagram the left horizontal arrow is injective, because if $\eta\in\bigl(\uldM(\ulM)/a\uldM(\ulM)\bigr)^\sdtau$ satisfies $\eta(\sigma^*m)\in a\Omega^1_{A/\BF_q}$ for all $m\in(\ulM/a\ulM)^\tau$, then $(\ulM/a\ulM)^\tau\otimes_{\BF_q}\BC\cong\ulM/a\ulM$ implies that $\eta(\sigma^*m)\in a\Omega^1_{A_\BC/\BC}$ for all $m\in M$, whence $\eta\in a\uldM(\ulM)$. This arrow is surjective because both $\Hom_{A/(a)}\bigl((\ulM/a\ulM)^\tau,\,\Omega^1_{A/\BF_q}/a\Omega^1_{A/\BF_q}\bigr)$ and $\bigl(\uldM(\ulM)/a\uldM(\ulM)\bigr)^\sdtau$ are locally free $A/(a)$-modules of rank $\rk\ulE$, and hence, are finite dimensional $\BF_q$-vector spaces of the same dimension, because $\ulE$ is $A$-finite.

This implies that $\Xi$ induces an isomorphism $\bigl(\uldM(\ulM)/a\uldM(\ulM)\bigr)^\sdtau\isoto\bigl(\uldM(\ulE)/a\uldM(\ulE)\bigr)^\sdtau$. Since $(\uldM/a\uldM)^\sdtau\otimes_{\BF_q}\BC\cong\uldM/a\uldM$ for every dual $A$-motive $\uldM$ over $\BC$, we conclude that $\Xi$ is an isomorphism $\uldM(\ulM)/a\uldM(\ulM)\isoto\uldM(\ulE)/a\uldM(\ulE)$. In particular if we take the element $a\in A$ from part 8 which annihilates the cokernel of $\Xi$ this shows that $\coker(\Xi)=0$ and that $\Xi$ is an isomorphism. Altogether we have proved the theorem
\end{proof}

Along with the proof of the theorem we also showed the following

\begin{corollary}\label{CorMandDMofE}
Let $\ulE$ be an abelian and $A$-finite Anderson $A$-module, and let $\uldM(\ulE)$ and $\ulM=\ulM(\ulE)$ be its associated (dual) $A$-motive. Let $a\in A$ and consider the dual $A$-motive $\uldM(\ulM):=\bigl(\Hom_{A_\BC}(\sigma^*M,\Omega^1_{A_\BC/\BC})\,,\,\tau_M\dual\bigr)$ from Proposition~\ref{PropDualizing}. Then the following diagram consisting of isomorphisms of $A/(a)$-modules is commutative
\begin{equation}\label{EqDiagMandDMofE}
\xymatrix  { 
\Hom_{A/(a)}\bigl((\ulM/a\ulM)^\tau,\,\Omega^1_{A/\BF_q}/a\Omega^1_{A/\BF_q}\bigr) \ar[d]^{\cong} & \bigl(\uldM(\ulM)/a\uldM(\ulM)\bigr)^\sdtau \ar[l]^{\qquad\qquad\cong}\ar[r]_{\cong}^{\TS\Xi} & \bigl(\uldM(\ulE)/a\uldM(\ulE)\bigr)^\sdtau \ar[d]_{\cong} \\
\ulE[a](\BC) \ar@{=}[rr] & & \ulE[a](\BC)\,, \\
\qquad \bigl(h_\eta\colon m\mapsto\eta(\sigma^*m)\bigr) \qquad \ar@{|->}[d] & \qquad \eta \qquad \ar@{|->}[l]\ar@{|->}[r]^{\TS\Xi} & \qquad \dm_\eta \qquad \ar@{|->}[d] \\
\qquad P \qquad \ar@{=}[rr] & & **{!R(0.55) =<13pc,2pc>} \objectbox{\delta_1\bigl(a^{-1}(\dm_\eta-\sdtau_\dM(\sdsigma^*\dm_\eta)\bigr)\,,}
}
\end{equation}
where the left horizontal arrow sends $\eta\in\bigl(\uldM(\ulM)/a\uldM(\ulM)\bigr)^\sdtau$ to $h_\eta:=(\eta\circ\sigma^*)|_{(\ulM/a\ulM)^\tau}$, where the right horizontal arrow is the isomorphism $\Xi$ from Theorem~\ref{ThmMandDMofE}, where the left vertical map is (up to a minus sign motivated by Theorem~\ref{ThmPeriodIsomForE} below) Anderson's isomorphism \cite[Proposition~1.8.3]{Anderson86} which sends 
$h\in\Hom_{A/(a)}\bigl((\ulM/a\ulM)^\tau,\,\Omega^1_{A/\BF_q}/a\Omega^1_{A/\BF_q}\bigr)$ to the point $P\in\ulE[a](\BC)$ satisfying $m(P)=-\Res_\infty a^{-1}h(m)$ for all $m\in(\ulM/a\ulM)^\tau$, and where the right vertical map is the isomorphism $\dm\mapsto \delta_1\bigl(a^{-1}(\dm-\sdtau_\dM(\sdsigma^*\dm)\bigr)$ from Proposition~\ref{PropSwitcheroo}.
\end{corollary}

\begin{proof}
The proof of the corollary was given in step 9 of the proof of Theorem~\ref{ThmMandDMofE}.
\end{proof}

The theorem naturally leads to the following

\begin{question}\label{QuestMandDMofE}
If $\ulE$ is an abelian and $A$-finite Anderson $A$-module, the inverse of the isomorphism $\Xi$ from Theorem~\ref{ThmMandDMofE} defines a perfect pairing of $A_\BC$-modules
\[
\dM(\ulE)\otimes_{A_\BC}\sigma^* M(\ulE)\;\longto\;\Omega^1_{A_\BC/\BC}\,,\quad \dm\otimes\sigma^* m\;\longmapsto\;\Xi^{-1}(\dm)(\sigma^*m)\,.
\]
Is it possible to give a direct description of this pairing, that is an explicit formula of the differential form $\Xi^{-1}(\dm)(\sigma^*m)$ in terms of $\dm$ and $m$ ?
\end{question}

\comment{Prove Question \ref{QuestMandDMofE} for uniformizable mixed abelian and $A$-finite Anderson $A$-modules as an application of the Hodge-theory.}

For Drinfeld $\BF_q[t]$-modules the question has an affirmative answer as follows.

\begin{example}\label{exampleDModDMotdDmot}
Let $C=\BP^1_{\BF_q}$, $A=\BF_q[t]$, $A_{\BC}=\BC[t]$, $\theta:=\charmorph(t)$, and $J=(t-\theta)$. Also we choose $z=\tfrac{1}{t}$ as the uniformizing parameter at $\infty$. Then $\Omega^1_{A_\BC/\BC}=\BC[t]\cdot dt$ and $dt=-\tfrac{1}{z^2}\,dz$. Let $\ulE=(E,\phi)$ be a Drinfeld $\BF_q[t]$-module given by $E=\BG_{a,\BC}$ and 
\[
\phi_t\;=\;\psi_{\theta}+\psi_{\alpha_1}\circ \tau+\ldots+\psi_{\alpha_r}\circ \tau^r
\]
with $\alpha_i\in\BC$ and $\alpha_r\ne0$. Then the powers $\dm_k:=\tau^k$ for $k=0,\ldots,r-1$ form a $\BC[t]$-basis of $\dM=\Hom_{\BF_q,\BC}(\BG_a,E)$ on which $\sdtau_\dM$ acts via $\sdtau_\dM(\tau^i)=\tau^{i+1}$ for $0\le i<r-1$ and 
\begin{eqnarray*}
\sdtau_\dM(\tau^{r-1})\es=\es \tau^r&=& \phi_{t}\circ\psi_{1/\alpha_r^{q^{-r}}}-\psi_{\theta/\alpha_r^{q^{-r}}}-\tau\circ\psi_{\alpha_1^{q^{-1}}/\alpha_r^{q^{-r}}}-\ldots-\tau^{r-1}\circ\psi_{\alpha_{r-1}^{q^{-(r-1)}}/\alpha_r^{q^{-r}}}\\[2mm]
&=& (t-\theta)/\alpha_r^{q^{-r}}-\alpha_1^{q^{-1}}/\alpha_r^{q^{-r}}\cdot \tau-\ldots-\alpha_{r-1}^{q^{-(r-1)}}/\alpha_r^{q^{-r}}\cdot \tau^{r-1}\,.
\end{eqnarray*}
Thus with respect to this basis of $\dM$ and the induced basis of $\sdsigma^\ast\dM$ the $\BC[t]$-linear map $\sdtau_\dM$ is given by the matrix
\[
\check\Phi
\;=\;\left( \raisebox{9.6ex}{$
\xymatrix @C=0.6pc @R=0.3pc {
0 \ar@{.}[drrdrrdrrd] \ar@{.}[rrrrrr] && && && 0 \ar@{.}[dddd]& (t-\theta)/\alpha_r^{q^{-r}}\\
1\ar@{.}[drrdrrdrrd] && && && & -\alpha_1^{q^{-1}}/\alpha_r^{q^{-r}}\ar@{.}[dddd]\\
0\ar@{.}[ddd]\ar@{.}[drrdrrd] && && && &  \\
\\
&& && && 0 &\\
0 \ar@{.}[rrrr] && && 0 && 1 &\hspace{-3mm} \es-\alpha_{r-1}^{q^{-(r-1)}}/\alpha_r^{q^{-r}} \\
}$}
\right)
\]
In particular $\ulE$ is $A$-finite.

On the other hand the powers $m_j:=\tau^j$ for $j=0,\ldots,r-1$ also form a $\BC[t]$-basis of $M=\Hom_{\BF_q,\BC}(E,\BG_a)$ on which $\tau_M$ acts via $\tau_M(\tau^i)=\tau^{i+1}$ for $0\le i<r-1$ and 
\begin{eqnarray*}
\tau_M(\tau^{r-1})\es=\es \tau^r&=&\psi_{1/\alpha_r}\circ\phi_t-\psi_{\theta/\alpha_r}-\psi_{\alpha_1/\alpha_r}\circ \tau-\ldots-\psi_{\alpha_{r-1}/\alpha_r}\circ \tau^{r-1}\\[2mm]
&=& (t-\theta)/\alpha_r-\alpha_1/\alpha_r\cdot \tau-\ldots-\alpha_{r-1}/\alpha_r\cdot \tau^{r-1}\,.
\end{eqnarray*}
Thus with respect to this basis of $M$ and the induced basis of $\sigma^\ast M$ the $\BC[t]$-linear map $\tau_M$ is given by the matrix
\[
\Phi\;=\;\left( \raisebox{9ex}{$
\xymatrix @C=0.6pc @R=0.3pc {
0 \ar@{.}[drrdrrdrrd] \ar@{.}[rrrrrr] && && && 0 \ar@{.}[dddd]& (t-\theta)/\alpha_r\\
1\ar@{.}[drrdrrdrrd] && && && & -\alpha_1/\alpha_r\ar@{.}[dddd]\\
0\ar@{.}[ddd]\ar@{.}[drrdrrd] && && && &  \\
\\
&& && && 0 &\\
0 \ar@{.}[rrrr] && && 0 && 1 &\hspace{-3mm} \quad-\alpha_{r-1}/\alpha_r 
}$}
\right)
\]
In particular $\ulE$ is also abelian.

Let $\eta_\ell\in\dM(\ulM)=\Hom_{A_\BC}(\sigma^*M,\Omega^1_{A_\BC/\BC})$ for $\ell=0,\ldots,r-1$ be the basis dual to $(\sigma^*m_j)_j$ which is given by $\eta_\ell(\sigma^*m_j)=\delta_{j\ell}\,dt=-\tfrac{\delta_{j\ell}}{z^2}\,dz$, where $\delta_{j,\ell}$ is the Kronecker delta. We want to compute the matrix representing the isomorphism $\Xi$ from Theorem~\ref{ThmMandDMofE} with respect to the bases $(\eta_\ell)_\ell$ and $(\dm_k)_k$. For this purpose we have to compute $\sdsigma^{i*}(\stau_M^{-i-1}m_j)\in \sigma^*M\otimes_{A_\BC}\BC\dpl z\dpr$ modulo $z^2$, because $\eta_\ell\bigl(\bigoplus_j z^2\,\BC\dbl z\dbr\cdot\sigma^* m_j\bigr)\subset\BC\dbl z\dbr dz$ and the elements of the latter have residue $0$ at $\infty$. We set $\alpha_i:=0$ for $i>r$ and observe $\tfrac{1}{t-\theta}=\tfrac{z}{1-\theta z}\in z\,\BC\dbl z\dbr$. By induction on $i$ one easily verifies that the matrix $\Phi^{-1}\cdot\ldots\cdot\sdsigma^{i*}\Phi^{-1}$, which represents $\sdsigma^{i*}\stau_M^{-i-1} \;:=\; \stau_M^{-1}\circ\ldots\circ\sdsigma^{i*}\stau_M^{-1}$ with respect to the basis $(m_j)_j$, is congruent to
\[
\left( \raisebox{18ex}{$
\xymatrix @C=0.6pc @R=0.3pc {
\DS\sdsigma^{i*}\Bigl(\frac{\alpha_{1+i}\,z}{1-\theta z}\Bigr) & \DS\sdsigma^{(i-1)*}\Bigl(\frac{\alpha_i\,z}{1-\theta z}\Bigr) \ar@{.}[rr] & & \DS\frac{\alpha_1\,z}{1-\theta z} & 1 \ar@{.}[drrdrrdrrd] && 0 \ar@{.}[drrdrrd] \ar@{.}[rrrr] && && 0 \ar@{.}[ddd]\\
\DS\sdsigma^{i*}\Bigl(\frac{\alpha_{2+i}\,z}{1-\theta z}\Bigr) \ar@{.}[ddd] & \DS\sdsigma^{(i-1)*}\Bigl(\frac{\alpha_{1+i}\,z}{1-\theta z}\Bigr) \ar@{.}[ddd] \ar@{.}[rr] & & \DS\frac{\alpha_2\,z}{1-\theta z} \ar@{.}[ddd] & 0\ar@{.}[ddddddd]\ar@{.}[drrdrrdrrd] \\
\\
& & & & && && && 0 \\
\DS\sdsigma^{i*}\Bigl(\frac{\alpha_{r-1}\,z}{1-\theta z}\Bigr) & \DS\sdsigma^{(i-1)*}\Bigl(\frac{\alpha_{r-2}\,z}{1-\theta z}\Bigr) \ar@{.}[rr] & & \DS\frac{\alpha_{r-i-1}\,z}{1-\theta z} & && && && 1\\
\DS\sdsigma^{i*}\Bigl(\frac{\alpha_r\,z}{1-\theta z}\Bigr) \ar@{.}[ddd] & \DS\sdsigma^{(i-1)*}\Bigl(\frac{\alpha_{r-1}\,z}{1-\theta z}\Bigr) \ar@{.}[ddd] \ar@{.}[rr] & & \DS\frac{\alpha_{r-i}\,z}{1-\theta z} \ar@{.}[ddd] & && && && 0 \ar@{.}[ddd]\\
\\
\\
\DS\sdsigma^{i*}\Bigl(\frac{\alpha_{r+i}\,z}{1-\theta z}\Bigr) & \DS\sdsigma^{(i-1)*}\Bigl(\frac{\alpha_{r-1+i}\,z}{1-\theta z}\Bigr) \ar@{.}[rr] & & \DS\frac{\alpha_r\,z}{1-\theta z} & 0 \ar@{.}[rrrrrr] && && && 0
}$}
\right)
\]
modulo $z^2\,\BC\dbl z\dbr^{r\times r}$ for $i=0,\ldots,r-1$, and to
\[
\left( \raisebox{9.3ex}{$
\xymatrix @C=0.3pc @R=0.3pc {
0 \ar@{.}[rrrr] \ar@{.}[ddddd] && && 0 \ar@{.}[drdrdrr] && \DS\sdsigma^{(r-1)*}\Bigl(\frac{\alpha_r\,z}{1-\theta z}\Bigr) \ar@{.}[rr] \ar@{.}[drdr] & & \DS\sdsigma^{(i-r+1)*}\Bigl(\frac{\alpha_{i-r+2}\,z}{1-\theta z}\Bigr) \ar@{.}[dd] \\
\\
&& && && & & \DS\sdsigma^{(i-r+1)*}\Bigl(\frac{\alpha_r\,z}{1-\theta z}\Bigr) \\
&& && && & & 0 \ar@{.}[dd] \\
\\
0 \ar@{.}[rrrrrrrr] && && && & & 0
}$}
\right)
\]
for $i=r-1,\ldots,2r-2$, and to the zero matrix for $i\ge 2r-1$. It follows that 
\[
\Res_\infty \eta_\ell\bigl(\sdsigma^{i*}(\tau_M^{-i-1}m_j)\bigr)\;=\;\left\{\begin{array}{cl} -\sdsigma^{(i-j)*}(\alpha_{\ell+1+i-j}) & \quad\text{for } j\le i\,,\\ 0 & \quad\text{for } j>i\,,\end{array}\right.
\]
and hence, $m_j\circ\dm_{\eta_\ell}=-\sum_{i=j}^{2r-2}\alpha_{\ell+1+i-j}^{q^j}\tau^i=-\tau^j\cdot\sum_{k=0}^{r-1-\ell}\tau^k\alpha_{k+\ell+1}^{q^{-k}}$ for $k=i-j$. These equations are equivalent to $\Xi(\eta_\ell)=\dm_{\eta_\ell}=-\sum_{k=0}^{r-1-\ell}\alpha_{k+\ell+1}^{q^{-k}}\dm_k$. Therefore, $\Xi$ is represented with respect to the bases $(\eta_\ell)_\ell$ and $(\dm_k)_k$ by the matrix
\begin{equation}\label{EqMatrixXi}
X\;:=\;-\left( \raisebox{9.3ex}{$
\xymatrix @C=0.3pc @R=0.3pc {
\alpha_1\!\!\!\!\!\! & \alpha_2 \ar@{.}[rrr] & & & \alpha_r \\
\alpha_2^{q^{-1}}\!\!\!\!\!\! \ar@{.}[ddd] & \alpha_3^{q^{-1}} \ar@{.}[rr] \ar@{.}[dd] & & \alpha_r^{q^{-1}} & 0 \ar@{.}[ddd]\\
\\
& \alpha_r^{q^{2-r}} \ar@{.}[ruru] \\
\alpha_r^{q^{1-r}}\!\!\!\!\!\! & 0 \ar@{.}[rrr] \ar@{.}[rururu] & & & 0
}$}
\right)\;\in\;\GL_r(\BC)\;\subset\;\GL_r(\BC[t])\,.
\end{equation}
Note that the compatibility of $\Xi$ with $\tau_M$ and $\sdtau_\dM$ from equation~\eqref{EqThmMandDMofECompatib} corresponds to the equation
\begin{equation}\label{EqRelationMatrixXi}
X\cdot\Phi^T\;=\;
-\left( \raisebox{6ex}{$
\xymatrix @C=0pc @R=0pc {
t-\theta & 0 \ar@{.}[rrr] & & & 0 \\
0 \ar@{.}[ddd] & \alpha_2^{q^{-1}} \ar@{.}[rrr] \ar@{.}[ddd] & & & \alpha_r^{q^{-1}} \\
& & & & 0 \ar@{.}[dd]\\
\\
0 & \alpha_r^{q^{1-r}} \ar@{.}[rururu] & 0 \ar@{.}[rr] \ar@{.}[ruru] & & 0
}$}
\right)
\;=\;\check\Phi\cdot\sdsigma^*(X)\,,
\end{equation}
which is easily verified. In particular, if 
\[
X^{-1}\;=\;\left( \raisebox{5ex}{$
\xymatrix @C=0pc @R=0.3pc {
0 \ar@{.}[rrr]\ar@{.}[dd] & & & 0 & \beta_{0,r-1}\ar@{.}[ddd] \\
\\
0 \ar@{.}[rurur] \\
\beta_{r-1,0} \ar@{.}[rururur] \ar@{.}[rrrr] & & & & \beta_{r-1,r-1}
}$}
\right)
\]
denotes the inverse of the matrix $X$ from \eqref{EqMatrixXi} then the pairing from Question~\ref{QuestMandDMofE} is explicitly given by
\[
\sum_{k=0}^{r-1}\df_k\,\dm_k \;\otimes\;\sum_{j=0}^{r-1}f_j\,\sigma^*m_j \es\longmapsto\es \sum_{j=0}^{r-1}\sum_{k=r-1-j}^{r-1} f_j\beta_{j,k}\df_k\,dt
\]
with $\df_k,f_j\in\BC[t]$ for $0\le j,k\le r-1$.

\medskip\noindent
(b) More generally let $\ulE=(E,\phi)$ be an Anderson $\BF_q[t]$-module given by $E=\BG_{a,\BC}^d$ and 
\[
\phi_t\;=\;\Delta_0+\Delta_1\circ \tau+\ldots+\Delta_s\circ \tau^s
\]
with $\Delta_i\in\BC^{d\times d}$, such that $(\Delta_0-\theta)^d=0$. Assume that $\Delta_s\in\GL_d(\BC)$. Then (with the Kronecker delta) the elements 
\[
\dm_{k,\nu}\;:=\;\left(\begin{array}{c}\TS\delta_{1,\nu}\tau^k \\ \vdots\\ \TS\delta_{d,\nu}\tau^k \end{array}\right)\!:\;x\;\longmapsto\;\left(\begin{array}{c}\TS\delta_{1,\nu}x^{q^k} \\ \vdots\\ \TS\delta_{d,\nu}x^{q^k} \end{array}\right)
\]
for $\nu=1,\ldots,d$ and $k=0,\ldots,s-1$ form a $\BC[t]$-basis of $\dM=\Hom_{\BF_q,\BC}(\BG_a,E)$. And the elements 
\[
m_{j,\nu}\;:=\;(\delta_{1,\nu}\tau^j,\ldots,\delta_{d,\nu}\tau^j):\;\left(\begin{array}{c} \TS x_1\\ \vdots \\ \TS x_d \end{array}\right)\;\longmapsto\;x_\nu^{q^j}
\]
for $\nu=1,\ldots,d$ and $j=0,\ldots,s-1$ form a $\BC[t]$-basis of $M=\Hom_{\BF_q,\BC}(E,\BG_a)$. A similar computation as in (a) shows that with respect to these bases of $\dM$ and $M$ and the induced bases of $\sdsigma^\ast\dM$ and $\sigma^\ast M$ the $\BC[t]$-linear maps $\sdtau_\dM$ and $\tau_M$ are given by the matrices
\[
\check\Phi
\;=\;\left( \raisebox{9.6ex}{$
\xymatrix @C=0.6pc @R=0.3pc {
0 \ar@{.}[drrdrrdrrd] \ar@{.}[rrrrrr] && && && 0 \ar@{.}[dddd]& (t-\Delta_0)\cdot\sdsigma^{s*}(\Delta_s^{-1})\\
\Id_d\ar@{.}[drrdrrdrrd] && && && & -\sdsigma^*(\Delta_1)\cdot\sdsigma^{s*}(\Delta_s^{-1})\ar@{.}[dddd]\\
0\ar@{.}[ddd]\ar@{.}[drrdrrd] && && && &  \\
\\
&& && && 0 &\\
0 \ar@{.}[rrrr] && && 0 && \Id_d &\hspace{-3mm} \es -\sdsigma^{(s-1)*}(\Delta_{s-1})\cdot\sdsigma^{s*}(\Delta_s^{-1}) \\
}$}
\right)
\]
and 
\[
\Phi\;=\;\left( \raisebox{9ex}{$
\xymatrix @C=0.6pc @R=0.3pc {
0 \ar@{.}[drrdrrdrrd] \ar@{.}[rrrrrr] && && && 0 \ar@{.}[dddd]& (t-\Delta_0^T)\cdot(\Delta_s^{-1})^T\\
\Id_d\ar@{.}[drrdrrdrrd] && && && & -\Delta_1^T\cdot(\Delta_s^{-1})^T \ar@{.}[dddd]\\
0\ar@{.}[ddd]\ar@{.}[drrdrrd] && && && &  \\
\\
&& && && 0 &\\
0 \ar@{.}[rrrr] && && 0 && \Id_d &\hspace{-3mm} -\Delta_{s-1}^T\cdot(\Delta_s^{-1})^T 
}$}
\right)
\]
In particular $\ulE$ is $A$-finite and abelian of dimension $d$ and rank $r:=sd$ and pure of weight $-s$.

Let $\eta_{\ell,\lambda}\in\dM(\ulM)=\Hom_{A_\BC}(\sigma^*M,\Omega^1_{A_\BC/\BC})$ for $\lambda=1,\ldots,d$ and $\ell=0,\ldots,s-1$ be the basis dual to $(m_{j,\nu})_{(j,\nu)}$ which is given by $\eta_{\ell,\lambda}(\sigma^*m_{j,\nu})=\delta_{j,\ell}\delta_{\nu,\lambda}\,dt=-\tfrac{\delta_{j,\ell}\delta_{\nu,\lambda}}{z^2}\,dz$. A similar computation as in (a) then shows that $\Xi$ is represented with respect to the bases $(\eta_{\ell,\lambda})_{(\ell,\lambda)}$ and $(\dm_{k,\nu})_{(k,\nu)}$ by the matrix
\[
X\;:=\;-\left(\qquad \raisebox{7.3ex}{$
\xymatrix @C=0.3pc @R=0.3pc {
\!\!\!\!\!\Delta_1\!\!\!\!\! & \Delta_2 \ar@{.}[rrr] & & & \Delta_s \\
\!\!\!\!\!\sdsigma^*\Delta_2\!\!\!\!\! \ar@{.}[ddd] & \sdsigma^*\Delta_3 \ar@{.}[rr] \ar@{.}[dd] & & \sdsigma^*\Delta_s & 0 \ar@{.}[ddd]\\
\\
& \sdsigma^{(s-2)*}\Delta_s \ar@{.}[ruru] \\
\!\!\!\!\!\!\!\!\!\!\sdsigma^{(s-1)*}\Delta_s\!\!\!\!\!\!\!\!\!\! & 0 \ar@{.}[rrr] \ar@{.}[rururu] & & & 0
}$}
\right)\;\in\;\GL_r(\BC)\;\subset\;\GL_r(\BC[t])
\]
which satisfies
\[
X\cdot\Phi^T\;=\;
-\left( \raisebox{5.2ex}{$
\xymatrix @C=0pc @R=0pc {
t-\Delta_0 & 0 \ar@{.}[rrrrrr] && && && 0 \\
0 \ar@{.}[ddd] & \sdsigma^*\Delta_2 \ar@{.}[rrrrrr] \ar@{.}[ddd] && && && \sdsigma^*\Delta_s \\
& && && && 0 \ar@{.}[dd]\\
\\
0 & \sdsigma^{(s-1)*}\Delta_s \ar@{.}[rrurrurru] && 0 \ar@{.}[rrrr] \ar@{.}[rrurru] && && 0
}$}
\right)
\;=\;\check\Phi\cdot\sdsigma^*(X)\,.
\]
\end{example}

\begin{corollary}\label{CorDriModisAFinite}
Let $\ulE=(E,\phi)$ be a Drinfeld $A$-module. Then $\ulE$ is abelian and $A$-finite.
\end{corollary}

\begin{proof}
Fix an element $t\in A\setminus\BF_q$ and consider the finite flat ring homomorphism $\wt A:=\BF_q[t]\into A$. By restricting $\phi|_{\wt A}\colon\wt A\to\End_{\BF_q,\BC}(E)$ we view $\ulE$ as a Drinfeld $\BF_q[t]$-module. Then $M(\ulE)$ and $\dM(\ulE)$ are finite free modules over $\wt A_\BC=\BC[t]$ by Example~\ref{exampleDModDMotdDmot}. Therefore, they are finite and torsion free, hence locally free modules over the Dedekind domain $A_\BC$.
\end{proof}

\subsection{Analytic theory of $A$-finite Anderson $A$-modules}\label{SectAnalytThyAFinite}

We equip the $\BC$-vector spaces of matrices $\BC^{d'\times d}$ and vectors $\BC^d=\BC^{d\times1}$ with the maximum norm $\|\,.\,\|$ given by $\|(x_{ij})\|:=\max\{\,|x_{ij}|\colon \text{all }i,j\,\}$. Then $\|BC\|\le\|B\|\cdot\|C\|$ for all matrices $B,C$. All norms on these spaces are equivalent by \cite[Theorem~13.3]{Schikhof} and induce the same topology.

\begin{lemma}\label{Lemma5.16}
Let  $f\colon\BG_{a,\BC}^d\to\BG_{a,\BC}^{d'}$ be a homomorphism of $\BF_q$-module schemes over $\BC$. Then $f$ induces a continuous $\BF_q$-linear map $f\colon \BG_{a,\BC}^d(\BC)=\BC^d\longto \BG_{a,\BC}^{d'}(\BC)=\BC^{d'}$. More precisely, there is a constant $C\in\BR_{\ge0}$ such that $\|f(y)\|\le C\cdot\|y\|$ for every $y\in\BC^d$ with $\|y\|\le1$.
\end{lemma}

\begin{proof}
Under the isomorphism $\Hom_{\BF_q,\BC}(\BG_{a,\BC}^d,\BG_{a,\BC}^{d'})\cong\BC\{\tau\}^{d'\times d}$ from Lemma~\ref{LemmaLucas} we write $f=\sum_{i\ge0}B_i\tau^i$ with $B_i\in\BC^{d'\times d}$ and $B_i=0$ for $i\gg0$. Let $C:=\max\{\,\|B_i\|\colon i\ge0\,\}$. For $y\in\BC^d$ with $\|y\|\le 1$ we have $\|\sigma^{i*}(y)\|=\|y\|^{q^i}\le\|y\|$, and therefore 
\begin{eqnarray*}
\|f(y)\| & = & \Bigl\|\,\sum_{i\ge0}B_i\sigma^{i*}(y)\Bigr\| \\[2mm]
& \le & \max\bigl\{\,\|B_i\sigma^{i*}(y)\|\colon i\ge0\,\bigr\} \\[2mm]
& \le & \max\bigl\{\,\|B_i\|\cdot\|\sigma^{i*}(y)\|\colon i\ge0\,\bigr\} \\[2mm]
& \le & C\cdot\|y\|\,.
\end{eqnarray*}
Since $f\colon \BC^d\to\BC^{d'}$ is $\BF_q$-linear, this shows that $f$ is continuous.
\end{proof}

\begin{definition}\label{DefDivisionTower}
Fix an $a\in A\setminus\BF_q$ and an $x\in E(\BC)$. 
\begin{enumerate}
\item \label{DefDivisionTower_A}
A sequence $x_{(0)}, x_{(1)},x_{(2)},\ldots\index{tdt@$\tdt$} \in E(\BC)$ is an \emph{$a$-division tower above $x$} if
\[
\phi_a(x_{(n)})= x_{(n-1)} \quad \text{for all } n>0 \qquad\text{and} \qquad \phi_a(x_{(0)}) = x \,.
\]
\item \label{DefDivisionTower_B}
An $a$-division tower $(x_{(n)})_{n\ge0}$ is said to be \emph{convergent} if for some (or, equivalently, any) isomorphism $\rho\colon E\isoto\BG_{a,\BC}^d$ of $\BF_q$-module schemes, $\lim\limits_{n\to\infty}\rho(x_{(n)})=0$ in the $\BC$-vector space $\BG_{a,\BC}^d(\BC)=\BC^d$.
\end{enumerate}
\end{definition}

\begin{proof}
We must explain, why the definition in~\ref{DefDivisionTower_B} is independent of $\rho$. For this purpose let $\tilde\rho\colon E\isoto\BG_{a,\BC}^d$ be another isomorphism. Then $\tilde\rho\circ\rho^{-1}\in\Aut_{\BF_q,\BC}(\BG_{a,\BC}^d)$ induces a homeomorphism $\tilde\rho\circ\rho^{-1}\colon\BC^d\to\BC^d$ by Lemma~\ref{Lemma5.16}. It follows that $\lim_{n\to\infty}\|\rho(x_{(n)})\|=0$ if and only if $\lim_{n\to\infty}\|\tilde\rho(x_{(n)})\|=0$ as claimed.
\end{proof}

If $\ulE$ is $A$-finite (or abelian) then $a$-division towers exist above every $x$. This follows from Theorem~\ref{ThmDivisionTower} (or respectively Proposition~\ref{PropCompTateModEandM}\ref{PropCompTateModEandM_B} below). But there may or may not exist convergent ones.

\begin{theorem}[{\cite{ABP_Rohrlich}}]\label{ThmDivTowersAndLie}
Let $\ulE$ be an Anderson $A$-module over $\BC$, let $x\in E(\BC)$, and let $a\in A\setminus\BF_q$. Then there is a canonical bijection
\begin{eqnarray}\label{EqDivTowersAndLie}
\{\,\xi\in\Lie E\colon \exp_\ulE(\xi)=x\,\} & \isoto & \{\,\text{convergent $a$-division towers above $x$}\,\}\nonumber \\[2mm]
\xi\qquad\qquad & \longmapsto & \Bigl(\exp_\ulE\bigl(\Lie\phi_a^{-n-1}(\xi)\bigr)\Bigr)_{n\in\BN_0}
\end{eqnarray}
If $\rho\colon E\isoto\BG_{a,\BC}^d$ is an isomorphism of $\BF_q$-module schemes and $\Lie\rho\colon\Lie E\isoto\BC^d$ is the induced isomorphism of Lie algebras then 
\begin{equation}\label{EqDivTowersAndLie2}
\lim_{n\to\infty}\Lie(\rho\circ\phi_a^{n+1}\circ\rho^{-1})\bigl(\rho(x_{(n)})\bigr)\;=\;(\Lie\rho)(\xi)
\end{equation}
holds in $\BC^d$ for all $\xi\in\Lie E$ with $x_{(n)}:=\exp_\ulE\bigl(\Lie\phi_a^{-n-1}(\xi)\bigr)$ for $n\ge0$.
\end{theorem}

\noindent
{\it Remark.} Equation~\eqref{EqDivTowersAndLie2} is the analog of the fact, that for a real or complex Lie group $G$ the exponential function $\exp_G\colon\Lie G\to G$ has derivative $1$ near the identity element of $G$ (with respect to any coordinate system). For example $\lim\limits_{n\to\infty}a^n\bigl(\exp(a^{-n}\xi) - 1\bigr)=\xi$ for $G=\BG_m$, where $\xi$ is any complex number and $a\in\BZ\setminus\{-1,0,1\}$.

\begin{proof}
The element $x_{(n)}:=\exp_\ulE\bigl(\Lie\phi_a^{-n-1}(\xi)\bigr)\in E(\BC)$ satisfies $\phi_a(x_{(n)})=\exp_\ulE\bigl(\Lie\phi_a^{-n}(\xi)\bigr)$. This equals $x_{(n-1)}$ when $n>0$ and it equals $x$ when $n=0$, hence, $(x_{(n)})_n$ is an $a$-division tower above $x$. By Lemmas~\ref{LemmaConvergent} and \ref{LemmaLog}, it is convergent and so the map is well defined.

If $\xi,\xi'\in\Lie E$ satisfy $\exp_\ulE\bigl(\Lie\phi_a^{-n}(\xi)\bigr)=\exp_\ulE\bigl(\Lie\phi_a^{-n}(\xi')\bigr)$ for all $n\ge0$ then Lemma~\ref{LemmaConvergent} implies that $\Lie\phi_a^{-n}(\xi)$ and $\Lie\phi_a^{-n}(\xi')$ converge to $0$ in $\Lie E$ and therefore $\Lie\phi_a^{-n}(\xi)=\Lie\phi_a^{-n}(\xi')$ for $n\gg0$ by Lemma~\ref{LemmaLog}. This implies $\xi=\xi'$, and hence the map is injective.

To prove surjectivity, let $(x_{(n)})_n$ be a convergent $a$-division tower above $x$. Since $(x_{(n)})$ converges to $0$ there is an $n_0\in\BN_0$ such that $\log_\ulE(x_{(n)})$ exists by Lemma~\ref{LemmaLog} for all $n\ge n_0$. We set $\xi:=\Lie\phi_a^{n_0+1}\bigl(\log_\ulE(x_{(n_0)})\bigr)$. Then $\Lie\phi_a^{n+1}\bigl(\log_\ulE(x_{(n)})\bigr)=\Lie\phi_a^{n_0+1}\bigl(\log_\ulE(\phi_a^{n-n_0}(x_{(n)}))\bigr)=\Lie\phi_a^{n_0+1}\bigl(\log_\ulE(x_{(n_0)})\bigr)=\xi$ for $n\ge n_0$ by Lemma~\ref{LemmaLog}. Therefore, $x_{(n)}=\exp_\ulE\bigl(\Lie\phi_a^{-n-1}(\xi)\bigr)$ for all $n\ge n_0$, and for $n<n_0$ we compute $x_{(n)}=\phi_a^{n_0-n}(x_{(n_0)})=\phi_a^{n_0-n}\bigl(\exp_\ulE(\Lie\phi_a^{-n_0-1}(\xi))\bigr)=\exp_\ulE\bigl(\Lie\phi_a^{-n-1}(\xi)\bigr)$.

It remains to prove \eqref{EqDivTowersAndLie2}. With respect to the coordinate system $\rho$ and $\Lie\rho$ we write $\phi_a$ as a matrix $\Delta_a:=\rho\circ\phi_a\circ\rho^{-1}=\sum_{i\ge0}\Delta_{a,i}\,\tau^i\in\BC\{\tau\}^{d\times d}$ and $\exp_\ulE$ as a matrix $\sum_{i=0}^\infty E_i\,\tau^i:=\rho\circ\exp_\ulE\circ(\Lie\rho)^{-1}$ with $\Delta_{a,i},E_i\in\BC^{d\times d}$ and $\Delta_{a,0}=\Lie(\rho\circ\phi_a\circ\rho^{-1})$ and $E_0=\Id_d$. By replacing $\rho$ by $\tilde\rho:=B\circ\rho$ for a matrix $B\in\GL_d(\BC)\subset\BC\{\tau\}^{d\times d}=\End_{\BF_q,\BC}(\BG_{a,\BC}^d)$ we can write $\Delta_{a,0}=\charmorph(a)(\Id_d+N)$ with strictly upper triangular (nilpotent) $N$ having only entries $0$ and $1$. This replacement is allowed because $\tilde\rho\circ\rho^{-1}=B$ is an automorphism of the $\BC$-vector space $\BC^d$. Then $\Lie(\rho\circ\phi_a^{n+1}\circ\rho^{-1})\bigl(\rho(x_{(n)})\bigr)=\Delta_{a,0}^{n+1}\sum_{i=0}^\infty E_i\sigma^{i*}\bigl(\Delta_{a,0}^{-n-1}\Lie\rho(\xi)\bigr)$. We consider the maximum norm $\|\,.\,\|$ on $\BC^d$ and $\BC^{d\times d}$. For $i>0$ the term $\Delta_{a,0}^{n+1}E_i\sigma^{i*}\bigl(\Delta_{a,0}^{-n-1}\Lie\rho(\xi)\bigr)$ equals
\begin{equation}\label{ExTermI}
\charmorph(a)^{n+1}(\Id_d+N)^{n+1}E_i\,\charmorph(a)^{-q^i(n+1)}(\Id_d+N)^{-n-1}\sigma^{i*}\bigl(\Lie\rho(\xi)\bigr)\,,
\end{equation}
and has norm less or equal to $\|E_i\|\,\|\Lie\rho(\xi)\|^{q^i}|\charmorph(a)|^{-(q^i-1)(n+1)}$, because $\|\,.\,\|$ is compatible with matrix multiplication and $\|\Id_d+N\|=1$. Since $|\charmorph(a)|>1$ and $\exp_\ulE$ converges on all of $\Lie E$, that is $\lim_{i\to\infty}\|E_i\|\,\|\Lie\rho(\xi)\|^{q^i}=0$, the terms \eqref{ExTermI} go to zero uniformly in $i$ when $n\to\infty$. Therefore, $\lim_{n\to\infty}\Delta_{a,0}^{n+1}\sum_{i=0}^\infty E_i\sigma^{i*}\bigl(\Delta_{a,0}^{-n-1}\Lie\rho(\xi)\bigr)=\Delta_{a,0}^{n+1}E_0\Delta_{a,0}^{-n-1}\Lie\rho(\xi)=\Lie\rho(\xi)$, proving \eqref{EqDivTowersAndLie2}. 
\end{proof}

From now on we assume that $\ulE$ is $A$-finite. The following theorem of Anderson \cite{ABP_Rohrlich} is crucial for the theory of uniformizability. Let $a\in A\setminus\BF_q$ and set $\dM_a:= \varprojlim \dM/a^n \dM$. If $v_1,\ldots,v_s$ are the maximal ideals of $A$ which contain $a$ then $\varprojlim A_\BC/(a^n)=\prod_{i=1}^sA_{\BC,v_i}$ and $\dM_a=\prod_{i=1}^s\dM\otimes_{A_\BC}A_{\BC,v_i}$. The latter equals the completion of $\dM$ at the closed subscheme $\Var(a)\subset\Spec A_\BC$. Since $\Var(a)\subset\FC_\BC\setminus\Disc$ there are natural inclusions $\CO(\FC_\BC\setminus\Disc)\into\prod_{i=1}^sA_{\BC,v_i}$ and $\dM\otimes_{A_\BC}\CO(\FC_\BC\setminus\Disc)\into\dM_a$. 

\begin{theorem}[{\cite{ABP_Rohrlich}}]\label{ThmDivisionTower}
\begin{enumerate}
\item \label{ThmDivisionTower_PartA}
Let $\ulE$ be an $A$-finite Anderson $A$-module and let $(\dM,\sdtau_\dM)=\uldM(\ulE)$ be its dual $A$-motive. Let $\dm\in\dM$ and $x:=\delta_1(\dm)=\dm(1)\in E(\BC)$. Then Proposition~\ref{PropSwitcheroo} defines a canonical bijection 
\begin{equation}\label{EqThmDivisionTower}
\bigl\{\,\dm'\in \dM_a\colon \sdtau_\dM(\sdsigma^*\dm')-\dm'=\dm\,\bigr\} \es \isoto \es \bigl\{\,\text{$a$-division towers }(x_{(n)})_n\text{ above }x\,\bigr\}
\end{equation}
as follows. Let $\dm'\in\dM_a$ satisfy $\sdtau_\dM(\sdsigma^*\dm')-\dm'=\dm$. For each $n\in\BN_0$ choose an $\dm'_n\in\dM$ with $\dm'\equiv\dm'_n\mod a^{n+1}\dM_a$. There is a uniquely determined $\dm''_n\in\dM$ with $a^{n+1}\dm''_n=\dm+\dm'_n-\sdtau_\dM(\sdsigma^*\dm'_n)$. Then $x_{(n)}:=\delta_1(\dm''_n)$.
\item \label{ThmDivisionTower_PartB}
Let $\dm'$ correspond to the $a$-division tower $(x_{(n)})_n$ under the bijection \eqref{EqThmDivisionTower}. Then the following are equivalent:
\begin{enumerate}
\item \label{ThmDivisionTower_A}
$\dm'\in\dM\otimes_{A_\BC}\CO(\FC_\BC\setminus\Disc)\subset\dM_a$,
\item \label{ThmDivisionTower_B}
$\dm'\in\dM\otimes_{A_\BC}\CO\bigl(\dotFC_\BC\setminus\bigcup_{i\in\BN_{>0}}\Var(\ssigma^{i\ast}J)\bigr)\subset\dM_a$,
\item \label{ThmDivisionTower_C}
$(x_{(n)})_n$ is convergent,
\item \label{ThmDivisionTower_D}
with respect to some (or, equivalently, any) isomorphism $\rho\colon E\isoto\BG_{a,\BC}^d$ of $\BF_q$-module schemes the sequence $\charmorph(a)^n\cdot\rho(x_{(n)})$ is bounded in the $\BC$-vector space $\BG_{a,\BC}^d(\BC)=\BC^d$.
\end{enumerate}
If these conditions hold and if $\xi\in\Lie E$ is the element from Theorem~\ref{ThmDivTowersAndLie} that corresponds to the convergent $a$-division tower $(x_{(n)})_n$, that is $x_{(n)}=\exp_\ulE\bigl(\Lie\phi_a^{-n-1}(\xi)\bigr)$ for all $n$, then $\xi=\delta_0(\dm'+\dm)$ for the map $\delta_0\colon\dM\to\Lie E$ from Proposition~\ref{prop:commutingdiagrams}.
\end{enumerate}
\end{theorem}

\begin{proof}
1. By Proposition~\ref{PropSwitcheroo} the definition of $x_{(n)}$ is independent of the chosen $\dm'_n$. In particular we can take $\dm'_{n-1}=\dm'_n$ and $\dm''_{n-1}=a\dm''_n$ to obtain $\phi_a(x_{(n)})=\delta_1(a\dm''_n)=\delta_1(\dm''_{n-1})=x_{(n-1)}$ and $\phi_a(x_{(0)})=\delta_1(a\dm''_0)=\delta_1(\dm)=x$. This defines the bijection \eqref{EqThmDivisionTower}. Note that we explicitly describe its inverse in part 5 below.

\medskip\noindent
2. To prove \ref{ThmDivisionTower_PartB}, note that trivially \ref{ThmDivisionTower_B}$\Longrightarrow$\ref{ThmDivisionTower_A} and \ref{ThmDivisionTower_D}$\Longrightarrow$\ref{ThmDivisionTower_C}, because $|\charmorph(a)|>1$.

\medskip\noindent
3. To prove \ref{ThmDivisionTower_C}$\Longrightarrow$\ref{ThmDivisionTower_D} for any isomorphism $\rho\colon E\isoto\BG_{a,\BC}^d$ of $\BF_q$-module schemes, we write $\rho\circ\phi_a\circ\rho^{-1}=:\Delta_a=\sum_{j\ge0}\Delta_{a,j}\,\tau^j\in\BC\{\tau\}^{d\times d}=\End_{\BF_q,\BC}(\BG_{a,\BC}^d)$ with $\Delta_{a,j}\in\BC^{d\times d}$ and $\Delta_{a,j}=0$ for $j\gg0$. By replacing $\rho$ by $\tilde\rho:=B\circ\rho$ for a matrix $B\in\GL_d(\BC)\subset\BC\{\tau\}^{d\times d}$ we can write $\Delta_{a,0}=\charmorph(a)(\Id_d+N)$ with strictly upper triangular (nilpotent) $N$ having only entries $0$ and $1$. This replacement is allowed because $\tilde\rho\circ\rho^{-1}=B$ is an automorphism of the $\BC$-vector space $\BG_{a,\BC}^d(\BC)$. Consider the maximum norm $\|x\|=\max\{|x_i|\colon i=0\ldots d\}$ for $x=(x_1,\ldots,x_d)^T\in\BC^d$ and the norm $\|y\|:=\|\rho(y)\|$ on $y\in E(\BC)$ induced via $\rho$. As $\Delta_{a,0}^{-1}=\charmorph(a)^{-1}(\Id_d-N+N^2-\ldots)$ we find $\|x\|=\|\Delta_{a,0}^{-1}\Delta_{a,0}\,x\|\le|\charmorph(a)|^{-1}\|\Delta_{a,0}\,x\|\le|\charmorph(a)|^{-1}|\charmorph(a)|\cdot\|x\|=\|x\|$, whence $\|\Delta_{a,0}\,x\|=|\charmorph(a)|\cdot\|x\|$. If $n\gg0$ then $\|x_{(n)}\|\ll 1$ by assumption \ref{ThmDivisionTower_C}, whence $\|\sigma^{j*}\rho(x_{(n)})\|=\|x_{(n)}\|^{q^j}\ll\|x_{(n)}\|$ for $j>0$. So $|\charmorph(a)|>1$ implies $\|\Delta_{a,j}\,\sigma^{j*}\rho(x_{(n)})\|<|\charmorph(a)|\cdot\|x_{(n)}\|=\|\Delta_{a,0}\,\rho(x_{(n)})\|$ for $n\gg0$ and all $j>0$. Thus $\|x_{(n-1)}\|=\|\phi_a(x_{(n)})\|=\|\sum_{j\ge0}\Delta_{a,j}\,\sigma^{j*}\rho(x_{(n)})\|=|\charmorph(a)|\cdot\|x_{(n)}\|$ for $n\gg0$, and this yields the boundedness of  $|\charmorph(a)|^n\cdot \|x_{(n)}\|$ and $\charmorph(a)^n\cdot\rho(x_{(n)})$.

\medskip\noindent
4. To prove \ref{ThmDivisionTower_D}$\Longrightarrow$\ref{ThmDivisionTower_B} and \ref{ThmDivisionTower_A}$\Longrightarrow$\ref{ThmDivisionTower_C} we choose an isomorphism $\rho\colon E\isoto\BG_{a,\BC}^d$ of $\BF_q$-module schemes and consider the induced $A_\BC$-isomorphism $\dagger\colon\dM\isoto\BC\{\sdtau\}^{1\times d}$, $\dm\mapsto\dm^\dagger$ from \eqref{EqDagger}. Moreover, under the finite flat ring homomorphism $\BF_q[t]\to A,\,t\mapsto a$ we have $A_\BC/(a^n)=A_\BC\otimes_{\BC[t]}\BC[t]/(t^n)$ and $\prod_{i=1}^sA_{\BC,v_i}=A_\BC\otimes_{\BC[t]}\BC\dbl t\dbr$, as well as $\CO(\FC_\BC\setminus\Disc)=A_\BC\otimes_{\BC[t]}\BC\langle t\rangle$; see \eqref{EqC<t>}. We also abbreviate $\theta:=\charmorph(a)$ and for a real number $s$ we use the notation
\begin{equation}\label{EqConvPowerSeries}
\BC\langle\tfrac{t}{\theta^s}\rangle\;:=\;\bigl\{\,{\TS\sum\limits_{i=0}^\infty} b_it^i\colon b_i\in\BC,\, \lim\limits_{i\to\infty}|b_i|\cdot|\theta|^{si}= 0\,\bigr\}\qquad\text{and}\qquad\BC\langle t\rangle\;:=\;\BC\langle\tfrac{t}{\theta^0}\rangle\,.
\end{equation}
We consider $\dM$ as a finite (locally) free module over $\BC[t]$ of rank $r$. We choose a $\BC[t]$-basis $\CB$ of $\dM$ and use it to identify $\dM_a\cong\BC\dbl t\dbr^{\oplus r}$ and $\dM\otimes_{A_\BC}\CO(\FC_\BC\setminus\Disc)\cong\BC\langle t\rangle^{\oplus r}$. Let $\|\,.\,\|$ denote the maximum norms on $\BC^r$, $\BC^d$ and $\BC^{1\times d}$, and consider the norm $\|y\|:=\|\rho(y)\|$ on $y\in E(\BC)$ and the norm $\|\sum_j c_j\sdtau^j\|_{_\sdtau}:=\sup\{\|c_i\|\colon i\ge0\}$ on $\BC\{\sdtau\}^{1\times d}$ and $\dM$ where $c_j\in\BC^{1\times d}$. For all $s\in\BR$ consider also the norm $\|\sum_i b_i t^i\|_{_s}:=\sup\{\|b_i\|\,|\charmorph(a)|^{si}\colon i\ge0\}$ on $\BC[t]^{\oplus r}$ and $\dM$ where $b_i\in\BC^r$. When $s\le s'$ these norms satisfy the inequalities $\|\,.\,\|_{_s}\le\|\,.\,\|_{_{s'}}$. Note that $\BC\langle\tfrac{t}{\theta^s}\rangle^{\oplus r}$ is the completion of $\BC[t]^{\oplus r}$ with respect to the norm $\|\,.\,\|_{_s}$, which therefore extends to $\BC\langle\tfrac{t}{\theta^s}\rangle^{\oplus r}$. 

\medskip\noindent
5. We now assume that \ref{ThmDivisionTower_D} holds for our fixed isomorphism $\rho$. For each $n\in\BN_0$ we let 
\[
(\dm''_n)^\dagger\;:=\;\rho(x_{(n)})^T\cdot\sdtau^0\;\in\;\BC^{1\times d}\,\sdtau^0\;\subset\;\BC\{\sdtau\}^{1\times d}\;\cong\;\dM\,. 
\]
We set $\dm''_{-1}:=\dm$ and $x_{(-1)}:=x$. Then $\delta_1(\dm''_n)=x_{(n)}$ for all $n\ge-1$, and hence, $\delta_1(t\dm''_n-\dm''_{n-1})=\phi_a(x_{(n)})-x_{(n-1)}=0$ implies that $t\dm''_n-\dm''_{n-1}=y_n-\sdtau_\dM(\sdsigma^*y_n)$ for an element $y_n\in\dM$ for $n\ge0$. Moreover, the elements $(t\dm''_n-\dm''_{n-1})^\dagger=(\dm''_n)^\dagger\cdot(\Delta_a)^\dagger-(\dm''_{n-1})^\dagger=\rho(x_{(n)})^T\cdot(\Delta_a)^\dagger-\rho(x_{(n-1)})^T$ lie in the finite dimensional $\BC$-vector space $W:=\bigoplus_{j=0}^\ell\BC^{1\times d}\,\sdtau^j$ where $\ell$ is the maximal $\sdtau$-degree of the entries of the matrix $(\Delta_a)^\dagger\in\BC\{\sdtau\}^{d\times d}$ corresponding to $\phi_a$. If $(y_n)^\dagger=:\sum_j c_j\sdtau^j\in\BC\{\sdtau\}^{1\times d}$ then $(t\dm''_n-\dm''_{n-1})^\dagger=\sum_j c_j\sdtau^j-\sdtau\cdot\sum_j c_j\sdtau^j=\sum_j\bigl(c_j-\sdsigma^*(c_{j-1})\bigr)\sdtau^j$. Writing $(t\dm''_n-\dm''_{n-1})^\dagger=:\sum_{j=0}^\ell\tilde c_j\sdtau^j$ we compute $\sdsigma^*(c_{j-1})=c_j-\tilde c_j$. Together with $c_j=0$ for $j\gg 0$ this implies $c_j=0$ for all $j\ge \ell$ and $c_j=\sum_{k=0}^j\sdsigma^{(j-k)*}(\tilde c_k)$ for $j<\ell$. So $(y_n)^\dagger\in W$. In particular, the series $\sum_{n=0}^\infty y_nt^n$ in $\dM_a\cong\BC\dbl t\dbr^{\oplus r}$ satisfies
\begin{equation}\label{EqFordm'}
\sdtau_\dM\bigl(\sdsigma^*({\TS\sum\limits_{n=0}^\infty}y_nt^n)\bigr)-({\TS\sum\limits_{n=0}^\infty}y_nt^n) \; = \; {\TS\sum\limits_{n=0}^\infty}(t^n\dm''_{n-1}-t^{n+1}\dm''_n) \; = \; \dm''_{-1}\; = \; \dm \,,
\end{equation}
whence $\dm'=\sum_{n=0}^\infty y_nt^n$ by Proposition~\ref{PropSwitcheroo}.

Moreover, our assumption \ref{ThmDivisionTower_D} that $|\charmorph(a)|^n\cdot\|\rho(x_{(n)})^T\|_{_\sdtau}$ is bounded together with $|\charmorph(a)|^n\cdot\|t\dm''_n\|_{_\sdtau}=|\charmorph(a)|^n\cdot\|\rho(x_{(n)})^T\cdot(\Delta_a)^\dagger\|_{_\sdtau}\le|\charmorph(a)|^n\cdot\|\rho(x_{(n)})^T\|_{_\sdtau}\cdot\|(\Delta_a)^\dagger\|_{_\sdtau}$ implies that $|\charmorph(a)|^n\cdot\|(t\dm''_n-\dm''_{n-1})^\dagger\|_{_\sdtau}=\|\sum_{j=0}^\ell\charmorph(a)^n\cdot\tilde c_j\sdtau^j\|_{_\sdtau}=\max\{\|\charmorph(a)^n\cdot\tilde c_j\|\colon 0\le j\le \ell\}$ is bounded, say by a constant $C_1\ge 1$. Therefore, $\|\sdsigma^{(j-k)*}(\tilde c_k)\|\le|\charmorph(a)|^{-nq^{k-j}}C_1^{q^{k-j}}\le|\charmorph(a)|^{-n/q^\ell}C_1$ for $0\le k\le j$ and thus $\|c_j\|\le|\charmorph(a)|^{-n/q^\ell}C_1$, whence $\|y_n\|_{_\sdtau}\le|\charmorph(a)|^{-n/q^\ell}C_1$. Fix an $s$ with $0<s<q^{-\ell}$. Since $\BC$ is complete with respect to $|\,.\,|$ the restrictions of the norms $\|\,.\,\|_{_\sdtau}$ and $\|\,.\,\|_{_s}$ to the finite dimensional $\BC$-vector space $W$ are equivalent by \cite[Theorem~13.3]{Schikhof}. Thus there is a constant $C_2$ with $\|\,.\,\|_s\le C_2\cdot\|\,.\,\|_\sdtau$ on $W$. Since $\dm''_n\in W$ we obtain in particular $\|t^{n+1}\dm''_n\|_{_s}=|\charmorph(a)|^{s(n+1)}\|\dm''_n\|_{_s}\le|\charmorph(a)|^{s(n+1)}\|\dm''_n\|_{_\sdtau}C_2=|\charmorph(a)|^{s(n+1)}\|\rho(x_{(n)})\|C_2=|\charmorph(a)|^{-n(1-s)+s}|\charmorph(a)|^n\|\rho(x_{(n)})\|C_2$ for all $n$, and hence, $\lim_{n\to\infty}\|t^{n+1}\dm''_n\|_{_s}=0$. Moreover, $\|y_n\|_{_s}\le\|y_n\|_{_\sdtau}C_2\le|\charmorph(a)|^{-n/q^\ell}C_1C_2$ for all $n$, whence $\lim_{n\to\infty}\|y_nt^n\|_{_s}=\lim_{n\to\infty}\|y_n\|_{_s}\,|\charmorph(a)|^{sn}=0$. This shows that even $\dm'\in\BC\langle\tfrac{t}{\theta^s}\rangle^{\oplus r}$ and equation \eqref{EqFordm'} holds in $\BC\langle\tfrac{t}{\theta^s}\rangle^{\oplus r}$. 

The matrix $\check\Phi\in\BC[t]^{r\times r}$ representing $\sdtau_\dM$ with respect to the basis $\CB$ has determinant $c\cdot(t-\theta)^d$ for a $c\in\BC\mal$ due to the elementary divisor theorem and the condition that $\coker\check\Phi\cong\dM/\sdtau_\dM(\sdsigma^*\dM)\cong\BC^d$ is annihilated by $(t-\theta)^d\in J^d$. Let $\check\Phi^\ad\in\BC[t]^{r\times r}$ be the adjoint matrix which satisfies $\check\Phi^\ad\check\Phi=c\cdot(t-\theta)^d\Id_r$. Recall the element $\ell_\zeta^{\SSC -}:=\prod_{i=0}^\infty\bigl(1-\tfrac{t}{\theta^{q^i}}\bigr)\in \CO(\dotFC_\BC)$ from Example~\ref{ExCarlitzMotiveUnif}(b) which satisfies $\ell_\zeta^{\SSC -}=-\tfrac{1}{\theta}(t-\theta)\cdot\sigma^*\ell_\zeta^{\SSC -}$. Multiplying \eqref{EqFordm'} with $\tfrac{\sigma^*(\ell_\zeta^{\SSC -})^d}{(-\theta)^d c}\check\Phi^\ad$, setting $y':=\sigma^*(\ell_\zeta^{\SSC -})^d\cdot\dm'\in\BC\langle\tfrac{t}{\theta^s}\rangle^{\oplus r}$ and applying $\sigma^*$ we obtain
\[
y'\;=\;\sigma^*\Bigl(\frac{1}{(-\theta)^d c}\sigma^*(\ell_\zeta^{\SSC -})^d\check\Phi^\ad\dm+\frac{1}{(-\theta)^d c}\check\Phi^\ad y'\Bigr)
\]
Since $\ssigma^*(y')\in\BC\langle\tfrac{t}{\theta^{qs}}\rangle^{\oplus r}$ this shows that $y'\in\BC\langle\tfrac{t}{\theta^{qs}}\rangle^{\oplus r}$ and iteratively $y'\in\BC\langle\tfrac{t}{\theta^{s'}}\rangle^{\oplus r}$ for all $s'=q^ks$, whence $y'\in\dM\otimes_{A_\BC}\CO(\dotFC_\BC)$ and $\dm'=\ssigma^*(\ell_\zeta^{\SSC -})^{-d}y'\in\dM\otimes_{A_\BC}\ssigma^*(\ell_\zeta^{\SSC -})^{-d}\CO(\dotFC_\BC)$. If $P\in\dotFC_\BC\setminus\bigcup_{i\in\BN_{>0}}\Var(\ssigma^{i\ast}J)$ is a point, that is $P=\Var(I)$ for a maximal ideal $I\subset\CO(\dotFC_\BC)$ with $I\ne\ssigma^{i\ast}J$ for all $i\in\BN_{>0}$, such that $P$ lies in the zero locus of $\sigma^*(\ell_\zeta^{\SSC -})$, then we make the

\medskip\noindent
Claim: $\dm'\in\dM\otimes_{A_\BC}\CO_{\dotFC_\BC,P}$ for the local ring $\CO_{\dotFC_\BC,P}$ of $\dotFC_\BC$ at $P$

\medskip\noindent
When the claim holds for all those $P$, we derive $\dm'\in\dM\otimes_{A_\BC}\CO\bigl(\dotFC_\BC\setminus\bigcup_{i\in\BN_{>0}}\Var(\ssigma^{i\ast}J)\bigr)$, that is assertion \ref{ThmDivisionTower_B}.

To prove the claim let $n\in\BN_{>0}$ be the integer with $t-\theta^{q^n}\in I$, which exists because $\sigma^*(\ell_\zeta^{\SSC -})$ vanishes at $P$. Then $t-\theta^{q^{n-j}}\in \sdsigma^{j*}I$ and $\sdsigma^{j*}I\ne J$ for all $0<j\le n$. Thus $\sdtau_\dM\colon\sdsigma^*\dM\otimes_{A_\BC}\CO_{\dotFC_\BC,\Var(\sdsigma^{j*}I)}\isoto\dM\otimes_{A_\BC}\CO_{\dotFC_\BC,\Var(\sdsigma^{j*}I)}$ is an isomorphism. For $j=n$ it follows from $\sigma^*(\ell_\zeta^{\SSC -})|_{t=\theta}\ne0$ that $\sigma^*(\ell_\zeta^{\SSC -})\in(\CO_{\dotFC_\BC,\Var(\sdsigma^{n*}I)})\mal$ and $\dm'=\sigma^*(\ell_\zeta^{\SSC -})^{-d}y'\in\dM\otimes_{A_\BC}\CO_{\dotFC_\BC,\Var(\sdsigma^{n*}I)}$. Therefore,
\[
\dm'\;=\;\ssigma^*\bigl(\sdtau_\dM^{-1}(\dm+\dm')\bigr)\;\in\;\ssigma^*\bigl(\sdtau_\dM^{-1}(\dM\otimes_{A_\BC}\CO_{\dotFC_\BC,\Var(\sdsigma^{n*}I)})\bigr)\;=\;\dM\otimes_{A_\BC}\CO_{\dotFC_\BC,\Var(\sdsigma^{(n-1)*}I)}
\]
and iteratively this yields $\dm'\in\dM\otimes_{A_\BC}\CO_{\dotFC_\BC,\Var(\sdsigma^{j*}I)}$ for $j=n,\ldots,0$. So our claim and with it assertion \ref{ThmDivisionTower_B} is proved.

\medskip\noindent
6. Conversely, to prove \ref{ThmDivisionTower_A}$\Longrightarrow$\ref{ThmDivisionTower_C}, we keep the notation from part 4 above and write $\dm'$ as $\sum_{i=0}^\infty b_it^i\in\BC\dbl t\dbr^{\oplus r}$ with $b_i\in\BC^r$ and assume \ref{ThmDivisionTower_A}, that is $\lim_{i\to\infty}b_i=0$ in $\BC^r$. For each $n\in\BN$ we set $\dm'_{n}:=\sum\limits_{i=0}^n b_it^i\in\dM$ and $\dm'_{> n}:=\sum\limits_{i=n+1}^\infty b_it^i$. Then 
\[
\dm''_n\;:=\;t^{-n-1}(\dm-\sdtau_\dM(\sdsigma^*\dm'_{n})+\dm'_{n})\;=\;t^{-n-1}(\sdtau_\dM(\sdsigma^*\dm'_{> n})-\dm'_{> n})\;\in\;\dM\,. 
\]
Note that the entries of $\dm''_n$ are polynomials in $\BC[t]$ whose degree is bounded by a bound which is independent of $n$ and only depends on the degrees of the entries of $\dm$ and of the matrix $\check\Phi\in\BC[t]^{r\times r}$ representing $\sdtau_\dM$ with respect to the basis $\CB$. It follows that all $\dm''_n$ lie in a finite dimensional $\BC$-vector space $V$. By \cite[Theorem~13.3]{Schikhof} the restrictions of $\|\,.\,\|_{_0}$ and $\|\,.\,\|_{_\sdtau}$ to $V$ are equivalent. From $\lim_{i\to\infty}b_i=0$ it follows that $\lim_{n\to\infty}\|\dm'_{>n}\|_{_0}=0$. Thus $\|\sdtau_\dM(\sdsigma^*\dm'_{>n})\|_{_0}\le \|\check\Phi\|_{_0}\|\dm'_{>n}\|_{_0}^{1/q}$ implies $\lim_{n\to\infty}\|\dm''_n\|_{_0}=0$, and hence, $\lim_{n\to\infty}\|\dm''_n\|_{_\sdtau}=0$. If $(\dm''_n)^\dagger=\sum_j c_j\sdtau^j\in\BC\{\sdtau\}^{1\times d}$ then $n\gg0$ implies $\|\dm''_n\|_{_\sdtau}=\max\{\|c_j\|\colon j\ge0\}\le 1$ and thus $\rho(x_{(n)})=\rho\bigl(\delta_1(\dm''_n)\bigr)=\sum_j\sigma^{j*}(c_j)^T$ satisfies \[
\|x_{(n)}\|\;\le\;\max\{\|\sigma^{j*}(c_j)\|\colon j\ge0\}\;\le\;\max\{\|c_j\|\colon j\ge0\}\;=\;\|\dm''_n\|_{_\sdtau}\,. 
\]
Therefore, $(x_{(n)})_n$ is convergent. Thus \ref{ThmDivisionTower_A} implies \ref{ThmDivisionTower_C}.

\medskip\noindent
7. Finally, for the last statement of the theorem we keep the notation from parts 4 and 6 above and assume moreover, that $\dm'=\sum_{i=0}^\infty b_it^i$ satisfies \ref{ThmDivisionTower_B}. Let $1<s<q$. Then $\Spm\BC\langle\tfrac{t}{\theta^s}\rangle\subset\dotFC_\BC\setminus\bigcup_{i\in\BN_{>0}}\Var(\ssigma^{i\ast}J)$ and this implies $\sum_{i=0}^\infty b_it^i\in\BC\langle\tfrac{t}{\theta^s}\rangle^{\oplus r}$, that is $\lim_{i\to\infty}\|b_i\|\,|\charmorph(a)|^{si}=0$; see \eqref{EqConvPowerSeries}. Fix a real number $\epsilon>0$ with $\epsilon\le\|\check\Phi\|_{_{s/q}}^{q/(q-1)}$. Then there is an $n_0\in\BN$ such that $\|b_i\|\,|\charmorph(a)|^{is/q}\le\|b_i\|\,|\charmorph(a)|^{is}<\epsilon$ for all $i\ge n_0$. So $n\ge n_0$ implies $\|\dm'_{>n}\|_{_{s/q}}\le\|\dm'_{>n}\|_{_s}<\epsilon\le\|\check\Phi\|_{_{s/q}}\epsilon^{1/q}$ and 
\[
\|\sdtau_\dM(\sdsigma^*\dm'_{>n})\|_{_{s/q}}\;\le\;\|\check\Phi\|_{_{s/q}}\cdot\|\sdsigma^*\dm'_{>n}\|_{_{s/q}}\;=\;\|\check\Phi\|_{_{s/q}}\cdot\|\dm'_{>n}\|_{_s}^{1/q}\;<\;\|\check\Phi\|_{_{s/q}}\epsilon^{1/q},
\] 
and hence, $\|\dm''_n\|_{_{s/q}}<|\charmorph(a)|^{-(n+1)s/q}\|\check\Phi\|_{_{s/q}}\epsilon^{1/q}$. We write $(\dm''_n)^\dagger=\sum_ic_i\sdtau^i\in\BC\{\sdtau\}^{1\times d}$. This time we use that by \cite[Theorem~13.3]{Schikhof} the restrictions of $\|\,.\,\|_{_{s/q}}$ and $\|\,.\,\|_{_\sdtau}$ to $V$ are equivalent. So there is a constant $C_3$ such that $\|\dm''_n\|_{_\sdtau}:=\sup\{\|c_i\|\colon i\ge0\}<|\charmorph(a)|^{-(n+1)s/q}\|\check\Phi\|_{_{s/q}}\epsilon^{1/q}C_3$ for all $n\ge n_0$. By enlarging $n_0$ we may assume that $|\charmorph(a)|^{-(n+1)s/q}\|\check\Phi\|_{_{s/q}}\epsilon^{1/q}C_3\le1$. Therefore, $\|c_i\|\le1$, whence $\|\sigma^{i*}c_i\|=\|c_i\|^{q^i}\le\|c_i\|^q$ for all $i\ge1$. So 
\[
\|\rho\bigl(\delta_1(\dm''_n)\bigr)-(\Lie\rho)\bigl(\delta_0(\dm''_n)\bigr)\|\;=\;\|\sum_{i\ge1}\sigma^{i*}(c_i)^T\|\;<\;|\charmorph(a)|^{-(n+1)s}\|\check\Phi\|_{_{s/q}}^q\epsilon C_3^q. 
\]
By choosing the isomorphism $\rho\colon E\isoto\BG_{a,\BC}^d$ appropriately in the beginning we may assume that $\Lie(\rho\circ\phi_a\circ\rho^{-1})=\charmorph(a)(\Id_d+N)$ for a nilpotent matrix $N$ with only $0$ and $1$ as entries. This yields 
\begin{eqnarray}\label{EqLastStatement1}
& \lim_{n\to\infty}\bigl\|\Lie(\rho\circ\phi_a^{n+1}\circ\rho^{-1})\bigl(\rho(\delta_1(\dm''_n))-(\Lie\rho)(\delta_0(\dm''_n))\bigr)\bigr\|\es\le\\[2mm]
& \mbox{ }\qquad\qquad\qquad\qquad\qquad\qquad\le\es\lim_{n\to\infty}|\charmorph(a)|^{(n+1)(1-s)}\|\check\Phi\|_{_{s/q}}^q\epsilon C_3^q\es=\es0\nonumber
\end{eqnarray}
By Theorem~\ref{ThmDivTowersAndLie} we have $\lim_{n\to\infty}\Lie(\rho\circ\phi_a^{n+1}\circ\rho^{-1})\bigl(\rho(\delta_1(\dm''_n))\bigr)=(\Lie\rho)(\xi)$. So we must compute 
\begin{eqnarray}\label{EqLastStatement2}
\Lie(\rho\circ\phi_a^{n+1}\circ\rho^{-1})\circ(\Lie\rho)(\delta_0(\dm''_n)) & = & (\Lie\rho)\bigl(\delta_0(t^{n+1}\dm''_n)\bigr)\\[2mm]
& = & (\Lie\rho)\bigl(\delta_0(\dm+\dm'_n-\sdtau_\dM(\sdsigma^*\dm'_n))\bigr)\nonumber\\[2mm]
& = & (\Lie\rho)\bigl(\delta_0(\dm+\dm'_n)\bigr)\,.\nonumber
\end{eqnarray}
Since the projection $\delta_0\colon\BC\langle\tfrac{t}{\theta^s}\rangle^{\oplus d}\onto\dM/J^d\dM\isoto\Lie E$ from Proposition~\ref{prop:commutingdiagrams} is continuous with respect to $\|\,.\,\|_{_s}$ and $\lim_{n\to\infty}\|\dm'-\dm'_n\|_{_s}=\lim_{n\to\infty}\|\dm'_{>n}\|_{_s}=0$, we find $\lim_{n\to\infty}\delta_0(\dm+\dm'_n)=\delta_0(\dm+\dm')$. In combination with \eqref{EqLastStatement1} and \eqref{EqLastStatement2} this proves that $\xi=\delta_0(\dm+\dm')$ and establishes the theorem.
\end{proof}

\begin{corollary}[{\cite{ABP_Rohrlich}}]\label{CorDivisionTower}
Let $C=\BP^1_{\BF_q}$, $A=\BF_q[t]$, $A_\BC=\BC[t]$ and $\theta=\charmorph(t)$. Then $\CO(\FC_\BC\setminus\Disc\bigr)=\BC\langle t\rangle$. Fix an isomorphism $\rho\colon E\isoto\BG_{a,\BC}^d$ of $\BF_q$-module schemes and write $\rho\circ\phi_t\circ\rho^{-1}=:\Delta_t=\sum_{j\ge0}\Delta_{t,j}\,\tau^j\in\BC\{\tau\}^{d\times d}=\End_{\BF_q,\BC}(\BG_{a,\BC}^d)$ with $\Delta_{t,j}\in\BC^{d\times d}$ and $\Delta_{t,j}=0$ for $j\gg0$. For $\nu\ge0$ consider the columns of the matrix $\sum_{j\ge0}\Delta_{t,\nu+j}\tau^j\in\BC\{\tau\}^{d\times d}$ as elements of $\BC\{\tau\}^d\cong\dM$ via $\rho$. Note that this matrix is zero for $\nu\gg1$. In the situation of Theorem~\ref{ThmDivisionTower} let $(x_{(n)})_n$ be a $t$-division tower above $x$ and let 
\[
f\;:=\;\sum_{n=0}^\infty \rho(x_{(n)})t^n \;\in\; \BC\dbl t\dbr^d
\]
be the associated \emph{Anderson generating function}. Then the bijection \eqref{EqThmDivisionTower} from Theorem~\ref{ThmDivisionTower} sends $(x_{(n)})_n$ to the element 
\begin{equation}\label{EqCorDivisionTower}
\dm'\;=\;-\sum_{\nu\ge1}\,\Bigl(\sum_{j\ge0}\Delta_{t,\nu+j}\tau^j\Bigr)\sigma^{\nu*}(f)\;\in\;\dM_t\;=\;\dM\otimes_{\BC[t]}\BC\dbl t\dbr\,.
\end{equation}
Moreover, the $t$-division tower $(x_{(n)})_n$ is convergent if and only if $f\in\BC\langle t\rangle^d$.
\end{corollary}

\begin{proof}
In step 5 of the proof of Theorem~\ref{ThmDivisionTower} we obtain $\rho(x_{(n-1)})=\sum_{j\ge0}\Delta_{t,j}\cdot\sigma^{j*}\rho(x_{(n)})$ and 
\begin{eqnarray*}
(t\dm''_n-\dm''_{n-1})^\dagger & = & \rho(x_{(n)})^T\cdot(\Delta_t)^\dagger-\rho(x_{(n-1)})^T\\[2mm]
& = &  \rho(x_{(n)})^T\cdot\sum_{j\ge0}\sdsigma^{j*}(\Delta_{t,j})^T\sdtau^j-\sigma^{j*}\rho(x_{(n)})^T\cdot\sum_{j\ge0}\Delta_{t,j}^T\\[2mm]
& = & \sum_{j\ge1} (\sdtau^j-1)\cdot\sigma^{j*}\rho(x_{(n)})^T\cdot\Delta_{t,j}^T\\
& = & \sum_{j\ge1}(\sdtau-1)\cdot\Bigl(\sum_{i=0}^{j-1}\sdtau^i\Bigr)\cdot\sigma^{j*}\rho(x_{(n)})^T\cdot\Delta_{t,j}^T \\[2mm]
& \stackrel{j=\nu+i}{=} & (\sdtau-1)\cdot\Bigl(\sum_{\nu\ge1}\sum_{i\ge0}\sigma^{\nu*}\rho(x_{(n)})^T\cdot\sdtau^i\cdot\Delta_{t,\nu+i}^T\Bigr) \\[2mm]
& = & (\sdtau-1)\cdot\Bigl(\sum_{\nu\ge1}\sum_{i\ge0}\Delta_{t,\nu+i}\cdot\tau^i\cdot\sigma^{\nu*}\rho(x_{(n)})\Bigr)^\dagger\,.
\end{eqnarray*}
Since also $(t\dm''_n-\dm''_{n-1})^\dagger=(y_n-\sdtau(\sdsigma^*y_n))^\dagger=(1-\sdtau)\cdot (y_n)^\dagger$ Proposition~\ref{prop:commutingdiagrams} implies that $y_n=-\sum_{\nu\ge1}\sum_{i\ge0}\Delta_{t,\nu+i}\cdot\tau^i\cdot\sigma^{\nu*}\rho(x_{(n)})$. Multiplying with $t^n$ and summing over all $n\ge0$ yields 
\[
\dm'\;=\;\sum_{n=0}^\infty y_n t^n\;=\;-\sum_{\nu\ge1}\Bigl(\sum_{i\ge0}\Delta_{t,\nu+i}\cdot\tau^i\Bigr)\cdot\sum_{n=0}^\infty \sigma^{\nu*}\rho(x_{(n)})t^n
\]
and establishes \eqref{EqCorDivisionTower}. Finally, if $(x_{(n)})_n$ is convergent then by definition $f\in\BC\langle t\rangle^d$. Conversely, the latter together with \eqref{EqCorDivisionTower} implies that $\dm'\in\dM\otimes_{\BC[t]}\BC\langle t\rangle$. By Theorem~\ref{ThmDivisionTower} this is equivalent to $(x_{(n)})_n$ being convergent.
\end{proof}

The following corollary is the analog in terms of dual $A$-motives of Sinha's diagram \cite[4.2.3]{Sinha97}.

\begin{corollary}\label{Cor=Boiler2.5.5}
Let $\ulE$ be an $A$-finite Anderson $A$-module and let $(\dM,\sdtau_\dM)=\uldM(\ulE)$ be its dual $A$-motive. For every $\dm'\in\dM\otimes_{A_\BC}\CO\bigl(\dotFC_\BC\setminus\bigcup_{i\in\BN_{>0}}\Var(\ssigma^{i\ast}J)\bigr)$ such that $\dm:=\sdtau_\dM(\sdsigma^*\dm')-\dm'\in\dM$ we have
\[
\exp_\ulE\bigl(\delta_0(\dm'+\dm)\bigr)\;=\;\delta_1(\dm)\,.
\]
\end{corollary}

\begin{proof}
This follows from the last statement of Theorem~\ref{ThmDivisionTower} and Theorem~\ref{ThmDivTowersAndLie}.
\end{proof}

\begin{corollary}\label{CorTo_Boiler2.5.5}
The morphism $\delta_0\colon\dM\to\Lie E$ from Proposition~\ref{prop:commutingdiagrams} restricts to an $A$-iso\-morphism
\[
\delta_0\colon \bigl(\dM\otimes_{A_\BC}\CO\bigl(\dotFC_\BC\setminus{\TS\bigcup_{i\in\BN_{>0}}}\Var(\ssigma^{i\ast}J)\bigr)\bigr)^\sdtau \; \isoto \; \Lambda(\ulE)\;=\;\ker(\exp_\ulE)\,.
\]
\end{corollary}

\begin{proof}
Let $\dm'\in\bigl(\dM\otimes_{A_\BC}\CO\bigl(\dotFC_\BC\setminus\bigcup_{i\in\BN_{>0}}\Var(\ssigma^{i\ast}J)\bigr)\bigr)^\sdtau$, that is $\dm:=\sdtau_\dM(\sdsigma^*\dm')-\dm'=0$. Then $x:=\delta_1(\dm)=0$. By Theorems~\ref{ThmDivisionTower} and \ref{ThmDivTowersAndLie} both sides of the claimed isomorphism are in bijection with the set of convergent $a$-division towers above $0$. By the last statement of Theorem~\ref{ThmDivisionTower} the combined bijection equals $\delta_0$, which is $A$-linear by Proposition~\ref{prop:commutingdiagrams}.
\end{proof}

\subsection{Purity and mixedness}\label{SectAModPurity}
Before we define purity of Anderson $A$-modules which are abelian or $A$-finite in terms of the corresponding (dual) $A$-motives, we show that the functors $\ulE\mapsto\ulM(\ulE)$ and $\ulE\mapsto\uldM(\ulE)$ are exact.

\begin{proposition}\label{PropQuotientAMod}
Let $\ulE'\subset\ulE$ be an Anderson $A$-submodule. Then the quotient $\ulE'':=\ulE/\ulE'$ exists as an Anderson $A$-module with $\dim\ulE''=\dim\ulE-\dim\ulE'$. 
\begin{enumerate}
\item \label{PropQuotientAMod_A}
$\ulE$ is abelian if and only if  both $\ulE'$ and $\ulE''$ are abelian. In this case $\rk\ulE''=\rk\ulE-\rk\ulE'$ and the induced sequence of $A$-motives
\[
\xymatrix {
0 \ar[r] & \ulM(\ulE'') \ar[r] & \ulM(\ulE) \ar[r] & \ulM(\ulE') \ar[r] & 0 
}
\]
is exact in the sense of Remark~\ref{RemQuillenAMot}(b) (that is, the sequence of the underlying $A_{\BC}$-modules is exact).
\item \label{PropQuotientAMod_B}
$\ulE$ is $A$-finite if and only if both $\ulE'$ and $\ulE''$ are $A$-finite. In this case $\rk\ulE''=\rk\ulE-\rk\ulE'$ and the induced sequence of dual $A$-motives
\[
\xymatrix {
0 \ar[r] & \uldM(\ulE') \ar[r] & \uldM(\ulE) \ar[r] & \uldM(\ulE'') \ar[r] & 0 
}
\]
is exact in the sense of Remark~\ref{RemQuillenDualAMot}(b) (that is, the sequence of the underlying $A_{\BC}$-modules is exact).
\end{enumerate}
\end{proposition}

\begin{proof}
Let $\ulE=(E,\phi)$ and $\ulE'=(E',\phi')$. Then the quotient $E'':=E/E'$ is a smooth irreducible group scheme with $\dim\ulE''=\dim\ulE-\dim\ulE'$ by \cite[Theorem~II.6.8]{Borel} and isomorphic to a power of $\BG_{a,\BC}$ by \cite[Proposition~VII.11]{Serre}. It inherits an action $\phi''\colon A\to\End_{\BC}(E'')$ of $A$ satisfying \eqref{EqDefAndersonAModuleA} in Definition \ref{DefAndersonAModule}\ref{DefAndersonAModule_A}, because $\Lie E''=\Lie E/\Lie E'$. Indeed, $E\to E''$ is smooth because $E'$ is smooth over $\BC$ and so $\Lie E\to\Lie E''$ is surjective by \cite[\S\,2.2, Proposition~8]{BLR} with $\Lie E'$ contained in its kernel. By reasons of dimension $\Lie E'$ equals the kernel of $\Lie E\onto\Lie E''$. We obtain an exact sequence of Anderson $A$-modules
\begin{equation}\label{EqExactE}
\xymatrix {
0 \ar[r] & \ulE' \ar[r]^{f'} & \ulE \ar[r]^{f''\;} & \ulE'' \ar[r] & 0 \,.
}
\end{equation}

\medskip\noindent
\ref{PropQuotientAMod_A} We apply the contravariant functor $\ulM(\,.\,)$ from Definition \ref{DefAbelianAMod}. This yields an exact sequence of $A_{\BC}$-modules
\begin{equation}\label{EqExactM}
\xymatrix {
0 \ar[r] & M(\ulE'') \ar[r] & M(\ulE) \ar[r] & M(\ulE')  \,.
}
\end{equation}
It is exact on the left because $E\onto E''$ is surjective. It is also exact in the middle by the universal mapping property of the quotient $E''$; see \cite[II.6.1]{Borel}. If $\ulE'$ and $\ulE''$ are abelian, that is $M(\ulE')$ and $M(\ulE'')$ are finite locally free over the Dedekind domain $A_\BC$, then also $M(\ulE)$ is finite locally free and $\ulE$ is abelian. Conversely, if $M(\ulE)$ is finite locally free, then also $M(\ulE'')$ is, and $\ulE''$ is abelian.

If $\ulE$ is abelian it remains to prove that $M(\ulE)\to M(\ulE')$ is surjective and $\ulE'$ is abelian. We consider the quotient $\ulTM:=\ulM(\ulE)/\ulM(\ulE'')$ which injects into $\ulM(\ulE')$. Since $\ulM(\ulE)$ is finitely generated both over $A_\BC$ and over $\BC\{\tau\}$, so is $\ulTM$. Since $\ulM(\ulE')$ has no $\BC\{\tau\}$-torsion the same holds for $\ulTM$, and so $\ulTM$ is locally free over $A_\BC$ by \cite[Lemma~1.4.5]{Anderson86}. Therefore, $\ulTM$ is an effective $A$-motive. If $\ulTM\cong\ulM(\ulE')$ this will imply that $\ulE'$ is abelian. By \cite[Theorem~1]{Anderson86} there exists an abelian Anderson $A$-module $\ulTE$ with $\ulTM=\ulM(\ulTE)$ and a morphism $\ulTE\to\ulE$ induced from $\ulM(\ulE)\onto\ulTM$. Any $\BC\{\tau\}$-basis $(\tilde m_1,\ldots,\tilde m_{\tilde d})$ of $\ulTM$ provides an isomorphism $\tilde m_1\times\ldots\times\tilde m_{\tilde d}\colon \wtE\isoto\BG_{a,\BC}^{\tilde d}$ of $\BF_q$-module schemes, and if $\BG_{a,\BC}=\Spec\BC[x]$ then the $\tilde x_j:=\tilde m_j^*(x)$ for $j=1,\ldots,\tilde\delta$ are free generators of the polynomial algebra $\Gamma(\wtE,\CO_{\wtE})=\BC[\tilde x_1,\ldots,\tilde x_{\tilde d}]$ over $\BC$. Since $\ulTM$ is a quotient of $\ulM(\ulE)$ the $\tilde m_j^*(x)$ lie in the image of $\Gamma(E,\CO_E)$. Therefore, $\ulTE\to\ulE$ is a closed immersion. Let $m_j'$ be the image of $\tilde m_j$ in $M(\ulE')$. Sending $\tilde x_j$ to $(m_j')^*(x)$ defines a $\BC$-homomorphism $\Gamma(\wtE,\CO_{\wtE})\to\Gamma(E',\CO_{E'})$. In this way the maps $\ulM(\ulE)\onto\ulTM\into\ulM(\ulE')$ induce morphisms
\[
\xymatrix {
\ulE'\ar[r] & \ulTE\ar[r] & \ulE\ar[r] & \ulE''\,.
}
\]
Since the composite map $M(\ulE'')\to M(\ulE)\to\ulTM$ is the zero map, the closed immersion $\wtE\into E$ factors through the kernel of $E\to E''$, which equals $E'$. So $\ulE'\to\ulTE$ must be an isomorphism. This shows that $\ulM(\ulE')=\ulM(\ulTE)=\ulTM$ onto which $\ulM(\ulE)$ surjects. Thus the sequence \eqref{EqExactM} is also exact on the right. From this also the formula for $\rk\ulE''$ follows.

\medskip\noindent
\ref{PropQuotientAMod_B} We apply the covariant functor $\uldM(\,.\,)$ from Definition \ref{DefDualMOfE} to the sequence \eqref{EqExactE}. This yields an exact sequence of $A_{\BC}$-modules
\begin{equation}\label{EqExactDM}
\xymatrix {
0 \ar[r] & \dM(\ulE') \ar[r] & \dM(\ulE) \ar[r] & \dM(\ulE'')  \,.
}
\end{equation}
It is exact on the left because $E'\into E$ is a closed immersion. It is also exact in the middle because $E'$ equals the fiber of $E\onto E''$ above $0$. If $\ulE'$ and $\ulE''$ are $A$-finite, that is $\dM(\ulE')$ and $\dM(\ulE'')$ are finite locally free over the Dedekind domain $A_\BC$, then also $\dM(\ulE)$ is finite locally free and $\ulE$ is $A$-finite. Conversely, if $\dM(\ulE)$ is finite locally free, then also $\dM(\ulE')$ is, and $\ulE'$ is $A$-finite.

If $\ulE$ is $A$-finite it remains to prove that $\dM(\ulE)\to\dM(\ulE'')$ is surjective and $\ulE''$ is $A$-finite. We consider the quotient $\uldN:=\uldM(\ulE)/\uldM(\ulE')$ which injects into $\uldM(\ulE'')$. Since $\uldM(\ulE)$ is finitely generated both over $A_\BC$ and over $\BC\{\sdtau\}$, so is $\uldN$. Since $\uldM(\ulE'')$ has no $\BC\{\sdtau\}$-torsion the same holds for $\uldN$, and so $\uldN$ is locally free over $A_\BC$ by the $\sdtau$-analog of \cite[Lemma~1.4.5]{Anderson86}. Therefore, $\uldN$ is an effective dual $A$-motive. If $\uldN\cong\uldM(\ulE'')$ this will imply that $\ulE''$ is $A$-finite. By Theorem~\ref{theorem:equivalence} there exists an $A$-finite Anderson $A$-module $\ulTE$ with $\uldN=\uldM(\ulTE)$ and morphisms $\tilde f\colon\ulE\to\ulTE$ and $g\colon\ulTE\to\ulE''$ induced from $\uldM(\ulE)\onto\uldN\into\uldM(\ulE'')$ and satisfying $f''=g\circ\tilde f$. 

Since the composite map $\uldM(\ulE')\to\uldM(\ulE)\to\uldN$ is the zero map, the morphism $\tilde f\circ f'\colon \ulE'\into \ulE\to \ulTE$ is the zero morphism by Theorem~\ref{theorem:equivalence}. By the universal mapping property \cite[II.6.1]{Borel} of the quotient $E/E'=E''$ the morphism $\tilde f\colon E\to \wtE$ factors as $\tilde f=h\circ f''$ for a morphism $h\colon E''\to\wtE$. Again by the universal mapping property, $f''=gh\circ f''$ implies that $gh=\id_{E''}$. Therefore, $\uldM(g)\circ\uldM(h)=\id_{\uldM(\ulE'')}$ and $\uldM(g)$ is surjective. As it is injective by construction we have $\uldN\cong\uldM(\ulE'')$ and the proposition is proved.
\end{proof}

\begin{corollary}
The category of abelian, respectively $A$-finite, Anderson $A$-modules is an exact category in the sense of Quillen~\cite[\S2]{Quillen} (see Remark~\ref{RemQuillenAMot}(b) for explanations) if one calls the sequences $\ulE'\to\ulE\to\ulE''$ of Anderson $A$-modules \emph{exact} where $\ulE'\subset\ulE$ is an Anderson $A$-submodule and $\ulE'':=\ulE/\ulE'$ is the quotient from Proposition~\ref{PropQuotientAMod}. The functors $\ulE\mapsto\ulM(\ulE)$ from Theorem~\ref{ThmAnderson}, respectively $\ulE\mapsto\uldM(\ulE)$ from Theorem~\ref{theorem:equivalence}\ref{theorem:equivalence_B}, are exact equivalences, that is, a sequence $\ulE'\to\ulE\to\ulE''$ is exact if and only if the induced sequence of $A$-motives, respectively dual $A$-motives, is exact. 
\end{corollary}

\begin{proof}
We start by proving the second assertion. By Proposition~\ref{PropQuotientAMod} the functors map exact sequences to exact sequences. 

Let $\ulE'\to\ulE\to\ulE''$ be a sequence of abelian Anderson $A$-modules whose associated sequence of $A$-motives $0\to\ulM(\ulE'')\to\ulM(\ulE)\to\ulM(\ulE')\to0$ is exact in the sense of Remark~\ref{RemQuillenAMot}(b). Consider an isomorphism $\rho'=(\rho'_1,\ldots,\rho'_{d'})\colon E'\isoto\BG_{a,\BC}^{d'}$ where $\rho'_i\colon E'\to\BG_{a,\BC}=\Spec\BC[x]$ is the projection onto the $i$-th coordinate. Then $\rho'_i\in\ulM(\ulE')$ and $\Gamma(E',\CO_{E'})$ is generated by $\rho_i^*(x)$. Since $\ulM(\ulE)$ surjects onto $\ulM(\ulE')$, we see that $\rho'_i$ lies in the image of $\Gamma(E,\CO_{E})\to\Gamma(E',\CO_{E'})$, and hence $\ulE'\to \ulE$ is a closed immersion. Let $\,\wt{\!\ulE}:=\ulE/\ulE'$ be the quotient from Proposition~\ref{PropQuotientAMod}. Then the $A$-motives $\ulM(\ulE'')$ and $\ulM(\,\wt{\!\ulE})$ both equal the kernel of $\ulM(\ulE)\onto\ulM(\ulE')$ by Proposition~\ref{PropQuotientAMod}. By Theorem~\ref{ThmAnderson} this shows that $\ulE''\cong\,\wt{\!\ulE}$, and hence the sequence $\ulE'\to\ulE\to\ulE''$ is exact as desired.

On the other hand, let $\ulE'\to\ulE\to\ulE''$ be a sequence of $A$-finite Anderson $A$-modules whose associated sequence of effective dual $A$-motives $0\to\uldM(\ulE')\to\uldM(\ulE)\to\uldM(\ulE'')\to0$ is exact in the sense of Remark~\ref{RemQuillenDualAMot}(b), that is on the underlying $A_\BC$-modules. Applying the snake lemma to 
\[
\xymatrix {
0 \ar[r] & \sdsigma^*\dM(\ulE') \ar[r] \ar@{^{ (}->}[d]_{\sdtau_{\dM(\ulE')}} & \sdsigma^*\dM(\ulE) \ar[r] \ar@{^{ (}->}[d]_{\sdtau_{\dM(\ulE)}} & \sdsigma^*\dM(\ulE'') \ar[r] \ar@{^{ (}->}[d]_{\sdtau_{\dM(\ulE'')}} & 0 \\
0 \ar[r] & \dM(\ulE') \ar[r] & \dM(\ulE) \ar[r] & \dM(\ulE'') \ar[r] & 0
}
\]
yields by \eqref{EqDualMAndE2} that the sequence on tangent spaces at the origin $0\to\Lie E'\to\Lie E\to\Lie E''\to0$ is exact. Analogously, \eqref{EqDualMAndE1} yields that the sequence $0\to E'(\BC)\to E(\BC)\to E''(\BC)\to0$ is exact. Both sequences together show that $E'\into E$ is a closed immersion. Let $\,\wt{\!\ulE}:=\ulE/\ulE'$ be the quotient from Proposition~\ref{PropQuotientAMod}. Then the dual $A$-motives $\uldM(\ulE'')$ and $\uldM(\,\wt{\!\ulE})$ both equal the cokernel of $\uldM(\ulE')\into\uldM(\ulE)$ by Proposition~\ref{PropQuotientAMod}. By Theorem~\ref{theorem:equivalence}\ref{theorem:equivalence_B} this shows that $\ulE''\cong\,\wt{\!\ulE}$, and hence the sequence $\ulE'\to\ulE\to\ulE''$ is exact as desired.

The first statement now follows from Remark~\ref{RemQuillenAMot}(b), respectively Remark~\ref{RemQuillenDualAMot}(b).
\end{proof}

\begin{definition}\label{DefMixedAModule}
\begin{enumerate}
\item \label{DefMixedAModuleA}
An abelian Anderson $A$-module $\ulE$ of dimension $d$ and rank $r$ is \emph{pure} if $\ulM(\ulE)$ is pure. In this case, we set $\weight\ulE=-\weight\ulM(\ulE)=-\frac{d}{r}$; see \cite[Lemma~1.10.1]{Anderson86}.
\item \label{DefMixedAModuleB}
An abelian Anderson $A$-module $\ulE$ is \emph{mixed} if it possesses an increasing \emph{weight filtration} by abelian Anderson $A$-submodules $W_{\mu\,}\ulE$ for $\mu\in\BQ$ such that $\Gr_\mu^W\ulE:=W_{\mu\,}\ulE/\bigl(\bigcup_{\mu'<\mu}W_{\mu'}\ulE\bigr)$ is a pure abelian Anderson $A$-module of weight $\mu$ for all $\mu\in\BQ$, and such that $\dim\ulE=\sum_{\mu\in\BQ}\dim\Gr_\mu^W\ulE$.
\item \label{DefMixedAModuleC}
An $A$-finite Anderson $A$-module $\ulE$ of dimension $d$ and rank $r$ is \emph{pure} if $\uldM(\ulE)$ is pure. In this case, we set $\weight\ulE=\weight\uldM(\ulE)=-\frac{d}{r}$. (This formula follows from the analog of \cite[Lemma~1.10.1]{Anderson86} using Proposition~\ref{prop:commutingdiagrams}.)
\item \label{DefMixedAModuleD}
An $A$-finite Anderson $A$-module $\ulE$ is \emph{mixed} if it possesses an increasing \emph{weight filtration} by $A$-finite Anderson $A$-submodules $W_{\mu\,}\ulE$ for $\mu\in\BQ$ such that $\Gr_\mu^W\ulE:=W_{\mu\,}\ulE/\bigl(\bigcup_{\mu'<\mu}W_{\mu'}\ulE\bigr)$ is a pure $A$-finite Anderson $A$-module of weight $\mu$ for all $\mu\in\BQ$, and such that $\dim\ulE=\sum_{\mu\in\BQ}\dim\Gr_\mu^W\ulE$.
\end{enumerate}
\end{definition}

\begin{remark}
(a) The set $\{W_{\mu\,}\ulE:\mu\in\BQ\}$ of closed irreducible and reduced subschemes of $E$ is totally ordered by inclusion, and hence finite by reasons of dimension. Therefore, also $\bigcup_{\mu'<\mu}W_{\mu'}\ulE$ and $\bigcap_{\tilde\mu>\mu}W_{\tilde\mu\,}\ulE$ belong to this set and are Anderson $A$-submodules of $\ulE$. By Proposition~\ref{PropQuotientAMod} they are abelian, respectively $A$-finite if $\ulE$ is, and the quotient $\Gr_\mu^W\ulE:=W_{\mu\,}\ulE/\bigl(\bigcup_{\mu'<\mu}W_{\mu'}\ulE\bigr)$ is again an abelian, respectively $A$-finite Anderson $A$-module of dimension $\dim\Gr_\mu^W\ulE\;=\;\dim W_{\mu\,}\ulE-\dim\bigl(\bigcup_{\mu'<\mu}W_{\mu'}\ulE\bigr)$.

\medskip\noindent
(b) The \emph{weights of $\ulE$} are the jumps of the weight filtration; that is, those real numbers $\mu$ for which 
\[
\TS\bigcup_{\mu'<\mu}W_{\mu'}\ulE \;\subsetneq\;\bigcap_{\tilde\mu>\mu}W_{\tilde\mu\,}\ulE\,.
\]
By (a) the condition $\sum_{\mu\in\BQ}\dim\Gr_\mu^W\ulE=\dim\ulE$ in Definition \ref{DefMixedAModule}\ref{DefMixedAModuleB} is equivalent to the conditions that all jumps lie in $\BQ$, that $W_{\mu\,}\ulE \;=\;\bigcap_{\tilde\mu>\mu}W_{\tilde\mu\,}\ulE$ for all $\mu\in\BQ$, that $W_{\mu\,}\ulE=(0)$ for $\mu\ll0$, and that $W_{\mu\,}\ulE=\ulE$ for $\mu\gg0$; compare Remarks~\ref{RemQFiltr} and \ref{RemAMotWts}.

\medskip\noindent
(c) By Definition~\ref{DefMixedAModule}\ref{DefMixedAModuleA} and \ref{DefMixedAModuleC} all weights of a mixed abelian, respectively $A$-finite Anderson $A$-module are negative. In particular, a Drinfeld $A$-module of rank $r$ is pure of weight $-\tfrac{1}{r}$ !
 
\medskip\noindent
(d) Every pure abelian, respectively $A$-finite Anderson $A$-module of weight $\mu$ is also mixed with $W_{\mu'}\ulE=(0)$ for $\mu'<\mu$, and $W_{\mu'}\ulE=\ulE$ for $\mu'\ge\mu$, and $\Gr_\mu^W\ulE=\ulE$.
\end{remark}

\begin{theorem}\label{ThmMixedEandM}
\begin{enumerate}
\item 
An abelian (respectively $A$-finite) Anderson $A$-module $\ulE$ is mixed if and only if its associated $A$-motive $\ulM(\ulE)$ (respectively dual $A$-motive $\uldM(\ulE)$) is mixed. In this case the weights of $\ulE$ are the negatives of the weights of $\ulM(\ulE)$ (respectively equal to the weights of $\uldM(\ulE)$). 
\item
If an Anderson $A$-module $\ulE$ is both abelian and $A$-finite, then it is mixed (respectively pure) as an abelian Anderson $A$-module if and only if $\ulE$ is so as an $A$-finite Anderson $A$-module. In this case its weight filtrations and weights as an abelian, respectively $A$-finite Anderson $A$-module coincide.
\end{enumerate}
\end{theorem}

\begin{proof}
First let $\ulE$ be abelian and let $\ulM=\ulM(\ulE)$. Assume that $\ulE$ is mixed. We set 
\[
\TS W_{-\mu\,}\ulM\;:=\;\ker\bigl(\ulM\onto\ulM(\bigcup\limits_{\mu'<\mu}W_{\mu'}\ulE)\bigr)\,.
\]
Then $W_\bullet\ulM$ is an increasing filtration of $\ulM$ by saturated $A$-sub-motives. Equivalently, if $\mu_1<\ldots<\mu_n$ are the jumps of the weight filtration $W_\bullet\ulE$, set in addition $\mu_0:=-\infty, \mu_{n+1}:=+\infty$, and $W_{\mu_0\,}\ulE=(0)$. Then $W_{\mu_i\,}\ulE=W_{\mu'}\ulE\subsetneq W_{\mu_{i+1}\,}\ulE$ for all $\mu_i\le\mu'<\mu_{i+1}$ and hence, for any $\mu$ with $\mu_i<\mu\le\mu_{i+1}$ we have $\bigcup_{\mu'<\mu}W_{\mu'}\ulE=W_{\mu_i\,}\ulE$ and $W_{-\mu\,}\ulM=\ker\bigl(\ulM\onto\ulM(W_{\mu_i\,}\ulE)\bigr)$. In particular, if $\mu_i\le\mu<\tilde\mu\le\mu_{i+1}$, then 
 \begin{equation}\label{EqRemMixedEandM}
 W_{-\tilde\mu\,}\ulM\;=\;\ker\bigl(\ulM\onto\ulM(W_{\mu_i\,}\ulE)\bigr)\;=\;\ker\bigl(\ulM\onto\ulM(W_{\mu\,}\ulE)\bigr)\,.
 \end{equation}
 This yields the following diagram with exact rows
 \[
 \xymatrix {
 0 \ar[r] & **{!U(0.36) =<6pc,2pc>} \objectbox{\bigcup\limits_{-\tilde\mu<-\mu}W_{-\tilde\mu\,}\ulM} \ar[r] \ar@{^{ (}->}[d] & \ulM \ar[r] \ar@{=}[d] & \ulM(W_{\mu\,}\ulE) \ar[r] \ar@{->>}[d] & 0\,\, \\
 0 \ar[r] & **{!U(0.06) =<3.5pc,2pc>} \objectbox{W_{-\mu\,}\ulM} \ar[r] & \ulM \ar[r] & **{!U(0.31) =<6.3pc,2pc>} \objectbox{\ulM\bigl(\bigcup\limits_{\mu'<\mu}W_{\mu'}\ulE\bigr)} \ar[r] & 0 \,.
 }
 \]
So Proposition~\ref{PropQuotientAMod} and the snake lemma yield $\Gr_{-\mu}^W\ulM\cong\ker\bigl(\ulM(W_{\mu\,}\ulE)\onto\ulM(\bigcup_{\mu'<\mu}W_{\mu'}\ulE)\bigr)\cong\ulM(\Gr_\mu^W\ulE)$. The latter is pure of weight $-\mu$ and therefore $\ulM(\ulE)$ is mixed with weight filtration $W_\bullet\ulM$ which jumps at $-\mu_n<\ldots<-\mu_1$. 

Conversely, let $\ulM$ be mixed with weight filtration $W_\mu\ulM$. If $\nu_1<\ldots<\nu_n$ are the weights of $\ulM$ set in addition $\nu_0:=-\infty$ and $\nu_{n+1}:=+\infty$ and set $W_{\nu_0}\ulM=(0)$ and $W_{\nu_{n+1}}:=\ulM$. For $\mu\in\BQ$ with $-\nu_{i+1}\le\mu<-\nu_i$ for some $i$ let $W_\mu\ulE$ be the abelian Anderson $A$-submodule of $\ulE$ with $\ulM(W_\mu\ulE)=\ulM/W_{\nu_i}\ulM$. (Note that $\ulM(W_\mu\ulE)$ is effective and finite free over $\BC\{\tau\}$ because $\ulM$ is.) Then the considerations above show that $\ulE$ is mixed with respect to this weight filtration.

Now let $\ulE$ be $A$-finite and let $\uldM=\uldM(\ulE)$. If $\ulE$ is mixed then setting 
\begin{equation}\label{EqThmMixedEandMWeights}
W_\mu\uldM\;:=\;\uldM(W_\mu\ulE)
\end{equation}
and applying Proposition~\ref{PropQuotientAMod} shows that $\Gr_\mu^W\uldM(\ulE)=\uldM(\Gr_\mu^W\ulE)$. Therefore, $\uldM$ is mixed with the same weights than $\ulE$. Conversely, if $\uldM=\uldM(\ulE)$ is mixed with weight filtration $W_\mu\uldM$ let $W_\mu\ulE$ be the $A$-finite Anderson $A$-submodule of $\ulE$ with $\uldM(W_\mu\ulE)=W_\mu\uldM$. (Note that $W_\mu\uldM$ is effective and finite free over the noetherian ring $\BC\{\sdtau\}$ because $\uldM$ is.) Then $\uldM(\Gr_\mu^W\ulE)=\Gr_\mu^W\uldM(\ulE)$ by Proposition~\ref{PropQuotientAMod}, and $\ulE$ is mixed with the same weights than $\uldM$.

If $\ulE$ is both abelian and $A$-finite then $\uldM\bigl(\ulM(\ulE)\bigr)=\uldM(\ulE)$ by Theorem~\ref{ThmMandDMofE}. The last statement of the theorem therefore follows from Proposition~\ref{PropDualMixed}.
\end{proof}

\subsection{Uniformizability}\label{SectAModUniformizability}

\begin{definition}\label{DefAModuleUniformizable}
An Anderson $A$-module is called \emph{uniformizable} if its exponential $\exp_\ulE$ is surjective.
\end{definition}

\begin{remark}\label{RemAModuleUniformizable}
(a) If $\ulE$ is uniformizable and $a\in A$, the snake lemma applied to
\[
\xymatrix @C+3pc {
0 \ar[r] & \Lambda(\ulE) \ar[r] \ar[d]_{\TS\Lie\phi_a} & \Lie E \ar[r]^{\TS\exp_\ulE} \ar[d]_{\TS\Lie\phi_a} & E(\BC) \ar[r] \ar[d]^{\TS\phi_a} & 0 \\
0 \ar[r] & \Lambda(\ulE) \ar[r] & \Lie E \ar[r]^{\TS\exp_\ulE} & E(\BC) \ar[r] & 0
}
\]
together with the fact that $\Lie\phi_a$ is an automorphism of $\Lie E$ yields $\ulE[a](\BC)\cong \Lambda(\ulE)/a\Lambda(\ulE)$.

\medskip\noindent
(b) By \cite[Theorem~4]{Anderson86} an \emph{abelian} Anderson $A$-module $\ulE$ is uniformizable if and only if $\Lambda(\ulE):=\ker(\exp_\ulE)$ is a locally free $A$-module of rank equal to $\rk\ulE$, if and only if its associated $A$-motive $\ulM(\ulE)$ is uniformizable. The analog for dual $A$-motives is the following theorem of Anderson.
\end{remark}

\begin{theorem}[{\cite{ABP_Rohrlich}}]\label{ThmUnifAModandDM}
Let $\ulE$ be an $A$-finite Anderson $A$-module and let $\uldM=\uldM(\ulE)$ be its associated dual $A$-motive. Then the following are equivalent
\begin{enumerate}
\item \label{ThmUnifAModandDM_A}
$\ulE$ is uniformizable,
\item \label{ThmUnifAModandDM_B}
$\ker(\exp_\ulE)$ is a locally free $A$-module of rank equal to $\rk\ulE$,
\item \label{ThmUnifAModandDM_C}
$\uldM(\ulE)$ is uniformizable.
\end{enumerate}
If these conditions hold then the map $\delta_0$ from Corollary~\ref{CorTo_Boiler2.5.5} provides an isomorphism of $A$-modules $\delta_0\colon\Lambda(\uldM)\isoto\Lambda(\ulE)$.
\end{theorem}

\begin{proof}
If $\ulE$ is $A$-finite, that is $\uldM=(\dM,\sdtau_\dM)$ is finite locally free over $A_\BC$ then for every $a\in A\setminus\BF_q$ Proposition~\ref{PropSwitcheroo} implies $\ulE[a](\BC)\cong(\dM/a\dM)^\sdtau\cong\bigl(A/(a)\bigr)^{\oplus\rk\ulE}$. 

If we assume \ref{ThmUnifAModandDM_A}, this observation together with Remark~\ref{RemAModuleUniformizable}(a) implies that the discrete $A$-submodule $\Lambda(\ulE):=\ker(\exp_\ulE)\subset\Lie E$ is locally free of rank $\rk\ulE$, whence \ref{ThmUnifAModandDM_B}; compare the proof of \cite[Theorem~4.6.9]{Goss}.

By Corollary~\ref{CorTo_Boiler2.5.5}, condition \ref{ThmUnifAModandDM_B} implies that $\Lambda(\uldM)$ is a locally free $A$-module of rank $\rk\ulE:=\rk\uldM$. So \ref{ThmUnifAModandDM_C} follows from Lemma~\ref{LemmaDualUniformizable}\ref{LemmaDualUniformizableB}.

Finally, assume \ref{ThmUnifAModandDM_C} and let $x\in E(\BC)$. By \eqref{EqDualMAndE1} in Proposition~\ref{prop:commutingdiagrams} there is an $\dm\in\dM$ with $\delta_1(\dm)=x$. By Lemma~\ref{LemmaDualUniformizable}\ref{LemmaDualUniformizableC} there is an $\dm'\in\dM\otimes_{A_\BC}\CO(\FC_\BC\setminus\Disc)$ with $\sdtau_\dM(\sdsigma^*\dm')-\dm'=\dm$. By Theorem~\ref{ThmDivisionTower} the element $\xi:=\delta_0(\dm'+\dm)$ satisfies $\exp_\ulE(\xi)=x$. This proves that $\exp_\ulE$ is surjective, that is \ref{ThmUnifAModandDM_A}.

The last assertion follows from Corollary~\ref{CorTo_Boiler2.5.5} and Proposition~\ref{Prop4.23}.
\end{proof}

The following corollary is the analog of \cite[Corollary~3.3.6]{Anderson86} for uniformizable $A$-finite Anderson $A$-modules.

\begin{corollary}\label{CorUnifGeneratesLieE}
If $\ulE$ is a uniformizable $A$-finite Anderson $A$-module, then the set $\Lambda(\ulE)$ generates the $\BC$-vector space $\Lie E$.
\end{corollary}

\begin{proof}
This question does not depend on the ring $A$, so we fix an element $t\in A\setminus\BF_q$ and the finite flat inclusion $\wt A:=\BF_q[t]\subset A$. Let $\xi\in\Lie E$ and consider $x:=\exp_\ulE(\xi)\in E(\BC)$ and the convergent $t$-division tower $x_{(n)}:=\exp_\ulE\bigl(\Lie\phi_t^{-n-1}(\xi)\bigr)$ above $x$ from Theorem~\ref{ThmDivTowersAndLie}. We set $\uldM(\ulE)=(\dM,\sdtau_\dM)$ and choose an element $\dm\in\dM$ with $\delta_1(\dm)=x$; see Proposition~\ref{prop:commutingdiagrams}. By Theorem~\ref{ThmDivisionTower} there exists an $\dm'\in\dM\otimes_{A_\BC}\CO\bigl(\dotFC_\BC\setminus\bigcup_{i\in\BN_{>0}}\Var(\ssigma^{i\ast}J)\bigr)$ such that $\xi=\delta_0(\dm'+\dm)$.

Now choose a basis $\check\CB$ of $\dM$ over $\wt A_\BC=\BC[t]$ and write $\sdtau_\dM$ with respect to $\check\CB$ as a matrix $\check\Phi\in\GL_r\bigl(\BC[t][\tfrac{1}{t-\charmorph(t)}]\bigr)$. By Theorem~\ref{ThmUnifAModandDM} and Lemma~\ref{LemmaDualUniformizableBF_q[t]} there is a rigid analytic trivialization $\check\Psi\in\GL_{r}(\BC\langle t\rangle)$ satisfying $\sdsigma^\ast\check\Psi=\check\Psi\cdot\check\Phi$. We set $f:=\check\Psi\cdot(\dm'+\dm)\in\BC\langle t\rangle^r$, where we denote the column vectors representing $\dm'$ and $\dm$ with respect to $\check\CB$ again by $\dm'$, respectively $\dm$. We now consider $f\mod(t-\charmorph(t))^{\dim\ulE}$ as a $\BC$-linear combination of elements $f_1,\ldots,f_n\in\BF_q[t]^r$. Then $\xi=\delta_0(\check\Psi^{-1}\cdot f)$ lies in the $\BC$-span of the $\delta_0(\check\Psi^{-1}\cdot f_i)$ by Proposition~\ref{prop:commutingdiagrams}, because $\delta_0$ is $\BC$-linear. Since the $\check\Psi^{-1}\cdot f_i$ lie in $\Lambda(\uldM)$ by Lemma~\ref{LemmaDualUniformizableBF_q[t]} the corollary follows from Theorem~\ref{ThmUnifAModandDM}.
\end{proof}

\begin{remark}\label{RemScattering}
We review Anderson's theory of scattering matrices \cite[Chapter 3]{Anderson86}. Let $\ulE$ be an abelian Anderson $A$-module over $\BC$ and let $\ulM=\ulM(\ulE)$ be its associated effective $A$-motive. Assume that $\ulE$, and hence also $\ulM$ are uniformizable. In particular $\Lambda(\ulE)=\ker(\exp_\ulE)$ and $\Lambda(\ulM)=\bigl(M\otimes_{A_{\BC}}\CO(\dotFC_\BC)\bigr)^\tau$ are locally free $A$-modules of rank equal to $\rk\ulE$. By \cite[Corollary~2.12.1]{Anderson86} there is an isomorphism
\begin{equation}\label{EqAnderson2.12.1}
\beta_A\colon\Lambda(\ulE)\isoto\Hom_A(\Lambda(\ulM),\Omega^1_{A/\BF_q})\,,\es\lambda\longmapsto m\dual_{A,\lambda}\,,\qquad\text{where}\quad m\dual_{A,\lambda}\colon m\mapsto \omega_{A,\lambda,m}
\end{equation}
is determined by the residues $\Res_\infty(a\cdot\omega_{A,\lambda,m})=-m\bigl(\exp_\ulE(\Lie\phi_a(\lambda))\bigr)\in\BF_q$ for all $a\in Q$. Note that indeed $m\bigl(\exp_\ulE(\Lie\phi_a(\lambda))\bigr)\in\BF_q$ is well defined. Namely, we choose an $a'\in A$ with $aa'\in A$ and we approximate $m\in\Lambda(\ulM)$ by an element $m'\in M(\ulE)$ such that $m- m'\in a'\cdot(M\otimes_{A_{\BC}}\CO(\dotFC_\BC))$. Then we define $m\bigl(\exp_\ulE(\Lie\phi_a(\lambda))\bigr):=m'\bigl(\exp_\ulE(\Lie\phi_a(\lambda))\bigr)\in\BC$ which is independent of $a'$ and $m'$. Since $m\in\Lambda(\ulM)$ we conclude that $m\bigl(\exp_\ulE(\Lie\phi_a(\lambda))\bigr)=\bigl(\tau_M(\sigma^*m)\bigr)\bigl(\exp_\ulE(\Lie\phi_a(\lambda))\bigr):=m\bigl(\exp_\ulE(\Lie\phi_a(\lambda))\bigr)^q\in\BF_q$ as desired.

We next reduce to the situation of abelian $t$-modules in which Anderson defines scattering matrices. By Lemma~\ref{LemmaTSep}\ref{LemmaTSep_B} there is a $t\in A$ such that $Q$ is a finite separable extension of $\wt Q:=\BF_q(t)$. Then $\Omega^1_{Q/\BF_q}=\Omega^1_{A/\BF_q}\otimes_AQ=Q\,dt$ by \cite[Theorems~25.1 and 25.3]{MatsumuraRingTh}. We set $\wt A:=\BF_q[t]\subset A$. This inclusion corresponds to a morphism $C\to\BP^1_{\BF_q}$ under which the preimage of $\Spec\BF_q[t]$ is $\Spec A$ and the preimage of $\wt\infty:=\Var(\frac{1}{t})$ is $\infty$. We view all $A$-modules as $\wt A$-modules and all $A_{\BC}$-modules as modules over $\wt A_{\BC}=\BC[t]$. Then the trace map 
\begin{equation}\label{EqTraceQ}
\Tr_{Q/\wt Q}\colon\,\Omega^1_{Q/\BF_q}=\,Q\,dt\;\longto\;\Omega^1_{\BF_q(t)/\BF_q}=\,\BF_q(t)\,dt\,,\quad a\,dt\mapsto\Tr_{Q/\wt Q}(a)\,dt
\end{equation}
satisfies $\Res_\infty(\omega)=\Res_{\wt\infty}(\Tr_{Q/\wt Q}\omega)$ for all $\omega\in\Omega^1_{Q/\BF_q}$ by \cite[Formula (9.16) on p.~299]{VillaSalvador}. In particular, consider the isomorphism of $\BF_q[t]$-modules which is analogous to \eqref{EqAnderson2.12.1}
\[
\beta_{\wt A}\colon \Lambda(\ulE)\isoto\Hom_{\BF_q[t]}(\Lambda(\ulM),\BF_q[t]dt),\es\lambda\longmapsto m\dual_{\wt A,\lambda}\,,
\]
where $m\dual_{\wt A,\lambda}\colon m\mapsto \omega_{\wt A,\lambda,m}$ is determined by $\Res_{\wt\infty}(a\cdot\omega_{\wt A,\lambda,m})=-m\bigl(\exp_\ulE(\Lie\phi_a(\lambda))\bigr)\in\BF_q$ for all $a\in\wt Q$. It satisfies $\omega_{\wt A,\lambda,m}=\Tr_{Q/\wt Q}(\omega_{A,\lambda,m})$ because in $\wt Q_{\wt\infty}\:dt=\BF_q\dpl\frac{1}{t}\dpr dt$ both can be written in the form $\omega_{\wt A,\lambda,m}=\sum_k\tilde b_k t^k dt$ and $\Tr_{Q/\wt Q}(\omega_{A,\lambda,m})=\sum_k b_k t^k dt$ with
\begin{eqnarray}\label{EqEandM3}
b_k&=&-\Res_{\wt\infty}(t^{-k-1}\Tr_{Q/\wt Q}\omega_{A,\lambda,m})\nonumber\\[2mm]
&=&-\Res_\infty(t^{-k-1}\omega_{A,\lambda,m})\nonumber\\[2mm]
&=&m\bigl(\exp_\ulE(\Lie\phi_t^{-k-1}(\lambda))\bigr)\\[2mm]
&=&-\Res_{\wt\infty}(t^{-k-1}\omega_{\wt A,\lambda,m})\nonumber\\[2mm]
&=& \tilde b_k\,.\nonumber
\end{eqnarray}
Note that in particular, $b_k=0$ for $k<0$.

Now Anderson's theory of scattering matrices \cite[\S3]{Anderson86} proceeds as follows. Fix an $\BF_q[t]$-basis $(\lambda_1,\ldots,\lambda_r)$ of $\Lambda(\ulE)$, where $r=\rk_{\BF_q[t]}\Lambda(\ulE)=\rk_{\BC[t]}M$, and a $\BC[t]$-basis $\CB=(m_1,\ldots,m_r)$ of $M$, and define the scattering matrix $\Psi$, where $i$ is the row index and $j$ is the column index
\begin{equation}\label{EqScatteringMatrix}
\Psi\;:=\;\Bigl(\sum_{k=0}^\infty m_i\bigl(\exp_\ulE(\Lie\phi_t^{-k-1}(\lambda_j))\bigr)t^k\Bigr)_{i,j=1,\ldots,r}\es.
\end{equation}
Its entries lie in $\BC\langle\tfrac{t}{\theta^s}\rangle$ for all $s<1$, see \eqref{EqConvPowerSeries}, because for any isomorphism $\rho\colon E\isoto\BG_{a,\BC}^d$ of $\BF_q$-module schemes, Lemma~\ref{LemmaLog} implies for $k\to\infty$ the estimate
\[
\|\rho\bigl(\exp_\ulE(\Lie\phi_t^{-k-1}(\lambda_j))\bigr)\|\;=\;\|\Lie\rho\bigl(\Lie\phi_t^{-k-1}(\lambda_j)\bigr)\|\;=\;O(|\theta|^{-k-1}\cdot\|\Lie\rho(\lambda_j)\|)\,,
\]
and then Lemma~\ref{Lemma5.16} applied to $f=m_i\circ\rho^{-1}$ implies 
\[
\|m_i\bigl(\exp_\ulE(\Lie\phi_t^{-k-1}(\lambda_j))\bigr)\|\;=\;O(|\theta|^{-k-1})\qquad\text{for }k\to\infty\,.
\]
Note that our scattering matrix $\Psi$ is the negative of Anderson's \cite[\S\,3.2]{Anderson86}. This is motivated by Example~\ref{exmp:scatteringmatrix} and Theorems~\ref{ThmHPofEandM} and \ref{ThmPeriodIsomForE} below.

If $\tau_M$ is represented with respect to the basis $\CB$ by the matrix $\Phi\in M_r\bigl(\BC[t]\bigr)\cap\GL_r\bigl(\BC[t][\frac{1}{t-\theta}]\bigr)$, then Anderson \cite[Proof of Lemma~3.2.1]{Anderson86} shows that $\sigma^\ast\Psi^T=\Psi^T\Phi$ and that the columns of $(\Psi^{-1})^T$ form an $\BF_q[t]$-basis $\CC=(n_1,\ldots,n_r)$ of $\Lambda(\ulM)$. In particular $\Psi\in\GL_r(\BC\langle t\rangle)$. In terms of Lemma~\ref{LemmaUniformizableBF_q[t]} this means that $\Psi$ is a rigid analytic trivialization of $\Phi$. More precisely, the $\ell$-th column of $(\Psi^{-1})^T$ is the coordinate vector of $n_\ell$ with respect to the basis $\CB$. Therefore, with respect to the bases $\CC$ and $\CB$ the morphism $h_\ulM\colon\Lambda(\ulM)\otimes_A\CO(\dotFC_\BC)\to M\otimes_{A_{\BC}}\CO(\dotFC_\BC)$ is represented by $(\Psi^{-1})^T\in M_r\bigl(\CO(\dotFC_\BC)\bigr)$. Since $\coker\Phi\cong\coker\tau_M$ is a $\BC[t]/(t-\theta)^d$-module with dimension $d=\dim\ulE$ as $\BC$-vector space, it follows by the elementary divisor theorem that $\det\Phi\in(t-\theta)^d\cdot\BC\mal$\!. Together with $\sigma^*\Psi^T\in M_r\bigl(\BC\langle\tfrac{t}{\theta}\rangle\bigr)$ this implies that $\Psi^T\in M_r\bigl((t-\theta)^{-d}\BC\langle\frac{t}{\theta}\rangle\bigr)$. 

In fact, we show that the $\BF_q[t]$-basis $(\lambda_1,\ldots,\lambda_r)$ of $\Lambda(\ulE)$ is even mapped under $\beta_{\wt A}$ to the basis of $\Hom_{\BF_q[t]}(\Lambda(\ulM),\BF_q[t]dt)$ which is dual to $\CC$. Namely, if $e_\ell=(\delta_{1\ell},\ldots,\delta_{r\ell})^T$ is the $\ell$-th standard basis vector and $(\Psi^{-1})^T e_\ell=(g_1,\ldots,g_r)^T$ is the $\ell$-th column of $(\Psi^{-1})^T$, then $n_\ell=\sum_i g_i m_i$ and by \eqref{EqEandM3} we obtain 
\begin{equation}\label{EqEandM5}
\omega_{\wt A,\lambda_j,n_\ell}\;=\;\sum_{k=0}^\infty n_\ell\bigl(\exp_\ulE(\Lie\phi_t^{-k-1}(\lambda_j))\bigr)t^k\,dt\;=\;(g_1,\ldots,g_r)\cdot\Psi\cdot e_j \,dt\;=\;e_\ell^T e_j\,dt\;=\;\delta_{\ell j}\,dt\,.
\end{equation}
\end{remark}

\begin{example}[{Cf.~\cite[\S 4.2]{pellarin-08}}]
\label{exmp:scatteringmatrix}
We continue with our Example~\ref{exampleDModDMotdDmot}(a) of Drinfeld $\BF_q[t]$-modules. Let $\lambda\in\Lambda(\ulE)$ be a period. The corresponding convergent $t$-division tower from Theorem~\ref{ThmDivTowersAndLie} is $\bigl(\exp_\ulE(\theta^{-n-1}\lambda)\bigr)_{n=0}^\infty$ and the corresponding \emph{Anderson generating function} from Corollary~\ref{CorDivisionTower} is
\begin{equation}\label{EqExampleScatteringmatrix1}
f_\lambda(t)\;:=\;\sum_{n=0}^\infty\exp_\ulE(\theta^{-n-1}\lambda)t^n \;\in\; \BC\langle t\rangle\,.
\end{equation}
Multiplying the equation 
\begin{eqnarray*}
\exp_\ulE(\theta^{-n}\lambda) &=& \phi_t\bigl(\exp_\ulE(\theta^{-n-1}\lambda)\bigr)\\
&=&\theta\cdot\exp_\ulE(\theta^{-n-1}\lambda)+\alpha_1\cdot\exp_\ulE(\theta^{-n-1}\lambda)^{q}+\ldots+\alpha_r\cdot\exp_\ulE(\theta^{-n-1}\lambda)^{q^r},
\end{eqnarray*}
by $t^n$ and summing up we get 
\begin{eqnarray}
\nonumber
\theta \,f_\lambda(t)+\alpha_1 \,\sigma^*f_\lambda(t)+\ldots+\alpha_r \,\sigma^{r*}f_\lambda(t)
&=&\sum_{n=0}^\infty \exp_\ulE(\theta^{-n}\lambda) t^n\\
&=& \exp_\ulE(\lambda)+tf_\lambda(t) \es=\es tf_\lambda(t).
\label{eq:useful}
\end{eqnarray}
We claim that the solution $\dm'_\lambda\in\uldM_t$ corresponding to $\lambda$ by Corollary~\ref{CorDivisionTower} of the equation $\sdtau_\dM(\sdsigma^*\dm')=\dm'$ is given with respect to the basis $(\dm_k)_k$ by the coordinate vector
\[
X\cdot\left(\begin{array}{c} \sigma^*(f_\lambda)\\ \vdots \\ \sigma^{r*}(f_\lambda) \end{array}\right),
\]
where $X$ is the matrix from \eqref{EqMatrixXi}, that is $\dm'_\lambda=-\sum_{\nu=1}^r\sum_{k=0}^{r-1}\sdsigma^{k*}(\alpha_{\nu+k})\cdot\sigma^{\nu*}(f_\lambda)\cdot\dm_k$. Indeed, the $1\times 1$-matrix $\sum_{j\ge0}\Delta_{t,\nu+j}\tau^j\in\BC\{\tau\}$ from Corollary~\ref{CorDivisionTower} is identified with the element $\sum_{k\ge0}\sdsigma^{k*}(\alpha_{\nu+k})\cdot\dm_k\in\dM$ for all $\nu=1,\ldots,r$. Note that the inverse assignment $\dm'_\lambda\mapsto\lambda$ is given by the map $\delta_0$ from Corollary~\ref{CorTo_Boiler2.5.5}; see also Theorem~\ref{ThmUnifAModandDM}.

Let $\lambda_1,\ldots,\lambda_r$ be an $\BF_q[t]$-basis of $\Lambda(\ulE)$. For $i=1,\ldots,r$ we write $f_i:=f_{\lambda_i}$ and $\dm'_i:=\dm'_{\lambda_i}$ for the corresponding solutions. From the linear independence of the sets $\{\lambda_1,\ldots,\lambda_r\}$ and $\{\dm'_1,\ldots,\dm'_r\}:=\{\delta_0^{-1}(\lambda_1),\ldots,\delta_0^{-1}(\lambda_r)\}$ it follows that $\{f_1,\ldots,f_r\}$ is linearly independent over $\BF_q[t]$.  From the description of $f_\lambda$ in \eqref{EqExampleScatteringmatrix1} we see that the matrix
\[
\Psi\;:=\; \left(\begin{array}{cccc}
       f_1^{}    &f_2     &\cdots& f_r\\
       \sigma^*f_1&\sigma^*f_2&\cdots&\sigma^*f_r\\
       \vdots &\vdots   &      &\vdots  \\
       \sigma^{(r-1)*}f_1 &\sigma^{(r-1)*}f_2&\cdots& \sigma^{(r-1)*}f_r
     \end{array}\right)
\]
is the scattering matrix from \eqref{EqScatteringMatrix}, and equation \eqref{eq:useful} shows that indeed $\Psi^T\cdot\Phi=\sigma^*\Psi^T$. In particular the columns of $(\Psi^T)^{-1}$ are the coordinate vectors with respect to the basis $(m_i)$ of an $\BF_q[t]$-basis $\CC$ of $\Lambda(\ulM)$; see Lemma~\ref{LemmaUniformizableBF_q[t]}.

We set $\check\Psi:=\sigma^*\Psi^{-1}\cdot X^{-1}\in\GL_r(\BC\langle t\rangle)$ so that the columns of $\check\Psi^{-1}=X\cdot\sigma^*\Psi$ are the coordinate vectors of the $\BF_q[t]$-basis $(\dm'_1,\ldots\dm'_r)$ of $\Lambda(\uldM)$ with respect to the basis $(\dm_k)_k$. Moreover, equation~\eqref{EqRelationMatrixXi}, that is $X\cdot\Phi^T=\check\Phi\cdot\sdsigma^*X$ shows that
\[
\check\Phi\cdot\sdsigma^*\check\Psi^{-1} \;=\; \check\Phi\cdot\sdsigma^*X\cdot\Psi \;=\; X\cdot\Phi^T\cdot\Psi \;=\; X\cdot\sigma^*\Psi \;=\; \check\Psi^{-1}
\]
and $\check\Psi$ is hence a rigid analytic trivialization of $\check\Phi$ in the sense of Lemma~\ref{LemmaDualUniformizableBF_q[t]}. By Corollaries~\ref{CorLambdaConvRadius} and \ref{CorLambdaDualConvRadius} the entries of the matrices $\Psi^{-1}$ and $\check\Psi$ even converge for all $t\in\BC$. Note that the matrix equations obtained above correspond to the isomorphisms of $\BC\langle t\rangle$-modules from Theorem~\ref{ThmMandDMofE}, Propositions~\ref{PropDualizingUnif} and \ref{Prop4.23} in the following diagram
\[
\xymatrix @C+2pc @R+1pc {
\Hom_{\BC[t]}\bigl(\sigma^*\ulM,\,\Omega^1_{\BC[t]/\BC}\bigr)\otimes_{\BC[t]}\BC\langle t\rangle \ar[rr]^{\TS\Xi} & & \uldM\otimes_{\BC[t]}\BC\langle t\rangle \\
\Hom_{\BF_q[t]}\bigl(\Lambda(\ulM),\,\Omega^1_{\BF_q[t]/\BF_q}\bigr)\otimes_{\BF_q[t]}\BC\langle t\rangle \ar[u]^{\TS(\sigma^*h_\ulM\dual)^{-1}} & \Lambda(\ulE)\otimes_{\BF_q[t]}\BC\langle t\rangle \ar[l]^{\TS\qquad\qquad\quad\beta_{\BF_q[t]}} & \Lambda(\uldM)\otimes_{\BF_q[t]}\BC\langle t\rangle \ar[l]^{\TS\delta_0\otimes\id} \ar[u]^{\TS h_\uldM}
}
\]
Namely $(\delta_0\otimes\id)^{-1}$, respectively $\beta_{\BF_q[t]}$, send the basis $(\lambda_i)_i$ of $\Lambda(\ulE)$ to the basis $(\dm'_i)$ of $\Lambda(\uldM)$, respectively to the dual of the basis of $\CC$ of $\Lambda(\ulM)$; see Remark~\ref{RemScattering}. Moreover, $h_\ulM$ is represented with respect to the basis $\CC$ and the basis $(m_j)_j$ by the matrix $(\Psi^T)^{-1}$, so $(\sigma^*h_\ulM\dual)^{-1}$ is represented by the matrix $\sigma^*\Psi$ with respect to the basis $(\eta_\ell)_\ell$ of $\Hom_{\BC[t]}\bigl(\sigma^*\ulM,\,\Omega^1_{\BC[t]/\BC}\bigr)$ which is dual to $(\sigma^*m_j)_j$ and the basis dual to $\CC$. And finally $\Xi$ is represented with respect to the bases $(\eta_\ell)_\ell$ and $(\dm_i)_i$ by the matrix $X$.
\end{example}

This example also suggests that the columns of the matrices $\sum_{j\ge0}\Delta_{t,\nu+j}\tau^j\in\BC\{\tau\}^{d\times d}$ from Corollary~\ref{CorDivisionTower}, when viewed as elements of $\BC\{\tau\}^d\cong\dM$ via an isomorphism $\rho\colon E\isoto\BG_{a,\BC}^d$ of $\BF_q$-module schemes, are relevant for an explicit description of the isomorphism $\Xi$ from Theorem~\ref{ThmMandDMofE} and the pairing of Question~\ref{QuestMandDMofE}.

\subsection{The associated Hodge-Pink structure}\label{SectAModHPStr}

Let $\ulE=(E,\phi)$ be a uniformizable mixed abelian, respectively $A$-finite Anderson $A$-module of dimension $d$ and rank $r$ over $\BC$. Consider the exponential exact sequence
\[
0\;\longto\;\Lambda(\ulE)\;\longto\;\Lie E\;\xrightarrow{\;\exp_\ulE}\;E(\BC)\;\longto\;0\,,
\]
where $\Lambda:=\Lambda(\ulE):=\ker(\exp_\ulE)$. It is a discrete $A$-submodule which is projective of rank $r$ by \cite[Theorem~4]{Anderson86}, respectively Theorem~\ref{ThmUnifAModandDM}. We extend the action of $A$ on the $\BC$-vector space $\Lie E$ to an action of $Q_\BC=Q\otimes_{\BF_q}\BC$ by letting $\tilde a/a\in Q$ with $\tilde a,a\in A$ act via $\Lie\phi_{\tilde a/a}:=(\Lie\phi_{\tilde a})\circ(\Lie\phi_a)^{-1}$. Note that $\Lie\phi_a$ is invertible for $a\ne0$ because $(\Lie\phi_a-\charmorph(a))^d=0$ on $\Lie E$ and $\charmorph(a)\ne0$. Since $J^d=0$ on $\Lie E$ and $A_{\BC}/J^d=\BC\dbl z-\zeta\dbr/(z-\zeta)^d$ by Lemma~\ref{LemmaZ-Zeta}, we may view $\Lie E$ as a $\BC\dbl z-\zeta\dbr$-module. We obtain a well defined $\BC\dbl z-\zeta\dbr$-homomorphism $\gamma$ on the right in the sequence
\begin{equation}\label{Eq1.1}
\xymatrix @R=0.1pc {
0 \ar[r] & \Fq \ar[r] & \Lambda(\ulE)\otimes_A \BC\dbl z-\zeta\dbr \ar[r]^{\DS\gamma} & \Lie E \ar[r] & 0 \\
& & \lambda\otimes\sum_i b_i(z-\zeta)^i \ar@{|->}[r] & \sum_i b_i\cdot(\Lie\phi_z-\zeta)^i(\lambda)\,,
}
\end{equation}
and we let $\Fq$ be its kernel. The sequence \eqref{Eq1.1} is exact on the right by Anderson~\cite[Corollary~3.3.6]{Anderson86} when $\ulE$ is abelian, respectively by Corollary~\ref{CorUnifGeneratesLieE} when $\ulE$ is $A$-finite. So the pair $(\Lambda,\Fq)$ determines the $A$-module $\Lie E$ and via $\exp_\ulE$ also $\ulE$. We further set 
\[H\;:=\;\Hodge_1(\ulE)\;:=\;\Lambda(\ulE)\otimes_A Q\qquad \text{and}\qquad W_\mu H := \Hodge_1(W_{\mu\,}\ulE)\,.
\]
Then $\ulHodge_1(\ulE):=(H,W_\bullet H,\Fq)$ is a $Q$-pre Hodge-Pink structure all of whose weights are negative. It satisfies $(z-\zeta)^d\Fp\subset\Fq\subset\Fp$ and hence, $F^{-d}H_\BC=H_\BC$ and $F^1H_\BC=(0)$. Recall that if $\ulE$ is pure, then, by our convention, its weight is $-\frac{d}{r}$ and so $\ulHodge_1(\ulE)$ is a pure $Q$-pre Hodge-Pink structure of weight $-\frac{d}{r}$. By Theorem~\ref{ThmHodgeConjecture} and Theorem~\ref{ThmHPofEandM} below, respectively by Theorem~\ref{ThmDualHodgeConjecture} and Theorem~\ref{ThmHPofEandDualM} below, $\ulHodge_1(\ulE)$ is in fact a $Q$-Hodge-Pink structure.

\begin{definition}\label{DefHPStructOfE}
Let $\ulE$ be a uniformizable mixed abelian, respectively $A$-finite Anderson $A$-module over $\BC$. The $Q$-Hodge-Pink structure $\ulHodge_1(\ulE)$ constructed above is called the \emph{$Q$-Hodge-Pink structure associated with $\ulE$}. We also set $\ulHodge^1(\ulE):=\ulHodge_1(\ulE)\dual$ in $\QHodgeCat$. The functor $\ulHodge_1$ is covariant and $\ulHodge^1$ is contravariant in $\ulE$.
\end{definition}

\begin{remark}
This construction parallels the classical situation in which abelian Anderson $A$-modules are replaced by abelian varieties and rational mixed Hodge structures are associated with them. Let $E$ be an abelian variety of dimension $d$ over the (classical) complex numbers (which we denote $\BC$ in the rest of this remark). Then $E(\BC)=\Lie E/\Lambda(E)$ where $\Lambda(E)$ is a $\BZ$-lattice of rank $2d$ in $\Lie E$. This lattice is functorially in $E$ described as the Betti-homology group $\Lambda(E)=\Koh_1(E,\BZ)$. There is a natural surjection on the right in the sequence
\[
\xymatrix @R=0pc {
0 \ar[r] & \ol{\Lie E} \ar[r] & \Koh_1(E,\BZ)\otimes_\BZ\BC \ar[r] & \Lie E \ar[r] & 0\,, \\
& & \lambda\otimes b \ar@{|->}[r] & b\cdot\lambda\,.
}
\]
The subspace $F^0\Koh_1(E,\BZ)\otimes_\BZ\BC:=\ol{\Lie E}$ constitutes the Hodge filtration on the Betti-homology of $E$.
\end{remark}

\begin{theorem}\label{ThmHPofEandDualM}
Let $\ulE$ be a uniformizable mixed $A$-finite Anderson $A$-module over $\BC$ and let $\uldM=\uldM(\ulE)$ be its associated effective mixed dual $A$-motive. Then $\ulHodge_1(\ulE)$ and $\ulHodge_1(\uldM)$ are canonically isomorphic. In particular, $\ulHodge_1(\ulE)$ and $\ulHodge^1(\ulE)$ are mixed $Q$-Hodge-Pink structures.
\end{theorem}

\begin{proof}
Since $\Lambda(\uldM)=\bigl(\dM\otimes_{A_\BC}\CO\bigl(\dotFC_\BC\setminus{\TS\bigcup_{i\in\BN_{>0}}}\Var(\ssigma^{i\ast}J)\bigr)\bigr)^\sdtau$ by Proposition~\ref{PropMaphdM}\ref{PropMaphdMA}, Corollary~\ref{CorTo_Boiler2.5.5} provides an $A$-isomorphism $\delta_0\colon\Lambda(\uldM)\isoto\Lambda(\ulE)$. By \eqref{EqThmMixedEandMWeights} in the proof of Theorem~\ref{ThmMixedEandM} it satisfies $W_\mu\Lambda(\uldM)=\Lambda(W_\mu\uldM)=\Lambda\bigl(\uldM(W_\mu\ulE)\bigr)\isoto\Lambda(W_\mu\ulE)=W_\mu\Lambda(\ulE)$, that is, it is compatible with the weight filtrations. Moreover, $\delta_0$ fits into the commutative diagram
\[
\xymatrix @C+2pc {
0 \ar[r] & \Fq \ar[r]\ar[d]^{\TS\cong} & \Lambda(\uldM)\otimes_A\BC\dbl z-\zeta\dbr \ar[r]_{\qquad\TS h_\dM}\ar[d]^{\TS\cong}_{\TS\delta_0\otimes\id} & \dM/\sdtau_\dM(\sdsigma^*\dM) \ar[r]\ar[d]^{\TS\cong}_{\TS\delta_0} & 0\;\\
0 \ar[r] & \Fq \ar[r] & \Lambda(\ulE)\otimes_A \BC\dbl z-\zeta\dbr \ar[r] & \Lie E \ar[r] & 0 \,,
}
\]
that is, it is compatible with the Hodge-Pink lattices. The last statement follows from Theorem~\ref{ThmDualHodgeConjecture}.
\end{proof}

\begin{theorem}\label{ThmHPofEandM}
Let $\ulE$ be a uniformizable mixed abelian Anderson $A$-module over $\BC$ and let $\ulM=\ulM(\ulE)$ be its associated mixed $A$-motive. Consider the $Q$-Hodge-Pink structure $\ul\Omega=(H,W_\bullet H,\Fq)$ which is pure of weight $0$ and given by $H=\Omega^1_{Q/\BF_q}=Q\,dz$ and $\Fq=\BC\dbl z-\zeta\dbr dz$. Then $\ulHodge_1(\ulE)$ and $\ulHodge_1(\ulM)\otimes\ul\Omega=\CHom(\ulHodge^1(\ulM),\ul\Omega)$ are canonically isomorphic.
\end{theorem}

Before we prove the theorem note that $\BC\dbl z-\zeta\dbr dz=\wh\Omega^1_{\BC\dbl z-\zeta\dbr/\BC}$ is the $\BC\dbl z-\zeta\dbr$-module of continuous differentials. Further note that $\ul\Omega\cong\UOne(0)$ and hence, $\ulHodge_1(\ulE)\cong\ulHodge_1(\ulM)$ and $\ulHodge^1(\ulE)\cong\ulHodge^1(\ulM)$.
Combining the theorem with Theorem~\ref{ThmHodgeConjecture} leads to the following

\begin{corollary}\label{CorHPofEandM}
If $\ulE$ is a uniformizable mixed abelian Anderson $A$-module, then $\ulHodge_1(\ulE)$ and $\ulHodge^1(\ulE)$ are mixed $Q$-Hodge-Pink structures.\qed
\end{corollary}

\begin{proof}[Proof of Theorem~\ref{ThmHPofEandM}]
Let $\ulM=\ulM(\ulE)=(M,\tau_M)$ and write $\ulHodge_1(\ulE)=(\Hodge_1(\ulE),W_\bullet\Hodge_1(\ulE),\Fq_\ulE)$ and $\ulHodge^1(\ulM)=(\Hodge^1(\ulM),W_\bullet\Hodge^1(\ulM),\Fq^\ulM)$ and $\CHom(\ulHodge^1(\ulM),\ul\Omega)=(\wt H_M,W_\bullet\wt H_M,\tilde\Fq_\ulM)$.

\medskip
\noindent
1. The isomorphism $\Hodge_1(\ulE)=\Lambda(\ulE)\otimes_AQ\isoto \wt H_M=\Hom_A(\Lambda(\ulM),\Omega^1_{A/\BF_q})\otimes_AQ$ in question will be induced from the isomorphism \cite[Corollary~2.12.1]{Anderson86}
\[
\beta_A\colon\Lambda(\ulE)\isoto\Hom_A(\Lambda(\ulM),\Omega^1_{A/\BF_q})\,,\es\lambda\longmapsto m\dual_{A,\lambda}\,,\qquad\text{where}\quad m\dual_{A,\lambda}\colon m\mapsto \omega_{A,\lambda,m}
\]
is determined by the residues $\Res_\infty(a\cdot\omega_{A,\lambda,m})= -m\bigl(\exp_\ulE(\Lie\phi_a(\lambda))\bigr)\in\BF_q$ for all $a\in Q$; see \eqref{EqAnderson2.12.1}. We verify its compatibility with the weight filtrations
\begin{eqnarray*}
W_\mu \Hodge_1(\ulE) & = & \Lambda(W_{\mu\,}\ulE)\otimes_AQ\qquad\text{and}\\[2mm]
W_\mu \wt H_M & = & \bigl\{\,h\in\Hom_Q(\Hodge^1(\ulM),\Omega^1_{Q/\BF_q}):\es h(W_{-\tilde\mu}\Hodge^1(\ulM))=0 \es\text{for all }\mu<\tilde\mu\,\bigr\}\,.
\end{eqnarray*}
Let $\mu_1<\ldots<\mu_n$ be the weights of $\ulE$ and set $\mu_0:=-\infty$ and $\mu_{n+1}:=+\infty$. If $\mu_i\le\mu<\tilde\mu\le\mu_{i+1}$, then $W_{-\tilde\mu}\Hodge^1(\ulM):=\ulHodge^1(W_{-\tilde\mu\,}\ulM)=\ker\bigl(\Hodge^1(\ulM)\onto\Hodge^1(\ulM(W_{\mu\,}\ulE))\bigr)$ by \eqref{EqRemMixedEandM} and Lemma~\ref{Lemma2.3}. This implies $W_\mu \wt H_M\;=\;\Hom_Q\bigl(\ulHodge^1(\ulM(W_{\mu\,}\ulE)),\Omega^1_{Q/\BF_q}\bigr)$. Therefore, $\beta_A\otimes\id_Q$ maps $W_\mu \Hodge_1(\ulE)$ isomorphically onto $W_\mu \wt H_M$ as desired.

\medskip
\noindent
2. We must show that $\beta_A\otimes\id_Q$ satisfies the compatibility $(\beta_A\otimes\id_{\BC\dpl z-\zeta\dpr})(\Fq_\ulE)=\tilde\Fq_\ulM$ with the Hodge-Pink lattices. As $\ulM$ is effective, $\Fq^\ulM$ sits in the exact sequence
\begin{equation}\label{EqEandM1}
\xymatrix{ 0 \ar[r] & \Fp^\ulM \ar[r] & \Fq^\ulM \ar[r] & \coker\tau_M \ar[r] & 0\,, }
\end{equation}
where $\Fp^\ulM=\Lambda(\ulM)\otimes_A\BC\dbl z-\zeta\dbr$ and $\coker\tau_M:=M/\tau_M(\sigma^*M)$. We also set 
\[
\tilde\Fp_\ulM\;:=\;\Hom_{\BC\dbl z-\zeta\dbr}(\Fp^\ulM,\BC\dbl z-\zeta\dbr dz)\;=\;\Hom_A(\Lambda(\ulM),\Omega^1_{A/\BF_q})\otimes_A\BC\dbl z-\zeta\dbr\,.
\]
Applying $\Hom_{\BC\dbl z-\zeta\dbr}(\fdot,\BC\dbl z-\zeta\dbr dz)$ to \eqref{EqEandM1} and observing that $J^{\dim\ulE}\cdot M\subset\tau_M(\sigma^*M)$ implies $\Hom_{\BC\dbl z-\zeta\dbr}(\coker\tau_M,\BC\dbl z-\zeta\dbr dz)=0$, yields the upper row in the following diagram of $\BC\dbl z-\zeta\dbr$-modules with exact rows
\begin{equation}\label{EqEandM2}
\xymatrix @R+1pc { 0 \ar[r] & \tilde\Fq_\ulM \ar[r]\ar@{=}[d] & \tilde\Fp_\ulM \ar[r]\ar@{=}[d] & \Ext^1_{\BC\dbl z-\zeta\dbr}(\coker\tau_M,\BC\dbl z-\zeta\dbr dz) \ar[r]\ar@{-->}[d]^{\DS\cong} & 0\\
 0 \ar[r] & \tilde\Fq_\ulM \ar[r] & \tilde\Fp_\ulM \ar[r]^{\DS\tilde\gamma_A\qquad} & \Hom_\BC(\coker\tau_M,\BC) \ar[r] & 0\\
0 \ar[r] & \Fq_\ulE \ar[r]\ar@{-->}[u]^{\DS\cong} & \Lambda(\ulE)\otimes_A\BC\dbl z-\zeta\dbr \ar[r]^{\DS\qquad\gamma} \ar[u]^{\DS\cong}_{\DS\beta_A\otimes\id_{\BC\dbl z-\zeta\dbr}} & \Lie E \ar[r]\ar[u]^{\DS\cong}_{\DS\alpha} & 0\,.
}
\end{equation}
In this diagram $\alpha$ is the isomorphism from \eqref{EqIsomLieE}
\[
\alpha\colon\Lie E\isoto \Hom_\BC(\coker\tau_M,\BC)\,,\es\lambda\longmapsto \bigl(\ol m\dual_\lambda\colon \ol m\mapsto(\Lie m)(\lambda)\bigr)
\]
The map $\gamma$ was defined in \eqref{Eq1.1} and the isomorphism $\beta_A\otimes\id_{\BC\dbl z-\zeta\dbr}$ is induced from the above isomorphism $\beta_A$. Finally, the map $\tilde\gamma_A$ is given by 
\newcommand{\EltOfPMDual}{m\dual}
\begin{eqnarray*}
& \tilde\gamma_A\colon \Hom_{\BC\dbl z-\zeta\dbr}(\Fp^\ulM,\BC\dbl z-\zeta\dbr dz) & \longto \es \Hom_\BC(\coker\tau_M,\BC)\,,\\[1mm]
& \mbox{\qquad\qquad\qquad}\EltOfPMDual & \longmapsto \es \bigl(\ol m\mapsto -\Res_{z=\zeta}(\EltOfPMDual(m))\bigr)\,.
\end{eqnarray*}
Here $\ol m=m\mod\tau_M(\sigma^*M)$ and $\EltOfPMDual(m)\in\BC\dpl z-\zeta\dpr dz$ is defined as 
\[
\EltOfPMDual(m)\;:=\;(\EltOfPMDual\otimes\id_{\BC\dpl z-\zeta\dpr})\bigl((h_\ulM\otimes\id_{\BC\dpl z-\zeta\dpr})^{-1}(m\otimes1)\bigr)
\]
where $(h_\ulM\otimes\id_{\BC\dpl z-\zeta\dpr})^{-1}(m\otimes1)$ is the preimage of $m\otimes1\in M\otimes_{A_{\BC}}\BC\dpl z-\zeta\dpr$ under the isomorphism $h_\ulM\otimes\id_{\BC\dpl z-\zeta\dpr}\colon \Lambda(\ulM)\otimes_A\BC\dpl z-\zeta\dpr\isoto M\otimes_{A_{\BC}}\BC\dpl z-\zeta\dpr$ from \eqref{EqhM}. Note that $\Res_{z=\zeta}(\EltOfPMDual(m))=0$ for all $m\in\tau_M(\sigma^*M)$ because for them $(h_\ulM\otimes\id_{\BC\dpl z-\zeta\dpr})^{-1}(m\otimes1)\in\Fp^\ulM$ and then $\EltOfPMDual\in\tilde\Fp_\ulM$ implies $\EltOfPMDual(m)\in\BC\dbl z-\zeta\dbr dz$. This proves that $\Res_{z=\zeta}(\EltOfPMDual(m))$ only depends on $\ol m$ and that the map $\tilde\gamma_A$ is well defined.

We show that $\tilde\gamma_A$ is $\BC\dbl z-\zeta\dbr$-linear. The $\BC\dbl z-\zeta\dbr$-action on $\Hom_\BC(\coker\tau_M,\BC)$ is induced from the action of $A_{\BC}$ on $\coker\tau_M$ which factors through $\BC\dbl z-\zeta\dbr/(z-\zeta)^d=A_{\BC}/J^d$ for $d=\dim\ulE$ by \eqref{EqIsomLieE} and the discussion thereafter. For $\EltOfPMDual\in\tilde\Fp_\ulM$, $f\in\BC\dbl z-\zeta\dbr$, and $\ol m\in\coker\tau_M$ this implies
\begin{eqnarray*}
\bigl(f\cdot\tilde\gamma_A(\EltOfPMDual)\bigr)(\ol m) & = & \tilde\gamma_A(\EltOfPMDual)(f\cdot\ol m) \\[2mm]
& = & -\Res_{z=\zeta}\bigl(\EltOfPMDual(fm)\bigr) \\[2mm]
& = & -\Res_{z=\zeta}\bigl((f\cdot \EltOfPMDual)(m)\bigr) \\[2mm]
& = & \tilde\gamma_A(f\cdot \EltOfPMDual)(\ol m)
\end{eqnarray*}
proving the $\BC\dbl z-\zeta\dbr$-linearity of $\tilde\gamma_A$.

To prove that $\tilde\Fq_\ulM:=\Hom_{\BC\dbl z-\zeta\dbr}(\Fq^\ulM,\BC\dbl z-\zeta\dbr dz)$ is the kernel of $\tilde\gamma_A$ first note that $\EltOfPMDual\in\tilde\Fq_\ulM$ and $(h_\ulM\otimes\id_{\BC\dpl z-\zeta\dpr})^{-1}(m\otimes1)\in\Fq^\ulM$ imply $\EltOfPMDual(m)\in\BC\dbl z-\zeta\dbr dz$ and hence, $\Res_{z=\zeta}(\EltOfPMDual(m))=0$ and $\tilde\Fq_\ulM\subset\ker\tilde\gamma_A$. Conversely, let $\EltOfPMDual\in\ker\tilde\gamma_A$. It follows for any $m\in M$ and any $n\in\BN_0$ that
\[
\Res_{z=\zeta}\Bigl((z-\zeta)^n\cdot (\EltOfPMDual\otimes\id_{\BC\dpl z-\zeta\dpr})\bigl((h_\ulM\otimes\id_{\BC\dpl z-\zeta\dpr})^{-1}(m\otimes1)\bigr)\Bigr)\;=\;\tilde\gamma_A\bigl((z-\zeta)^n\cdot \EltOfPMDual\bigr)(\ol m)\;=\;0\,.
\]
Therefore, $\EltOfPMDual\otimes\id_{\BC\dpl z-\zeta\dpr}\bigl((h_\ulM\otimes\id_{\BC\dpl z-\zeta\dpr})^{-1}(m\otimes1)\bigr)$ belongs to $\BC\dbl z-\zeta\dbr dz$. Since the $\BC\dbl z-\zeta\dbr$-module $\Fq^\ulM$ is generated by the elements $(h_\ulM\otimes\id_{\BC\dpl z-\zeta\dpr})^{-1}(m\otimes1)$ for $m\in M$ it follows that $\EltOfPMDual\in\Hom_{\BC\dbl z-\zeta\dbr}(\Fq^\ulM,\BC\dbl z-\zeta\dbr dz)=\tilde\Fq_\ulM$. We will show in step~3 below that the lower right square in of diagram~\eqref{EqEandM2} commutes. Therefore $\tilde\gamma_A$ is surjective, because $\gamma$ is. We conclude that also the middle row \eqref{EqEandM2} is exact and that $\Ext^1_{\BC\dbl z-\zeta\dbr}(\coker\tau_M,\BC\dbl z-\zeta\dbr dz)$ and $\Hom_\BC(\coker\tau_M,\BC)$ are isomorphic as quotients of $\tilde\Fp_\ulM$.

\medskip
\noindent
3. To prove the theorem it remains to show that $(\beta_A\otimes\id_{\BC\dpl z-\zeta\dpr})(\Fq_\ulE)=\tilde\Fq_\ulM$. For this it suffices to show that the lower right square in \eqref{EqEandM2} commutes; that is, $\alpha\circ\gamma=\tilde\gamma_A\circ(\beta_A\otimes\id_{\BC\dpl z-\zeta\dpr})$. By the $\BC\dbl z-\zeta\dbr$-linearity of the four maps this is equivalent to the following

\medskip\noindent
\emph{Claim~1.} The inverse isomorphism $\beta_A^{-1}\colon\Hom_A(\Lambda(\ulM),\Omega^1_{A/\BF_q})\isoto\Lambda(\ulE)$ is determined by the composition $\alpha\circ\gamma\circ\beta_A^{-1}$ which is given by $\tilde\gamma_A$; that is, by
\[
\xymatrix @R=0pc @C+1pc{
\Hom_A(\Lambda(\ulM),\Omega^1_{A/\BF_q}) \ar[r]^{\qquad\qquad\DS\beta_A^{-1}} & \Lambda(\ulE) \ar@{^{ (}->}[r] & \Lie E \ar[r]^{\DS\alpha\qquad\quad} & \Hom_\BC(\coker\tau_M,\BC)\\
m\dual\ar@{|->}[rrr] & & & \bigl(\ol m\mapsto -\Res_{z=\zeta}(m\dual\otimes 1)(m)\bigr)\,,
}
\]
where $m\dual\otimes1\in\Hom_A(\Lambda(\ulM),\Omega^1_{A/\BF_q})\otimes_A\BC\dbl z-\zeta\dbr=\tilde\Fp_\ulM$ is induced from $m\dual$.

\medskip

This can be made more explicit by choosing a coordinate system; that is, an isomorphism $\kappa=(\kappa_1,\ldots,\kappa_d)^T\colon E\isoto\BG_{a,\BC}^d$ of $\BF_q$-module schemes. The $\kappa_i\in\Hom_{\BF_q,\BC}(E,\BG_{a,\BC})=M$ then form a $\BC\{\tau\}$-basis of $M$, where $\tau$ is the $\sigma^\ast$-linear map $\tau\colon M\to M,m\mapsto\tau_M(\sigma^\ast m)$, and the $\ol\kappa_i:=\kappa_i\mod\tau_M(\sigma^*M)$ form a $\BC$-basis of $\coker\tau_M$ and yield an isomorphism $(\ol\kappa_1,\ldots,\ol\kappa_d)^T\colon\Lie E\isoto\BC^d$. In these terms the isomorphism $\alpha$ has the inverse
\[
\alpha^{-1}\colon \Hom_\BC(\coker\tau_M,\BC)\isoto\Lie E\,,\es \ol m\dual\longmapsto\bigl(\ol m\dual(\ol\kappa_1),\ldots,\ol m\dual(\ol\kappa_d)\bigr)^T
\]
and Claim~1 is equivalent to 

\medskip\noindent
\emph{Claim~2.} The inverse isomorphism $\beta_A^{-1}\colon\Hom_A(\Lambda(\ulM),\Omega^1_{A/\BF_q})\isoto\Lambda(\ulE)$ is given by 
\[
m\dual\;\longmapsto\; \bigl(-\Res_{z=\zeta}(m\dual\otimes 1(\kappa_1)),\ldots,-\Res_{z=\zeta}(m\dual\otimes 1(\kappa_d))\bigr)^T\,.
\]

\medskip

To prove Claim~2 we apply Anderson's theory of scattering matrices. We recall the notation introduced in Remark~\ref{RemScattering}. In particular $\wt A=\BF_q[t]\subset A$ is a finite flat ring extension for which the corresponding morphism of curves $C\to\BP^1_{\BF_q}$ is separable, $\CB=(m_1,\ldots,m_r)$ is a basis of $M$ over $\wt A_\BC=\BC[t]$, and $(\lambda_1,\ldots,\lambda_r)$ is an $\BF_q[t]$-basis of $\Lambda(\ulE)$, where $r=\rk_{\BF_q[t]}\Lambda(\ulE)=\rk_{\BC[t]}M$. Then
\[
\Psi\;:=\;\Bigl(\sum_{k=0}^\infty m_i\bigl(\exp_\ulE(\Lie\phi_t^{-k-1}(\lambda_j))\bigr)t^k\Bigr)_{i,j=1,\ldots,r}\;\in\; M_r\bigl((t-\theta)^{-d}\BC\langle\tfrac{t}{\theta}\rangle\bigr)\;.
\]
is Anderson's scattering matrix, where $\theta=\charmorph(t)\in\BC$ and $d=\dim\ulE$. The matrix $(\Psi^{-1})^T$ belongs to $M_r\bigl(\CO(\dotFC_\BC)\bigr)$ and its columns form an $\BF_q[t]$-basis $\CC=(n_1,\ldots,n_r)$ of $\Lambda(\ulM)$. With respect to the bases $\CC$ and $\CB$ the morphism $h_\ulM\colon\Lambda(\ulM)\otimes_A\CO(\dotFC_\BC)\to M\otimes_{A_{\BC}}\CO(\dotFC_\BC)$ is represented by $(\Psi^{-1})^T$. 

Under the induced morphism $C_\BC\to\BP^1_\BC$ the point $\Var(J)\in C_\BC$ maps to $\Var(t-\theta)\in\BP^1_\BC$. We extend the trace map from \eqref{EqTraceQ} to $\Tr_{\Quot(A_{\BC})/\BC(t)}\colon \Omega^1_{\Quot(A_{\BC})/\BC}\to\Omega^1_{\BC(t)/\BC}$. Then again by \cite[Formula (9.16) on p.~299]{VillaSalvador}
\begin{equation}\label{EqEandM4}
\Res_{t=\theta}(\Tr_{\Quot(A_{\BC})/\BC(t)}\omega)\;=\;\sum_{P\,|\Var(t-\theta)}\Res_{_{\SC P}}\omega
\end{equation}
for all $\omega\in\Omega^1_{\Quot(A_{\BC})/\BC}$ where the sum runs over all points $P\in C_\BC$ mapping to $\Var(t-\theta)$. Consider the rigid analytic closed disc $\Spm\BC\langle\frac{t}{\theta}\rangle=\{|t|\le|\theta|\}$ inside $(\BP^1_\BC)^\rig$ and its preimage $\Spm A_{\BC}\otimes_{\BC[t]}\BC\langle\frac{t}{\theta}\rangle$ inside $\FC_\BC$. By $\BC$-linearity and continuity \eqref{EqEandM4} extends to all differential forms $\omega\in(t-\theta)^{-d}A_{\BC}\otimes_{\BC[t]}\BC\langle\frac{t}{\theta}\rangle\,dt$ with pole above $\Var(t-\theta)$ of order at most $d$ and holomorphic on $(\Spm A_{\BC}\otimes_{\BC[t]}\BC\langle\frac{t}{\theta}\rangle)\setminus \Var(t-\theta)$. 

If we denote by ${}_{_{\SC\CB}}[\kappa_\ell]\in\BC[t]^r$ the coordinate vector of $\kappa_\ell$ with respect to the basis $\CB$ and by ${}_{_{\SC\CC}}\bigl[(h_\ulM\otimes\id_{\BC\dpl z-\zeta\dpr})^{-1}(\kappa_\ell\otimes1)\bigr]$ the coordinate vector with respect to the basis $\CC$, then
\[
{}_{_{\SC\CC}}\bigl[(h_\ulM\otimes\id_{\BC\dpl z-\zeta\dpr})^{-1}(\kappa_\ell\otimes1)\bigr]\;=\;\Psi^T\cdot{}_{_{\SC\CB}}[\kappa_\ell]\;=:\;(f_1,\ldots,f_r)^T\;\in\;(t-\theta)^{-d}\BC\langle{\TS\frac{t}{\theta}}\rangle^{\oplus r}\,.
\]
The map $\beta_A(\lambda_j)=m\dual_{A,\lambda_j}$ sends $n_i$ to $\omega_{A,\lambda_j,n_i}$ and hence, $(h_\ulM\otimes\id_{\BC\dpl z-\zeta\dpr})^{-1}(\kappa_\ell\otimes1)\;=\;\sum_i n_i\otimes f_i$ to $\sum_i f_i\,\omega_{A,\lambda_j,n_i}$. The latter is a differential form in $(t-\theta)^{-d} A_{\BC}\otimes_{\BC[t]}\BC\langle\frac{t}{\theta}\rangle\,dt$ which is holomorphic outside $\Var(J)$, because $h_\ulM^{-1}$ is an isomorphism on $(\Spm A_{\BC}\otimes_{\BC[t]}\BC\langle\frac{t}{\theta}\rangle)\setminus\Var(J)$. This differential form has trace 
\[
\TS\Tr_{\Quot(A_{\BC})/\BC(t)}(\sum_i f_i\,\omega_{A,\lambda_j,n_i})\;=\;\sum_i f_i\cdot\Tr_{Q/\wt Q}(\omega_{A,\lambda_j,n_i})\;=\;\sum_i f_i \,\omega_{\wt A,\lambda_j,n_i}\;=\;f_j\,dt\,;
\]
see \eqref{EqEandM5}. Applying \eqref{EqEandM4} yields
\begin{eqnarray*}
-\Res_{z=\zeta}(m\dual_{A,\lambda_j}\otimes1(\kappa_\ell))&:=& -\Res_{z=\zeta}\bigl((m\dual_{A,\lambda_j}\otimes\id_{\BC\dpl z-\zeta\dpr})(h_\ulM\otimes\id_{\BC\dpl z-\zeta\dpr})^{-1}(\kappa_\ell\otimes1)\bigr)\\[2mm]
&=& -\TS\Res_{_{\SC\Var(J)}}(\sum_i f_i\,\omega_{A,\lambda_j,n_i})\\[2mm]
&=& -\TS\Res_{t=\theta}\,\Tr_{\Quot(A_{\BC})/\BC(t)}(\sum_i f_i\,\omega_{A,\lambda_j,n_i})\\[2mm]
&=& -\TS\Res_{t=\theta}(f_j\,dt)\\[2mm]
&=& -\Res_{t=\theta}(e_j^T\Psi^T{}_{_{\SC\CB}}[\kappa_\ell]\,dt)\\[2mm]
&=& \ol\kappa_\ell(\lambda_j)\,.
\end{eqnarray*}
Here the last equation is \cite[Formula (3.3.3)]{Anderson86} taking into account that our scattering matrix $\Psi$ differs from Anderson's by a minus sign. This shows that 
\[
\lambda_j\;=\;\bigl(-\Res_{z=\zeta}(\beta_A(\lambda_j)\otimes1(\kappa_1)),\ldots,-\Res_{z=\zeta}(\beta_A(\lambda_j)\otimes1(\kappa_d))\bigr)^T
\]
and indeed the inverse isomorphism $\beta_A^{-1}$ has the form described in Claim~2. This finishes the proof of Claim~2, Claim~1 and the theorem.
\end{proof}

We also record the following theorem, which we will prove after Lemma~\ref{LemmaTraceOmega} below.

\begin{theorem}\label{ThmCompatXi}
Let $\ulE$ be a uniformizable mixed Anderson $A$-module over $\BC$ which is both abelian and $A$-finite, and let $\ulM=\ulM(\ulE)$ and $\uldM=\uldM(\ulE)$ be the associated (dual) $A$-motive. Then the isomorphisms above are also compatible with the isomorphisms from Theorems~\ref{ThmHofMandDualM}, \ref{ThmHPofEandDualM} and \ref{ThmHPofEandM} and the isomorphism $\Xi\colon\uldM(\ulM)\isoto\uldM(\ulE)$ from Theorem~\ref{ThmMandDMofE}, in the sense that the following diagram commutes
\begin{equation}\label{EqCompatXiHodge}
\xymatrix @C=9pc { 
\ulHodge_1\bigl(\uldM(\ulM)\bigr)  \ar[r]_{\TS\cong}^{\TS\ulHodge_1(\Xi)} & \ulHodge_1\bigl(\uldM(\ulE)\bigr) \ar[d]_{\TS\cong}^{\TS\rm Theorem~\ref{ThmHPofEandDualM}\es} \\
\ulHodge_1(\ulM)\otimes\ul\Omega \ar[r]_{\TS\cong}^{\TS\rm\es Theorem~\ref{ThmHPofEandM}} \ar[u]_{\TS\cong}^{\TS\rm Theorem~\ref{ThmHofMandDualM}\es} &  \ulHodge_1(\ulE)
}
\end{equation}
\end{theorem}

\subsection{Cohomology realizations}\label{SectCohAMod}

Let $\ulE$ be an Anderson $A$-module over $\BC$ with exponential function $\exp_\ulE\colon\Lie E\to E(\BC)$ and let $\Lambda(\ulE):=\ker(\exp_\ulE)$. We assume that $\ulE$ is abelian or $A$-finite. Anderson defined the \emph{Betti cohomology realization} of $\ulE$ to be
\[
\Koh_{1,\Betti}(\ulE,B):=\Lambda(\ulE)\otimes_A B
\quad\text{and}\quad \Koh_\Betti^1(\ulE,B):=\Hom_A(\Lambda(\ulE),B)
\]
for any $A$-algebra $B$; see \cite[Definition~1.3.6]{Goss94}. This is most useful when $\ulE$ is uniformizable, in which case both are locally free $B$-modules of rank equal to $\rk\ulE$ and $\Hodge_1(\ulE)=\Koh_{1,\Betti}(\ulE,Q)$; see Remark~\ref{RemAModuleUniformizable}(b), respectively Theorem~\ref{ThmUnifAModandDM}. By \cite[Corollary~2.12.2]{Anderson86} (respectively Theorem~\ref{ThmUnifAModandDM}) this realization provides for $B=Q$ an exact faithful functor on abelian (respectively $A$-finite) uniformizable Anderson $A$-modules.

\medskip

Moreover, let $v$ be a finite place of $C$, that is a closed point $v\in\dotC$ and let $A_v$ be the $v$-adic completion of $A$, and $Q_v$ the fraction field of $A_v$. Let $T_v\ulE:=\Hom_A\bigl(Q_v/A_v,\,\ulE(\BC)\bigr)$ be the \emph{$v$-adic Tate module} of $\ulE$. The \emph{$v$-adic cohomology realization} of $\ulE$ is defined as
\[
\begin{array}{lllllll}
\Koh_{1,v}(\ulE,A_v) & := & T_v\ulE & \quad\text{and}\quad & \Koh_{1,v}(\ulE,Q_v) & := & T_v\ulE\otimes_{A_v}Q_v\qquad\text{and}\\[2mm]
\Koh^1_v(\ulE,A_v) & := & \Hom_{A_v}(T_v\ulE,A_v) & \quad\text{and}\quad & \Koh^1_v(\ulE,Q_v) & := & \Hom_{A_v}(T_v\ulE,Q_v)\,;
\end{array}
\]
see \cite[\S\,1.2]{Goss94}. These are free $A_v$-modules, respectively $Q_v$-vector spaces of rank equal to $\rk\ulE$ by Remarks~\ref{RemRankE} and \ref{Rem2RankE}. Indeed, after fixing an integer $e$ such that $v^e\subset A$ is a principal ideal and choosing a generator $a$ of $v^e$ we can identify $A[\tfrac{1}{a}]/A\isoto Q_v/A_v$. Then there is an isomorphism
\begin{eqnarray}\label{EqTateModDivisionTower}
T_v\ulE \es\isoto\es \invlim[n]\bigl(\ulE[a^n](\BC),\phi_a\bigr)& := & \bigl\{(P_n)_n\in\TS\prod\limits_{n\in\BN}\ulE[a^n](\BC)\colon \phi_a(P_{n+1})=P_n\bigr\}\\
& = & \bigl\{\,\text{$a$-division towers }(P_n)_n\text{ above }0\,\bigr\}\,. \nonumber
\end{eqnarray}
This isomorphism sends $f\in\Hom_A\bigl(Q_v/A_v,\,\ulE(\BC)\bigr)$ to the tuple $P_n:=f(a^{-n})$. It is indeed an isomorphism, because from $(P_n)_n$ we can reconstruct $f\colon A[\tfrac{1}{a}]/A\to\ulE(\BC)$ as $f(c\,a^{-n})=\phi_c(P_n)$ for $c\in A, n\in\BN$. 

By Proposition~\ref{PropCompTateModEandM}\ref{PropCompTateModEandM_B} below, respectively Proposition~\ref{PropCompTateModEandDualM}\ref{PropCompTateModEandDualM_A} below, we obtain covariant functors \linebreak $\Koh_{1,v}(\,.\,,A_v)$ on abelian, respectively $A$-finite Anderson $A$-modules, which are exact and faithful, because they can be compared with the corresponding functors on the associated (dual) $A$-motives. If $\ulE$ is defined over a subfield $L$ of $\BC$ then $\Koh_{1,v}(\ulE,A_v)$ carries a continuous action of $\Gal(L^\sep/L)$ and the $v$-adic realization factors through the category ${\tt Mod}_{A_v[\Gal(L^\sep/L)]}$. Moreover, if $L$ is a \emph{finitely generated} field then
\begin{equation}\label{EqTateConjAModule}
\Koh_{1,v}(\,.\,,A_v)\colon \; \Hom(\ulE,\ulE')\otimes_A A_v \;\isoto\; \Hom_{A_v[\Gal(L^\sep/L)]}\bigl(\Koh_{1,v}(\ulE,A_v),\Koh_{1,v}(\ulE',A_v)\bigr)
\end{equation}
is an isomorphism for abelian, respectively $A$-finite, Anderson $A$-modules $\ulE$ and $\ulE'$. This is the analog of the \emph{Tate conjecture} and follows by Proposition~\ref{PropCompTateModEandM}\ref{PropCompTateModEandM_B}, respectively Proposition~\ref{PropCompTateModEandDualM}\ref{PropCompTateModEandDualM_A} from \eqref{EqTateConjAMotives}, respectively \eqref{EqTateConjDualAMotives}.

\begin{proposition}\label{PropWeightsTateModuleAModule}
Let $\ulE$ be a pure or mixed Anderson $A$-module, which is defined over a \emph{finite field extension} $L$ of $Q$ and is abelian or $A$-finite. Let $\CP$ be a place of $L$, not lying above $\infty$ or $v$, where $\ulE$ has good reduction, and let $\BF_\CP$ be its residue field. Then the geometric Frobenius $\Frob_\CP$ of $\CP$ has a well defined action on $\Koh_{1,v}(\ulE,A_v)$ and each of its eigenvalues lies in the algebraic closure of $Q$ in $\BC$ and has absolute value $(\#\BF_\CP)^\mu$ for a weight $\mu$ of $\ulE$. Dually every eigenvalue of $\Frob_\CP$ on $\Koh^1_v(\ulE,A_v)$ has absolute value $(\#\BF_\CP)^{-\mu}$ for a weight $\mu$ of $\ulE$. These eigenvalues are independent of $v$.
\end{proposition}

\noindent
{\it Remark.} The \emph{geometric Frobenius} $\Frob_\CP$ of $\CP$ is the inverse of the \emph{arithmetic Frobenius} $\Frob_\CP^{-1}$, which satisfies $\Frob_\CP^{-1}(x) \equiv x^{\#\BF_\CP}\mod\CP$ for $x\in\CO_L$.

\begin{proof}
This follows by Proposition~\ref{PropCompTateModEandM}\ref{PropCompTateModEandM_B}, respectively Proposition~\ref{PropCompTateModEandDualM}\ref{PropCompTateModEandDualM_A} from the corresponding facts for $\ulM(\ulE)$, respectively $\uldM(\ulE)$ proved in Propositions~\ref{PropWeightsTateModule}, respectively \ref{PropWeightsTateModuleDual}.
\end{proof}

\medskip

Finally, if $\ulE$ is \emph{abelian}, let $\ulM=(M,\tau_M)=\ulM(\ulE)$ be the associated $A$-motive. Then the \emph{de Rham cohomology realization} of $\ulE$ is defined to be
\begin{eqnarray*}
\Koh^1_\dR(\ulE,\BC) & := & \Hom_A(\Omega^1_{A/\BF_q},\,\sigma^*M/J\cdot\sigma^*M)\,,\\[2mm]
\Koh^1_\dR(\ulE,\BC\dbl z-\zeta\dbr) & := & \Hom_A\bigl(\Omega^1_{A/\BF_q},\,\sigma^*M\otimes_{A_\BC}\BC\dbl z-\zeta\dbr\bigr)\quad\text{and}\\[2mm]
\Koh_{1,\dR}(\ulE,\BC\dbl z-\zeta\dbr) & := & \Hom_{A_\BC}(\sigma^*M,\,\wh\Omega^1_{\BC\dbl z-\zeta\dbr/\BC})\,,
\end{eqnarray*}
where $\wh\Omega^1_{\BC\dbl z-\zeta\dbr/\BC}=\BC\dbl z-\zeta\dbr dz$ is the $\BC\dbl z-\zeta\dbr$-module of continuous differentials. We define the \emph{Hodge-Pink lattices} of $\ulE$ as the $\BC\dbl z-\zeta\dbr$-submodules 
\[
\begin{array}{rcccl}
\Fq^\ulE & := & \Hom_A\bigl(\Omega^1_{A/\BF_q},\,\tau_M^{-1}(M)\otimes_{A_\BC}\BC\dbl z-\zeta\dbr\bigr) & \subset & \Koh^1_\dR\bigl(\ulE,\BC\dpl z-\zeta\dpr\bigr) \quad\text{and}\\[2mm]
\Fq_\ulE & := & (\tau_M\dual\otimes\id_{\BC\dpl z-\zeta\dpr})\bigl(\Hom_{A_\BC}(M,\,\wh\Omega^1_{\BC\dbl z-\zeta\dbr/\BC})\bigr) & \subset & \Koh_{1,\dR}\bigl(\ulE,\BC\dpl z-\zeta\dpr\bigr)\,.
\end{array}
\]

On the other hand, if $\ulE$ is \emph{$A$-finite}, let $\uldM=(\dM,\sdtau_\ulM)=\uldM(\ulE)$ be the associated dual $A$-motive. Then the \emph{de Rham cohomology realization} of $\ulE$ is defined to be
\begin{eqnarray*}
\Koh^1_\dR(\ulE,\BC) & := & \Hom_\BC(\dM/J\dM,\BC)\,,\\[2mm]
\Koh^1_\dR(\ulE,\BC\dbl z-\zeta\dbr) & := & \Hom_{A_\BC}(\dM,\BC\dbl z-\zeta\dbr)\quad\text{and}\\[2mm]
\Koh_{1,\dR}(\ulE,\BC\dbl z-\zeta\dbr) & := & \dM\otimes_{A_\BC}\BC\dbl z-\zeta\dbr\,.
\end{eqnarray*}
We define the \emph{Hodge-Pink lattices} of $\ulE$ as the $\BC\dbl z-\zeta\dbr$-submodules 
\[
\begin{array}{rcccl}
\Fq^\ulE & := & (\sdtau_\dM\dual)^{-1}\bigl(\Hom_{A_\BC}(\sdsigma^*\dM,\,\BC\dbl z-\zeta\dbr)\bigr) & \subset & \Koh^1_\dR\bigl(\ulE,\BC\dpl z-\zeta\dpr\bigr) \quad\text{and}\\[2mm]
\Fq_\ulE & := & \sdtau_\dM(\sdsigma^*\dM)\otimes_{A_\BC}\BC\dbl z-\zeta\dbr & \subset & \Koh_{1,\dR}\bigl(\ulE,\BC\dpl z-\zeta\dpr\bigr)\,.
\end{array}
\]
In both cases the Hodge-Pink filtrations $F^i \Koh^1_\dR(\ulE,\BC)$ and $F^i \Koh_{1,\dR}(\ulE,\BC)$ of $\ulE$ are recovered as the images of $\Koh^1_\dR\bigl(\ulE,\BC\dbl z-\zeta\dbr\bigr)\cap(z-\zeta)^i\Fq^\ulE$ in $\Koh^1_\dR(\ulE,\BC)$ and of $\Koh_{1,\dR}\bigl(\ulE,\BC\dbl z-\zeta\dbr\bigr)\cap(z-\zeta)^i\Fq_\ulE$ in $\Koh_{1,\dR}(\ulE,\BC)$ like in Remark~\ref{Rem1.4}. All these structures are compatible with the natural duality between $H^1_\dR$ and $H_{1,\dR}$.

\begin{remark}\label{RemDeRhamOfE}
Let $\ulE$ be an abelian Anderson $A$-module and let $\ulM=(M,\tau_M)=\ulM(\ulE)$ be its associated $A$-motive. Our definition of $\Koh^1_\dR(\ulE,\BC)$ and its Hodge filtration coincides with the one of Goss~\cite[Definition~2.6.1]{Goss94}. If $\ulE$ is a \emph{Drinfeld $A$-module}, the de Rham cohomology realization $\Koh^1_\dR(\ulE,\BC)$ of a $\ulE$ was earlier defined by Deligne, Anderson, Gekeler and Jing Yu as the $\BC$-vector space of extension classes 
\[
0 \longto \BG_{a,\BC} \longto E^* \longto \ulE \longto 0
\]
of group schemes with $A$-action together with an $A$-equivariant splitting of the induced exact sequence of Lie algebras. Here $a\in A$ acts on $\BG_{a,\BC}$ via $\psi_{\charmorph(a)}$; see for example \cite[\S\,1.5]{Goss94} or \cite[\S\,3.4]{BrownawellPapanikolas16}.

There is an equivalent formulation as follows, see \cite[\S\,2]{Gekeler89} and \cite{Yu90}, which was extended to abelian Anderson $A$-modules by Brownawell and Papanikolas~\cite[\S\,3]{BrownawellPapanikolas02}. An \emph{$\BF_q$-linear biderivation} of $A$ into $\tau_M(\sigma^*M)$ is an $\BF_q$-homomorphism
\[
\eta\colon A\to \tau_M(\sigma^*M)\,,\es a\mapsto \eta_a\quad\text{such that}\quad \eta_{ab}=\charmorph(a)\cdot\eta_b+b\cdot\eta_a
\]
$\eta$ is called \emph{inner} if there is an element $m\in M$ with $\eta_a=\charmorph(a)\cdot m-a\cdot m\in\tau_M(\sigma^*M)$ for all $a\in A$. The condition $\charmorph(a)\cdot m-a\cdot m\in\tau_M(\sigma^*M)$ holds for example if $m\in\tau_M(\sigma^*M)$ in which case $\eta$ is called \emph{strictly inner}. Let $D(\ulE,\BC)$ (respectively $D_i(\ulE,\BC)$, respectively $D_{si}(\ulE,\BC)$) be the $\BC$-vector space of $\BF_q$-linear biderivations of $A$ into $\tau_M(\sigma^*M)$ (respectively inner, respectively strictly inner ones). Then define
\[
\Koh^1_\dR(\ulE,\BC)\es:=\es D(\ulE,\BC)/D_{si}(\ulE,\BC)\,.
\]
For Drinfeld $A$-modules $\ulE$ the isomorphism between these two definitions of $\Koh^1_\dR(\ulE,\BC)$ is given by sending $\eta\in D(\ulE,\BC)$ to the extension $E^*=\BG_{a,\BC}\times_\BC E$ with the action of $a\in A$ by $\left(\begin{smallmatrix} \psi_{\charmorph(a)}\; & \;\eta_a \\ 0\; & \;\phi_a \end{smallmatrix}\right)$ and observing $\eta_a\in M(\ulE)=\Hom_{\BF_q,\BC}(E,\BG_{a,\BC})$; see \cite[Theorem~1.5.4]{Goss94}.
Finally, Gekeler~\cite[(2.13)]{Gekeler89} defined the Hodge filtration of the Drinfeld $A$-module $\ulE$ by setting $F^0\Koh^1_\dR(\ulE,\BC)\;=\;\Koh^1_\dR(\ulE,\BC)$ and $F^2\Koh^1_\dR(\ulE,\BC)=(0)$ and
\[
F^1\Koh^1_\dR(\ulE,\BC)\es:=\es D_i(\ulE,\BC)/D_{si}(\ulE,\BC)\es\subset\es\Koh^1_\dR(\ulE,\BC)\,.
\]
For general abelian Anderson $A$-modules the relation to extension classes of group schemes was developed by Brownawell and Papanikolas~\cite[\S\,3.3]{BrownawellPapanikolas02} but they did not define the Hodge filtration.
\end{remark}

The following result, which justifies our definition of $\Koh^1_\dR(\ulE,\BC)$ and $\Fq^\ulE$ above, can be found in \cite[Lemmas~4.3 and 4.4]{Gekeler90}.

\begin{lemma}\label{LemmaGekeler}
Let $\ulE$ be an abelian Anderson $A$-module over $\BC$. Then there is a canonical isomorphism
\begin{equation}\label{EqLemmaGekeler1}
D(\ulE,\BC)/D_{si}(\ulE,\BC) \; \isoto \; \Hom_A\bigl(\Omega^1_{A/\BF_q},\,\sigma^*M/J\cdot\sigma^*M\bigr)\,.
\end{equation}
If $\ulE$ is a Drinfeld $A$-module over $\BC$ then \eqref{EqLemmaGekeler1} restricts to an isomorphism
\begin{equation}\label{EqLemmaGekeler2}
D_i(\ulE,\BC)/D_{si}(\ulE,\BC) \; \isoto \; \Hom_A\bigl(\Omega^1_{A/\BF_q},\,\tau_M^{-1}(J\cdot M)/J\cdot\sigma^*M\bigr)\,.
\end{equation}
In particular our definition of $\Koh^1_\dR(\ulE,\BC)$ and of $F^i\Koh^1_\dR(\ulE,\BC)$ coincides with the definition of Deligne, Anderson, Gekeler, Yu, Brownawell and Papanikolas which we recalled in Remark~\ref{RemDeRhamOfE}.
\end{lemma}

\begin{proof}
Let $\Delta:=\ker(A\otimes_{\BF_q}A\to A,\,a\otimes b\mapsto ab)$. Then $\Delta\otimes_{A\otimes A}A_\BC=J\subset A_\BC$. When we view $\tau_M(\sigma^*M)$ as an $A\otimes_{\BF_q}A$-module with $(a\otimes b)\cdot m := a\charmorph(b)\cdot m:=\psi_b\circ m\circ\phi_a$ for $m\in\tau_M(\sigma^*M)\subset M(\ulE)$ then by \cite[\S\,III.10.10, Proposition~17]{BourbakiAlgebra}
\begin{equation}\label{EqLemmaGekeler}
\begin{array}{rcccl}
D(\ulE,\BC) & \isoto & \Hom_{A\otimes A}\bigl(\Delta,\,\tau_M(\sigma^*M)\bigr) & \isoto & \Hom_{A_\BC}(J,\,\sigma^*M)\\[2mm]
\eta & \longmapsto & \bigl((a\otimes1-1\otimes a) \mapsto \eta_a\bigr) & \longmapsto & \bigl((a\otimes1-1\otimes\charmorph(a)) \mapsto \tau_M^{-1}(\eta_a)\bigr)\,.
\end{array}
\end{equation}
The last isomorphism is induced from $\tau_M\colon\sigma^*M\isoto\tau_M(\sigma^*M)$ and from the fact that $\tau_M(\sigma^*M)$ is an $A_\BC$-module. If $x\in J$ and  $\sum_i x_ig_i\in J\cdot\Hom_{A_\BC}(J,\sigma^*M)$ with $x_i\in J$ and $g_i\in\Hom_{A_\BC}(J,\sigma^*M)$, then $(\sum_i x_ig_i)(x):=\sum_i g_i(x_i x)=x\cdot m$ for $m=\sum_i g_i(x_i)\in\sigma^*M$. Therefore, $\sum_ix_ig_i$ corresponds under the isomorphism \eqref{EqLemmaGekeler} to the strictly inner derivation $\bigl(\eta\colon a\mapsto(\charmorph(a)-a)\cdot\tau_M(-m)\bigr)\in D_{si}(\ulE,\BC)$. On the other hand, since $J$ is an invertible $A_\BC$-module, we may identify $\sigma^*M\cong J\cdot\Hom_{A_\BC}(J,\sigma^*M)$ and write every $m\in\sigma^*M$ in the form $\sum_i x_i g_i\in J\cdot\Hom_{A_\BC}(J,\sigma^*M)$. This shows that $D_{si}(\ulE,\BC)\isoto J\cdot\Hom_{A_\BC}(J,\sigma^*M)$ under the isomorphism \eqref{EqLemmaGekeler}. Finally, \eqref{EqLemmaGekeler1} follows from $\Delta/\Delta^2=\Omega^1_{A/\BF_q}$ and the induced identification
\begin{eqnarray*}
\Hom_{A_\BC}(J,\sigma^*M)\big/J\cdot\Hom_{A_\BC}(J,\sigma^*M) & = & \Hom_{\BC}(J/J^2,\,\sigma^*M/J\cdot\sigma^*M)\\[2mm]
& = & \Hom_A(\Omega^1_{A/\BF_q},\,\sigma^*M/J\cdot\sigma^*M)\,.
\end{eqnarray*}
Moreover, if $\ulE$ is a Drinfeld $A$-module then $J\cdot M\subset\tau_M(\sigma^*M)$. Therefore, we can consider the morphism induced from \eqref{EqLemmaGekeler} 
\[
\begin{array}{rcccl}
D(\ulE,\BC) & \linj & \Hom_{A_\BC}\bigl(J,\tau_M^{-1}(J\cdot M)\bigr)) & \leftisoto & M\\[2mm]
\bigl(\eta\colon a\mapsto(\charmorph(a)-a)\cdot(-m)\bigr) & \longleftmapsto & \bigl(x\mapsto \tau_M^{-1}(xm)\bigr) & \longleftmapsto & m\,.
\end{array}
\]
Its image equals $D_i(\ulE,\BC)$ and $M\isoto\Hom_{A_\BC}\bigl(J,\tau_M^{-1}(J\cdot M)\bigr))$ is an isomorphism because $J$ is an invertible $A_\BC$-module. Therefore, $D_i(\ulE,\BC)\cong \Hom_{A_\BC}\bigl(J,\tau_M^{-1}(J\cdot M)\bigr))$ and 
\begin{eqnarray*}
D_i(\ulE,\BC)/D_{si}(\ulE,\BC) & \isoto & \Hom_{A_\BC}\bigl(J,\tau_M^{-1}(J\cdot M)\bigr)\big/J\cdot\Hom_{A_\BC}(J,\sigma^*M) \\[2mm]
& \isoto & \Hom_{\BC}\bigl(J/J^2,\,\tau_M^{-1}(J\cdot M)/J\cdot\sigma^*M\bigr)
\\[2mm]
& \isoto & \Hom_A\bigl(\Omega^1_{A/\BF_q},\,\tau_M^{-1}(J\cdot M)/J\cdot\sigma^*M\bigr)\,.
\end{eqnarray*}
This proves the lemma.
\end{proof}

\begin{proposition}\label{PropCompTateModEandM}
Let $\ulE$ be an abelian Anderson $A$-module over $\BC$ and let $\ulM=\ulM(\ulE)$ be the associated $A$-motive. 
\begin{enumerate}
\item\label{PropCompTateModEandM_B} 
There is a perfect pairing of $A_v$-modules
\[
\Koh_{1,v}(\ulE,A_v)\times \Koh^1_v(\ulM,A_v)\;\longto\;\Hom_{\BF_q}(Q_v/A_v,\BF_q)\,,\quad(f, m) \longmapsto m\circ f\,,
\]
where $m\circ f\colon Q_v/A_v\to\BG_{a,\BC}(\BC)=\BC$ factors through $\BF_q$ by the $\tau$-invariance of $m$. It induces isomorphisms
\begin{eqnarray*}
\Koh^1_v(\ulM,A_v) & \isoto & \Koh^1_v(\ulE,A_v)\otimes_{A_v}\Hom_{\BF_q}(Q_v/A_v,\BF_q)\qquad\text{and}\\[2mm]
\Koh_{1,v}(\ulE,A_v) & \isoto & \Koh_{1,v}(\ulM,A_v)\otimes_{A_v}\Hom_{\BF_q}(Q_v/A_v,\BF_q)\,.
\end{eqnarray*}
\item \label{PropCompTateModEandM_C}
There is a canonical isomorphism of $\BC\dbl z-\zeta\dbr$-modules
\[
\Koh^1_\dR(\ulM,\BC\dbl z-\zeta\dbr) \es \isoto \es \Koh^1_\dR(\ulE,\BC\dbl z-\zeta\dbr)\otimes_{\BC\dbl z-\zeta\dbr}\;\wh\Omega^1_{\BC\dbl z-\zeta\dbr/\BC}\,,
\]
which is compatible with the Hodge-Pink lattices.
\item \label{PropCompTateModEandM_A} 
If $\ulE$ is uniformizable, there is a perfect pairing of $A$-modules
\[
\Koh_{1,\Betti}(\ulE,A)\times \Koh^1_\Betti(\ulM,A)\;\longto\;\Omega^1_{A/\BF_q}\,,\quad(\lambda,m)\longmapsto \omega_{A,\lambda,m}
\]
where $\omega_{A,\lambda,m}$ is determined by the residues $\Res_\infty(a\cdot\omega_{A,\lambda,m})=-m\bigl(\exp_\ulE(\Lie\phi_a(\lambda))\bigr)\in\BF_q$ for all $a\in Q$. 
\end{enumerate}
\end{proposition}

\begin{proof}
\ref{PropCompTateModEandM_B} The existence of the perfect pairing follows from Anderson~\cite[Proposition~1.8.3]{Anderson86}. The rest follows from this.

\medskip\noindent
\ref{PropCompTateModEandM_C} By the universal property of the tensor product $\wh\Omega^1_{\BC\dbl z-\zeta\dbr/\BC}=\BC\dbl z-\zeta\dbr\otimes_A\Omega^1_{A/\BF_q}$ and our definitions
\begin{eqnarray*}
\Koh^1_\dR(\ulE,\BC\dbl z-\zeta\dbr) & := & \Hom_A\bigl(\Omega^1_{A/\BF_q},\,\sigma^*M\otimes_{A_\BC}\BC\dbl z-\zeta\dbr\bigr)\\[2mm]
& \,= & \Hom_{\BC\dbl z-\zeta\dbr}\bigl(\wh\Omega^1_{\BC\dbl z-\zeta\dbr/\BC},\,\Koh^1_\dR(\ulM,\BC\dbl z-\zeta\dbr)\bigr)
\end{eqnarray*}
and this is compatible with the Hodge-Pink lattices.

\medskip\noindent
\ref{PropCompTateModEandM_A} The perfect pairing was established by Anderson~\cite[Corollary~2.12.1]{Anderson86} and already used by us in \eqref{EqAnderson2.12.1} and in Theorem~\ref{ThmHPofEandM}.
\end{proof}

\begin{remark}\label{RemContDiff}
Let $\BF_v$ be the residue field of $A_v$. Then there is a canonical isomorphism of $A_v$-modules
\begin{equation}\label{EqTraceOnHom}
\Hom_{\BF_v}(Q_v/A_v,\BF_v) \;\isoto\;\Hom_{\BF_q}(Q_v/A_v,\BF_q)\,,\quad f\;\longmapsto\; \Tr_{\BF_v/\BF_q}\circ f
\end{equation}
given by composition with the trace map $\Tr_{\BF_v/\BF_q}\colon\BF_v\to\BF_q$. Indeed, $Q_v/A_v=\bigcup_n v^{-n}A_v/A_v$ is a union of finite dimensional $\BF_v$-vector spaces, and the $A_v$-homomorphism 
\[
\TS \Hom_{\BF_v}\bigl(\bigcup\limits_n v^{-n}A_v/A_v,\BF_v\bigr) \;\isoto\;\Hom_{\BF_q}\bigl(\bigcup\limits_n v^{-n}A_v/A_v,\BF_q\bigr)\,,\quad f \;\mapsto\; \Tr_{\BF_v/\BF_q}\circ f
\]
is injective, whence bijective by dimension reasons, because an element of these Hom sets is non-zero if and only if it is surjective onto $\BF_v$, respectively $\BF_q$. So the injectivity follows from the surjectivity of $\Tr_{\BF_v/\BF_q}$.

Furthermore, the $A_v$-module $\Hom_{\BF_v}(Q_v/A_v,\BF_v)$ is canonically isomorphic to the module of continuous differential forms $\wh\Omega^1_{A_v/\BF_v}$ under the map
\begin{equation}\label{EqOmegaAndHom}
\wh\Omega^1_{A_v/\BF_v} \; \isoto \; \Hom_{\BF_v}(Q_v/A_v,\BF_v)\,,\quad\omega\longmapsto\bigl(a\mapsto \Res_v^{\BF_v}(a\omega)\bigr)\,,
\end{equation}
where $\Res_v^{\BF_v}\colon\wh\Omega^1_{A_v/\BF_v}\to\BF_v$ is the residue map.
After choosing a uniformizing parameter $z$ of $A_v$ we can identify $A_v=\BF_v\dbl z\dbr$ and $\wh\Omega^1_{A_v/\BF_v}\cong\BF_v\dbl z\dbr dz$ and the inverse map is given by $\Hom_{\BF_v}(Q_v/A_v,\BF_v)\to\BF_v\dbl z\dbr dz,\,f\mapsto \sum_{i=0}^\infty f(z^{-1-i})z^idz$.

Combining \eqref{EqTraceOnHom} and \eqref{EqOmegaAndHom} and putting $\Res_v:=\Tr_{\BF_v/\BF_q}\circ\Res_v^{\BF_v}\colon\wh\Omega^1_{A_v/\BF_v}\to\BF_q$ yields the isomorphism
\begin{equation}\label{EqOmegaAndHom2}
\wh\Omega^1_{A_v/\BF_v} \; \isoto \; \Hom_{\BF_q}(Q_v/A_v,\BF_q)\,,\quad\omega\longmapsto\bigl(a\mapsto \Res_v(a\omega)\bigr)\,.
\end{equation}
\end{remark}

To obtain a comparison isomorphism between Betti cohomology and de Rham cohomology of Drinfeld modules, Gekeler~\cite[\S\,2]{Gekeler89} defined a kind of ``cycle integration'' as follows. He shows that for each $\eta\in D(\ulE,\BC)$ there exists a uniquely determined power series $F_\eta(X)=\sum_{i=0}^\infty f_i X^{q^i}$ in one variable $X$ such that
\begin{equation}\label{EqFEta}
F_\eta(\charmorph(a)\cdot X)-\charmorph(a)\cdot F_\eta(X)\;=\;\eta_a(\exp_\ulE(X))
\end{equation}
for all $a\in A$. (See \cite[\S\,3.2]{BrownawellPapanikolas02} for the generalization to abelian Anderson $A$-modules.) This defines a pairing
\begin{equation}\label{EqGekelersPairing}
\Koh_{1,\Betti}(\ulE,A)\times\Koh^1_\dR(\ulE,\BC)\;\longto\;\BC\,,\quad(\lambda,\eta)\;\mapsto\;\TS\int_\lambda\eta\;:=\;F_\eta(\lambda)\;\in\;\BC\,.
\end{equation}
We generalize this as follows.

\begin{theorem} \label{ThmPeriodIsomForE}
If $\ulE$ is a uniformizable abelian Anderson $A$-module there are canonical \emph{comparison isomorphisms}, sometimes also called \emph{period isomorphisms} for all $v$
\begin{eqnarray*}
h_{\Betti,v}\colon \Koh_{1,\Betti}(\ulE,A_v) \; = \; \Lambda(\ulE)\otimes_A A_v & \isoto & \Koh_{1,v}(\ulE,A_v) \; = \; \Hom_A\bigl(Q_v/A_v,\,\ulE(\BC)\bigr)\,,\\[2mm]
\lambda\otimes y & \longmapsto & \bigl(x\mapsto \exp_\ulE(\Lie\phi_{xy}(\lambda))\bigr)
\end{eqnarray*}
where $xy$ is viewed as an element of $A[\tfrac{1}{a}]/A$ for $v^e=(a)$ as above, and 
\[
\begin{array}[b]{rccl}
h_{\Betti,\,\dR}\colon & \Koh^1_\Betti(\ulE,\BC\dbl z-\zeta\dbr) & \isoto & \Koh^1_\dR(\ulE,\BC\dbl z-\zeta\dbr) \quad\text{and}\\[2mm]
h_{\Betti,\,\dR}\colon & \Koh^1_\Betti(\ulE,\BC) & \isoto & \Koh^1_\dR(\ulE,\BC)
\end{array}
\]
which are compatible with the Hodge-Pink lattices and Hodge-Pink filtration provided on the Betti realization $\Koh^1_\Betti(\ulE,Q)=\Hodge^1(\ulE)$ via the associated Hodge-Pink structure $\ulHodge^1(\ulE)$. All these isomorphisms are compatible with the comparison isomorphisms from Theorem~\ref{ThmCompIsomBettiDRAMotive} and Proposition~\ref{PropCompTateModEandM}.

Moreover, if $\ulE$ is a Drinfeld $A$-module, our comparison isomorphism $h_{\Betti,\,\dR}$ coincides with Gekeler's which is given by ``cycle integration'' 
\[
h_{\Betti,\,\dR}^{-1}\colon \Koh^1_\dR(\ulE,\BC)\;\isoto\;\Koh^1_\Betti(\ulE,\BC)\;=\;\Hom_A(\Lambda(\ulE),\BC)\,,\quad\eta\longmapsto(\lambda\mapsto\TS\int_\lambda\eta)\,.
\]
\end{theorem}

\begin{proof}
Clearly, the $A_v$-homomorphism $h_{\Betti,v}$ is well defined. In order to show that $h_{\Betti,v}$ is an isomorphism it suffices to prove that it is compatible with the comparison isomorphisms from Theorem~\ref{ThmCompIsomBettiDRAMotive} and Proposition~\ref{PropCompTateModEandM}. For this purpose we show that the following diagram commutes
\begin{equation}\label{EqThmPeriodIsomForE}
\xymatrix { 
\Koh_{1,\Betti}(\ulE,A) \otimes_A A_v \ar[d]_{\cong}\ar[r]^{\qquad h_{\Betti,v}} & \Koh_{1,v}(\ulE, A_v) \ar[d]^{\cong}\\
\Hom_A\bigl(\Koh^1_\Betti(\ulM,A),\,\Hom_{\BF_q}(Q_v/A_v,\BF_q)\bigr) & \Hom_{A_v}\bigl(\Koh^1_v(\ulM,A_v),\,\Hom_{\BF_q}(Q_v/A_v,\BF_q)\bigr) \ar[l]_{\es\cong}
}
\end{equation}
By Proposition~\ref{PropCompTateModEandM}\ref{PropCompTateModEandM_A} and the identification $\wh\Omega^1_{A_v/\BF_v}=\Hom_{\BF_q}(Q_v/A_v,\BF_q)$ from Remark~\ref{RemContDiff}, the left vertical arrow is given for $\lambda\in\Koh_{1,\Betti}(\ulE,A)$ and $y\in A_v$ and $m\in\Koh^1_\Betti(\ulM,A)$ by the assignment \mbox{$\lambda\otimes y\longmapsto(m\mapsto\omega_{A,\lambda,m}\otimes y)$} where $\omega_{A,\lambda,m}\otimes y\in\Omega^1_{A/\BF_q}\otimes_A A_v=\wh\Omega^1_{A_v/\BF_v}$ is identified with the map $\omega_{A,\lambda,m}\otimes y\colon Q_v/A_v\to\BF_q,$ $x\mapsto \Res_v(\omega_{A,\lambda,m}\otimes xy):=\Tr_{\BF_v/\BF_q}\bigl(\Res_v^{\BF_v}(\omega_{A,\lambda,m}\otimes xy)\bigr)$; see Remark~\ref{RemContDiff}. When we view $xy$ as an element of $A[\tfrac{1}{a}]/A$ for $v^e=(a)$ as above then the global differential form $xy\cdot\omega_{A,\lambda,m}\in\Omega^1_{Q/\BF_q}$ is holomorphic outside $v$ and $\infty$, and therefore $\Res_v(xy\cdot\omega_{A,\lambda,m})=-\Res_\infty(xy\cdot\omega_{A,\lambda,m})=m\bigl(\exp_\ulE(\Lie\phi_{xy}(\lambda))\bigr)$ by \cite[Definition~9.3.10 and Theorem~9.3.22]{VillaSalvador}. According to Proposition~\ref{PropCompTateModEandM}\ref{PropCompTateModEandM_B} and Theorem~\ref{ThmCompIsomBettiDRAMotive} this coincides with the composition of the other three maps in diagram~\eqref{EqThmPeriodIsomForE} as claimed.

We define $h_{\Betti,\,\dR}$ to be the composition of the isomorphisms
\begin{eqnarray*}
\Koh^1_\Betti(\ulE,\BC\dbl z-\zeta\dbr) & \isoto & \Hom_{\BC\dbl z-\zeta\dbr}\bigl(\wh\Omega^1_{\BC\dbl z-\zeta\dbr/\BC},\,\Koh^1_\Betti(\ulM,\BC\dbl z-\zeta\dbr)\bigr) \\[2mm]
& \isoto & \Hom_{\BC\dbl z-\zeta\dbr}\bigl(\wh\Omega^1_{\BC\dbl z-\zeta\dbr/\BC},\,\Koh^1_\dR(\ulM,\BC\dbl z-\zeta\dbr)\bigr) \\[2mm]
& \isoto & \Koh^1_\dR(\ulE,\BC\dbl z-\zeta\dbr)
\end{eqnarray*}
from Theorem~\ref{ThmCompIsomBettiDRAMotive} and Proposition~\ref{PropCompTateModEandM}\ref{PropCompTateModEandM_A} and \ref{PropCompTateModEandM_C}. The compatibility with the Hodge-Pink lattices was established in Theorem~\ref{ThmHPofEandM} for the first of these isomorphisms, and in Theorem~\ref{ThmCompIsomBettiDRAMotive} and Proposition~\ref{PropCompTateModEandM} for the other two.

\medskip
To prove that for a Drinfeld $A$-module our period isomorphism $h_{\Betti,\,\dR}$ is equal to Gekeler's, we describe the pairing
\begin{equation}\label{EqCompGekeler}
\Koh^1_\dR(\ulE,\BC\dbl z-\zeta\dbr)\times\Koh_{1,\Betti}(\ulE,A) \;\longto\;\BC\dbl z-\zeta\dbr
\end{equation}
induced by our $h_{\Betti,\,\dR}$. Let $f\in\Koh^1_\dR(\ulE,\BC\dbl z-\zeta\dbr):=\Hom_A\bigl(\Omega^1_{A/\BF_q},\,\sigma^*M\otimes_{A_\BC}\BC\dbl z-\zeta\dbr\bigr)=\Hom_{\BC\dbl z-\zeta\dbr}\bigl(\wh\Omega^1_{\BC\dbl z-\zeta\dbr/\BC},\,\Koh^1_\dR(\ulM,\BC\dbl z-\zeta\dbr)\bigr)$. Under the period isomorphism $h_{\Betti,\,\dR}=\sigma^*h_\ulM$ for $\ulM$ from Theorem~\ref{ThmCompIsomBettiDRAMotive} this $f$ is sent to $\sigma^*h_\ulM^{-1}\circ f\in\Hom_{\BC\dbl z-\zeta\dbr}\bigl(\wh\Omega^1_{\BC\dbl z-\zeta\dbr/\BC},\,\Koh^1_\Betti(\ulM,\BC\dbl z-\zeta\dbr)\bigr)$. For $\lambda\in\Koh_{1,\Betti}(\ulE,A):=\Lambda(\ulE)$ consider the element $\beta_A(\lambda)=m_{A,\lambda}\dual\in\Hom_A(\Lambda(\ulM),\Omega^1_{A/\BF_q})$ from \eqref{EqAnderson2.12.1}, which sends $m\in\Lambda(\ulM)$ to the differential form $\omega_{A,\lambda,m}\in\Omega_{A/\BF_q}$. Then our pairing \eqref{EqCompGekeler} sends $(f,\lambda)$ to 
\begin{equation}\label{EqCompGekeler2}
(m_{A,\lambda}\dual\otimes_A\id_{\BC\dbl z-\zeta\dbr})\circ\sigma^*h_\ulM^{-1}\circ f\es\in\es\End_{\BC\dbl z-\zeta\dbr}(\wh\Omega^1_{\BC\dbl z-\zeta\dbr/\BC})\es=\es\BC\dbl z-\zeta\dbr\,.
\end{equation}

To compute this element we apply Anderson's theory of scattering matrices \cite[\S\,3]{Anderson86} and recall the notation from Remark~\ref{RemScattering}. In particular $\wt A=\BF_q[t]\subset A$ is a finite flat ring extension for which the corresponding morphism of curves $C\to\BP^1_{\BF_q}$ is separable, $\CB=(m_1,\ldots,m_r)$ is a basis of $M$ over $\wt A_\BC=\BC[t]$, and $(\lambda_1,\ldots,\lambda_r)$ is an $\BF_q[t]$-basis of $\Lambda(\ulE)$, where $r=\rk_{\BF_q[t]}\Lambda(\ulE)=\rk_{\BC[t]}M$. Then
\[
\Psi\;:=\;\Bigl(\sum_{k=0}^\infty m_i\bigl(\exp_\ulE(\theta^{-k-1}\lambda_j)\bigr)t^k\Bigr)_{i,j=1,\ldots,r}\;\in\; M_r\bigl((t-\theta)^{-d}\BC\langle\tfrac{t}{\theta}\rangle\bigr)\;.
\]
is Anderson's scattering matrix, where $\theta=\charmorph(t)\in\BC$ and $d=\dim\ulE$. The matrix $(\Psi^{-1})^T$ belongs to $M_r\bigl(\CO(\dotFC_\BC)\bigr)$ and its columns form an $\BF_q[t]$-basis $\CC=(n_1,\ldots,n_r)$ of $\Lambda(\ulM)$. With respect to the bases $\CC$ and $\CB$ the morphism $h_\ulM\colon\Lambda(\ulM)\otimes_A\CO(\dotFC_\BC)\to M\otimes_{A_{\BC}}\CO(\dotFC_\BC)$ is represented by $(\Psi^{-1})^T$. At every point $P\in C_\BC$ lying above $\Var(t-\theta)\in\BP^1_\BC$ the element $t-\theta$ is a uniformizing parameter by Lemma~\ref{LemmaZ-Zeta}. Therefore, $P$ is unramified and $A\otimes_{\BF_q[t]}\BC\dbl t-\theta\dbr=\prod_{P|\Var(t-\theta)}\wh\CO_{C_\BC,P}=\prod_{P|\Var(t-\theta)}\BC\dbl t-\theta\dbr$. Let $pr\colon A\otimes_{\BF_q[t]}\BC\dbl t-\theta\dbr\onto\wh\CO_{C_\BC,\Var(J)}=\BC\dbl z-\zeta\dbr$ be the projection onto the factor for $P=\Var(J)$. The trace map $\Tr_{Q/\BF_q(t)}\colon Q\to\BF_q(t)$ corresponds under this product decomposition to the map 
\[
\Tr_{Q/\BF_q(t)}\otimes_{\BF_q(t)}\id_{\BC\dbl t-\theta\dbr}\colon\prod_{P|\Var(t-\theta)}\BC\dbl t-\theta\dbr\;\longto\;\BC\dbl t-\theta\dbr\,,\quad(f_P)_P\;\mapsto\;\sum_P f_P\,.
\]

We now view $\beta_A(\lambda)=m_{A,\lambda}\dual\in\Hom_A(\Lambda(\ulM),\Omega^1_{A/\BF_q})$ as an element of $\Hom_{\BF_q[t]}(\Lambda(\ulM),\Omega^1_{A/\BF_q})$ and consider $(m_{A,\lambda}\dual\otimes_{\BF_q[t]}\id_{\BC\dbl t-\theta\dbr})\in\Hom_{\BC\dbl t-\theta\dbr}\bigl(\Lambda(\ulM)\otimes_{\BF_q[t]}\BC\dbl t-\theta\dbr,\,\Omega^1_{A/\BF_q}\otimes_{\BF_q[t]}\BC\dbl t-\theta\dbr\bigr)$. Let $f_i\in\BC\dbl t-\theta\dbr$ be such that $\sum_in_i\otimes f_i\in\Lambda(\ulM)\otimes_{\BF_q[t]}\BC\dbl t-\theta\dbr=\Lambda(\ulM)\otimes_A\prod_{P|\Var(t-\theta)}\wh\CO_{C_\BC,P}$ is the element whose component at $P=\Var(J)$ is $(\sigma^*h_\ulM^{-1}\circ f)(dt)$ and whose components at $P\ne\Var(J)$ are $0$. Writing $\lambda=\sum_j c_j\lambda_j$ with $c_j\in\BF_q[t]$ we obtain
\begin{eqnarray*}
\bigl((m_{A,\lambda}\dual\otimes_A\id_{\BC\dbl z-\zeta\dbr})\circ\sigma^*h_\ulM^{-1}\circ f\bigr)(dt) & = & pr\circ(m_{A,\lambda}\dual\otimes_{\BF_q[t]}\id_{\BC\dbl t-\theta\dbr})(\TS\sum\limits_in_i\otimes f_i)\\[2mm]
& = & (\Tr_{Q/\BF_q(t)}\otimes_{\BF_q(t)}\id_{\BC\dbl t-\theta\dbr})(\TS\sum\limits_i f_i\cdot \omega_{A,\lambda,n_i})\\[2mm]
& = & \TS\sum\limits_i f_i\cdot \omega_{\wt A,\lambda,n_i}\\[2mm]
& = & \TS\sum\limits_i f_i\cdot c_i\,dt\,,
\end{eqnarray*}
where the third equality was proved in \eqref{EqEandM3} and the last equality in \eqref{EqEandM5}. Thus by \eqref{EqCompGekeler2} our pairing \eqref{EqCompGekeler} sends $(f,\lambda)$ to $\sum_i f_i c_i\in\BC\dbl t-\theta\dbr\cong\BC\dbl z-\zeta\dbr$. 

We compare this to Gekeler's pairing \eqref{EqGekelersPairing}. For our $f\in\Koh^1_\dR(\ulE,\BC\dbl z-\zeta\dbr)$ consider the element $\tau_M(f(dt))\in\tau_M(\sigma^*M)\otimes_{A_\BC}\BC\dbl z-\zeta\dbr$. Its reduction modulo $z-\zeta$ in $\tau_M(\sigma^*M)\otimes_{A_\BC}A_\BC/J$ induces by \eqref{EqLemmaGekeler1} and \eqref{EqLemmaGekeler} an element $\eta\mod D_{si}(\ulE,\BC)$ in $D(\ulE,\BC)/D_{si}(\ulE,\BC)$ with $\eta_t\equiv\tau_M(f(dt))\mod (z-\zeta)$, because $J/J^2=\BC\cdot(t-\theta)$. So modulo $J$ we have $\sum_i n_i\otimes f_i\equiv(\sigma^*h_\ulM^{-1}\circ f)(dt)\equiv(\sigma^*h_\ulM^{-1}\circ\tau_M^{-1})(\eta_t)=h_\ulM^{-1}(\eta_t)\mod J$. Let ${}_\CB[\eta_t]\in\BC[t]^r$ be the coordinate vector of $\eta_t\in\tau_M(\sigma^*M)\subset M$ with respect to the $\BC[t]$-basis $\CB$ of $M$. Then $\Psi^T\cdot{}_\CB[\eta_t]$ is the coordinate vector of $h_\ulM^{-1}(\eta_t)$ with respect to the $\BC\langle\tfrac{t}{\theta}\rangle$-basis $\CC$ of $\Lambda(\ulM)\otimes_{\BF_q[t]}\BC\langle\tfrac{t}{\theta}\rangle$. Therefore, using \eqref{EqFEta} we compute modulo $J$
\begin{eqnarray*}
(f_1,\ldots,f_r)^T & \equiv & \Psi^T\cdot{}_\CB[\eta_t]\\[2mm]
& \equiv & \left(\sum_{k=0}^\infty \eta_t\bigl(\exp_\ulE(\theta^{-k-1}\lambda_j)\bigr)t^k\right)_{j=1\ldots r}^T\\[2mm]
& \equiv & \left(t\cdot\sum_{k=0}^\infty F_\eta(\theta^{-k}\lambda_j)t^{k-1}-\sum_{k=0}^\infty \theta\cdot F_\eta(\theta^{-k-1}\lambda_j)t^k\right)_{j=1\ldots r}^T\\[2mm]
& \equiv & \left(F_\eta(\lambda_j) + (t-\theta)\cdot\sum_{k=1}^\infty F_\eta(\theta^{-k}\lambda_j)t^{k-1}\right)_{j=1\ldots r}^T\\[2mm]
& \equiv & \bigl(F_\eta(\lambda_j)\bigr)_{j=1\ldots r}^T\es\mod J\,.
\end{eqnarray*}
Since $(c_1,\ldots,c_r)^T$ is the coordinate vector of $\lambda$ with respect to the $\BF_q[t]$-basis $(\lambda_1,\ldots,\lambda_r)$ of $\Lambda(\ulE)$ we conclude $\sum_j c_j f_j\equiv \sum_j c_j F_\eta(\lambda_j)\equiv F_\eta(\lambda)\mod J$. So modulo $J\cdot\BC\dbl z-\zeta\dbr=(z-\zeta)$ our pairing \eqref{EqCompGekeler} specializes to Gekeler's pairing \eqref{EqGekelersPairing}. This completes the proof of the theorem.
\end{proof}

\medskip

\begin{proposition}\label{PropCompTateModEandDualM}
Let $\ulE$ be an $A$-finite Anderson $A$-module over $\BC$ and let $\uldM=\uldM(\ulE)$ be the associated dual $A$-motive. 
\begin{enumerate}
\item \label{PropCompTateModEandDualM_A}
The isomorphism \eqref{EqThmDivisionTower} from Theorem~\ref{ThmDivisionTower} induces canonical isomorphisms of $A_v$-modules
\[
\Koh_{1,v}(\uldM,A_v) \; \isoto \; \Koh_{1,v}(\ulE,A_v) \qquad\text{and}\qquad \Koh^1_v(\uldM,A_v) \; \isoto \; \Koh^1_v(\ulE,A_v) \,.
\]
\item \label{PropCompTateModEandDualM_C}
There are canonical isomorphisms of $\BC\dbl z-\zeta\dbr$-modules
\begin{eqnarray*}
\Koh^1_{\dR}(\uldM,\BC\dbl z-\zeta\dbr) & \isoto & \Koh^1_{\dR}(\ulE,\BC\dbl z-\zeta\dbr)\qquad\text{and}\\[2mm]
\Koh_{1,\dR}(\uldM,\BC\dbl z-\zeta\dbr) & \isoto & \Koh_{1,\dR}(\ulE,\BC\dbl z-\zeta\dbr)\,,
\end{eqnarray*}
which are compatible with the Hodge-Pink lattices.
\item \label{PropCompTateModEandDualM_B}
If $\ulE$ is uniformizable, the map $\delta_0$ from Proposition~\ref{prop:commutingdiagrams} and Corollary~\ref{CorTo_Boiler2.5.5} provides canonical isomorphisms of $A$-modules
\begin{eqnarray*}
\delta_0\colon & \Koh_{1,\Betti}(\uldM,A) & \isoto \es \Koh_{1,\Betti}(\ulE,A)\qquad\text{and}\\[2mm]
(\delta_0\dual)^{-1}\colon & \Koh^1_{\Betti}(\uldM,A) & \isoto \es \Koh^1_{\Betti}(\ulE,A)
\end{eqnarray*}
\end{enumerate}
If $\ulE$ is both abelian and $A$-finite and $\ulM=\ulM(\ulE)$ is the $A$-motive of $\ulE$ from Definition~\ref{DefAbelianAMod}, the isomorphisms above are also compatible with the isomorphisms from Propositions~\ref{PropCohAMot} and \ref{PropCompTateModEandM} and the isomorphism $\Xi\colon\uldM(\ulM)\isoto\uldM$ from Theorem~\ref{ThmMandDMofE}, in the sense that the diagram
\begin{equation}\label{EqDiagCompatXi}
\xymatrix @C+7pc {
\Koh_{1,*}\bigl(\uldM(\ulM)\bigr) \ar[r]_{\TS\cong}^{\TS\Koh_{1,*}(\Xi)} & \Koh_{1,*}\bigl(\uldM(\ulE)\bigr) \ar[d]_{\TS\cong} \\
\Hom\bigl(\Koh^1_*(\ulM),\Omega\bigr) \ar[r]_{\TS\cong}^{\TS\rm\es Proposition~\ref{PropCompTateModEandM}} \ar[u]_{\TS\cong}^{\TS\rm Proposition~\ref{PropCohAMot}\es} &  \Koh_{1,*}(\ulE)
}
\end{equation}
is commutative where $*\in\{\Betti,\dR,v\}$ and $\Omega=\Omega^1_{A/\BF_q}$ for $*=\Betti$, respectively $\Omega=\wh\Omega^1_{\BC\dbl z-\zeta\dbr/\BC}=\BC\dbl z-\zeta\dbr dz$ for $*=\dR$, respectively $\Omega=\wh\Omega^1_{A_v/\BF_v}=\Hom_{\BF_q}(Q_v/A_v,\BF_q)$ for $*=v$.
\end{proposition}

\begin{proof}
\ref{PropCompTateModEandDualM_A} Let $e$ be a positive integer such that $v^e=(a)\subset A$ is a principal ideal. Then $\Koh_{1,v}(\ulE,A_v)=\{\,\text{$a$-division towers }(P_n)_n\text{ above }0\,\}$ by \eqref{EqTateModDivisionTower}. Note that our definition of the map $\Koh_{1,v}(\uldM,A_v)\to\Koh_{1,v}(\ulE,A_v)$ corresponds to Anderson's ``switcheroo''; see \cite{ABP_Rohrlich} or \cite[Lemma~4.1.23 and Theorem~4.1.24(i)]{JuschkaDipl}. 

\medskip\noindent
\ref{PropCompTateModEandDualM_C} By definition $\Koh^1_{\dR}(\uldM,\BC\dbl z-\zeta\dbr) \, := \, \Hom_{A_\BC}(\dM,\,\BC\dbl z-\zeta\dbr)\,=:\,\Koh^1_{\dR}(\ulE,\BC\dbl z-\zeta\dbr)$.

\medskip\noindent
\ref{PropCompTateModEandDualM_B} was proved in Theorem~\ref{ThmUnifAModandDM}.

\medskip
Let now $\ulE$ be both abelian and $A$-finite. For $\Koh_{1,\dR}$ we enforce the compatibility by defining the isomorphism on the bottom of diagram~\eqref{EqDiagCompatXi} between 
\begin{eqnarray*}
\Hom_{\BC\dbl z-\zeta\dbr}\bigl(\Koh^1_\dR(\ulM,\BC\dbl z-\zeta\dbr),\,\BC\dbl z-\zeta\dbr dz\bigr) & = & \Hom_{A_\BC}(\sigma^*M,\Omega^1_{A_\BC/\BC})\otimes_{A_\BC}\BC\dbl z-\zeta\dbr\\[2mm]
& =: & \Koh_{1,\dR}\bigl(\uldM(\ulM),\BC\dbl z-\zeta\dbr\bigr)
\end{eqnarray*}
and $\Koh_{1,\dR}(\ulE,\BC\dbl z-\zeta\dbr):=\Koh_{1,\dR}\bigl(\uldM(\ulE),\BC\dbl z-\zeta\dbr\bigr)$ to be $\Koh_{1,\dR}\bigl(\Xi,\BC\dbl z-\zeta\dbr\bigr)$. 

The compatibility for $\Koh_{1,v}$ follows from Corollary~\ref{CorMandDMofE} and the isomorphism~\eqref{EqOmegaAndHom2}, taking into account that $\Res_v\bigl(a^{-1}h(m)\bigr)=-\Res_\infty\bigl(a^{-1}h(m)\bigr)$ for $(a)=v^e\subset A$ by \cite[Definition~9.3.10 and Theorem~9.3.22]{VillaSalvador}.

For $\Koh_{1,\Betti}$ we fix an element $t\in A\setminus\BF_q$ such that $Q$ is separable over $\BF_q(t)$ and consider the finite flat ring homomorphism $\wt A:=\BF_q[t]\into A$. Let $\wt\infty$ be the complement of $\Spec\wt A$ in $\BP^1_{\BF_q}$. All members of diagram~\eqref{EqDiagCompatXi} are finite projective $A$-modules and we consider them as finite projective $\wt A$-modules
. We use the identification $\Koh^1_\Betti(\ulM,A)=\Lambda(\ulM)=\Koh^1_\Betti(\ulM,\wt A)$ and the isomorphism $\Tr_{A/\wt A}\colon\Hom_{A}(\Koh^1_{\Betti}(\ulM,A),\Omega^1_{A/\BF_q})\isoto\Hom_{\wt A}(\Koh^1_{\Betti}(\ulM,\wt A),\Omega^1_{\wt A/\BF_q})$ from Lemma~\ref{LemmaTraceOmega} below. Let $(n_i)$ be an $\wt A$-basis of $\Koh^1_\Betti(\ulM,\wt A)$ and let $(\lambda_j)$ be the $\wt A$-basis of $\Koh_{1,\Betti}(\ulE,\wt A)$ which is dual to $(n_i)$ under the pairing from Proposition~\ref{PropCompTateModEandM}\ref{PropCompTateModEandM_A}, that is $\omega_{\wt A,\lambda_j,n_i}=\delta_{ij}\,dt$. Let $(\eta_j)$ with $\eta_j\in\Lambda(\uldM(\ulM))\subset\dM(\ulM)\otimes_{A_\BC}\CO\bigl(\dotFC_\BC\setminus\bigcup_{i\in\BN_{>0}}\Var(\ssigma^{i\ast}J)\bigr)$ be the $\wt A$-basis of $\Hom_{\wt A}(\Koh^1_{\Betti}(\ulM,\wt A),\Omega^1_{\wt A/\BF_q})$ which is the image of $(\lambda_j)$ under the isomorphism from Proposition~\ref{PropCompTateModEandM}, that is $\eta_j\colon m\mapsto\omega_{\wt A,\lambda_j,m}$. Let $\dn_j:=\Xi(\eta_j)\in\Koh_{1,\Betti}(\uldM(\ulE),\wt A)$ and $\lambda'_j:=\delta_0(\dn_j)\in\Koh_{1,\Betti}(\ulE,\wt A)$. We must show that $\lambda'_j=\lambda_j$ for all $j$, or equivalently $\omega_{\wt A,\lambda_j,n_i}=\omega_{\wt A,\lambda'_j,n_i}=\sum_{k=0}^\infty n_i\bigl(\exp_\ulE(\Lie\phi_t^{-k-1}(\lambda'_j))\bigr)t^k dt$ for all $i$ and $j$; see \eqref{EqEandM3}.

Fix a $k$ and write $\dn_j=\dm''_{j,k}+t^{k+1}\dm'_{j,k}$ with $\dm'_{j,k}\in\dM(\ulE)\otimes_{A_\BC}\CO\bigl(\dotFC_\BC\setminus\bigcup_{i\in\BN_{>0}}\Var(\ssigma^{i\ast}J)\bigr)$ and $\dm''_{j,k}\in\dM(\ulE)$. Then $\dm''_{j,k}+t^{k+1}\dm'_{j,k}=\dn_j=\sdtau_\dM(\sdsigma^*\dn_j)=\sdtau_\dM(\sdsigma^*\dm''_{j,k})+t^{k+1}\sdtau_\dM(\sdsigma^*\dm'_{j,k})$ and
\begin{equation}\label{Eqlambda'j}
\lambda'_j\;=\;\delta_0(\dn_j)\;=\;\delta_0\bigl(\dm''_{j,k}+t^{k+1}\dm'_{j,k}-\sdtau_\dM(\sdsigma^*\dm''_{j,k})\bigr)\;=\;\Lie\phi_t^{k+1}\delta_0\bigl(\sdtau_\dM(\sdsigma^*\dm'_{j,k})\bigr)\,,
\end{equation}
because $\delta_0\bigl(\sdtau_\dM(\sdsigma^*\dm''_{j,k})\bigr)=0$. Let $\dm_{j,k}:=\sdtau_\dM(\sdsigma^*\dm'_{j,k})-\dm'_{j,k}=t^{-k-1}(\dm''_{j,k}-\sdtau_\dM(\sdsigma^*\dm''_{j,k}))\in\dM(\ulE)$. Then \eqref{Eqlambda'j} and Corollary~\ref{Cor=Boiler2.5.5} imply
\[
\exp_\ulE\bigl(\Lie\phi_t^{-k-1}(\lambda'_j)\bigr)\;=\;\exp_\ulE\bigl(\delta_0(\dm'_{j,k}+\dm_{j,k})\bigr)\;=\;\delta_1(\dm_{j,k})\,.
\]
We write $n_i=m_{i,k}+t^{k+1}m'_{i,k}$ for an $m_{i,k}\in M$ and an $m'_{i,k}\in M\otimes_{A_{\BC}}\CO(\dotFC_\BC)$. Then we obtain $m_{i,k}-\tau_M(\sigma^*m_{i,k})=n_i-t^{k+1}m'_{i,k}-\tau_M(\sigma^*n_i)+t^{k+1}\tau_M(\sigma^*m'_{i,k})=t^{k+1}(\tau_M(\sigma^*m'_{i,k})-m'_{i,k})$. Setting $\eta''_{j,k}:=\Xi^{-1}(\dm''_{j,k})\in\dM(\ulM)=\Hom_{A_\BC}(\sigma^*M,\Omega^1_{A_\BC/\BC})$ and using \eqref{Eqm(dm'_eta)(1)}, we compute
\begin{eqnarray*}
n_i\bigl(\exp_\ulE(\Lie\phi_t^{-k-1}(\lambda'_j))\bigr) &:=& m_{i,k}\bigl(\exp_\ulE(\Lie\phi_t^{-k-1}(\lambda'_j))\bigr) \\[2mm]
&=& (m_{i,k}\circ\dm_{j,k})(1) \\[2mm]
&=& -\Res_{\wt\infty}t^{-k-1}\eta''_{j,k}(\sigma^*m_{i,k}) \\[2mm]
&=& \Res_{t=0}t^{-k-1}\eta''_{j,k}(\sigma^*m_{i,k}) \\[2mm]
&=& \Res_{t=0}t^{-k-1}\eta_j(\sigma^*n_i) \\[2mm]
&=& -\Res_{\wt\infty}t^{-k-1}\eta_j(n_i) \\[2mm]
&=& -\Res_{\wt\infty}t^{-k-1}\delta_{ij}\,dt \\[2mm]
&=& \delta_{ij}\delta_{k0}\,,
\end{eqnarray*}
where in lines four and six we use \cite[Theorem~9.3.22]{VillaSalvador} and that $\eta''_{j,k}(\sigma^*m_{i,k})\in\Omega^1_{A_\BC/\BC}$ and $\eta_j(n_i)=\eta_j(\sigma^*n_i)\in\Omega^1_{A/\BF_q}$, as $n_i=\sigma^*n_i$ in $\Lambda(\ulM)$, and in line five we use that $t^{-k-1}\bigl(\eta_j(\sigma^*n_i)-\eta''_{j,k}(\sigma^*m_{i,k})\bigr)=t^{-k-1}\bigl(\eta_j(\sigma^*n_i-\sigma^*m_{i,k})+(\eta_j-\eta''_{j,k})(\sigma^*m_{i,k})\bigr)=\eta_j(\sigma^*m'_{i,k})+\Xi^{-1}(\dm'_{j,k})(\sigma^*m_{i,k})$ is holomorphic at $t=0$. By \eqref{EqEandM3} this implies 
\[
\omega_{\wt A,\lambda'_j,n_i} \;=\; \sum_{k=0}^\infty n_i\bigl(\exp_\ulE(\Lie\phi_t^{-k-1}(\lambda'_j))\bigr)t^k dt \;=\; \delta_{ij}dt \;=\; \omega_{\wt A,\lambda_j,n_i}
\]
as desired.
\end{proof}

\begin{example}\label{ExCarlitzModulePeriod}
Let $C=\BP^1_{\BF_q}$, $A=\BF_q[t]$, $z=\frac{1}{t}$, $\theta=\charmorph(t)=\frac{1}{\zeta}\in\BC$, and let $\ulE=(\BG_{a,\BC},\phi_t=\theta+\tau)$ be the \emph{Carlitz module}. It is uniformizable, abelian and $A$-finite and its (dual) $A$-motive was described in Examples~\ref{ExCarlitz}, \ref{ExCarlitzPeriod}, \ref{ExDualCarlitz}, and \ref{ExDualCarlitzPeriod}. Let $\eta\in\BC$ satisfy $\eta^{q-1}=-\zeta$. By \cite[p.~47 bottom]{Thakur} the period lattice $\Lambda(\ulE):=\ker(\exp_\ulE)$ is generated by the \emph{Carlitz period} $\wt\pi:=\bigl(\eta^q\prod_{i=1}^\infty(1-\zeta^{q^i-1})\bigr)^{-1}$, which is the function field analog of $2i\pi$. In particular, the Carlitz period equals $\delta_0\bigl((\eta^q\check\ell_\zeta^{\SSC -})^{-1}\bigr)$ for the generator $(\eta^q\check\ell_\zeta^{\SSC -})^{-1}$ of $\Lambda\bigl(\uldM(\ulE)\bigr)$ from Example~\ref{ExDualCarlitzPeriod}. The compatibility of Proposition~\ref{PropCompTateModEandDualM} implies various interesting identities, like for example
\[
\sum_{k=0}^\infty\exp_\ulE(\theta^{-k-1}\wt\pi)t^k\;=\;(\eta\ell_\zeta^{\SSC -})^{-1}\;=\;\eta^{-1}\cdot\prod_{i=0}^\infty(1-\zeta^{q^i}t)^{-1}
\]
for the ($1\times1$-)scattering matrix; see Remark~\ref{RemScattering}
\end{example}

\begin{lemma}\label{LemmaTraceOmega}
Let $\wt A\into A$ be a finite flat morphism such that $Q/\Quot(\wt A)$ is a separable field extension (where $\Spec A$ and $\Spec\wt A$ are smooth affine curves over $\BF_q$). Then for any field extension $k/\BF_q$ and any finite projective $A_k$-module $P$ the map
\[
\Tr_{A/\wt A}\colon \Hom_{A_k}(P,\Omega^1_{A_k/k}) \;\longto\;\Hom_{\wt A_k}(P,\Omega^1_{\wt A_k/k})\,,\quad f\;\longmapsto\;\Tr_{A/\wt A}\circ f
\]
is an isomorphism of $A_k$-modules.
\end{lemma}

\begin{proof}
This is a special case of \cite[Corollary 3.4(c), p.~384]{HartshorneResidues}, which is reproved in elementary terms by \cite[Theorem~4.1.5]{Sinha97}, respectively \cite[Lemma~4.2.1]{Anderson86} when $\wt A=\BF_q[t]$.
\end{proof}

\begin{proof}[{Proof of Theorem~\ref{ThmCompatXi}}]
The commutativity of diagram~\eqref{EqDiagCompatXi} for $\Koh_{1,\Betti}$ implies the commutativity of diagram~\eqref{EqCompatXiHodge} on the level of the underlying $Q$-vector spaces. This suffices, because the compatibility with the weight filtrations and the Hodge-Pink lattices was proved in Theorems~\ref{ThmHofMandDualM}, \ref{ThmHPofEandDualM} and \ref{ThmHPofEandM}.
\end{proof}

\begin{theorem} \label{ThmPeriodIsomForAFiniteE}
If $\ulE$ is a uniformizable $A$-finite Anderson $A$-module there are canonical \emph{comparison isomorphisms}, sometimes also called \emph{period isomorphisms} for all $v$
\begin{eqnarray*}
h_{\Betti,v}\colon \Koh_{1,\Betti}(\ulE,A_v) \; = \; \Lambda(\ulE)\otimes_A A_v & \isoto & \Koh_{1,v}(\ulE,A_v) \; = \; \Hom_A\bigl(Q_v/A_v,\,\ulE(\BC)\bigr)\,,\\[2mm]
\lambda\otimes y & \longmapsto & \bigl(x\mapsto \exp_\ulE(\Lie\phi_{xy}(\lambda))\bigr)
\end{eqnarray*}
where $xy$ is viewed as an element of $A[\tfrac{1}{a}]/A$ for $v^e=(a)$ as above, and 
\[
\begin{array}[b]{rccl}
h_{\Betti,\,\dR}\colon & \Koh_{1,\Betti}(\ulE,\BC\dbl z-\zeta\dbr) & \isoto & \Koh_{1,\dR}(\ulE,\BC\dbl z-\zeta\dbr) \quad\text{and}\\[2mm]
h_{\Betti,\,\dR}\colon & \Koh_{1,\Betti}(\ulE,\BC) & \isoto & \Koh_{1,\dR}(\ulE,\BC)
\end{array}
\]
which are compatible with the Hodge-Pink lattices and Hodge-Pink filtration provided on the Betti realization $\Koh_{1,\Betti}(\ulE,Q)=\Hodge_1(\ulE)$ via the associated Hodge-Pink structure $\ulHodge_1(\ulE)$. All these isomorphisms are compatible with the comparison isomorphisms from Theorem~\ref{ThmCompIsomBettiDRDualAMotive} and Proposition~\ref{PropCompTateModEandDualM}.
\end{theorem}

\begin{proof}
From Theorem~\ref{ThmDivisionTower} we obtain the commutativity of the diagram
\[
\xymatrix @C=7pc {
\Koh_{1,\Betti}(\uldM,A)\otimes_A A_v \ar[r]_{\TS\cong}^{\TS h_{\Betti,v}} \ar[d]_{\TS\delta_0\otimes 1}^{\TS\cong} & \Koh_{1,v}(\uldM,A_v) \ar[d]_{\TS\cong} \\
\Koh_{1,\Betti}(\ulE,A)\otimes_A A_v \ar[r]^{\TS h_{\Betti,v}} & \Koh_{1,v}(\ulE,A_v)
}
\]
where the right vertical isomorphism was defined in Proposition~\ref{PropCompTateModEandDualM}. This proves that $h_{\Betti,v}$ is an isomorphism and compatible with the comparison isomorphisms from Theorem~\ref{ThmCompIsomBettiDRDualAMotive} and Proposition~\ref{PropCompTateModEandDualM}. 

We define $h_{\Betti,\,\dR}$ as the composition $(h_\dM\otimes\id_{\BC\dbl z-\zeta\dbr})\circ(\delta_0^{-1}\otimes\id_{\BC\dbl z-\zeta\dbr})\colon$
\[
\Koh_{1,\Betti}(\ulE,\BC\dbl z-\zeta\dbr)\;\isoto\;\Koh_{1,\Betti}(\uldM,\BC\dbl z-\zeta\dbr) \;\isoto\; \Koh_{1,\dR}(\uldM,\BC\dbl z-\zeta\dbr)\;=:\;\Koh_{1,\dR}(\ulE,\BC\dbl z-\zeta\dbr)\,.
\]
All compatibilities follow immediately.
\end{proof}

\begin{remark}\label{RemChangeOfA}
Let $\ulE$ be an abelian, respectively $A$-finite Anderson $A$-module. Then the various comparison isomorphisms between the cohomology realizations of $\ulE$, of $\ulM=\ulM(\ulE)$ and $\uldM=\uldM(\ulE)$ are compatible with a change of the ring $A$ as follows. Let $\wt A\subset A$ be a subring such that $Q$ is a finite \emph{separable} extension of $\wt Q=\Quot(\wt A)$ and let $\pi\colon C\to\wt C$ be the corresponding finite flat morphism of projective curves. Then $\wt\infty:=\pi(\infty)$ is the complement of $\Spec\wt A\subset\wt C$ and $\pi^{-1}(\wt\infty)=\{\infty\}$, and so $\pi^{-1}(\Spec\wt A)=\Spec A$ and $A$ is a finite locally free $\wt A$-module of rank $\rk_{\wt A}A=[Q:\wt Q]$. In this way $\ulE$ becomes an abelian (respectively $\wt A$-finite) Anderson $\wt A$-module and $\ulM$ (respectively $\uldM$) is its associated (dual) $\wt A$-motive. We have $\rk_{\wt A}\ulE=[Q:\wt Q]\cdot\rk_A\ulE$ and $\dim_{\wt A}\ulE=\dim_A\ulE$. When we compute the cohomology modules of $\ulE$ as an $A$-module (respectively $\wt A$-module) we add the index $A$ (respectively $\wt A$) to the notation, and similarly for $\ulM$ and $\uldM$. 

\medskip\noindent
(a) \es Then the Betti (co)homology satisfies
\begin{eqnarray*}
\Koh_{1,\Betti,A}(\ulE,A) \es = & \Lambda(\ulE) & = \es \Koh_{1,\Betti,\wt A}(\ulE,\wt A)\,,\\[2mm]
\Koh^1_{\Betti,A}(\ulM,A) \es = & \Lambda(\ulM) & = \es \Koh^1_{\Betti,\wt A}(\ulM,\wt A)\,,\quad\text{and}\\[2mm]
\Koh_{1,\Betti,A}(\uldM,A) \es = & \Lambda(\uldM) & = \es \Koh_{1,\Betti,\wt A}(\uldM,\wt A)\,.
\end{eqnarray*}
The isomorphisms from Propositions~\ref{PropCompTateModEandM} and \ref{PropCompTateModEandDualM} and \ref{PropCohAMot} are compatible with the change of rings $\wt A\subset A$ via the following commutative diagrams
\[
\xymatrix @C+1pc @R=0.8pc {
\Koh_{1,\Betti,A}(\uldM,A) \ar@{=}[d] \ar[r]^{\delta_0}_\cong & \Koh_{1,\Betti,A}(\ulE,A)\; \ar@{=}[d] \\
\Koh_{1,\Betti,\wt A}(\uldM,\wt A) \ar[r]^{\delta_0}_\cong & \Koh_{1,\Betti,\wt A}(\ulE,\wt A)
}
\]
and
\[
\xymatrix @C+1pc {
\Koh_{1,\Betti,A}(\ulE,A)\otimes_A\Koh^1_{\Betti,A}(\ulM,A) \ar[r]^{\qquad\qquad\qquad\beta_A} & \Omega^1_{A/\BF_q}\,, \ar[d]^{\Tr_{A/\wt A}} & (\lambda,m) \longmapsto \omega_{A,\lambda,m}\,,\\
\Koh_{1,\Betti,\wt A}(\ulE,\wt A)\otimes_{\wt A}\Koh^1_{\Betti,\wt A}(\ulM,\wt A) \ar[r]^{\qquad\qquad\qquad\beta_{\wt A}} \ar@{->>}[u] & \Omega^1_{\wt A/\BF_q}\,, & (\lambda,m) \longmapsto \omega_{\wt A,\lambda,m}\,.
}
\]
The proof is similar to \eqref{EqEandM3} and also follows from Lemma~\ref{LemmaTraceOmega}.

\medskip\noindent
(b) \es For the de Rham cohomology let $\tilde z$ be a uniformizing parameter of $\wt C$ at $\wt\infty$ and let $\tilde\zeta:=\charmorph(\tilde z)$. Then $\tilde z-\tilde\zeta$ is a uniformizing parameter at the point $\wt J:=(\tilde a\otimes 1-1\otimes\charmorph(\tilde a):\tilde a\in\wt A)\subset\wt A_\BC$, and also at every point $P\in C_\BC$ lying above $\Var(\wt J)\in\wt C_\BC$ by Lemma~\ref{LemmaZ-Zeta}. Therefore, $P$ is unramified and
\begin{equation}\label{EqRemChangeOfADecomp}
A\otimes_{\wt A}\BC\dbl\tilde z-\tilde\zeta\dbr=\prod_{P|\Var(\wt J)}\wh\CO_{C_\BC,P}=\prod_{P|\Var(\wt J)}\BC\dbl\tilde z-\tilde\zeta\dbr\,. 
\end{equation}
Let $pr\colon A\otimes_{\wt A}\BC\dbl\tilde z-\tilde\zeta\dbr\onto\wh\CO_{C_\BC,\Var(J)}=\BC\dbl z-\zeta\dbr$ be the projection onto the factor for $P=\Var(J)$. This induces the left column in the following diagram
\[
\xymatrix @C+4pc @R-0.5pc {
\Koh^1_{\dR,A}\bigl(\ulM,\BC\dbl z-\zeta\dbr\bigr) & \Koh^1_{\Betti,A}\bigl(\ulM,\BC\dbl z-\zeta\dbr\bigr) \ar[l]^\cong_{h_{\Betti,\,\dR,A}} \\
\sigma^*M\otimes_{A_\BC}\BC\dbl z-\zeta\dbr \ar@{=}[u] & \Lambda(\ulM)\otimes_A\BC\dbl z-\zeta\dbr \ar[l]^\cong_{h_{\Betti,\,\dR,A}} \ar@{=}[u]\\
\sigma^*M\otimes_{A_\BC}\bigl(A_\BC\otimes_{\wt A_\BC}\BC\dbl\tilde z-\tilde\zeta\dbr\bigr) \ar@{->>}[u]_{pr} & \Lambda(\ulM)\otimes_A\bigl(A\otimes_{\wt A}\BC\dbl\tilde z-\tilde\zeta\dbr\bigr) \ar[l]^\cong_{h_{\Betti,\,\dR,\wt A}} \ar@{->>}[u]^{pr} \\
\Koh^1_{\dR,\wt A}\bigl(\ulM,\BC\dbl\tilde z-\tilde\zeta\dbr\bigr) \ar@{=}[u] & \Koh^1_{\Betti,\wt A}\bigl(\ulM,\BC\dbl\tilde z-\tilde\zeta\dbr\bigr)\,. \ar[l]^\cong_{h_{\Betti,\,\dR,\wt A}} \ar@{=}[u] 
}
\]
If moreover $\ulM$ is uniformizable, also the right column exists and the diagram is commutative, where the horizontal isomorphisms are the period isomorphisms from Theorem~\ref{ThmCompIsomBettiDRAMotive}. There are similar diagrams for (uniformizable) dual $A$-motives and for (uniformizable) abelian or $A$-finite Anderson $A$-modules, which fit into the comparison diagrams
\[
\xymatrix @C+1pc @R-1pc {
\Hom_{\wt Q}\bigl(\Omega^1_{\wt Q/\BF_q},\,\Koh^1_{\dR,\wt A}(\ulM,\BC\dbl\tilde z-\tilde\zeta\dbr)\bigr) \ar@{->>}[r]^{pr} \ar@{=}[d] & \Hom_Q\bigl(\Omega^1_{Q/\BF_q},\,\Koh^1_{\dR,A}(\ulM,\BC\dbl z-\zeta\dbr)\bigr) \ar@{=}[d] \\
\Koh^1_{\dR,\wt A}(\ulE,\BC\dbl\tilde z-\tilde\zeta\dbr) \ar@{->>}[r]^{pr} & \Koh^1_{\dR,A}(\ulE,\BC\dbl z-\zeta\dbr) \\
}
\]
respectively
\[
\xymatrix @C+2pc @R-1pc {
\Koh^1_{\dR,\wt A}(\ulE,\BC\dbl\tilde z-\tilde\zeta\dbr) \ar@{->>}[r]^{pr} \ar@{=}[d] & \Koh^1_{\dR,A}(\ulE,\BC\dbl z-\zeta\dbr)  \ar@{=}[d] \\
\Koh^1_{\dR,\wt A}(\uldM,\BC\dbl\tilde z-\tilde\zeta\dbr) \ar@{->>}[r]^{pr}& \Koh^1_{\dR,A}(\uldM,\BC\dbl z-\zeta\dbr)
}
\]
Note that $\Omega^1_{\wt Q/\BF_q}=\wt Q\,d\tilde z$ and $\Omega^1_{Q/\BF_q}=Q\,d\tilde z$.

\medskip\noindent
(c) \es For the Hodge-Pink structures, (a) and (b) imply the compatibility 
\[
(\id_H,\id_{W_\bullet H},pr)\colon\;\ulHodge^1_{\wt A}(\ulM)\;=\;(H,W_\bullet H,\Fq)\;\longto\;\ulHodge^1_A(\ulM)\;=\;\bigl(H,W_\bullet H,pr(\Fq)\bigr)\,, 
\]
where $pr$ is the projection of \eqref{EqRemChangeOfADecomp} onto the factor for $P=\Var(J)$.

\medskip\noindent
(d) \es For $\Koh^1_v$ let $\tilde v\in\wt C\setminus\{\wt\infty\}$ be a closed point and let $v_1,\ldots,v_n$ be the points of $C\setminus\{\infty\}$ lying above $\tilde v$. Then $A\otimes_{\wt A}\wt A_{\tilde v}=\prod_{i=1}^n A_{v_i}$. This induces the decomposition 
\[
\Koh^1_{\tilde v,\wt A}(\ulM,\wt A_{\tilde v})\;=\;\prod_{i=1}^n\Koh^1_{v_i,A}(\ulM,A_{v_i})
\]
and similarly for dual $A$-motives and for abelian or $A$-finite Anderson $A$-modules. All comparison isomorphism are compatible with these product decompositions.
\end{remark}

 
\section{Applications}\label{Sect3}
\setcounter{equation}{0}

Theorem~\ref{ThmHodgeConjecture}, that is the Hodge conjecture, has consequences for the motivic Galois groups of (dual) $A$-motives and Anderson $A$-modules from Definitions~\ref{DefMotGp} and \ref{DefMotGpDual} and the Hodge-Pink groups of mixed $Q$-Hodge-Pink structures from Definition~\ref{DefHPGp}. For a uniformizable dual $A$-motive $\uldM$ our motivic Galois group $\Gamma_\uldM$ equals the motivic Galois group defined by Papanikolas~\cite[\S\,3.5.2]{Papanikolas}. We also explain further results known about this group.

\begin{theorem}\label{ThmHodgeConjectureE}
Let $\ulM$ (respectively $\uldM$) be a uniformizable mixed (dual) $A$-motive and let $\ulH:=\ulHodge^1(\ulM)$ (respectively $\ulH:=\ulHodge^1(\uldM)$) be the associated mixed $Q$-Hodge-Pink structure. Then the motivic Galois group $\Gamma_\ulM$ (respectively $\Gamma_\uldM$) is canonically isomorphic to the Hodge-Pink group $\Gamma_{\ulH}$.
\end{theorem}

\begin{proof}
This is a direct consequence of the canonical equivalence $\llangle\ulM\rrangle\isoto\llangle\ulH\rrangle$ (respectively $\llangle\uldM\rrangle\isoto\llangle\ulH\rrangle$) from Theorem~\ref{ThmHodgeConjecture}\ref{ThmHodgeConjectureD} (respectively Theorem~\ref{ThmDualHodgeConjecture}\ref{ThmDualHodgeConjectureD}).
\end{proof}

\begin{proposition}\label{PropPropertiesofGalGp}
The motivic Galois group $\Gamma_\ulM$ of a uniformizable mixed $A$-motive $\ulM$ is smooth and connected.
\end{proposition}

\begin{proof}
It was proved by Pink~\cite[Proposition~9.4 and 9.6]{PinkHodge} that $\Gamma_{\ulHodge^1(\ulM)}$ is connected and reduced and satisfies $\Gamma_{\Frob_{q^n}^*\ulHodge^1(\ulM)}\cong\Gamma_{\ulHodge^1(\ulM)}\times_{Q,\Frob_{q^n}}Q$ for every $n\in\BN$. So in particular, also $\Gamma_{\ulHodge^1(\ulM)}\times_{Q,\Frob_{q^n}}Q$ is reduced. Since every finite purely inseparable extension of $Q$ is contained in an extension of the form $\Frob_{q^n}\colon Q\to Q$ by \cite[Proof of Corollary~II.2.12]{Silverman}, this implies by \cite[IV$_2$, Proposition~4.6.1(d)]{EGA} that $\Gamma_{\ulHodge^1(\ulM)}$ is geometrically reduced, and hence smooth. The statement for $\Gamma_\ulM$ follows from Theorem~\ref{ThmHodgeConjectureE}.
\end{proof}

The $v$-adic cohomology realization $\Koh^1_v(\ulM,Q_v)$ of a (uniformizable) $A$-motive $\ulM$ defines an exact tensor functor \eqref{EqVAdicFiberFunctor}. If $\ulM$ is defined over a subfield $L\subset\BC$, the elements of $\Gal(L^\sep/L)$ act on $\Koh^1_v(\ulM',Q_v)$ for $\ulM'\in\llangle\ulM\rrangle$ as tensor automorphisms. If $\ulM$ is uniformizable this action is compatible with the comparison isomorphism $h_{\Betti,v}\colon\Koh^1_\Betti(\,.\,,A)\otimes_A Q_v\isoto\Koh^1_v(\,.\,,Q_v)$. This induces homomorphisms of groups
\begin{equation}\label{EqGalRep}
\Gal(L^\sep/L)\;\longto\;\Gamma_\ulM(Q_v) \qquad\text{and}\qquad \Gal(L^\sep/L)\;\longto\;\Gamma_\ulM(\BA_Q^f)\,,
\end{equation}
where $\BA_Q^f:=\wh A\otimes_A Q$ denotes the finite adeles of $Q$. Here $\wh A:=\invlim A/I$ is the projective limit where $I$ runs over the ideals of $A$ different from $(0)$. Richard Pink and his group also proved the following

\begin{theorem}\label{PinksResults}
Let $\ulE$ be a Drinfeld $A$-module and let $\ulH=(H,W_\bullet H,\Fq):=\ulHodge_1(\ulE)$ be its $Q$-Hodge-Pink structure. Then
\begin{enumerate}
\item \label{PinksResults_A}
$\Gamma_\ulH$ equals the centralizer $\Cent_{\GL(H)}\End_\BC(\ulE)$ of $\End_\BC(\ulE)$ inside $\GL(H)$.
\end{enumerate}
Assume now that $\ulE$ is defined over a finitely generated subfield $L\subset\BC$ such that $\End_L(\ulE)=\End_\BC(\ulE)$.
\begin{enumerate}
\stepcounter{enumi}
\item \label{PinksResults_B}
For every place $v$ the image of $\Gal(L^\sep/L)\to\Cent_{\GL(\Koh_{1,v}(\ulE,Q_v))}\End_\BC(\ulE)$ is $v$-adically open.
\item \label{PinksResults_C}
The image of $\Gal(L^\sep/L)\to\Gamma_\ulH(\BA_Q^f)$ is open in the adelic topology.
\end{enumerate}
\end{theorem}

\begin{proof}
\ref{PinksResults_A} was proved by Pink~\cite[Theorem~10.3]{PinkHodge} taking into account that $\End_\BC(\ulE)\otimes_AQ\cong\End_\BC(\ulH)$ by Theorems~\ref{ThmAnderson} and \ref{ThmHodgeConjecture}\ref{ThmHodgeConjectureB}, respectively Theorems~\ref{theorem:equivalence} and \ref{ThmDualHodgeConjecture}\ref{ThmDualHodgeConjectureB}.

\medskip\noindent
\ref{PinksResults_B} was proved by Pink~\cite[Theorem~0.2]{Pink97a}.

\medskip\noindent
\ref{PinksResults_C} was proved in the formulation that the image of $\Gal(L^\sep/L)\to\bigl(\Cent_{\GL(H)}\End_\BC(\ulE)\bigr)(\BA_Q^f)$ is open by Pink and R\"utsche~\cite[Theorems~0.1 and 0.2]{PR09b} after previous work by Pink, Breuer, R\"utsche and Traulsen \cite{Pink97a,BreuerPink,PT06b,PR09a}. Using \ref{PinksResults_A} yields our formulation.
\end{proof}

\begin{remark}
Note that for $\ulM=\ulM(\ulE)$ when $\ulE$ is a Drinfeld module, Theorem~\ref{PinksResults}\ref{PinksResults_B} implies that $\Gamma_\ulM=\Cent_{\GL(H)}\End_\BC(\ulM)$. This point of view is taken in \cite[Theorem~3.5.4]{ChangPapa12}. Indeed, the inclusion $\Gamma_\ulM\subset\Cent_{\GL(H)}\End_\BC(\ulM)$ is automatic by Lemma~\ref{lem:prop-deligne-220}. Since the commutation with $\End_\BC(\ulM)$ is a linear condition, $\Cent_{\GL(H)}\End_\BC(\ulM)$ is an irreducible group. Therefore, if $\Gamma_\ulM$ was a proper subgroup, the image of \eqref{EqGalRep} could not be open in $\Cent_{\GL(H)}\End_\BC(\ulM)$ in contradiction to \ref{PinksResults_B}.

Therefore, Theorem~\ref{PinksResults}\ref{PinksResults_A} is equivalent to Theorem~\ref{ThmHodgeConjectureE} for Drinfeld modules.
\end{remark}

The motivic Galois group also carries information about transcendence. For example Papanikolas~\cite[Theorem~1.7]{Papanikolas} proved the following analog of Grothendieck's period conjecture.

\begin{theorem}\label{ThmPapanikolasPeriodConj}
Let $\uldM$ be a uniformizable dual $\BF_q[t]$-motive which is defined over the algebraic closure $L\subset\BC$ of $\BF_q(\theta)$ where $\theta=\charmorph(t)$. Let $\check\Psi$ be a rigid analytic trivialization of $\uldM$ as in Lemma~\ref{LemmaDualUniformizableBF_q[t]} and let $L_\uldM$ be the field extension of $L$ generated by the entries of the matrix $\check\Psi|_{t=\theta}$. Then the transcendence degree of $L_\uldM$ over $L$ is equal to the dimension of the algebraic group $\Gamma_\uldM$.
\end{theorem}

Papanikolas~\cite[Theorem~4.5.10]{Papanikolas} also shows that $\Gamma_\uldM$ equals the Galois group $\Gamma_{\check\Psi}$ of the \emph{Frobenius difference equation} $\sdsigma^\ast\check\Psi=\check\Psi\cdot\check\Phi$ corresponding to $\uldM$. The group $\Gamma_{\check\Psi}$ can be computed explicitly in many cases. This is a powerful tool which already lead to several transcendence results. For example it was applied to determine all algebraic relations among Carlitz logarithms by Papanikolas~\cite[Theorem~1.2.6]{Papanikolas}, respectively among Carlitz (Multi-)Zeta-values and Gamma-values by Anderson, Brownawell, Chang, Mishiba, Papanikolas, Thakur and Yu~\cite{ABP,ChangYu07,ChangPapaYu10,ChangPapaThakurYu,ChangPapaYu11,MishibaMZV}, respectively among periods and logarithms of Drinfeld-modules by Chang and Papanikolas~\cite{ChangPapa11,ChangPapa12,ChangThirdKind}; see the article of Chang~\cite{Chang12} in this volume for an overview of these results. 

There is also a comparison isomorphism between the $v$-adic cohomology and the de Rham cohomology of an $A$-motive defined over an extension of $Q_v$; see \cite[Remark~4.16]{HartlKim}. Analogous to and inspired by Theorem~\ref{ThmPapanikolasPeriodConj}, Mishiba~\cite{Mishiba12} related the transcendence degree of that comparison isomorphism to the dimension of the motivic Galois group of $\ulM$ and applied this to the Carlitz $A$-motive; see \cite[Remark~4.17 and Example~4.19]{HartlKim}.

\begin{example}\label{Ex2.9GalGp}
To end this section we compute the motivic Galois group of the uniformizable mixed $\BF_q[t]$-motive $\ulM=(M,\tau_M)$ with $M=A_{\BC}^{\oplus2}$ and $\tau_M=\Phi:=\left( \begin{array}{cc}t-\theta&b\\0&(t-\theta)^3 \end{array}\right)$ from Example~\ref{Ex2.9} and the associated dual $\BF_q[t]$-motive $\uldM=\uldM(\ulM)$ from Example~\ref{ExDual2.9}. Since $\ulM$ is an extension 
\begin{equation}\label{EqSeqMsplitsOrNot}
0\,\longto\,\UOne(1)\,\longto\,\ulM\,\longto\,\UOne(3)\,\longto\,0\,,
\end{equation}
the representation $\rho$ of $\Gamma_\ulM$ on $\Hodge^1(\ulM)$ can be written in upper diagonal matrix form such that the diagonal entries are representations corresponding to the simple constituents of $\ulM$. Therefore, $\Gamma_\ulM$ is a subgroup of $\bigl\{\left( \begin{smallmatrix} u&*\\0&\;u^3 \end{smallmatrix}\right)\bigr\}\subset\GL_{2,Q}$. There are now two cases, according to whether the extension \eqref{EqSeqMsplitsOrNot} splits or not. We will discuss a criterion for the splitting in Example~\ref{EqDiagramExMixedWeights} below.

If the extension splits, then $\llangle\ulM\rrangle=\llangle\UOne(1)\rrangle$ and $\Gamma_\ulM\cong\Gamma_{\UOne(1)}=\BG_{m,Q}$ by Example~\ref{ExGalGpOfCarlitz}. In this case $*=0$ and the isomorphism is given by $\BG_{m,Q}\isoto\Gamma_\ulM$, $u\mapsto\text{diag}(u,u^3)$.

Conversely, if $*=0$ the inclusion $\llangle\UOne(1)\rrangle\subset\llangle\ulM\rrangle$ is an equivalence of categories by Theorem~\ref{theorem:theorem-deligne-211}\ref{theorem:theorem-deligne-211_B}, because the corresponding group homomorphism $\Gamma_\ulM\isoto\BG_{m,Q}$, $\left( \begin{smallmatrix} u&0\\0&\;u^3 \end{smallmatrix}\right)\mapsto u$ is an isomorphism. This implies that \eqref{EqSeqMsplitsOrNot} splits. We conclude that if \eqref{EqSeqMsplitsOrNot} does not split, then $\Gamma_\ulM$ is the semi-direct product $\BG_{a,Q}\rtimes\BG_{m,Q}$, where $\BG_{m,Q}$ acts on $\BG_{a,Q}$ by multiplication with the character $u\mapsto u^2$.
\end{example}

 
\section{\texorpdfstring{$\sigma$}{sigma}-bundles}\label{Sect4}
\setcounter{equation}{0}

In this section we give the proof of Theorems~\ref{ThmHodgeConjecture} and \ref{ThmDualHodgeConjecture}, which uses in particular the classification of $\sigma$-bundles associated with uniformizable mixed (dual) $A$-motives.

\subsection{Definition of \texorpdfstring{$\sigma$}{sigma}-bundles} \label{SectDefSigmaBd}

Recall the punctured open unit disc $\PDisc=\{0<|z|<1\}$ around $\infty$ introduced at the beginning of Section~\ref{SectAMotUniformizability} and set 
\[
\dotCO\;:=\;\Gamma(\PDisc,\CO_{\PDisc})\;=\;\bigl\{\,\sum_{i\in\BZ}b_i z^i\colon b_i\in\BC, \;\lim_{i\to\pm\infty}|b_i|\,|\zeta|^{si}= 0\text{ for all }s>0\,\bigr\}\,. 
\]
This disc can be exhausted by the closed annuli $\{|\zeta|^s\le|z|\le|\zeta|^{s'}\}$ for $s,s'\in\BQ$ with $0<s'\le s$. Hence, $\PDisc$ is a quasi-Stein space in the sense of Kiehl~\cite[\S2]{KiehlAB}. In particular, the functor $\CF\mapsto\Gamma(\PDisc,\CF)$ is an equivalence between the category of locally free coherent sheaves on $\PDisc$ and the category of finite projective $\dotCO$-modules; see Gruson~\cite[Chapter V, Theorem~1 and Remark on p.~85]{Gruson68}. Note further, that the rings 
\begin{eqnarray*}
\BC\langle\tfrac{z}{\zeta^{s'}},\tfrac{\zeta^s}{z}\rangle & := & \Gamma\bigl(\{|\zeta|^s\le|z|\le|\zeta|^{s'}\}\,,\,\CO_{\{|\zeta|^s\le|z|\le|\zeta|^{s'}\}}\bigr)\\[2mm]
& = & \bigl\{\sum_{i\in\BZ}b_i z^i\colon b_i\in\BC,\;\lim_{i\to\pm\infty}|b_i|\,|\zeta|^{s''i}=0\text{ for all }s'\le s''\le s\,\bigr\} \qquad \text{and}\\[2mm]
\BC\langle\tfrac{z}{\zeta^s}\rangle & := & \Gamma\bigl(\{|z|\le|\zeta|^s\}\,,\,\CO_{\{|z|\le|\zeta|^s\}}\bigr)\\[2mm]
& = & \bigl\{\sum_{i\in\BN_0}b_i z^i\colon b_i\in\BC,\;\lim_{i\to+\infty}|b_i|\,|\zeta|^{si}=0\,\bigr\}
\end{eqnarray*}
are principal ideal domains by \cite[Proposition~4]{Lazard}. 

\begin{definition}
A \emph{$\sigma$-bundle (over $\dotCO$)} is a pair $\ulCF=(\CF,\tau_\CF)$ consisting of a finite projective $\dotCO$-module $\CF$ (or, equivalently, locally free coherent sheaf on $\PDisc$) together with an isomorphism $\tau_\CF\colon \sigma^\ast\CF\isoto\CF$. We define the \emph{rank} of $\ulCF$ as $\rk\ulCF:=\rk_\dotCO\CF$.

A \emph{homomorphism} $f\colon(\CF,\tau_\CF)\to(\CG,\tau_\CG)$ between $\sigma$-bundles is a homomorphism $f\colon\CF\to\CG$ of $\dotCO$-modules which satisfies $\tau_\CF\circ\sigma^*f=f\circ\tau_\CG$.

The \emph{$\tau$-invariants} of $(\CF,\tau_\CF)$ are defined as $\ulCF^\tau:=\{\,f\in\CF:\tau_\CF(\sigma^\ast f)=f\,\}$.
\end{definition}

If follows from Theorem~\ref{ThmHartlPink1}\ref{ThmHartlPink1a} below that the module $\CF$ underlying a $\sigma$-bundle is actually free.

\begin{example}\label{Ex4.2}
(a) The trivial $\sigma$-bundle is $(\CF,\tau_\CF)=(\dotCO,\id_{\dotCO})$. Its $\tau$-invariants are $(\dotCO,\id_{\dotCO})^\tau=\{\,f\in\dotCO:\sigma^\ast(f)=f\,\}=\BF_q\dpl z\dpr=Q_\infty$, because $f=\sum_{i\in\BZ}b_iz^i=\sigma^*(f)=\sum_{i\in\BZ}b_i^qz^i$ implies $b_i=b_i^q$, whence $b_i\in\BF_q$, and $\lim\limits_{i\to\pm\infty}|b_i|\,|\zeta|^{si}=0$ implies that there is an integer $n$ with $b_i=0$ for all $i<n$.

\medskip\noindent
(b) More generally, for relatively prime integers $d,r$ with $r>0$ we let $\ulCF_{d,r}$ be the $\sigma$-bundle consisting of $\CF_{d,r}=\dotCO^{\oplus r}$ with 
\[
\tau_{\CF_{d,r}}:= \left( \raisebox{6.2ex}{$
\xymatrix @C=0.3pc @R=0.3pc {
0 \ar@{.}[ddd]\ar@{.}[drdrdrdr] & 1\ar@{.}[drdrdr] & 0\ar@{.}[rr]\ar@{.}[drdr] & & 0 \ar@{.}[dd]\\
& & & & \\
& & & & 0 \\
0 & & & & 1\\
z^{-d} & 0 \ar@{.}[rrr]  & & & 0\\
}$}
\right).
\]

\medskip\noindent
(c) We exhibit the following $\tau$-invariants of $\ulCF_{_{\SC 1,1}}=(\dotCO,z^{-1})$. Let $\alpha\in\BC$ with $0<|\alpha|<1$. Then the product $\ell_\alpha^{\SSC -}:=\prod\limits_{i\in\BN_0}(1-{\TS\frac{\alpha^{q^i}}{z}})\in\dotCO$ has simple zeroes exactly at $z=\alpha^{q^i}$ for $i\in\BN_0$ and satisfies $(1-\frac{\alpha}{z})\sigma^\ast(\ell_\alpha^{\SSC -})=\ell_\alpha^{\SSC -}$. To obtain a non-zero $\ell_\alpha=\ell_\alpha^{\SSC +}\cdot \ell_\alpha^{\SSC -}\in\ulCF_{1,1}{}^\tau$ satisfying $z^{-1}\sigma^\ast(\ell_\alpha)=\ell_\alpha$ we need a function $\ell_\alpha^{\SSC +}=\sum\limits_{i\ge0}b_i z^i\in\dotCO$ with $b_0\ne0$ satisfying $\sigma^\ast(\ell_\alpha^{\SSC +})=(z-\alpha)\ell_\alpha^{\SSC +}$. The latter amounts to the equations $b_0^{q-1}=-\alpha$ and $b_i^q=b_{i-1}-\alpha b_i$ for $i>0$. Since $\BC$ is algebraically closed these equations can be solved recursively, yielding an element $\ell_\alpha\in\ulCF_{1,1}{}^\tau$, which due to $z^{-1}\sigma^\ast(\ell_\alpha)=\ell_\alpha$ has simple zeroes exactly at $z=\alpha^{q^i}$ for all $i\in\BZ$. Note that $\ell_\alpha$ is not canonically defined but depends on the chosen solutions $b_i$. A different choice replaces $\ell_\alpha^{\SSC +}$ by $\tilde\ell_\alpha^{\SSC +}=u\cdot\ell_\alpha^{\SSC +}$ for $u\in\BF_q\dbl z\dbr\mal$ because $u=\tilde\ell_\alpha^{\SSC +}/\ell_\alpha^{\SSC +}\in \BC\dbl z\dbr\mal$ satisfies $\sigma^\ast(u)=u$. One can prove that in fact, all $\tau$-invariants in $\ulCF_{1,1}{}^\tau$ are obtained in this way; see \cite[Theorem~5.4]{HartlPink1}.

\medskip\noindent
(d) On the other hand $\ulCF_{d,r}{}^\tau=(0)$ for $d<0$. Indeed, since $(\tau_{\CF_{d,r}})^r=z^{-d}\Id_r$, any such $\tau$-invariant $(f_1,\ldots,f_r)^T$ satisfies $f_j=z^{-d}\sigma^{r*}(f_j)$ for all $j$. If we write $f_j=\sum_{i\in\BZ}b_iz^i$ with $b_i\in\BC$ this implies $b_i=b_{i+d}^{q^r}=b_{i+kd}^{q^{kr}}$ for all $i,k\in\BZ$. As $|b_{i+kd}|\to0$ for $(i+kd)\to-\infty$, that is for $k\to +\infty$, this implies $b_i=0$ for all $i$.
\end{example}

The structure theory of $\sigma$-bundles was developed in \cite{HartlPink1}. 

\begin{theorem}\label{ThmHartlPink1}
\begin{enumerate}
\item \label{ThmHartlPink1a}
Any $\sigma$-bundle $\ulCF$ is isomorphic to $\bigoplus_i\ulCF_{d_i,r_i}$ for pairs of relatively prime integers $d_i,r_i$ with $r_i>0$, which are uniquely determined by $\ulCF$ up to permutation. They satisfy $\rk\ulCF=\sum_i r_i$ and we define the \emph{degree} of $\ulCF$ as $\deg\ulCF:=\sum_i d_i$.
\item  \label{ThmHartlPink1b}
There is a non-zero morphism $\ulCF_{d',r'}\to\ulCF_{d,r}$ if and only if $\frac{d'}{r'}\le\frac{d}{r}$.
\item  \label{ThmHartlPink1c}
Any $\sigma$-sub-bundle $\ulCF'\subset\ulCF_{d,r}{}^{\oplus n}$ satisfies $\deg\ulCF'\le\frac{d}{r}\cdot\rk\ulCF'$.
\item  \label{ThmHartlPink1d}
If $\ulCF'\subset\ulCF$ is an inclusion of $\sigma$-bundles with $\rk\ulCF'=\rk\ulCF=r$, then for any $s>0$ we have
\[
\deg\ulCF-\deg\ulCF'\,=\,\dim_\BC(\CF/\CF')|_{\{\,|\zeta|^{sq}<|z|\le|\zeta|^s\}}\,.
\]
\end{enumerate}
\end{theorem}

\begin{proof}
Statements \ref{ThmHartlPink1a} \ref{ThmHartlPink1b} and \ref{ThmHartlPink1c} are \cite[Theorem~11.1, Proposition~8.5, and Proposition~7.6, respectively]{HartlPink1}, but \ref{ThmHartlPink1c} also easily follows from \ref{ThmHartlPink1a} and \ref{ThmHartlPink1b}. Namely, $\ulCF'\cong\bigoplus_i\ulCF_{d_i,r_i}$ by \ref{ThmHartlPink1a} with $\frac{d_i}{r_i}\le\frac{d}{r}$ by \ref{ThmHartlPink1b} yields \ref{ThmHartlPink1c}.

\medskip\noindent
\ref{ThmHartlPink1d} We use the results of Lazard~\cite{Lazard} and normalize his valuation $v$ such that $v(\zeta)=1$. Then his ring $L_\BC[s,qs[$ is the ring of rigid analytic functions on $\{\,|\zeta|^{sq}<|z|\le|\zeta|^s\}$ and his ring $L_\BC[s,qs]$ is our $\BC\langle\tfrac{z}{\zeta^s},\tfrac{\zeta^{qs}}{z}\rangle$. Since the latter is a principal ideal domain we may choose bases of $\CF'\otimes_\dotCO\BC\langle\tfrac{z}{\zeta^s},\tfrac{\zeta^{qs}}{z}\rangle$ and $\CF\otimes_\dotCO\BC\langle\tfrac{z}{\zeta^s},\tfrac{\zeta^{qs}}{z}\rangle$ and write the inclusion $\ulCF'\subset\ulCF$ with respect to these bases as a matrix $T$. By the elementary divisor theorem there are matrices $U,V\in\GL_r\bigl(\BC\langle\tfrac{z}{\zeta^s},\tfrac{\zeta^{qs}}{z}\rangle\bigr)$ such that $UTV=\text{diag}(f_1,\ldots,f_r)$ is a diagonal matrix with diagonal entries $f_i\in\BC\langle\tfrac{z}{\zeta^s},\tfrac{\zeta^{qs}}{z}\rangle$. Changing $U$ we can multiply the $f_i$ with units and by \cite[Proposition~4]{Lazard} we may assume that they are monic polynomials in $\BC[z]$, all of whose zeroes $\alpha$ satisfy $|\zeta|^{qs}\le|\alpha|\le|\zeta|^s$. Considering those zeroes $\alpha$ of all the $f_i$ which satisfy $|\zeta|^{qs}<|\alpha|$ they even satisfy $|\zeta|^{s'}\le|\alpha|\le|\zeta|^s$ for an $s'$ with $s\le s'<qs$. We write $f_i=f_i'\cdot\tilde f_i$ with $f_i',\tilde f_i$ monic such that all zeros $\alpha$ of $f_i'$, respectively of $\tilde f_i$, satisfy $|\zeta|^{s'}\le|\alpha|\le|\zeta|^s$, respectively $|\alpha|=|\zeta|^{qs}$. Then $\tilde f_i$ is a unit in $L_\BC[s,qs[$ and 
\[
(\CF/\CF')|_{\{\,|\zeta|^{sq}<|z|\le|\zeta|^s\}}\;\cong\;\prod_{i=1}^r L_\BC[s,qs[\;/(f_i')\;=\;\prod_{i=1}^r\BC[z]/(f_i')\,,
\]
where the last equality follows by Euclidean division in $L_\BC[s,qs[$ in the style of \cite[Lemma~2]{Lazard}. This implies 
\[
\dim_\BC(\CF/\CF')|_{\{\,|\zeta|^{sq}<|z|\le|\zeta|^s\}}\;=\;\sum_{i=1}^r\deg_z f_i'\;=\;\dim_\BC \BC[z]/(f_1'\cdots f_r')\;=\;\dim_\BC (L_\BC[s,qs[)/(\det T)\,,
\]
because $\det T$ differs from $f_1'\cdots f_r'$ by a unit in $L_\BC[s,qs[$. 

We now compute $\det T$ in a different way. Namely, by Theorem~\ref{ThmHartlPink1}\ref{ThmHartlPink1a} there are isomorphisms $\ulCF\cong\bigoplus_i \ulCF_{d_i,r_i}$ and $\ulCF'\cong\bigoplus_j \ulCF_{d'_j,r'_j}$. These provide $\dotCO$-bases of $\ulCF$ and $\ulCF'$ with respect to which the inclusion $\ulCF'\subset\ulCF$ is given by a matrix $S$. Then $S\cdot\tau_{\CF'}=\tau_\CF\cdot\sigma^*S$ implies $\det S\cdot\pm z^{-\deg\ulCF'}=\pm z^{-\deg\ulCF}\cdot\sigma^*(\det S)$, and hence, $f:=\sqrt[q-1]{\pm1}\cdot\det S=z^{-e}\cdot\sigma^*(f)$ with $e:=\deg\ulCF-\deg\ulCF'$. From \cite[Proposition~1.4.4]{HartlPSp} it follows that $f=g\cdot\ell_{\alpha_1}\cdot\ldots\cdot\ell_{\alpha_e}$ with $g\in\BF_q\dpl z\dpr\mal$ and $|\zeta|^{qs}<|\alpha_i|\le|\zeta|^s$. Since $\ell_{\alpha_i}(z-\alpha_i)^{-1}$ is a unit in $L_\BC[s,qs[$, and the matrices $T$ and $S$ differ by a base change over $L_\BC[s,qs[$, we conclude that
\[
\dim_\BC L_\BC[s,qs[\;/(\det T)\;=\;\dim_\BC L_\BC[s,qs[\;/(f)\;=\;\dim_\BC L_\BC[s,qs[\;/\prod_{i=1}^e(z-\alpha_i)\;=\;e.
\]
The theorem follows.
\end{proof}

\subsection{The pair of \texorpdfstring{$\sigma$}{o}-bundles associated with an $A$-motive}
Consider a uniformizable $A$-motive $\ulM$ over $\BC$. Then $\ulCE(\ulM):=(\CE(\ulM),\tau_\CE):=\Lambda(\ulM)\otimes_A\ulCF_{0,1}$ is a $\sigma$-bundle with $\CE(\ulM):=\Lambda(\ulM)\otimes_A\dotCO$ and $\tau_\CE=\id$. By Proposition~\ref{PropMaphM}, $\CE(\ulM)$ coincides via $h_\ulM$ with $M\otimes_{A_\BC}\dotCO$ on $\PDisc\setminus\bigcup_{i\in\BN_0}\{z=\zeta^{q^i}\}$ and via $\sigma^*h_\ulM$ it coincides with $\sigma^*M\otimes_{A_\BC}\dotCO$ on $\PDisc\setminus\bigcup_{i>0}\{z=\zeta^{q^i}\}$. So it can be obtained as a modification of $M\otimes_{A_\BC}\dotCO$ at all places $z=\zeta^{q^i}$ for $i\ge0$.

But $\ulM$ also gives rise to a second $\sigma$-bundle as follows. The isomorphism $\tau_M$ is an isomorphism between $\sigma^\ast M$ and $M$ outside $z=\zeta$. So one can modify $M\otimes_{A_\BC}\dotCO$ at $z=\zeta^{q^i}$ for $i<0$ to obtain a $\sigma$-bundle $\ulCF(\ulM)=(\CF(\ulM),\tau_\CF)$ with 
\begin{eqnarray}\label{Eq4.1}
\CF(\ulM)\es:=&\bigl\{\,f\in M\otimes_{A_\BC}\dotCO[\ell_\zeta^{-1}]:&\tau_M^i(\sigma^{i\ast}f)\in M\otimes_{A_\BC}\BC\dbl z-\zeta\dbr\text{ for all }i\in\BZ\,\bigr\}\\[2mm]
=&\bigl\{\,f\in \CE(\ulM)[\ell_\zeta^{-1}]:&\tau_M^i(\sigma^{i\ast}f)\in M\otimes_{A_\BC}\BC\dbl z-\zeta\dbr\text{ for all }i\in\BZ\,\bigr\}\nonumber
\end{eqnarray}
and $\tau_\CF=\tau_M\otimes\id$. To see that this is indeed a $\sigma$-bundle, we view it as a sheaf. Then 
\begin{eqnarray*}
\Gamma\bigl(\{|\zeta|^s\le|z|\le|\zeta|^{s'}\}\,,\,\CF(\ulM)\bigr)&=& \TS\bigl\{\,f\in\Lambda(\ulM)\otimes_A\BC\langle\frac{z}{\zeta^{s'}},\frac{\zeta^s}{z}\rangle[\ell_\zeta^{-1}]:\es\tau_M^i(\sigma^{i\ast}f)\in M\otimes_{A_\BC}\BC\dbl z-\zeta\dbr\\[1mm]
&&\qquad\text{ for all }i\in\BZ\text{ with }|\zeta|^{q^is}\le|\zeta|\le|\zeta|^{q^is'}\,\bigr\}\,.
\end{eqnarray*}
The latter is a finite free module over the principal ideal domain $\BC\langle\frac{z}{\zeta^{s'}},\frac{\zeta^s}{z}\rangle$, because by Proposition~\ref{PropMaphM} it is contained in the free module $\ell_\zeta^{-d}\cdot\Lambda(\ulM)\otimes_A\BC\langle\frac{z}{\zeta^{s'}},\frac{\zeta^s}{z}\rangle$ if $J^d\cdot M\subset\tau_M(\sigma^*M)$.

Again by Proposition~\ref{PropMaphM}, $\CF(\ulM)$ coincides via $h_\ulM$ with $M\otimes_{A_\BC}\dotCO$ on $\PDisc\setminus\bigcup_{i<0}\{z=\zeta^{q^i}\}$ and via $\sigma^*h_\ulM$ it coincides with $\sigma^*M\otimes_{A_\BC}\dotCO$ on $\PDisc\setminus\bigcup_{i\le0}\{z=\zeta^{q^i}\}$. 

\begin{definition}\label{DefPairOfSigmaBundles}
The pair $(\ulCF(\ulM),\ulCE(\ulM))$ constructed above is called the \emph{pair of $\sigma$-bundles} associated with the uniformizable $A$-motive $\ulM$.
\end{definition}

Assume that $\ulM$ is effective with $\tau_M(\sigma^*M)\subsetneq M$. Then we visualize these $\sigma$-bundles over $\dotCO$ by the following diagram, in which the thick lines represent sheaves on $\PDisc$:
\begin{equation}\label{EqDiagSigmaBdOfM}
\newpsobject{showgrid}{psgrid}{subgriddiv=1,griddots=10,gridlabels=0pt}
 \psset{unit=3.1cm} 
 \begin{pspicture}(0,-0.4)(3.7,0.85) 
   \rput(0.4,0.61){$\CF(\ulM)$}
   \psplot[plotstyle=curve,plotpoints=200,linewidth=1pt,linecolor=black]{0.1}%
                {3.7}{0.7}
   \rput(3.4,0.08){$\CE(\ulM)$}
   \psplot[plotstyle=curve,plotpoints=200,linewidth=1pt,linecolor=black]{0.1}%
                {3.7}{6 7 div 0.7 mul 0.6 sub}
   \rput(1.62,0.3){$M\!\otimes_{A_\BC}\dotCO$}
   \psplot[plotstyle=curve,plotpoints=200,linewidth=1pt,linecolor=black]{0.45}{3.7}%
                {x 0.45 sub 16 exp 6 mul x 0.45 sub 16 exp 7 mul 1 add div 0.7 mul 0.2 3 div add}
   \psline[linewidth=1pt](0.1,0.066667)(0.45,0.066667)
   \rput(2.52,0.3){$\sigma^* M\!\otimes_{A_\BC}\dotCO$}
   \psplot[plotstyle=curve,plotpoints=200,linewidth=1pt,linecolor=black]{1.2}{3.7}%
                {x 1.3 sub 16 exp 6 mul x 1.3 sub 16 exp 7 mul 1 add div 0.7 mul 0.1 3 div add}
   \psline[linewidth=1pt](0.1,0.033333)(1.3,0.033333)
   %
   %
   \rput(3.85,-0.15){$\PDisc$}
   \psline[linewidth=0.5pt](0.1,-0.15)(3.7,-0.15)
   \rput(0.5,-0.3){$\dots$}
   \rput(1.13,-0.3){$z=\zeta^{1/q}$}
   \psline[linewidth=0.5pt](1.05,-0.25)(1.05,0.75)
   \rput(1.95,-0.3){$z=\zeta$}
   \psline[linewidth=0.5pt](1.95,-0.25)(1.95,0.75)
   \rput(2.88,-0.3){$z=\zeta^{q}$}  
   \psline[linewidth=0.5pt](2.85,-0.25)(2.85,0.75)
   \rput(3.4,-0.3){$\dots$}  
 \end{pspicture}
\end{equation}
\noindent
Sheaves drawn higher contain the ones drawn below. All sheaves coincide outside $\bigcup_{i\in\BZ}\{z=\zeta^{q^i}\}$. At those points in $\bigcup_{i\in\BZ}\{z=\zeta^{q^i}\}$ where two sheaves are drawn at almost the same height, they also coincide. Indeed, $\CE(\ulM)$ coincides via $h_\ulM$ with $M\otimes_{A_\BC}\dotCO$ outside $\bigcup_{i\in\BN_0}\{z=\zeta^{q^i}\}$ and via $\sigma^*h_\ulM$ with $\sigma^\ast M\otimes_{A_\BC}\dotCO$ outside $\bigcup_{i\in\BN_{>0}}\{z=\zeta^{q^i}\}$ and is contained in these modules by Proposition~\ref{PropMaphM}. Via $\tau_M$ also $M$ contains $\sigma^\ast M$ and differs from it only at $z=\zeta$. Finally, one sees that $M\otimes_{A_\BC}\dotCO$ is via $h_\ulM^{-1}$ contained in $\CF(\ulM)$ and they coincide outside $\bigcup_{i<0}\{z=\zeta^{q^i}\}$. Namely, the condition $\tau_M^i(\sigma^{i\ast}f)\in M\otimes_{A_\BC}\BC\dbl z-\zeta\dbr$ for $i<0$ is equivalent to (setting $j:=-i>0$)
\[
f\,\in\,\tau_M^j\bigl(\sigma^{j\ast}(M\otimes_{A_\BC}\BC\dbl z-\zeta\dbr)\bigr)\,=\,\tau_M^j(\sigma^{j\ast}M)\otimes_{A_\BC}\BC\dbl z-\zeta^{q^j}\dbr\,=\,M\otimes_{A_\BC}\BC\dbl z-\zeta^{q^j}\dbr\,.
\]
In particular, for $f\in M\otimes_{A_\BC}\dotCO$ the condition is satisfied for $i<0$ and obviously for $i\ge0$ proving $M\otimes_{A_\BC}\dotCO\subset\CF(\ulM)$. In terms of Definition~\ref{Def2.6} this also shows $\CE(\ulM)\otimes_\dotCO\BC\dbl z-\zeta\dbr=\Hodge^1(\ulM)\otimes_Q\BC\dbl z-\zeta\dbr=\Fp$ and $\CF(\ulM)\otimes_\dotCO\BC\dbl z-\zeta\dbr=\Fq$.

\begin{proposition}\label{Prop4.4}
Let $\ulM$ be a uniformizable mixed $A$-motive, $(\ulCF(\ulM),\ulCE(\ulM))$ the associated pair of $\sigma$-bundles and let $\ulHodge^1(\ulM)=(H,W_\bullet H,\Fq)$ be its mixed Hodge-Pink structure.
\begin{enumerate}
 \item \label{Prop4.4_a}
The $\tau$-invariants of $\ulCE(\ulM)$ are
\[
\ulCE(\ulM)^\tau\;=\;\bigl(\Lambda(\ulM)\otimes_A\ulCF_{0,1}\bigr)^\tau\;=\;\Lambda(\ulM)\otimes_A Q_\infty\;=\;H\otimes_Q Q_\infty\;=\;H_\infty\,.
\]
\item \label{Prop4.4_b}
We have $\deg\ulCE(\ulM)\;=\;0$ and $\deg\ulCF(\ulM)\;=\;\dim\ulM\;=\;\deg_\Fq\ulHodge^1(\ulM)$.

\item \label{Prop4.4_c}
If $\ulM$ is pure of weight $\mu=\frac{k}{l}$ with $(k,l)=1$, then $\ulCF(\ulM)\cong\ulCF_{k,l}^{\oplus\rk\ulM/l}$. In particular, 
\[
\deg_\Fq\ulHodge^1(\ulM)\;=\;\deg\ulCF(\ulM)\;=\;\frac{k\cdot\rk\ulM}{l}\;=\;\mu\cdot\rk\ulHodge^1(\ulM)\;=\;\deg^W\ulHodge^1(\ulM)\,.
\]
\item \label{Prop4.4_d} If $\ulM$ is mixed, then also $\deg_\Fq\ulHodge^1(\ulM)=\deg^W\ulHodge^1(\ulM)$ and $\deg_\Fq W_\mu\ulHodge^1(\ulM)=\deg^W W_\mu\ulHodge^1(\ulM)$ for all $\mu$.
\end{enumerate}
\end{proposition}

\begin{proof}
\ref{Prop4.4_a} is obvious from the construction of $\ulCE(\ulM)$. 

\medskip\noindent
\ref{Prop4.4_b} Since $\ulCE(\ulM)=H\otimes_Q\ulCF_{0,1}\cong\ulCF_{0,1}{}^{\oplus\dim_QH}$ it has degree zero. There is an integer $d\in\BN_0$ with $\tau_M(J^d\cdot\sigma^*M)\subset M$. It follows that $\tau_\CE(\ell_\zeta^d\cdot\CE)\subset\CF$. We consider the $\sigma$-bundle $\ulCE':=H\otimes_Q\ulCF_{-d,1}$ and the inclusions $\ulCE'\into\ulCE(\ulM)$, $f\mapsto\ell_\zeta^d\cdot f$ and $\ulCE'\into\ulCF(\ulM)$, $f\mapsto\tau_\CE(\ell_\zeta^d\cdot f)$. If $r=\rk\ulM=\dim_QH$ then $\deg\ulCE'=-dr$. By Theorem~\ref{ThmHartlPink1}\ref{ThmHartlPink1d} and Remark~\ref{RemPolygons}(a) we compute
\[
\deg\ulCF(\ulM)\;=\;\deg\ulCF(\ulM)-\deg\ulCE'-dr\;=\;\dim_\BC\Fq/(z-\zeta)^d\Fp-dr\;=\;\deg_\Fq\ulHodge^1(\ulM)
\]
because on the annulus $\{|\zeta|^q<|z|\le|\zeta|\}$ the quotient $\CF(\ulM)/\CE'$ equals $\Fq/(z-\zeta)^d\Fp$. The equality $\dim\ulM=\deg_\Fq\ulHodge^1(\ulM)$ follows directly from the definitions.

\medskip\noindent
\ref{Prop4.4_c} To prove that $\ulCF(\ulM)\cong\ulCF_{k,l}^{\oplus\rk\ulM/l}$, recall from Proposition~\ref{PropWeights}\ref{PropWeights_c} that $M$ extends to a locally free sheaf $\olM$ on $C_\BC$ on which $z^k\tau_M^l$ is an isomorphism locally at $\infty$. We consider the ring of rigid analytic functions $\BC\langle\frac{z}{\zeta}\rangle$ on the closed disc $\{|z|\le|\zeta|\}\subset\FC_\BC$ of radius $|\zeta|$ around $\infty$. Then we obtain an isomorphism $z^k\tau_M^l\colon \sigma^{l\ast}\bigl(\olM\otimes_{\CO_{C_\BC}}\BC\langle\frac{z}{\zeta}\rangle\bigr)\;=\;(\sigma^{l\ast}\olM)\otimes_{\CO_{C_\BC}}\BC\langle\frac{z}{\zeta^{q^l}}\rangle\;\isoto\;\olM\otimes_{\CO_{C_\BC}}\BC\langle\frac{z}{\zeta^{q^l}}\rangle$, because $z$ has no other poles or zeroes besides $\infty$ on the disc $\Disc$. Since $\BC\langle\frac{z}{\zeta}\rangle$ is a principal ideal domain, we can choose a basis $\{e_1,\ldots,e_r\}$ of $\olM\otimes_{\CO_{C_\BC}}\BC\langle\frac{z}{\zeta}\rangle$ with respect to which $\tau_M$ is given by a matrix $\Phi\in\BC\langle\frac{z}{\zeta^q}\rangle[z^{-1}]^{r\times r}$ and $z^k\tau_M^l$ by the matrix $U:=z^k\cdot \Phi\cdot\sigma^\ast(\Phi)\cdot\ldots\cdot\sigma^{(l-1)\ast}(\Phi)\in\GL_r\bigl(\BC\langle\frac{z}{\zeta^{q^l}}\rangle\bigr)$. We will prove the following:

\medskip\noindent
{\it Claim.} There is a matrix $S=\sum_{i=0}^\infty S_iz^i\in\GL_r\bigl(\BC\langle\frac{z}{\zeta}\rangle\bigr)$ with $U\cdot\sigma^{l\ast}(S)=S$.

\medskip\noindent
The equation is equivalent to $\sigma^{l\ast}(S)=U^{-1}S$. Writing $U^{-1}=\sum_{i=0}^\infty U_iz^i$ with $U_0\in\GL_r(\BC)$ we can solve the equation $\sigma^{l\ast}(S_0)=U_0S_0$ for $S_0\in\GL_r(\BC)$ by Lang's theorem~\cite[Corollary on p.~557]{Lang57} and then recursively solve the system of Artin-Schreier equations 
\[
\sigma^{l\ast}(S_0^{-1}S_j)-S_0^{-1}S_j\;=\;\sum_{i=0}^{j-1} S_0^{-1}U_0^{-1}U_{j-i} S_i
\]
for $S_j\in\BC^{r\times r}$. To compute the radius of convergence of $S$, let $c\ge1$ be a constant with $|U_i \zeta^{q^li}|\leq c$ for all $i$ where $|U_i \zeta^{q^li}|$ denotes the maximal absolute value of the entries of the matrix $U_i \zeta^{q^li}$. Then
\[
\sigma^{l\ast}(S_j \zeta^{j})\; = \; \sum_{i=0}^j \bigl(U_{j-i} \zeta^{q^l(j-i)}\bigr)(S_i \zeta^{i}) \,\zeta^{i(q^l-1)}\,.
\]
This implies the estimate $|S_j\zeta^{j}|^{q^l}\,=\,|\sigma^{l\ast}\bigl(S_j \zeta^{j}\bigr)|\,\le\,c\cdot \max\{\,|S_i\zeta^{i}|:\,0\leq i\leq j\,\}$, from which induction yields $|S_j\zeta^{j}|\leq c^{1/(q^l-1)}$ for all $j\ge0$. 
In particular $S\in \GL_r\bigl(\BC\langle\frac{z}{\zeta^{q^l}}\rangle\bigr)$. But now the equation $\sigma^{l\ast}(S)=U^{-1}S$ shows that $\sigma^{l\ast}(S)\in\GL_r\bigl(\BC\langle\frac{z}{\zeta^{q^l}}\rangle\bigr)$, hence, $S\in\GL_r\bigl(\BC\langle\frac{z}{\zeta}\rangle\bigr)$ proving the claim.

A consequence of the claim is that we may use $S$ to produce a new basis of $\olM\otimes_{\CO_{C_\BC}}\BC\langle\frac{z}{\zeta}\rangle$ with respect to which $z^k\tau_M^l=\Id_r$ is the identity matrix. Thus also $\ulM\otimes_{A_\BC}\CO_{\{0<|z|\le|\zeta|\}}\;=\;\ulCF(\ulM)\otimes_{\dotCO}\CO_{\{0<|z|\le|\zeta|\}}$ has a basis with respect to which $\tau_\CF^l=z^{-k}$. By Theorem~\ref{ThmHartlPink1}\ref{ThmHartlPink1a} this is only possible if $\ulCF(\ulM)\cong\ulCF_{k,l}^{\oplus\rk\ulM/l}$. In particular, $\deg\ulCF(\ulM)=k\cdot\rk\ulM/l$. This is what we wanted to prove.

\medskip\noindent
\ref{Prop4.4_d} If $\ulM$ is mixed, the construction of $\ulCF(\ulM)$ applies to $W_{\mu\,}\ulM$ and $\Gr^W_{\mu\,}\ulM$ to yield an exact sequence
\begin{equation}\label{EqExactWeightSeq}
0\es\longto\es \DS\bigcup_{\mu'<\mu} \ulCF(W_{\mu'}\ulM)\es\longto\es \ulCF(W_{\mu\,}\ulM)\es\longto\es \ulCF(\Gr^W_{\mu\,}\ulM)\es\longto\es 0
\end{equation}
of $\sigma$-bundles. Indeed, the restriction of \eqref{EqExactWeightSeq} to $\{0<|z|\le|\zeta|\}$ equals the tensor product over $A_\BC$ of $0\to\bigcup_{\mu'<\mu}W_{\mu'}\ulM\to W_{\mu\,}\ulM\to\Gr^W_{\mu\,}\ulM\to0$ with $\CO_{\{0<|z|\le|\zeta|\}}$. Therefore, it is exact because $\Gr^W_{\mu\,}\ulM$ is locally free over $A_\BC$. Since $\sigma^{-1}(\{0<|z|\le|\zeta|\})=\{0<|z|\le|\zeta|^{1/q}\}$ successive application of the isomorphism $\tau_\CF^{-1}$ yields exactness of \eqref{EqExactWeightSeq} on all of $\PDisc$. In particular, $\CF(W_{\mu\,}\ulM)$ equals the intersection of $\CE(W_{\mu\,}\ulM)[\ell_\zeta^{-1}]$ with $\CF(\ulM)$ inside $\CE(\ulM)[\ell_\zeta^{-1}]$.

Since the degree is additive in the sequence~\eqref{EqExactWeightSeq}, \ref{Prop4.4_b} and \ref{Prop4.4_c} imply inductively for increasing $\mu$ that $\deg_\Fq W_{\mu\,}\ulHodge^1(\ulM)=\deg\ulCF(W_\mu\ulM)=\sum_{\mu'\le\mu}\mu'\cdot\rk(\Gr_{\mu'}^W\ulM)=:\deg^W W_{\mu\,}\ulHodge^1(\ulM)$ and so also $\deg_\Fq\ulHodge^1(\ulM)=\deg^W\ulHodge^1(\ulM)$.
\end{proof}

The reader should be warned however, that in the mixed case the weights $\frac{d_i}{r_i}$ of $\ulCF(\ulM)\cong\bigoplus_i\ulCF_{d_i,r_i}$ do not need to coincide with the weights of $\ulM$. 

\begin{example}\label{ExMixedWeights}
Recall the mixed $A$-motive $\ulM$ with weights 1 and 3 from Example~\ref{Ex2.9}, whose Hodge-Pink structure $\ulHodge^1(\ulM)$ has Hodge-Pink weights $(1,3)$ or $(0,4)$ if $(t-\theta)|b$ or $(t-\theta)\nmid b$, respectively. The motivic Galois group $\Gamma_\ulM$ of $\ulM$ was computed in Example~\ref{Ex2.9GalGp}.

The associated $\sigma$-bundles can be described by the following diagram. 
\begin{equation}\label{EqDiagramExMixedWeights}
\xymatrix @R=0.8pc {
0 \ar[r] & \ulCE(W_1\ulM) = \ulCF_{0,1}\ar@{^{ (}->}[dddd]^{\TS\cdot\eta\ell_\zeta^{\SSC -}} \ar[r] & \ulCE(\ulM) = \ulCF_{0,1}^{\oplus2}\ar@{^{ (}->}[dddd]^{\cdot\left( \begin{array}{cc}\eta\ell_\zeta^{\SSC -}&f\\0&(\eta\ell_\zeta^{\SSC -})^3 \end{array}\right)} \ar[r] & \ulCE(\Gr^W_3\ulM) = \ulCF_{0,1}\ar@{^{ (}->}[dddd]^{\TS\cdot(\eta\ell_\zeta^{\SSC -})^3} \ar[r] & 0 \\ \\ \\ \\
0 \ar[r] & W_1\ulM\otimes_{A_\BC}\dotCO \ar[r]\ar@{=}[d] & \ulM\otimes_{A_\BC}\dotCO \ar[r]\ar@{=}[d] & \Gr^W_3\ulM\otimes_{A_\BC}\dotCO \ar[r]\ar@{=}[d] & 0\\
& (\dotCO,t-\theta) \ar[ddd]^{\TS \cdot\eta^{-1}\ell_\zeta^{\SSC +}}& (\dotCO{}^{\oplus 2},\Phi) \ar[ddd] & (\dotCO,(t-\theta)^3) \ar[ddd]^{\TS \cdot(\eta^{-1}\ell_\zeta^{\SSC +})^3} & \\ \\ \\
0 \ar[r] & \ulCF_{1,1} = (\dotCO,z^{-1}) \ar[r] & \ulCF(\ulM)\cong\bigoplus_i\ulCF_{d_i,r_i} \ar[r] & \ulCF_{3,1} = (\dotCO,z^{-3}) \ar[r] & 0}
\end{equation}
where $\ell_\zeta^{\SSC -}$ and $\ell_\zeta^{\SSC +}$ were defined in Example~\ref{Ex4.2}(c). In particular, one sees that $\ulCE(\ulM)\into\ulM\otimes_{A_\BC}\dot\CO$ is an isomorphism outside $\bigcup_{j\in\BN_{0}}\{z=\zeta^{q^j}\}$ and $\ulM\otimes_{A_\BC}\dot\CO\into\ulCF(\ulM)$ is an isomorphism outside $\bigcup_{j<0}\{z=\zeta^{q^j}\}$.

We determine the isomorphy type of $\ulCF(\ulM)$ as in Theorem~\ref{ThmHartlPink1}\ref{ThmHartlPink1a}. If $\frac{d_i}{r_i}>3$, then by Theorem~\ref{ThmHartlPink1}\ref{ThmHartlPink1b} the map of $\ulCF_{d_i,r_i}$ to $\ulCF_{3,1}$ is zero, so $\ulCF_{d_i,r_i}\subset\ulCF_{1,1}$ and again by Theorem~\ref{ThmHartlPink1}, $\frac{d_i}{r_i}\le1$, a contradiction. Similarly, if $\frac{d_i}{r_i}<1$ then the map $\ulCF_{1,1}\to\ulCF_{d_i,r_i}$ is zero, and $\ulCF_{3,1}\onto\ulCF_{d_i,r_i}$, a contradiction. So $1\le\tfrac{d_i}{r_i}\le3$. Since $\deg\ulCF(\ulM)=4$ the only possibilities are $\ulCF(\ulM)\cong\ulCF_{2,1}^{\oplus 2}$ or $\ulCF(\ulM)\cong\ulCF_{1,1}\oplus\ulCF_{3,1}$. 

The latter occurs if and only if the bottom horizontal sequence splits, that is if and only if there are $u,v\in\dotCO$ not both zero which define the map $(u,v)\colon\ulM\otimes_{A_\BC}\dotCO\to\ulCF_{1,1}$, ${x \choose y}\mapsto ux+vy$. This implies that 
\[
(u,v)\left( \begin{array}{cc}\eta\ell_\zeta^{\SSC -}&f\\0&(\eta\ell_\zeta^{\SSC -})^3 \end{array}\right)\;=\;\Bigl(\eta\ell_\zeta^{\SSC -}\,u\,,\;uf+(\eta\ell_\zeta^{\SSC -})^3\,v\Bigr)
\]
defines a morphism $\ulCF_{0,1}^{\oplus 2}\to\ulCF_{1,1}$. Since $\eta\ell_\zeta^{\SSC -}\,u=z^{-1}\sigma^*(\eta\ell_\zeta^{\SSC -}\,u)$ and it vanishes at $z=\zeta$, it also vanishes at $z=\zeta^{q^i}$ for all $i\in\BZ$. Therefore, it is divisible by $\ell_\zeta$, whence $u=\eta^{-1}\ell_\zeta^{\SSC +}\cdot\tilde u$ with $\tilde u=\sigma^*\tilde u\in\BF_q\dpl z\dpr$. If $\tilde u=0$ then $(\eta\ell_\zeta^{\SSC -})^3\,v=z^{-1}\sigma^*((\eta\ell_\zeta^{\SSC -})^3\,v)$. Since this vanishes at $z=\zeta$ of order three, it also vanishes at $z=\zeta^{q^i}$ for all $i\in\BZ$ of order three, and hence, it is divisible by $(\ell_\zeta)^3$, that is $v=(\eta^{-1}\ell_\zeta^{\SSC +})^3\cdot\tilde v$ with $\tilde v=z^2\sigma^*(\tilde v)\in\ulCF_{-2,1}{}^\tau$. By Example~\ref{Ex4.2}(d) this implies $\tilde v=0$ in contradiction to $v\ne0$. So to split the bottom horizontal sequence we must have $u\ne0$. 

We claim that in the case where $(t-\theta)|b$ in $\BC[t]$, the bottom horizontal sequence splits if and only if the sequence 
\begin{equation}\label{EqDiagramExMixedWeights2}
0\longto W_1\ulHodge^1(\ulM)\otimes_QQ_\infty\longto\ulHodge^1(\ulM)\otimes_QQ_\infty\longto\ulHodge^1(\Gr^W_3\ulM)\otimes_QQ_\infty\longto0
\end{equation}
of ``$Q_\infty$-Hodge structures'' splits. Namely, in this case $f$ is divisible by $\ell_\zeta^{\SSC -}$ by Example~\ref{Ex2.9} and therefore $uf+(\eta\ell_\zeta^{\SSC -})^3\,v$ vanishes at $z=\zeta^{q^i}$ for all $i\in\BN_0$ and moreover for all $i\in\BZ$ because it is a $\tau$-invariant in $\ulCF_{1,1}$. Thus $uf+(\eta\ell_\zeta^{\SSC -})^3\,v=\ell_\zeta\cdot\tilde h$ for an element $\tilde h=\sigma^*\tilde h\in\BF_q\dpl z\dpr$. This shows that in diagram~\eqref{EqDiagramExMixedWeights} the top row is split by the morphism $(\tilde u,\tilde h)\colon\ulCE(\ulM)\to\ulCE(W_1\ulM)$. Since $\ulHodge^1(\ulM)\otimes_QQ_\infty=\ulCE(\ulM)^\tau$ this defines the splitting of \eqref{EqDiagramExMixedWeights2} on the level of the underlying $Q_\infty$-vector spaces. It is automatically compatible with the weight filtration here. Moreover, the splitting $(\tilde u, \tilde h)$ is compatible with the splitting of the bottom row in diagram~\eqref{EqDiagramExMixedWeights} and this shows that the splitting respects the Hodge-Pink lattices.

Conversely, by construction of the $\sigma$-bundles $\ulCE(\ulM)$ and $\ulCF(\ulM)$ every splitting of \eqref{EqDiagramExMixedWeights2} induces a compatible splitting of the top and bottom row in diagram~\eqref{EqDiagramExMixedWeights}. Therefore, $\ulCF(\ulM)\cong\ulCF_{1,1}\oplus\ulCF_{3,1}$.

Note that when the extension $0\longto(\BC[t],t-\theta)\longto\ulM\longto\bigl(\BC[t],(t-\theta)^3\bigr)\longto0$, see \eqref{EqSeqMsplitsOrNot}, splits then also \eqref{EqDiagramExMixedWeights2} splits, but the converse is false in general. Namely, by Theorem~\ref{ThmHodgeConjecture}\ref{ThmHodgeConjectureB} which we are going to prove, the former occurs if and only if the associated sequence of $Q$-Hodge-Pink structures analogous to \eqref{EqDiagramExMixedWeights2} splits. This is the case if and only if $\tilde h/\tilde u\in Q\subset Q_\infty=\BF_q\dpl z\dpr$.
\end{example}

\begin{remark}\label{RemSigmaPolygon}
In general, one defines the \emph{$\sigma$-bundle polygon} $SP(\ulM)$ of $\ulM$ as the piecewise linear function on $[0,n]$ whose slope on $[j-1,j]$ is the $j$-th smallest of the weights $\frac{d_i}{r_i}$ where $\ulCF(\ulM)\cong\bigoplus_i\ulCF_{d_i,r_i}$. Then the $\sigma$-bundle polygon lies above the weight polygon $WP(\ulM)$ from Remark~\ref{RemPolygons}(b) and both have the same endpoint, $SP(\ulM)\ge WP(\ulM)$; see \cite[Proposition~1.6.6]{HartlPSp} or Theorem~\ref{ThmImageOfH} below. In particular, after we have proved Theorem~\ref{ThmHodgeConjecture}\ref{ThmHodgeConjectureA}, Remark~\ref{RemPolygons}(b) yields $SP(\ulM)\ge WP(\ulM)\ge HP(\ulM)$ and Example~\ref{ExMixedWeights} illustrates this.
\end{remark}

\subsection{The pair of \texorpdfstring{$\sigma$}{o}-bundles associated with a dual $A$-motive}
To a uniformizable dual $A$-motive $\uldM=(\dM,\sdtau_\dM)$ we assign the pair of $\sigma$-bundles, which was associated in the previous section with the corresponding $A$-motive $\ulM:=\ulM(\uldM)=\bigl((\sdsigma^\ast\dM)\dual,\sdtau_\dM\dual\bigr)$. More precisely, we set 
\[
\ulCE(\uldM)\;:=\;\Lambda(\uldM)\dual\otimes_A\ulCF_{0,1}\;\cong\;\Lambda(\ulM)\otimes_A \ulCF_{0,1}\,.
\]
It is a $\sigma$-bundle with $\CE(\uldM):=\Lambda(\uldM)\dual\otimes_A\dotCO$ and $\stau_\CE=\id$. By Proposition~\ref{PropMaphdM}, $\CE(\uldM)$ coincides via $\sdsigma^*h_\uldM\dual$ with $(\sdsigma^\ast\dM)\dual\otimes_{A_\BC}\dotCO$ on $\PDisc\setminus\bigcup_{i\in\BN_0}\{z=\zeta^{q^i}\}$ and via $h_\uldM\dual$ with $\dM\dual\otimes_{A_\BC}\dotCO$ on $\PDisc\setminus\bigcup_{i>0}\{z=\zeta^{q^i}\}$. So it can be obtained as a modification of $(\sdsigma^\ast\dM)\dual\otimes_{A_\BC}\dotCO$ at all places $z=\zeta^{q^i}$ for $i\ge 0$.

Again $\uldM$ gives rise to a second $\sigma$-bundle as follows. The isomorphism $\sdtau_\dM\dual$ is an isomorphism between $\dM\dual$ and $(\sdsigma^\ast\dM)\dual$ outside $z=\zeta$. So one can modify $(\sdsigma^\ast\dM)\dual\otimes_{A_\BC}\dotCO$ at $z=\zeta^{q^i}$ for $i<0$ to obtain a $\sigma$-bundle $\ulCF(\uldM)=(\CF(\uldM),\stau_\CF)$ with 
\begin{eqnarray}\label{DualEq4.1}
\CF(\uldM)\es:=&\bigl\{\,f\in (\sdsigma^\ast\dM)\dual\otimes_{A_\BC}\dotCO[\ell_\zeta^{-1}]:&(\sdtau_\dM\dual)^{i}(\sdsigma^{-i\ast}f)\in (\sdsigma^\ast \dM)\dual\otimes_{A_\BC}\BC\dbl z-\zeta\dbr\text{ for all }i\in\BZ\,\bigr\}\nonumber\\[2mm]
=&\bigl\{\,f\in \CE(\uldM)[\ell_\zeta^{-1}]:&(\sdtau_\dM\dual)^{i}(\sdsigma^{-i\ast}f)\in (\sdsigma^* \dM)\dual\otimes_{A_\BC}\BC\dbl z-\zeta\dbr\text{ for all }i\in\BZ\,\bigr\},
\end{eqnarray}
and $\stau_\CF=\sdtau_\dM\dual\otimes\id$.

This is indeed a $\sigma$-bundle, because it can be viewed like $\CF(\ulM)$ above as a sheaf. Namely,
\begin{eqnarray*}
&& \Gamma\bigl(\{|\zeta|^s\le|z|\le|\zeta|^{s'}\}\,,\,\CF(\uldM)\bigr)\es=\es \TS\bigl\{\,f\in\Lambda(\uldM)\dual\otimes_A\BC\langle\frac{z}{\zeta^{s'}},\frac{\zeta^s}{z}\rangle
\colon\\[1mm]
&&\qquad\qquad(\sdtau_\dM^{\SSC\vee})^{i}(\sdsigma^{-i\ast}f)\in \sdtau_\dM(\sdsigma^* \dM)\otimes_{A_\BC}\BC\dbl z-\zeta\dbr\text{ for all }i\text{ with }|\zeta|^{q^is}\le|z|\le|\zeta|^{q^ir}\,\bigr\}\,.
\end{eqnarray*}
The latter is a finite free module over the principal ideal domain $\BC\langle\frac{z}{\zeta^{s'}},\frac{\zeta^s}{z}\rangle$, because by Proposition~\ref{PropMaphdM} it is contained in the free module $\ell_\zeta^{-d}\cdot\Lambda(\uldM)\otimes_A\BC\langle\frac{z}{\zeta^{s'}},\frac{\zeta^s}{z}\rangle$ if $J^d\cdot\dM\subset\sdtau_\dM(\sdsigma^*\dM)$.

Again by Proposition~\ref{PropMaphdM}, $\CF(\uldM)$ coincides via $\sdsigma^*h_\uldM\dual$ with $(\sdsigma^\ast\dM)\dual\otimes_{A_\BC}\dotCO$ on the space \\
$\PDisc\setminus\bigcup_{i<0}\{z=\zeta^{q^i}\}$ and via $h_\uldM\dual$ with $\dM\dual\otimes_{A_\BC}\dotCO$ on $\PDisc\setminus\bigcup_{i\le0}\{z=\zeta^{q^i}\}$.

\begin{definition}\label{DefDualPairOfSigmaBundles}
The pair $(\ulCF(\uldM),\ulCE(\uldM))$ constructed above is called the \emph{pair of $\sigma$-bundles} associated with the uniformizable dual $A$-motive $\uldM$.
\end{definition}

\begin{remark}\label{RemSigmaBdOfMAndDM}
The choice of a uniformizing parameter $z\in Q$ at $\infty$ will give rise to isomorphisms $\Omega^1_{Q/\BF_q}=Q\,dz\cong Q$ and $\Lambda\bigl(\ulM(\uldM)\bigr)\otimes_AQ=(\Lambda(\uldM)\dual\otimes_AQ)\otimes_Q\Omega^1_{Q/\BF_q}\cong\Lambda(\uldM)\dual\otimes_AQ$ and \\
$\bigl(\ulCE(\ulM(\uldM)),\ulCF(\ulM(\uldM))\bigr)\cong\bigl(\ulCE(\uldM),\ulCF(\uldM)\bigr)$; see Proposition~\ref{PropDualizingUnif}.
\end{remark}

Assume that $\uldM$ is effective with $\sdtau_\dM(\sdsigma^*\dM)\subsetneq \dM$. As in diagram~\eqref{EqDiagSigmaBdOfM} we visualize these $\sigma$-bundles over $\dotCO$ by the following diagram, in which the thick lines represent sheaves on $\PDisc$:

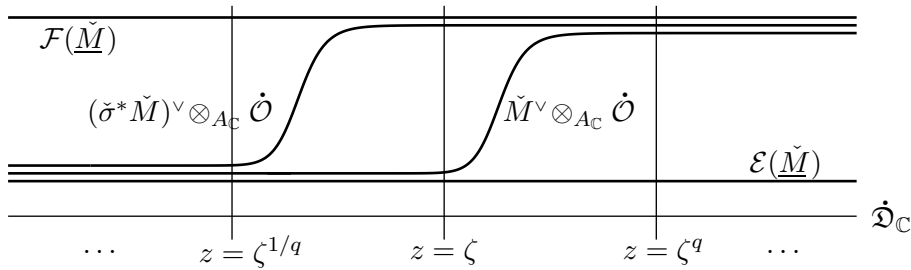
\psset{unit=3.1cm} 
 \begin{pspicture}(0,-0.4)(3.7,0.85) 
   \rput(0.4,0.61){$\CF(\uldM)$}
   \psplot[plotstyle=curve,plotpoints=200,linewidth=1pt,linecolor=black]{0.1}%
                {3.7}{0.7}
   \rput(3.4,0.08){$\CE(\uldM)$}
   \psplot[plotstyle=curve,plotpoints=200,linewidth=1pt,linecolor=black]{0.1}%
                {3.7}{6 7 div 0.7 mul 0.6 sub}
   \rput(0.83,0.3){$(\sdsigma^*\dM)\dual\!\otimes_{A_\BC}\dotCO$}
   \psplot[plotstyle=curve,plotpoints=200,linewidth=1pt,linecolor=black]{0.45}{3.7}%
                {x 0.45 sub 16 exp 6 mul x 0.45 sub 16 exp 7 mul 1 add div 0.7 mul 0.2 3 div add}
   \psline[linewidth=1pt](0.1,0.066667)(0.45,0.066667)
   \rput(2.48,0.3){$\dM\dual\!\otimes_{A_\BC}\dotCO$}
   \psplot[plotstyle=curve,plotpoints=200,linewidth=1pt,linecolor=black]{1.2}{3.7}%
                {x 1.3 sub 16 exp 6 mul x 1.3 sub 16 exp 7 mul 1 add div 0.7 mul 0.1 3 div add}
   \psline[linewidth=1pt](0.1,0.033333)(1.3,0.033333)
   %
   %
   \rput(3.85,-0.15){$\PDisc$}
   \psline[linewidth=0.5pt](0.1,-0.15)(3.7,-0.15)
   \rput(0.5,-0.3){$\dots$}
   \rput(1.13,-0.3){$z=\zeta^{1/q}$}
   \psline[linewidth=0.5pt](1.05,-0.25)(1.05,0.75)
   \rput(1.95,-0.3){$z=\zeta$}
   \psline[linewidth=0.5pt](1.95,-0.25)(1.95,0.75)
   \rput(2.88,-0.3){$z=\zeta^{q}$}  
   \psline[linewidth=0.5pt](2.85,-0.25)(2.85,0.75)
   \rput(3.4,-0.3){$\dots$}  
 \end{pspicture}

\noindent
Indeed, $\CE(\uldM)$ coincides via $h_\uldM\dual$ with $\dM\dual\otimes_{A_\BC}\dotCO$ outside $\bigcup_{i\in\BN_{>0}}\{z=\zeta^{q^i}\}$ and via $\sdsigma^*h_\uldM\dual$ it coincides with $(\sdsigma^\ast \dM)\dual\otimes_{A_\BC}\dotCO$ outside $\bigcup_{i\in\BN_0}\{z=\zeta^{q^i}\}$ and is contained in these modules by Proposition~\ref{PropMaphdM}. Via $\sdtau_\dM\dual$ also $(\sdsigma^\ast \dM)\dual$ contains $\dM\dual$ and differs from it only at $z=\zeta$. Finally, one sees that $(\sdsigma^* \dM)\dual\otimes_{A_\BC}\dotCO$ is contained in $\CF(\uldM)$ via $\sdsigma^*h_\uldM\dual$ and they coincide outside $\bigcup_{i< 0}\{z=\zeta^{q^i}\}$.

\begin{proposition}\label{PropDual4.4}
Let $\uldM$ be a uniformizable mixed dual $A$-motive, $(\ulCF(\uldM),\ulCE(\uldM))$ the associated pair of $\sigma$-bundles and $\ulHodge^1(\uldM)=(H,W_\bullet H,\Fq)$ its mixed Hodge-Pink structure.
\begin{enumerate}
\item \label{PropDual4.4_e} $\CE(\uldM)\otimes_\dotCO\BC\dbl z-\zeta\dbr=\Hodge^1(\uldM)\otimes_Q\BC\dbl z-\zeta\dbr=\Fp$ and $\CF(\uldM)\otimes_\dotCO\BC\dbl z-\zeta\dbr=\Fq\;\subset\;\Fp[\tfrac{1}{z-\zeta}]$.
\item \label{PropDual4.4_a}
The $\tau$-invariants of $\ulCE(\uldM)$ are
\[
\ulCE(\uldM)^\sdtau\;=\;\bigl(\Lambda(\uldM)\dual\otimes_A\ulCF_{0,1}\bigr)^\sdtau\;=\;\Lambda(\uldM)\dual\otimes_A Q_\infty\;=\;H\otimes_Q Q_\infty\;=\;H_\infty\,.
\]
\item \label{PropDual4.4_b}
We have $\deg\ulCE(\uldM)\;=\;0$ and $\deg\ulCF(\uldM)\;=\;\dim\uldM\;=\;\deg_\Fq\ulHodge^1(\uldM)$.

\item \label{PropDual4.4_c}
If $\uldM$ is pure of weight $\mu=-\frac{k}{l}$ with $(k,l)=1$, then $\ulHodge^1(\uldM)$ is pure of weight $-\mu=\tfrac{k}{l}$ and $\ulCF(\uldM)\cong\ulCF_{k,l}^{\oplus\rk\uldM/l}$. In particular, 
\[
\deg_\Fq\ulHodge^1(\uldM)\;=\;\deg\ulCF(\uldM)\;=\;\frac{k\cdot\rk\uldM}{l}\;=\;-\mu\cdot\rk\ulHodge^1(\uldM)\;=\;\deg^W\ulHodge^1(\uldM)\,.
\]
\item \label{PropDual4.4_d} If $\uldM$ is mixed, then also $\deg_\Fq\ulHodge^1(\uldM)=\deg^W\ulHodge^1(\uldM)$ and $\deg_\Fq W_\mu\ulHodge^1(\uldM)=\deg^W W_\mu\ulHodge^1(\uldM)$ for all $\mu$.
\end{enumerate}
\end{proposition}
\begin{proof}
We could adapt the proof of Proposition~\ref{Prop4.4}. However, everything also follows from combining Remark~\ref{RemSigmaBdOfMAndDM} and Proposition~\ref{PropDualMixed} with~\ref{Prop4.4}.
\end{proof}

Again, the reader should be warned that in the mixed case the weights $\frac{d_i}{r_i}$ of $\ulCF(\uldM)\cong\bigoplus_i\ulCF_{d_i,r_i}$ do not need to coincide with the negatives of the weights of $\uldM$. The analog of Example~\ref{ExMixedWeights} for dual $A$-motives is a case where this happens.

\subsection{Proof of Theorem~\ref{ThmHodgeConjecture}}

\noindent{\bfseries Proof of Theorem~\ref{ThmHodgeConjecture}\ref{ThmHodgeConjectureA}.}
 We want to show that $\ulHodge^1(\ulM)=(H,W_\bullet H,\Fq)$ is locally semistable. So let $H'_\infty\subset H_\infty$ be a $Q_\infty$-subspace and let $\ulH'_\infty=(H'_\infty,W_\bullet H'_\infty,\Fq')$ be the induced strict $Q_\infty$-subobject as in Definition~\ref{Def1.5}. We have to show that $\deg_\Fq\ulH'_\infty\le\deg^W\ulH'_\infty$ with equality for $H'_\infty=(W_\mu H)_\infty$. We consider two $\sigma$-bundles associated with $\ulH'_\infty$:
\begin{eqnarray*}
\ulCE'&:=&(\CE',\tau_{\CE'})\es:=\es\ulCE(\ulH'_\infty)\es:=\es H'_\infty\otimes_{Q_\infty}\ulCF_{0,1}\es\subset\es H_\infty\otimes_{Q_\infty}\ulCF_{0,1}\es=\es\ulCE(\ulM)\quad\text{and}\\[2mm]
\ulCF'&:=&(\CF',\tau_{\CF'})\es:=\es\ulCF(\ulH'_\infty)\es:=\es\bigl\{\,f\in\CE'[\ell_\zeta^{-1}]:\es \tau_{\CE'}^i(\sigma^{i\ast}f)\in\Fq'\text{ for all }i\in\BZ\,\bigr\}\,.
\end{eqnarray*}
That these are $\sigma$-bundles is seen in the same way as for $\ulCF(\ulM)$ from \eqref{Eq4.1}. Note that $\Fq'=\Fq\cap\bigl(H'_\infty\otimes_{Q_\infty}\BC\dpl z-\zeta\dpr\bigr)$ implies that $\CF'$ is the intersection of $\CE'[\ell_\zeta^{-1}]$ and $\CF(\ulM)$ inside $\CE[\ell_\zeta^{-1}]$. The two $\sigma$-bundles $\ulCE'$ and $\ulCF'$ coincide outside $\bigcup_{i\in\BZ}\{z=\zeta^{q^i}\}$ and satisfy 
\[
\Fp'\;:=\;H'_\infty\otimes_{Q_\infty}\BC\dbl z-\zeta\dbr\;=\;\CE'\otimes_{\dotCO}\BC\dbl z-\zeta\dbr\quad\text{and}\quad\CF'\otimes_{\dotCO}\BC\dbl z-\zeta\dbr\;=\;\Fq'\;\subset\;\Fp'[\tfrac{1}{z-\zeta}]\,.
\]
Since $\deg\ulCE'=0$ we compute as in the proof of Proposition~\ref{Prop4.4}\ref{Prop4.4_b} using Theorem~\ref{ThmHartlPink1}\ref{ThmHartlPink1d} that
\[
\deg\ulCF'\;=\;\deg\ulCF'-\deg\ulCE'\;=\;\deg_\Fq\ulH'_\infty\,.
\]
From the weight filtration $W_\mu H'_\infty=H'_\infty\cap(W_\mu H)_\infty$ the $\sigma$-bundle $\ulCF'$ inherits a weight filtration with saturated $\sigma$-sub-bundles $W_{\mu\,}\ulCF'=\ulCF(W_{\mu\,}\ulH'_\infty)$ being the intersection of $(W_\mu H'_\infty)\otimes_{Q_\infty}\dotCO[\ell_\zeta^{-1}]$ and $\CF(\ulM)$ inside $\CE[\ell_\zeta^{-1}]$. Moreover, $W_{\mu\,}\ulCF'$ equals the intersection $\CF(W_{\mu\,}\ulM)\cap\CF'$ inside $\CF$, because $W_\mu\CF$ is the intersection of $(W_\mu H)_\infty\otimes_{Q_\infty}\dotCO[\ell_\zeta^{-1}]$ and $\CF(\ulM)$ inside $\CE[\ell_\zeta^{-1}]$; see the proof of Proposition~\ref{Prop4.4}\ref{Prop4.4_d}. From the exact sequence
\begin{equation}\label{EqExactWeightSeq2}
0\es\longto\es \DS\bigcup_{\mu'<\mu} W_{\mu'}\ulCF'\es\longto\es W_{\mu\,}\ulCF'\es\longto\es \Gr^W_{\mu\,}\ulCF'\es\longto\es 0
\end{equation}
it follows that the natural morphism $\Gr^W_{\mu\,}\ulCF'\to\Gr^W_{\mu\,}\ulCF(\ulM)$ is injective. Since \\
$\Gr^W_{\mu\,}\ulCF(\ulM)=\ulCF(\Gr^W_{\mu\,}\ulM)\cong\ulCF_{k,l}^{(\rk W_{\mu\,}\ulM)/l}$ for $\mu=\frac{k}{l}$ with $(k,l)=1$ by Proposition~\ref{Prop4.4}\ref{Prop4.4_c}, Theorem~\ref{ThmHartlPink1} implies $\deg(\Gr^W_{\mu\,}\ulCF')\le\mu\cdot\rk(\Gr^W_{\mu\,}\ulCF')$. Using $\rk(\Gr^W_{\mu\,}\ulCF')=\dim_{Q_\infty}(\Gr^W_\mu H'_\infty)$ and the additivity of the degree in the exact sequence \eqref{EqExactWeightSeq2} we compute
\[
\deg_\Fq\ulH'_\infty\;=\;\deg\ulCF'\;=\;\sum_{\mu\in\BQ}\deg(\Gr^W_{\mu\,}\ulCF')\;\le\;\sum_{\mu\in\BQ} \mu\cdot\dim_{Q_\infty}(\Gr^W_\mu H'_\infty)\;=\;\deg^W\ulH'_\infty\,.
\]
Moreover, if $H'_\infty=(W_{\tilde\mu}H)_\infty$, then $W_\mu H'_\infty=(W_\mu H)_\infty$ and $W_{\mu\,}\ulCF'\,=\,\ulCF(W_{\mu\,}\ulM)\cap\ulCF'\,=\,\ulCF(W_{\mu\,}\ulM)$ for all $\mu\le\tilde\mu$ and so all the above inclusions and inequalities are equalities. This shows that $\ulHodge^1(\ulM)$ is locally semistable and finishes the proof of Theorem~\ref{ThmHodgeConjecture}\ref{ThmHodgeConjectureA}.
\qed

\medskip
\noindent{\bfseries Proof of Theorem~\ref{ThmHodgeConjecture}\ref{ThmHodgeConjectureB}.}
By construction the functor $\ulHodge^1$ is $Q$-linear. To prove exactness of $\ulHodge^1$ let $0\to\ulM'\to\ulM\to\ulM''\to0$ be an exact sequence of mixed $A$-motives. Then it follows from Lemma~\ref{Lemma2.3} and Proposition~\ref{PropPure}\ref{PropPure_d} that $0\to\Hodge^1(\ulM')\to\Hodge^1(\ulM)\to\Hodge^1(\ulM'')\to0$ is exact and strictly compatible with the weight filtrations. Consider the commutative diagram with exact rows
\[
\xymatrix {
0\ar[r] & M'\otimes_{A_{\BC}}\BC\dbl z-\zeta\dbr \ar[r]\ar@{^{ (}->}[d]_{\TS h_{\ulM'}^{-1}} & M\otimes_{A_{\BC}}\BC\dbl z-\zeta\dbr \ar[r]\ar@{^{ (}->}[d]_{\TS h_{\ulM}^{-1}} & M''\otimes_{A_{\BC}}\BC\dbl z-\zeta\dbr \ar[r]\ar@{^{ (}->}[d]_{\TS h_{\ulM''}^{-1}} & 0\\
0\ar[r] & \Fp'[\frac{1}{z-\zeta}] \ar[r] & \Fp[\frac{1}{z-\zeta}] \ar[r] & \Fp''[\frac{1}{z-\zeta}] \ar[r] & 0\,,
}
\]
where the vertical maps come from Proposition~\ref{PropMaphM}. Their images are the Hodge-Pink lattices $\Fq\subset\Fp[\tfrac{1}{z-\zeta}]$. Since $M'\subset M$ is saturated the sequence $0\to\ulHodge^1(\ulM')\to\ulHodge^1(\ulM)\to\ulHodge^1(\ulM'')\to0$ is strict exact. 

To prove that $\ulHodge^1$ is faithful let $f\colon \ulM\to\ulM'$ be a morphism of $A$-motives with $\ulHodge^1(f)=0$. This implies that $\Lambda(f)\colon \Lambda(\ulM)\to\Lambda(\ulM')$ is the zero map. Then by Definition~\ref{DefLambda} the map $f\otimes\id\colon M\otimes_{A_{\BC}}\CO(\FC_\BC\setminus\Disc)\to M'\otimes_{A_{\BC}}\CO(\FC_\BC\setminus\Disc)$ is the zero map and this implies $f=0$.

To prove that $\ulHodge^1$ is full let $g\colon \ulHodge^1(\ulM)\to\ulHodge^1(\ulM')$ be a non-zero morphism of $Q$-Hodge-Pink structures. It can be interpreted as an injection $\UOne\into\ulHodge^1(\ulM')\otimes\ulHodge^1(\ulM)\dual=\ulHodge^1(\ulM'\otimes\ulM\dual)$. It suffices to show that this Hodge-Pink sub-structure $\UOne\subset\ulHodge^1(\ulM'\otimes\ulM\dual)$ is of the form $\ulHodge^1(\ulM'')=\UOne$ for an $A$-sub-motive $\ulM''\subset\ulM'\otimes\ulM\dual$ in the category $\AMUMotCatIsog$ of uniformizable mixed $A$-motives up to isogeny. Then necessarily $\ulM''$ has rank $1$ and virtual dimension $0$ and hence, equals $\UOne$; see Example~\ref{ExampleCHMotive}. Therefore, $\ulM''=\UOne$ can be reinterpreted as a morphism $f\colon \ulM\to\ulM'$ with $\ulHodge^1(f)=g$. So Theorem~\ref{ThmHodgeConjecture}\ref{ThmHodgeConjectureB} follows from Theorem~\ref{ThmHodgeConjecture}\ref{ThmHodgeConjectureC}.
\qed

\medskip
\noindent{\bfseries Proof of Theorem~\ref{ThmHodgeConjecture}\ref{ThmHodgeConjectureC}.}
To show that the essential image of the functor $\ulHodge^1\colon \ulM\to\ulHodge^1(\ulM)$ is closed under forming subquotients we only need to treat the case of a Hodge-Pink sub-structure 
\[
\ulH'\;=\;(H', W_\bullet H',\Fq')\;\subset\;\ulHodge^1(\ulM)\;=\;(\Lambda(\ulM)\otimes_A Q,W_\bullet H,\Fq)\,,
\]
because by the exactness of $\ulHodge^1$, quotient objects can be handled via their associated kernel subobjects. By \cite[Proposition~4.7(c)]{PinkHodge} the inclusion $\ulH'\subset\ulHodge^1(\ulM)$ is automatically strict. We will prove the following

\medskip\noindent
{\it Claim 1.} There is a saturated $A$-sub-motive $\ulM'\subset\ulM$ with $\Hodge^1(\ulM')=H'\subset H$ and such that the Hodge-Pink lattice of $\ulM'$ equals $\Fq'$.

\medskip\noindent
We use the claim to prove Theorem~\ref{ThmHodgeConjecture}\ref{ThmHodgeConjectureC} as follows. By Proposition~\ref{PropPure}\ref{PropPure_c} the $A$-sub-motive $\ulM'$ is mixed with $W_{\mu\,}\ulM'=\ulM'\cap W_{\mu\,}\ulM\subset\ulM$. Then the exactness of $\ulHodge^1$ implies that 
\[
\Hodge^1(W_{\mu\,}\ulM')\;=\;\Hodge^1(\ulM')\cap\Hodge^1(W_{\mu\,}\ulM)\;=\;H'\cap W_\mu H\;=\;W_\mu H'
\]
and in particular $\ulHodge^1(\ulM')=\ulH'$.

To prove the claim, we set $\Lambda':=H'\cap\Lambda(\ulM)$ and consider the $\sigma$-sub-bundle $\ulCE'\;:=\;\Lambda'\otimes_A\ulCF_{0,1}\;=\;H'\otimes_Q\ulCF_{0,1}\subset\ulCE(\ulM)$ whose underlying module $\CE'=H'\otimes_Q\dotCO$ is a saturated submodule of $\CE(\ulM)$. As above we modify $\ulCE'$ at $\bigcup_{i\in\BZ}\{z=\zeta^{q^i}\}$ according to the inclusion $\Fp'=H'\otimes_Q\BC\dbl z-\zeta\dbr\subset\Fq'$ to obtain the $\sigma$-sub-bundle
\[
\ulCF'\;:=\;(\CF',\tau_{\CF'})\;:=\;\ulCF(\ulH')\;:=\;\bigl\{\,f\in\CE'[\ell_\zeta^{-1}]:\; \tau_{\CE'}^i(\sigma^{i\ast}f)\in\Fq'\text{ for all }i\in\BZ\,\bigr\}\;\subset\;\ulCF(\ulM)\,.
\]
Since $\Fq'=\Fq\cap H'\otimes_Q\BC\dpl z-\zeta\dpr$ this sub-bundle is also saturated. We now consider the admissible covering $\dotFC_\BC=\{0<|z|<|\zeta|^{q^{-1}}\}\cup\FC_\BC\setminus\{|z|\le|\zeta|\}$ of the rigid analytic curve $\dotFC_\BC$, and we define a saturated locally free subsheaf $\CM'\subset M\otimes_{A_{\BC}}\CO(\dotFC_\BC)$ of finite rank on $\dotFC_\BC$ together with an isomorphism $\tau_{\CM'}\colon \sigma^\ast\CM'[J^{-1}]\isoto\CM'[J^{-1}]$ by setting
\[
\begin{array}{lll@{\qquad\text{with}\qquad}lll}
\CM'|_{\FC_\BC\setminus\{|z|\le|\zeta|\}}&:=& \Lambda'\otimes_A\CO_{\FC_\BC\setminus\{|z|\le|\zeta|\}}&\tau_{\CM'}&:=&\id\\[2mm]
\CM'|_{\{0<|z|<|\zeta|^{q^{-1}}\}}&:=&\CF'|_{\{0<|z|<|\zeta|^{q^{-1}}\}}&\tau_{\CM'}&:=&\tau_{\CF'}
\end{array}
\]
and glueing the two pieces on the overlap $\{|\zeta|<|z|<|\zeta|^{q^{-1}}\}$ via the isomorphism 
\[
\Lambda'\otimes_A\CO_{\{|\zeta|<|z|<|\zeta|^{q^{-1}}\}}\;=\;\CE'|_{\{|\zeta|<|z|<|\zeta|^{q^{-1}}\}}\;=\;\CF'|_{\{|\zeta|<|z|<|\zeta|^{q^{-1}}\}}\,.
\]
Since $M\otimes_{A_{\BC}}\CO_{\FC_\BC\setminus\{|z|\le|\zeta|\}}\;\cong\;\Lambda(\ulM)\otimes_A\CO_{\FC_\BC\setminus\{|z|\le|\zeta|\}}$ by Proposition~\ref{PropLambdaConvRadius} and $\Lambda'\subset\Lambda(\ulM)$ is saturated, the subsheaf $\CM'\subset M\otimes_{A_{\BC}}\CO(\dotFC_\BC)$ is saturated.
Note that 
\[
\begin{array}{llcll}
(\sigma^*\CM')|_{\FC_\BC\setminus\{|z|\le|\zeta|^q\}}&=& \sigma^*\bigl(\CM'|_{\FC_\BC\setminus\{|z|\le|\zeta|\}}\bigr)&=& \Lambda'\otimes_A\CO_{\FC_\BC\setminus\{|z|\le|\zeta|^q\}}\\[2mm]
(\sigma^*\CM')|_{\{0<|z|<|\zeta|\}}&=& \sigma^*\bigl(\CM'|_{\{0<|z|<|\zeta|^{q^{-1}}\}}\bigr)&=& (\sigma^*\CF')|_{\{0<|z|<|\zeta|\}}
\end{array}
\]
and
\begin{eqnarray*}
(\sigma^*\CM')|_{\{|\zeta|^q<|z|<|\zeta|^{q^{-1}}\}}[\tfrac{1}{z-\zeta}] & = & \CE'|_{\{|\zeta|^q<|z|<|\zeta|^{q^{-1}}\}}[\tfrac{1}{z-\zeta}]\\[2mm]
& = & \CF'|_{\{|\zeta|^q<|z|<|\zeta|^{q^{-1}}\}}[\tfrac{1}{z-\zeta}]\es=\es\CM'|_{\{|\zeta|^q<|z|<|\zeta|^{q^{-1}}\}}[\tfrac{1}{z-\zeta}]\,.
\end{eqnarray*}
Therefore, $\tau_{\CM'}$ is an isomorphism between $\sigma^*\CM'$ and $\CM'$ outside $z=\zeta$. At $z=\zeta$ we have
\begin{eqnarray*}
(\sigma^\ast\CM')\otimes_{\CO(\dotFC_\BC)}\BC\dbl z-\zeta\dbr & = & \CE'\otimes_\dotCO\BC\dbl z-\zeta\dbr \es=\es \Fp'\qquad\text{and}\\[2mm]
\CM'\otimes_{\CO(\dotFC_\BC)}\BC\dbl z-\zeta\dbr & = & \CF'\otimes_\dotCO\BC\dbl z-\zeta\dbr \es=\es\Fq'\es\subset\es\Fp'[\tfrac{1}{z-\zeta}]\,.
\end{eqnarray*}
So indeed $\tau_{\CM'}\colon\sigma^\ast\CM'[J^{-1}]\isoto\CM'[J^{-1}]$ is a isomorphism. If $\Fp'\subset\Fq'$, then $\ulCE'\subset\ulCF'$ and therefore $\tau_{\CM'}\colon\sigma^\ast\CM'\to\CM'$ is a morphism with $\coker\tau_{\CM'}\;\cong\;(\CF'/\CE')\otimes_\dotCO\BC\dbl z-\zeta\dbr\;\cong\;\Fq'/\Fp'$. We visualize this case as follows.

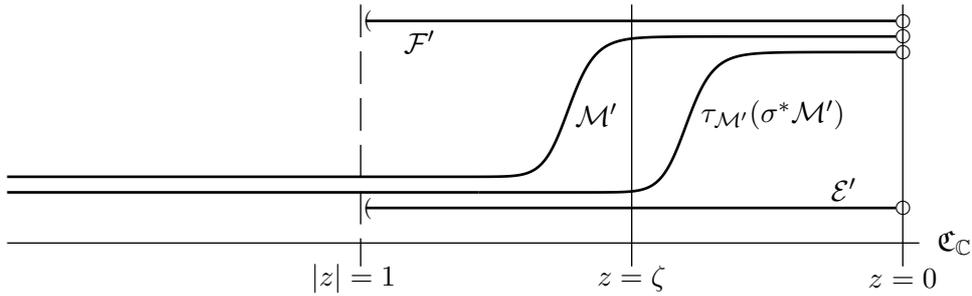
\newpsobject{showgrid}{psgrid}{subgriddiv=1,griddots=10,gridlabels=0pt}
 \psset{unit=3.1cm} 
 \begin{pspicture}(0,-0.4)(3.7,1) 
   \rput(1.85,0.71){$\CF'$}
   \psplot[plotstyle=curve,plotpoints=200,linewidth=1pt,linecolor=black]{1.62}%
                {3.87}{0.8}
   \rput(1.63,0.8){$\SC ($}
   \rput(3.93,0.8){\circle{0.06}}
   \rput(2.6,0.4){$\CM'$}
   \psplot[plotstyle=curve,plotpoints=200,linewidth=1pt,linecolor=black]{1.6}{3.87}%
                {x 1.6 sub 16 exp 6 mul x 1.6 sub 16 exp 7 mul 1 add div 0.7 mul 0.4 3 div add}
   \psline[linewidth=1pt](0.1,0.133333)(1.6,0.133333)
   \rput(3.93,0.733337){\circle{0.06}}
   \rput(3.35,0.4){$\tau_{\CM'}(\sigma^*\CM')$}
   \psplot[plotstyle=curve,plotpoints=200,linewidth=1pt,linecolor=black]{2.1}{3.87}%
                {x 2.1 sub 16 exp 6 mul x 2.1 sub 16 exp 7 mul 1 add div 0.7 mul 0.2 3 div add}
   \psline[linewidth=1pt](0.1,0.066667)(2.1,0.066667)
   \rput(3.93,0.666667){\circle{0.06}}
   \rput(3.65,0.08){$\CE'$}
   \psplot[plotstyle=curve,plotpoints=200,linewidth=1pt,linecolor=black]{1.62}%
                {3.87}{6 7 div 0.7 mul 0.6 sub}
   \rput(1.63,0){$\SC ($}
   \rput(3.93,0){\circle{0.06}}
   %
   %
   \rput(4.12,-0.15){$\FC_\BC$}
   \psline[linewidth=0.5pt](0.1,-0.15)(3.97,-0.15)
   %
   %
   \rput(1.57,-0.3){$|z|=1$}
   \psline[linewidth=0.5pt](1.6,-0.25)(1.6,-0.13)
   \psline[linewidth=0.5pt](1.6,-0.07)(1.6,0.07)
   \psline[linewidth=0.5pt](1.6,0.13)(1.6,0.27)
   \psline[linewidth=0.5pt](1.6,0.33)(1.6,0.47)
   \psline[linewidth=0.5pt](1.6,0.53)(1.6,0.67)
   \psline[linewidth=0.5pt](1.6,0.73)(1.6,0.87)
   \rput(2.75,-0.3){$z=\zeta$}
   \psline[linewidth=0.5pt](2.75,-0.25)(2.75,0.87)
   %
   %
   %
   \rput(3.9,-0.3){$z=0$}
   \psline[linewidth=0.5pt](3.9,-0.25)(3.9,0.87)
 \end{pspicture}

\noindent
This picture has to be interpreted in the same way as diagram~\eqref{EqDiagSigmaBdOfM}, except that here we see the entire rigid analytic curve $\dotFC_\BC=\FC_\BC\setminus\{z=0\}$ to which we have extended $\CM'$ and $\sigma^*\CM'$. Before we continue with the proof we make the following 

\begin{definition}\label{DefAnalyticAMotive}
The $\ul\CM(\ulH'):=(\CM',\tau_{\CM'})$ constructed above is called the \emph{analytic $A$-motive} and $(\ulCF(\ulH'),\ulCE(\ulH'))$ is called the \emph{pair of $\sigma$-bundles} associated with the $Q$-Hodge-Pink structure $\ulH'$.
\end{definition}

Especially for $H'=\Hodge(\ulM)$ we obtain $\Gamma\bigl(\dotFC_\BC,\ul\CM(\ulHodge^1(\ulM))\bigr)\cong\ulM\otimes_{A_{\BC}}\CO(\dotFC_\BC)$ by diagram~\eqref{EqDiagSigmaBdOfM} and Proposition~\ref{PropMaphM}. Recall that the $\tau$-invariants $\Lambda(\ulM)$ of $\ulM$ are computed as the $\tau$-invariants of $\ulM\otimes_{A_{\BC}}\CO\bigl(\dotFC_\BC\setminus\bigcup_{i\in\BN_0}\Var(\sigma^{i\ast}J)\bigr)$. For our $\ulH'$ the $\tau$-invariants of $\ul\CM(\ulH')$ are 
\[
\Bigl\{\,m\in\Gamma\bigl(\dotFC_\BC\setminus{\TS\bigcup\limits_{i\in\BN_0}}\Var(\sigma^{i\ast}J),\,\CM'\bigr):\tau_{\CM'}(\sigma^\ast m)=m\,\Bigr\}\;=\;\Lambda'\,;
\]
use \cite[Proposition~3.4]{BoeckleHartl}. We therefore must show that $\ul\CM(\ulH')\cong\ulM'\otimes_{A_{\BC}}\CO(\dotFC_\BC)$ for a saturated $A$-sub-motive $\ulM'\subset\ulM$. For this we use the following

\begin{lemma}\label{LemmaDescend}
The saturated analytic $A$-sub-motive $\ul\CM'\subset\ulM\otimes_{A_{\BC}}\CO(\dotFC_\BC)$ of rank $r':=\rk\CM'$ descends to a saturated $A$-sub-motive $\ulM'\subset \ulM$ with $\ul\CM'=\ulM'\otimes_{A_{\BC}}\CO(\dotFC_\BC)$ if and only if the saturated analytic $A$-sub-motive $\wedge^{r'}\ul\CM'\subset\wedge^{r'} \ulM\otimes_{A_{\BC}}\CO(\dotFC_\BC)$ descends to a saturated $A$-sub-motive $\ul N'\subset\wedge^{r'} \ulM$ with $\wedge^{r'}\ul\CM'=\ul N'\otimes_{A_{\BC}}\CO(\dotFC_\BC)$.
\end{lemma}

\begin{proof}
Clearly, the existence of $\ulM'$ implies the existence of $\ulN':=\wedge^{r'}\ulM'$. Conversely, if $\ulN'=(N',\tau_{N'})$ exists, we define $M':=\{m\in M\colon m\wedge n=0\text{ for all }n\in N'\}$. Then the equality $\wedge^{r'}\CM'=N'\otimes_{A_{\BC}}\CO(\dotFC_\BC)$ implies that the submodule $M'\subset M$ is cut out by the same linear conditions as $\CM'\subset M\otimes_{A_{\BC}}\CO(\dotFC_\BC)$. Thus $\CM'=M'\otimes_{A_{\BC}}\CO(\dotFC_\BC)$ because $\CO(\dotFC_\BC)$ is flat over $A_{\BC}$. 

To construct $\tau_{M'}$ note that the isomorphism $\sigma^\ast\colon A_{\BC}\to A_{\BC}$ is flat. Therefore, $\tau_M$ induces a map $\tau_{M'}\colon \sigma^\ast M'\to\{m\in M:m\wedge\tilde n=0\text{ for all }\tilde n\in\tau_{N'}(\sigma^\ast N')\subset N'\}$. The target of this map contains $M'$ and even equals $M'$ because $M'\subset M$ is saturated and $J^d N'\subset\tau_{N'}(\sigma^\ast N')$. Hence, $\ulM':=(M',\tau_{M'})$ is the desired $A$-sub-motive of $\ulM$.
\end{proof}

By the lemma we may set $r':=\rk\ulH'=\dim_QH'$, consider the $r'$-th exterior powers of everything and thus reduce to the case that $\rk\ulH'=1$. Then $\ulH'$ is necessarily pure of some weight $\mu\in\BZ$ and satisfies $\Fq'=(z-\zeta)^{-\mu}\,\Fp'$. Clearly, the $A$-motive $\UOne(\mu)$ of rank $1$ satisfies $\ulH'=\ulHodge^1(\UOne(\mu))$. But we have to prove that $\UOne(\mu)$ is an appropriate sub-motive of $\ulM$.
Since $\rk\ulH'=1$ we have $\ulCF(\ulH')\cong\ulCF_{d,1}$ for 
\[
d\;=\;\deg\ulCF(\ulH')\;=\;\deg\ulCF(\ulH')-\deg\ulCE(\ulH')\;=\;\deg_\Fq\ulH'\;=\;\deg^W\ulH'\;=\;\mu\,.
\]
Since $\ulCF_{\mu,1}=(\dotCO,\tau_{\CF_{\mu,1}}=z^{-\mu})$ contains the tautological $\CO_{\FD_\BC}$-lattice $\CO_{\FD_\BC}$, $\CM'$ extends to a locally free rigid analytic sheaf $\olCM'$ on $\FC_\BC$ with $\tau_{\CM'}\colon \sigma^\ast\olCM'\isoto\olCM'\bigl(\mu\cdot\infty-\mu\cdot\Var(J)\bigr)$, where the notation $\bigl(\mu\cdot\infty-\mu\cdot\Var(J)\bigr)$ means that we allow poles at $\infty$ of order less than or equal to $\mu$ and at $\Var(J)$ of order less than or equal to $-\mu$. Note that a pole with negative order is a zero.
Also since $\ulH'\subset\ulH$ is a strict subobject it already lies in $W_{\mu\,}\ulH=\ulHodge^1(W_{\mu\,}\ulM)$. We replace $\ulM$ by $W_{\mu\,}\ulM$ and thus assume that all weights of $\ulM$ are less than or equal to $\mu$. By Proposition~\ref{PropWeights}\ref{PropWeights_b} there is an extension of $\ulM$ to a locally free sheaf $\olM$ on $C_\BC$ with $\tau_M\colon \sigma^\ast\olM\to\olM\bigl(\mu\cdot\infty+\tilde d\cdot\Var(J)\bigr)$ for some $\tilde d\in\BZ$. 

We want to show that the inclusion $\CM'\into M\otimes_{A_{\BC}}\CO(\dotFC_\BC)$ extends to an inclusion $\olCM'\into\olM\otimes_{\CO_{C_\BC}}\CO_{\FC_\BC}$. Consider the ring $\BC\langle\frac{z}{\zeta}\rangle$ of rigid analytic functions on $\{|z|\le|\zeta|\}\subset\FC_\BC$. It is a principal ideal domain. So the module $\olM\otimes_{\CO_{C_\BC}}\BC\langle\frac{z}{\zeta}\rangle$ has a basis $\{e_1,\ldots,e_n\}$ with respect to which $z^\mu\tau_M\colon \sigma^\ast(\olM\otimes\BC\langle\frac{z}{\zeta}\rangle)=(\sigma^\ast\olM)\otimes\BC\langle\frac{z}{\zeta^q}\rangle\to\olM\otimes\BC\langle\frac{z}{\zeta^q}\rangle$ is given by a matrix $A=\sum_{i=0}^\infty A_iz^i\in \BC\langle\frac{z}{\zeta^q}\rangle^{n\times n}$. After tensoring with the ring $\BC\langle\frac{z}{\zeta},\frac{\zeta^q}{z}\rangle$ of rigid analytic functions on $\{|\zeta|^q\le|z|\le|\zeta|\}$, the inclusion $\ulCF_{\mu,1}\isoto\ulCF(\ulH')\into\ulCF(\ulM),\;1\mapsto f$ induces a map $\BC\langle\frac{z}{\zeta},\frac{\zeta^q}{z}\rangle\to\ulCF(\ulM)\otimes_{\dotCO}\BC\langle\frac{z}{\zeta},\frac{\zeta^q}{z}\rangle=\olM\otimes_{A_\BC}\BC\langle\frac{z}{\zeta},\frac{\zeta^q}{z}\rangle$ with $\tau_M(\sigma^\ast f)=z^{-\mu}f$. Hence, the coordinate vector $x\in\BC\langle\frac{z}{\zeta},\frac{\zeta^q}{z}\rangle^n$ of $f$ with respect to the basis $\{e_1,\ldots,e_n\}$ satisfies $A\sigma^\ast(x)=x$. We write $x=\sum_{i\in\BZ}x_iz^i$ and make the 

\medskip\noindent
{\it Claim 2.} There is an integer $k\in\BZ$ with $x_i=0$ for all $i\le -k$, in particular, $z^kx\in\BC\langle\frac{z}{\zeta}\rangle^n$.

\medskip\noindent
To prove the claim, assume the contrary and let $c\ge1$ with $|A_i\zeta^{qi}|\leq c$ for all $i$. Since $x\in\BC\langle\frac{z}{\zeta},\frac{\zeta}{z}\rangle^n$ we can find a negative integer $m$ with $x_m\neq0$, $|x_m\zeta^m|=:\tilde c\leq c^{-1}\le1$, and $|x_{m-i}\zeta^{m-i}|\leq \tilde c$ for all $i\geq0$. From $x_m=\sum_{i=0}^\infty A_i\cdot \sigma^*(x_{m-i})$ we obtain
\[
|x_m\zeta^m|\es\leq\es|\zeta^{(1-q)m}|\,\max_{i\geq0}\bigl\{\,|A_i\zeta^{qi}|\,|\sigma^*(x_{m-i})\zeta^{qm-qi}|\,\bigr\}\es<\es c\,\tilde c^q\es\leq\es \tilde c\,,
\]
a contradiction. This proves Claim 2.

We now replace $\olM$ by $\olM(\mbox{$k\cdot\infty$})$ and thus the basis $\{e_i\}$ by $\{z^{-k}e_i\}$ and $x$ by $z^kx\in\BC\langle\frac{z}{\zeta}\rangle^n$. This shows that $f\in\olM\otimes\BC\langle\frac{z}{\zeta}\rangle$ and hence, the inclusion $\ul\CM'\into\ulM\otimes_{A_{\BC}}\CO(\dotFC_\BC)$ extends to an inclusion $\bar f\colon \olCM'\into\olM\otimes_{\CO_{C_\BC}}\CO_{\FC_\BC}$. By the rigid analytic GAGA principle (see L\"utkebohmert~\cite[Theorem~2.8]{Lubo}) for the projective curve $C_\BC$ there is an algebraic subsheaf $\olM'\into\olM$ over $C_\BC$ together with an isomorphism $\tau_{M'}\colon \sigma^\ast\olM'\isoto\olM'\bigl(\mu\cdot\infty-\mu\cdot\Var(J)\bigr)$ such that $\olCM'=\olM'\otimes_{\CO_{C_\BC}}\CO_{\FC_\BC}$ and $\tau_{\CM'}=\tau_{M'}\otimes\id$. In particular, $\ulM':=\bigl(\Gamma(\dotC_\BC,\olM'),\tau_{M'}\bigr)\subset\ulM$ is the desired $A$-sub-motive with $\ulHodge^1(\ulM')=\ulH'$. This proves Claim 1 and hence also Theorem~\ref{ThmHodgeConjecture}\ref{ThmHodgeConjectureC}.
\qed

\medskip
\noindent{\bfseries Proof of Theorem~\ref{ThmHodgeConjecture}\ref{ThmHodgeConjectureD}.}
This follows directly from Theorem~\ref{ThmHodgeConjecture}\ref{ThmHodgeConjectureC}.
\qed

\medskip

We want to end this section by discussing which $Q$-Hodge-Pink structures come from uniformizable mixed $A$-motives. We give a criterion in terms of $\sigma$-bundles and the polygons from Remarks~\ref{RemPolygons} and \ref{RemSigmaPolygon}.

\begin{theorem}\label{ThmImageOfH}
Let $\ulH$ be a $Q$-Hodge-Pink structure. Then $\ulH=\ulHodge^1(\ulM)$ for a uniformizable mixed $A$-motive $\ulM$ if and only if for every $\mu=\frac{k}{l}$ with $(k,l)=1$ the $\sigma$-bundle is $\ulCF(\Gr_\mu^W\ulH)\cong\ulCF_{k,l}^{\oplus(\rk\Gr_\mu^W\ulH)/l}$, that is, if and only if the $\sigma$-bundle polygon of $\Gr_\mu^W\ulH$ and the weight polygon of $\Gr_\mu^W\ulH$ are equal, $SP(\Gr_\mu^W\ulH)=WP(\Gr_\mu^W\ulH)$. Since $WP(\Gr_\mu^W\ulH)$ has one single slope $\mu$, the latter holds if and only if $SP(\Gr_\mu^W\ulH)$ lies above $WP(\Gr_\mu^W\ulH)$ and both have the same endpoints, $SP(\Gr_\mu^W\ulH)\ge WP(\Gr_\mu^W\ulH)$. In this case the $\sigma$-bundle polygon of $W_{\mu\,}\ulH$ lies above the weight polygon for every $\mu$ and both polygons have the same endpoint, i.e.\ $SP(W_{\mu\,}\ulH)\ge WP(W_{\mu\,}\ulH)$.
\end{theorem}

We remark that the condition $SP(W_{\mu\,}\ulH)\ge WP(W_{\mu\,}\ulH)$ on the polygons of $W_{\mu\,}\ulH$ in general does not imply the condition on the polygons of $\Gr^W_{\mu\,}\ulH$ and the existence of $\ulM$.

\begin{proof}[Proof of Theorem~\ref{ThmImageOfH}]
To prove the first direction let $\ulH=\ulHodge^1(\ulM)$. Then Proposition~\ref{Prop4.4}\ref{Prop4.4_c} yields $\ulCF(\Gr_\mu^W\ulH)\cong\ulCF_{k,l}^{\oplus(\rk\Gr_\mu^W\ulH)/l}$. In particular, $SP(\Gr_\mu^W\ulH)=WP(\Gr_\mu^W\ulH)$ is the polygon with one single slope $\mu$. Consider the exact sequence \eqref{EqExactWeightSeq}. Using the convention that the sum of two polygons is defined to be the polygon whose slope multiset is the union of the slope multisets of its summands, we compute by induction on $\mu$
\begin{eqnarray*}
WP(W_{\mu\,}\ulH)&=&WP(\bigcup_{\mu'<\mu}W_{\mu'}\ulH)\;+\;WP(\Gr_\mu^W\ulH)\\[2mm]
&\le&SP(\bigcup_{\mu'<\mu}W_{\mu'}\ulH)\;+\;SP(\Gr_\mu^W\ulH)\\[2mm]
&\le&SP(W_{\mu\,}\ulH)\,.
\end{eqnarray*}
Here the first equality follows from the definition of the weight polygon, the first inequality is the induction hypothesis, and the final inequality follows from \cite[Proposition~1.5.18]{HartlPSp}.

\medskip

Conversely, let $\ulCF(\Gr_\mu^W\ulH)\cong\ulCF_{k,l}^{\oplus(\rk\Gr_\mu^W\ulH)/l}$ for $\mu=\frac{k}{l}$ with $(k,l)=1$. Let $d\in\BN_0$ be such that $(z-\zeta)^d\Fp\subset\Fq$. Recall the construction of the associated analytic $A$-motive $\ul\CM(\Gr_\mu^W\ulH)=:(\CM,\tau_\CM)$ before Definition~\ref{DefAnalyticAMotive}. It satisfies 
\[
\ul\CM(\Gr_\mu^W\ulH)\otimes_{\CO(\dotFC_\BC)}\CO_{\{0<|z|\le|\zeta|\}}\;=\;\ulCF(\Gr_\mu^W\ulH)\otimes_\dotCO\CO_{\{0<|z|\le|\zeta|\}}\;\cong\;\ulCF_{k,l}^{\oplus(\rk\Gr_\mu^W\ulH)/l}\otimes_\dotCO\CO_{\{0<|z|\le|\zeta|\}}\,.
\]
Inside the right hand side the tautological $\BC\langle\frac{z}{\zeta}\rangle$-lattice $\BC\langle\frac{z}{\zeta}\rangle^{\oplus(\rk\Gr_\mu^W\ulH)}$ defines an extension of $\ul\CM(\Gr_\mu^W\ulH)$ to a locally free rigid analytic sheaf $\olCM$ on $\FC_\BC$ with $\tau_{\CM}\colon \sigma^\ast\olCM\to\olCM\bigl(k\cdot\infty-d\cdot\Var(J)\bigr)$ such that $z^k\tau_\CM^l$ is an isomorphism locally at $\infty$. By the rigid analytic GAGA principle (see L\"utkebohmert~\cite[Theorem~2.8]{Lubo}) on the projective curve $C_\BC$ there is a locally free algebraic sheaf $\olM$ together with a homomorphism $\tau_{M}\colon \sigma^\ast\olM\to\olM\bigl(k\cdot\infty-d\cdot\Var(J)\bigr)$ such that $\olCM=\olM\otimes_{\CO_{C_\BC}}\CO_{\FC_\BC}$ and $\tau_{\CM}=\tau_{M}\otimes\id$. This implies that $z^k\tau_M^l$ is an isomorphism locally at $\infty$. In particular, $\ulM(\Gr_\mu^W\ulH):=\bigl(\Gamma(\dotC_\BC,\olM),\tau_M\bigr)$ is a pure $A$-motive of weight $\mu$ with $\ul\CM(\Gr_\mu^W\ulH)=\ulM(\Gr_\mu^W\ulH)\otimes_{A_{\BC}}\CO(\dotFC_\BC)$.

We now consider the exact sequences
\[
0\es\longto\es \DS\bigcup_{\mu'<\mu} \ul\CM(W_{\mu'}\ulH)\es\longto\es \ul\CM(W_{\mu\,}\ulH)\es\longto\es \ul\CM(\Gr^W_{\mu\,}\ulH)\es\longto\es 0\,.
\]
By induction on $\mu$ and application of Proposition~\ref{PropExtensions} below we obtain a mixed (algebraic) $A$-motive $\ulM$ with $W_{\mu\,}\ulM\otimes_{A_{\BC}}\CO(\dotFC_\BC)\cong\ul\CM(W_{\mu\,}\ulH)$ for all $\mu$. Since $\ulHodge^1(\ulM)$ is computed from $\ulM\otimes_{A_{\BC}}\CO(\dotFC_\BC)$ we find $\ulHodge^1(\ulM)\cong\ulH$ and the theorem is proved.
\end{proof}

\begin{proposition}\label{PropExtensions}
Let $\ulM',\ulM''$ be $A$-motives and let $0\to\ulM'\otimes_{A_{\BC}}\CO(\dotFC_\BC)\to\ul\CM\to\ulM''\otimes_{A_{\BC}}\CO(\dotFC_\BC)\to0$ be an exact sequence of analytic $A$-motives. Then there is an exact sequence of (algebraic) $A$-motives $0\to\ulM'\to\ulM\to\ulM''\to0$ and an isomorphism of extensions of analytic $A$-motives
\[
\xymatrix {
0\ar[r] & \ulM'\otimes_{A_{\BC}}\CO(\dotFC_\BC)\ar[r]\ar@{=}[d] & \ulM\otimes_{A_{\BC}}\CO(\dotFC_\BC)\ar[r]\ar[d]_{\TS\cong} & \ulM''\otimes_{A_{\BC}}\CO(\dotFC_\BC)\ar[r]\ar@{=}[d] & 0\\
0\ar[r] & \ulM'\otimes_{A_{\BC}}\CO(\dotFC_\BC)\ar[r] & \ul\CM\ar[r] & \ulM''\otimes_{A_{\BC}}\CO(\dotFC_\BC)\ar[r] & 0\,.
}
\]
\end{proposition}

\begin{proof}
1. Let $\CR:=\Gamma\bigl(\{0<|z|\le|\zeta|\},\CO_{\FC_\BC}\bigr)$ and $\CR^\sigma:=\sigma^\ast(\CR)=\Gamma\bigl(\{0<|z|\le|\zeta|^q\},\CO_{\FC_\BC}\bigr)$ be the rings of rigid analytic functions on the punctured discs $\{0<|z|\le|\zeta|\}$, respectively $\{0<|z|\le|\zeta|^q\}$. The exact sequence of projective $\CR$-modules $0\to M'\otimes_{A_{\BC}}\CR\to \CM\otimes_{\CO(\dotFC_\BC)}\CR\to M''\otimes_{A_{\BC}}\CR\to0$ splits and yields an isomorphism $\CM\otimes_{\CO(\dotFC_\BC)}\CR\cong (M'\oplus M'')\otimes_{A_{\BC}}\CR$ under which $\tau_\CM$ takes the form $\left(\begin{array}{cc} \tau_{M'} & f\circ\tau_{M''}\\ 0 & \tau_{M''} \end{array}\right)$ for a homomorphism $f\in\Hom_{A_{\BC}}(M'',M')\otimes_{A_{\BC}}\CR^\sigma$ which is in general not compatible with the $\tau$'s. Note that $f$ exists because $\tau_{M''}\otimes\id_{\CR^\sigma}$ is an isomorphism. A change of the splitting corresponds to an automorphism $\left(\begin{array}{cc} \id_{M'} & h\\ 0 & \id_{M''} \end{array}\right)$ of $(M'\oplus M'')\otimes_{A_{\BC}}\CR$. This replaces $\left(\begin{array}{cc} \tau_{M'} & f\circ\tau_{M''}\\ 0 & \tau_{M''} \end{array}\right)$ by 
\[
\left(\begin{array}{cc} \id_{M'} & h\\ 0 & \id_{M''} \end{array}\right) \cdot
\left(\begin{array}{cc} \tau_{M'} & f\circ\tau_{M''}\\ 0 & \tau_{M''} \end{array}\right) \cdot
\sigma^\ast\left(\begin{array}{cc} \id_{M'} & h\\ 0 & \id_{M''} \end{array}\right)^{-1}\es=\es\left(\begin{array}{cc} \tau_{M'} & \tilde f\circ\tau_{M''}\\ 0 & \tau_{M''} \end{array}\right)
\]
for $\tilde f=f+h-\tau_{M'}\circ\sigma^\ast(h)\circ\tau_{M''}^{-1}$.

By \cite[Proposition~1.4.1(b)]{HartlPSp} the functor $\ulCF=(\CF,\tau_\CF)\mapsto(\CF\otimes_\dotCO\CR,\tau_\CF\otimes\id_{\CR^\sigma})$ is an equivalence of categories between $\sigma$-bundles over $\dotCO$ and \emph{$\sigma$-bundles over $\CR$}. We now consider the $\sigma$-bundle $\ulCH:=(\CH,\tau_\CH\colon \sigma^\ast\CH\isoto\CH\otimes_\CR\CR^\sigma)$ over $\CR$ with $\CH:=\Hom_{A_{\BC}}(M'',M')\otimes_{A_{\BC}}\CR$ and $\tau_\CH\colon \sigma^\ast(h)\mapsto\tau_{M'}\circ\sigma^\ast(h)\circ\tau_{M''}^{-1}$. Then we just proved that the isomorphism classes of extensions of $\ulM''\otimes_{A_{\BC}}\CR$ by $\ulM'\otimes_{A_{\BC}}\CR$ are in bijection with
\[
\Koh^1(\ulCH)\;:=\;\coker\bigl(1-\tau_\CH\circ\sigma^\ast:\es\CH\to\CH\otimes_\CR\CR^\sigma,\,h\mapsto h-\tau_\CH(\sigma^\ast h)\bigr)\,;
\]
compare~\cite[Proposition~1.3.4]{HartlPSp} or \cite[Proposition~2.4]{HartlPink1}.

\medskip\noindent
2. To change the analytic extension $\ul\CM$ into an algebraic extension we now proceed as follows. We choose locally free sheaves $\olM'$ and $\olM''$ on $C_\BC$ which extend $M'$ and $M''$. Then $\tau_{M'}$ and $\tau_{M''}$ have poles of finite order on $\olM'$, respectively $\olM''$. Since $\BC\langle\frac{z}{\zeta}\rangle$ is a principal ideal domain we can choose a basis of $\olH:=\Hom_{\BC\langle\frac{z}{\zeta}\rangle}\bigl(\olM''\otimes\BC\langle\frac{z}{\zeta}\rangle,\olM'\otimes\BC\langle\frac{z}{\zeta}\rangle\bigr)$. With respect to this basis the element $f\in\Hom_{A_{\BC}}(M'',M')\otimes_{A_{\BC}}\CR^\sigma=\olH\otimes_{\BC\langle\frac{z}{\zeta}\rangle}\CR^\sigma$ associated with $\ul\CM$ in step 1 can be viewed as an element $f=\sum_{\nu\in\BZ}f_\nu z^\nu\in(\CR^\sigma)^{\oplus n}\subset \BC\langle\frac{z}{\zeta^q},\frac{\zeta^q}{z}\rangle^{\oplus n}$. Also $\tau_\CH$ is given by a matrix $T=(t_{ij})\in\GL_n\bigl(\BC\con[q]\bigr)$. Let $c>1$ be a constant with $\|T\|_q:=\max\{\|t_{ij}\|_q:1\le i,j\le n\}\le c$ where for $x=\sum_{\nu \in\BZ}x_\nu z^\nu\in\BC\langle\frac{z}{\zeta^q},\frac{\zeta^q}{z}\rangle$
\[
\|x\|_q\;:=\;\sup\{\,|x_\nu|\,|\zeta|^{\nu q}:\;\nu\in\BZ\,\}
\]
denotes the supremum norm on the annulus $\{|z|=|\zeta|^q\}$. By the convergence condition on $f$ there is an integer $m\le0$ with $\bigl\|\sum_{\nu \le m}f_\nu z^\nu \bigr\|_q\le C:=c^{\frac{2}{3-2q}}<1$. Consider the linear function $\alpha \colon \BC\langle\frac{z}{\zeta^q},\frac{\zeta^q}{z}\rangle^{\oplus n}\to \BC\langle\frac{\zeta^q}{z}\rangle^{\oplus n},\,x=\sum_{\nu\in\BZ}x_\nu z^\nu\mapsto\sum_{\nu\le m}x_\nu z^\nu$ which satisfies $\|\alpha (x)\|_q\le\|x\|_q$. Also note that any element $x=\sum_{\nu\le m}x_\nu z^\nu\in\BC\langle\frac{\zeta^q}{z}\rangle^{\oplus n}$ satisfies
\[
\|\sigma^\ast(x)\|_q \; =\; \sup\{\,|x_\nu^q|\,|\zeta|^{\nu q}:\;\nu\le m\,\}\;\leq\; \sup\{\,\bigl(|x_\nu|\,|\zeta|^{\nu q}\bigr)^q:\;\nu\le m\,\} \;=\; \|x\|_q^q\,.
\]
Recursively we define $g_0:=\alpha (f)$ and $g_k:=\alpha\bigl(T\sigma^\ast(g_{k-1})\bigr)\in \BC\langle\frac{\zeta^q}{z}\rangle^{\oplus n}$ for all $k\in\BN$. Then we show by induction that $\|g_k\|_q\le C^{1+\frac{k}{2}}$. Indeed $\|g_0\|_q\le C$ and we estimate
\[
\|g_k\|_q\;\le\;\|T\sigma^\ast(g_{k-1})\|_q\;\le\;\|T\|_q\,\|g_{k-1}\|_q^q\;\le\;c\cdot C^{(1+\frac{k-1}{2})q}\;\le\;C^{\frac{3}{2}-q+q+\frac{k-1}{2}}\;=\;C^{1+\frac{k}{2}}
\]
as claimed. This implies that $g:=\sum_{k=0}^\infty g_k$ converges in $\BC\langle\frac{\zeta^q}{z}\rangle^{\oplus n}\subset\BC\langle\frac{\zeta}{z}\rangle^{\oplus n}$. By construction $\alpha (g)=g$. We compute $\alpha (f+T\sigma^\ast(g)-g)=\alpha (f)+\sum_{k=0}^\infty \alpha (T\sigma^\ast(g_k))-\sum_{k=0}^\infty g_k=0$. Hence, $h:=f+T\sigma^\ast(g)-g\in \BC\con[q]^{\oplus n}$. From the formula $g=f+T\sigma^\ast(g)-h$ one inductively sees that $g\in \BC\langle\frac{\zeta^{q^j}}{z}\rangle^{\oplus n}$ for all $j\in\BN_0$, whence $g\in \CR^{\oplus n}$. 
Now consider the element
$
\tilde f:=\sigma^{-1\ast}(T^{-1}h)\in \BC\con^{\oplus n},
$
which satisfies $f-\tilde f=T\sigma^\ast(\tilde f-g)-(\tilde f-g)$.
This shows that the class of $f$ in $\Koh^1(\ulCH)$ is the same as the class of $\tilde f$ and we may identify $\CM\otimes_{\CO(\dotFC_\BC)}\CR= (M'\oplus M'')\otimes_{A_{\BC}}\CR$ such that $\tau_\CM=\left(\begin{array}{cc} \tau_{M'} & \tilde f\circ\tau_{M''}\\ 0 & \tau_{M''} \end{array}\right)$. Thus $\CM$ extends to a locally free rigid analytic sheaf $\olCM$ on $\FC_\BC$ with $\olCM\otimes\BC\langle\frac{z}{\zeta}\rangle=(\olM'\oplus \olM'')\otimes_{\CO_{C_\BC}}\BC\langle\frac{z}{\zeta}\rangle$. Since $\tau_{M'},\tau_{M''}$ and $\tilde f$ all have poles of finite order at $\infty$, also $\tau_\CM$ extends to $\tau_\CM\colon \sigma^\ast\olCM\to\olCM\bigl(l\cdot\infty-d\cdot\Var(J)\bigr)$ for some integers $l$ and $d$ with $\tau_\CM(\sigma^*\CM)\subset J^d\CM$. Again the rigid analytic GAGA principle \cite[Theorem~2.8]{Lubo} produces a locally free algebraic sheaf $\olM$ together with a homomorphism $\tau_{M}\colon \sigma^\ast\olM\to\olM\bigl(l\cdot\infty-d\cdot\Var(J)\bigr)$ such that $\olCM=\olM\otimes_{\CO_{C_\BC}}\CO_{\FC_\BC}$ and $\tau_{\CM}=\tau_{M}\otimes\id$. By construction $\ulM:=\bigl(\Gamma(\dotC_\BC,\olM),\tau_M\bigr)$ is the extension of $\ulM''$ by $\ulM'$ in the category of $A$-motives we were searching. This proves the proposition.
\end{proof}


\vfill

\parbox[t]{8.2cm}{ 
Urs Hartl  \\ 
Universit\"at M\"unster\\
Mathematisches Institut \\
Einsteinstr.~62\\
D -- 48149 M\"unster
\\ Germany
\\[1mm]
{\small \href{http://www.math.uni-muenster.de/u/urs.hartl/}{www.math.uni-muenster.de/u/urs.hartl/}}
} 
\parbox[t]{6.2cm}{ 
Ann-Kristin Juschka  \\ 
Universit\"at Heidelberg\\
Im Neuenheimer Feld 368\\
D -- 69120 Heidelberg\\
Germany
}


\begin{thebibliography}{BGR84}

\bibitem[And86]{Anderson86} G.~Anderson: \emph{$t$-Motives}, Duke Math.~J.\ {\bfseries 53} (1986), 457--502. 

\bibitem[ABP04]{ABP} G.~Anderson, D.~Brownawell, M.~Papanikolas: \emph{Determination of the algebraic, relations among special $\Gamma$-values in positive characteristic}, Annals of Math.\ (2) {\bfseries 160} (2004), 237--313; also available as \href{http://arxiv.org/abs/math.NT/0207168}{arXiv:math.NT/0207168}.

\bibitem[ABP02]{ABP_Rohrlich} G.~Anderson, D.~Brownawell, M.~Papanikolas: \emph{Rohrlich Conjecture Project}, unpublished notes from 2002.

\bibitem[BH07]{BoeckleHartl} G.~B{\"o}ckle, U.~Hartl: \emph{Uniformizable Families of $t$-motives}, Trans.\ Amer.\ Math.\ Soc.\ {\bfseries 359} (2007),  no.~8, 3933--3972; also available as \href{http://arxiv.org/abs/math/0411262}{arXiv:math.NT/0411262}.

\bibitem[Bor69]{Borel}A.~Borel: \emph{Linear algebraic groups}, W.~A.~Benjamin, Inc., New York-Amsterdam 1969 .

\bibitem[BH11]{BH1} M.~Bornhofen, U.~Hartl: \emph{Pure Anderson motives and abelian $\tau$-sheaves}, Math.\ Z.\ {\bfseries 268} (2011), 67--100; also available as \href{http://arxiv.org/abs/0709.2809}{arXiv:0709.2809}.

\bibitem[Bos14]{Bosch} S.~Bosch: \emph{Lectures on Formal and Rigid Geometry}, LNM {\bfseries 2105}, Springer-Verlag 2014; also available at \href{https://www.uni-muenster.de/IVV5WS/WebHop/user/weckerm/heft378.pdf}{https:/\!/www.uni-muenster.de/IVV5WS/WebHop/user/weckerm/}.

\bibitem[BGR84]{BGR} S.~Bosch, U.~G{\"u}ntzer, R.~Remmert: \emph{Non-Archimedean Analysis}, Grundlehren {\bfseries 261}, Springer-Verlag, Berlin etc.\ 1984. 

\bibitem[BL85]{BL85} S.~Bosch, W.~L{\"u}tkebohmert: \emph{Stable reduction and uniformization of abelian varieties~I}, Math.\ Ann.\ {\bfseries 270} (1985), no.~3, 349--379.
 
\bibitem[BLR90]{BLR} S.~Bosch, W.~L\"utkebohmert, M.~Raynaud: \emph{N\'eron models}, Ergebnisse der Mathematik und ihrer Grenzgebiete (3) {\bfseries 21}, Springer-Verlag, Berlin, 1990. 

\bibitem[Bou70]{BourbakiAlgebra} N.~Bourbaki: \emph{\'Elements de Math\'ematique, Alg\`ebre, Chapitres 1 \`a 3}, Hermann, Paris 1970. 

\bibitem[BP05]{BreuerPink} F.~Breuer, R.~Pink: \emph{Monodromy groups associated to non-isotrivial Drinfeld modules in generic characteristic}, in ``Number fields and function fields—two parallel worlds'', pp.~61–69, Progr.\ Math., {\bfseries 239}, Birkh{\"a}user Boston, Boston, MA, 2005; also available at \href{ https://people.math.ethz.ch/~pinkri/ftp/BreuerPink.pdf}{http:/\!/www.math.ethz.ch/$\sim$pinkri/}.

\bibitem[BP02]{BrownawellPapanikolas02} D.~Brownawell, M.~Papanikolas: \emph{Linear independence of gamma values in positive characteristic}, J.~Reine Angew.\ Math.\ {\bfseries 549} (2002), 91--148; also available as \href{http://arxiv.org/abs/math/0106054}{arXiv:math.NT/0106054}.

\bibitem[BP20]{BrownawellPapanikolas16} D.~Brownawell, M.~Papanikolas: \emph{A rapid introduction to Drinfeld modules, $t$-modules, and $t$-motives}, in ``$t$-motives: Hodge structures, transcendence and other motivic aspects'', eds.\ G.~B\"ockle, D.~Goss, U.~Hartl, M.~Papanikolas, EMS 2020, pp.~3--30
.

\bibitem[Cha12]{ChangThirdKind} C.-Y.~Chang: \emph{On periods of the third kind for rank $2$ Drinfeld module}, Math.~Z.\ {\bfseries 273} (2013), no.~3-4, 921--933; also available as \href{http://arxiv.org/abs/0909.0101}{arXiv:math/0909.0101}. 

\bibitem[Cha20]{Chang12} C.-Y.~Chang: \emph{Frobenius difference equations and difference Galois groups}, in ``$t$-motives: Hodge structures, transcendence and other motivic aspects'', eds.\ G.~B\"ockle, D.~Goss, U.~Hartl, M.~Papanikolas, EMS 2020, pp.~261--296
.

\bibitem[CP11]{ChangPapa11} C.-Y.~Chang and M.~Papanikolas: \emph{Algebraic relations among periods and logarithms of rank $2$ Drinfeld modules}, Amer.\ J.~Math.\ {\bfseries 133} (2011), 359--391; also available as \href{http://arxiv.org/abs/0807.3157}{arXiv:math.NT/0807.3157}.


\bibitem[CP12]{ChangPapa12} C.-Y.~Chang, M.~Papanikolas: \emph{Algebraic independence of periods and logarithms of Drinfeld modules}, With an appendix by Brian Conrad, J.~Amer.\ Math.\ Soc.\ {\bfseries 25} (2012), no.~1, 123--150; also available as \href{http://arxiv.org/abs/1005.5120}{arXiv:math.NT/1005.5120}.


\bibitem[CPTY10]{ChangPapaThakurYu} C.-Y.~Chang, M.~Papanikolas, D.~ Thakur, J.~Yu: \emph{Algebraic independence of arithmetic gamma values and Carlitz zeta values}, Adv.\ Math.\ {\bfseries 223} (2010), 1137--1154; also available as \href{http://arxiv.org/abs/0909.0096}{arXiv:math.NT/0909.0096}.


\bibitem[CPY10]{ChangPapaYu10} C.-Y.~Chang, M.~Papanikolas, J.~Yu: \emph{Geometric gamma values and zeta values in positive characteristic}, Int.\ Math.\ Res.\ Notices {\bfseries 2010} (2010), 1432-1455; also available as \href{http://arxiv.org/abs/0905.2876}{arXiv:math.NT/0905.2876}.


\bibitem[CPY11]{ChangPapaYu11} C.-Y.~Chang, M.~Papanikolas, J.~Yu: \emph{Frobenius difference equations and algebraic independence of zeta values in positive equal characteristic}, Algebra \& Number Theory {\bfseries 5} (2011), 111--129; also available as \href{http://arxiv.org/abs/0804.0038}{arXiv:math.NT/0804.0038}.


\bibitem[CY07]{ChangYu07} C.-Y.~Chang and J.~Yu: \emph{Determination of algebraic relations among special zeta values in positive characteristic}, Advances in Mathematics {\bfseries 216} (2007), 321--345.

\bibitem[Del71]{DeligneHodge2} P.~Deligne: \emph{Th\'eorie de Hodge II}, Inst.\ Hautes \'Etudes Sci.\ Publ.\ Math.\ {\bfseries 40} (1971), 5--57; also available at \href{http://numdam.org/numdam-bin/fitem?id=PMIHES_1971__40__5_0}{http:/\!/numdam.org/numdam-bin/fitem?id=PMIHES\_1971\_\_40\_\_5\_0}.

\bibitem[Del74]{DeligneHodge3} P.~Deligne: \emph{Th\'eorie de Hodge III}, Inst.\ Hautes \'Etudes Sci.\ Publ.\ Math.\ {\bfseries 44} (1974), 5--77; also available at \href{http://numdam.org/numdam-bin/fitem?id=PMIHES_1974__44__5_0}{http:/\!/numdam.org/numdam-bin/fitem?id=PMIHES\_1974\_\_44\_\_5\_0}.

\bibitem[Del90]{DeligneCatTann} P.~Deligne: \emph{Cat\'egories tannakiennes}, in ``The Grothendieck Festschrift'', Vol. II, pp.~111-–195, Progress in Math. {\bfseries 87}, Birkh\"auser Boston, Boston, MA, 1990. 

\bibitem[Del94]{Deligne94} P.~Deligne: \emph{Structures de Hodge mixtes r\'eelles}, in ``Motives (Seattle, WA, 1991)'', pp.~509--514, Proc.\ Sympos.\ Pure Math.\ {\bfseries 55}, Part 1, Amer.\ Math.\ Soc., Providence, RI, 1994.

\bibitem[Del06]{Deligne06} P.~Deligne: \emph{The Hodge conjecture}, in The millennium prize problems, pp.~45--53, Clay Math. Inst., Cambridge, MA, 2006; available at \href{http://www.claymath.org/sites/default/files/hodge.pdf}{http:/\!/www.claymath.org/sites/default/files/hodge.pdf}.

\bibitem[DM82]{DM82} P.~Deligne, J.~Milne: \emph{Tannakian Categories}, in ``Hodge Cycles, Motives, and Shimura Varieties'', pp.~101--228, LNM {\bfseries 900}, Springer-Verlag, New York 1982; also available at \href{http://www.jmilne.org/math/Books/DMOS.pdf}{http:/\!/www.jmilne.org/math}.

\bibitem[Dem69]{DemazureMotives} M.~Demazure: \emph{Motifs des vari\'et\'es alg\'ebriques}, S\'eminaire Bourbaki, Expos\'e No.~365, LNM {\bfseries 180}, Springer-Verlag, Heidelberg 1969; also available at \href{http://numdam.org/numdam-bin/fitem?id=SB_1969-1970__12__19_0}{http:/\!/numdam.org/numdam-bin/fitem?id=SB\_1969-1970\_\_12\_\_19\_0}.

\bibitem[Dri76]{Drinfeld} V.G.~Drinfeld: \emph{Elliptic Modules}, Math.\ USSR-Sb.\ {\bfseries 23} (1976), 561--592. 


\bibitem[EGA]{EGA} A.~Grothendieck: \emph{{\'E}lements de G{\'e}o\-m{\'e}trie Alg{\'e}\-brique}, Publ.\ Math.\ IHES {\bfseries 4}, {\bfseries 8}, {\bfseries 11}, {\bfseries 17}, {\bfseries 20}, {\bfseries 24}, {\bfseries 28}, {\bfseries 32}, Bures-Sur-Yvette, 1960--1967; see also Grundlehren {\bfseries 166}, Springer-Verlag, Berlin etc.\ 1971; also available at \href{http://www.numdam.org/numdam-bin/recherche?au=Grothendieck}{http:/\!/www.numdam.org/numdam-bin/recherche?au=Grothendieck}.

\bibitem[Eis95]{Eisenbud} D.~Eisenbud: \emph{Commutative Algebra with a View Toward Algebraic Geometry}, GTM {\bfseries 150}, Springer-Verlag, Berlin etc.\ 1995. 

\bibitem[Fin47]{Fine47} N.~Fine: \emph{Binomial coefficients modulo a prime}, Amer.\ Math.\ Monthly {\bfseries 54} (1947), 589--592; available at \href{http://www.jstor.org/stable/2304500}{http:/\!/www.jstor.org/stable/2304500}.

\bibitem[Fon82]{Fontaine82} J.-M.~Fontaine: \emph{Sur certains types de repr\'esentations $p$-adiques du groupe de Galois d'un corps local; construction d'un anneau de Barsotti-Tate}, Ann.\ of Math.\ (2) {\bfseries 115}  (1982), no.~3, 529--577; available at \href{http://www.jstor.org/stable/pdf/2007012.pdf}{http:/\!/www.jstor.org/stable/2007012}.

\bibitem[Gar02]{Gardeyn2} F.~Gardeyn: \emph{A Galois criterion for good reduction of $\tau$-sheaves}, J.~Number Theory\ {\bfseries 97} (2002), 447--471. 

\bibitem[Gek89]{Gekeler89} E.-U.~Gekeler: \emph{On the de Rham isomorphism for Drinfel'd modules}, J.~Reine Angew.\ Math.\ {\bfseries 401} (1989), 188--208.

\bibitem[Gek90]{Gekeler90} E.-U.~Gekeler: \emph{De Rham cohomology for Drinfel'd modules}, S\'eminaire de Th\'eorie des Nombres, Paris 1988–1989, pp.~57--85, Progr.\ Math.~{\bfseries 91}, Birkhäuser Boston, Boston, MA, 1990.

\bibitem[Gos94]{Goss94} D.~Goss: \emph{Drinfel'd modules: cohomology and special functions}, in ``Motives (Seattle, WA, 1991)'', Proc.\ Sympos.\ Pure Math.\ {\bfseries 55}, Part 2, pp.~309--362, Amer.\ Math.\ Soc., Providence, RI, 1994. 

\bibitem[Gos96]{Goss} D.~Goss: \emph{Basic Structures of Function Field Arithmetic}, Ergebnisse der Mathematik und ihrer Grenzgebiete (3) {\bfseries 35}, Springer-Verlag, Berlin-Heidelberg-New York 1996.

\bibitem[Gro69a]{GrothendieckStandard} A.~Grothendieck: \emph{Standard conjectures on algebraic cycles}, in Algebraic Geometry (Internat.\ Colloq., Tata Inst.\ Fund.\ Res., Bombay, 1968) pp.~193--199 Oxford Univ. Press, London 1969.

\bibitem[Gro69b]{GrothendieckHodge} A.~Grothendieck: \emph{Hodge's general conjecture is false for trivial reasons}, Topology {\bfseries 8} 1969, 299--303. 

\bibitem[Gru68]{Gruson68} L.~Gruson: \emph{Fibr\'es vectoriels sur un polydisque ultram\'etrique}, Ann.\ Sci.\ \'Ecole Norm.\ Sup.\ (4) {\bfseries 1} (1968), 45--89; also available at \href{http://www.numdam.org/numdam-bin/fitem?id=ASENS_1968_4_1_1_45_0}{http:/\!/www.numdam.org/numdam-bin/fitem?id=ASENS\_1968\_4\_1\_1\_45\_0}.

\bibitem[Har11]{HartlPSp} U.~Hartl: \emph{Period Spaces for Hodge Structures in Equal Characteristic}, Annals of Math. {\bfseries 173}, n.~3 (2011), 1241--1358; also available as \href{http://arxiv.org/abs/math/0511686}{arXiv:math.NT/0511686}.

\bibitem[Har17]{HartlIsog} U.~Hartl: \emph{Isogenies of abelian Anderson $A$-modules and $A$-motives}, Preprint 2017 on \href{http://arxiv.org/abs/1706.06807}{arXiv:math/1706.06807}.

\bibitem[HK16]{HartlKim} U.~Hartl, W.~Kim: \emph{Local Shtukas, Hodge-Pink Structures and Galois Representations}, Proceedings of the conference on ``$t$-motives: Hodge structures, transcendence and other motivic aspects'', BIRS, Banff, Canada 2009, eds.\ G.~Böckle, D.~Goss, U.~Hartl, M.~Papanikolas, EMS 2016; also available as \href{http://arxiv.org/abs/1512.05893}{arXiv:1512.05893}.

\bibitem[HP04]{HartlPink1} U.~Hartl, R.~Pink: \emph{Vector bundles with a Frobenius structure on the punctured unit disc}, Comp.\ Math.\ {\bfseries 140}, n.~3 (2004), 689--716; also available at \href{http://www.math.uni-muenster.de/u/urs.hartl/Publikat/SigmaBd.pdf}{http:/\!/www.math.uni-muenster.de/u/urs.hartl/Publikat/}.

\bibitem[HP18]{HartlPink2} U.~Hartl, R.~Pink: \emph{Analytic uniformization of $A$-motives}, in preparation.

\bibitem[Har66]{HartshorneResidues} R.~Hartshorne: \emph{Residues and duality}, Lecture Notes in Mathematics {\bfseries 20}, Springer-Verlag, Berlin-New York 1966.


\bibitem[Hod52]{Hodge52} W.V.D.~Hodge: \emph{The topological invariants of algebraic varieties}, Proceedings of the International Congress of Mathematicians, Cambridge, Mass., 1950, vol.~1, pp.~182--192, Amer.\ Math.\ Soc., Providence, RI, 1952.

\bibitem[Jus10]{JuschkaDipl} A.-K.~Juschka: \emph{The Hodge Conjecture For Function Fields}, Diploma thesis, University of Muenster, 2010; available at \href{http://www.math.uni-muenster.de/u/urs.hartl/Publikat/Diplomarbeit_Juschka_100415.pdf}{http:/\!/www.math.uni-muenster.de/u/urs.hartl}.

\bibitem[Kie67]{KiehlAB} R.~Kiehl: \emph{Theorem A und B in der nichtarchimedischen Funktionentheorie}, Invent.\ Math.\ {\bfseries 2} (1967), 256--273.

\bibitem[Kle72]{KleimanMotives} S.~Kleiman: \emph{Motives}, in Algebraic geometry, Oslo 1970 (Proc.\ Fifth Nordic Summer-School in Math., Oslo, 1970), pp.~53--82, Wolters-Noordhoff, Groningen, 1972. 

\bibitem[Lan57]{Lang57} S.~Lang: \emph{Algebraic groups over finite fields}, Amer.\ J.~Math.\ {\bfseries 78} (1956), 555--563; available at \href{http://www.jstor.org/stable/2372673}{http:/\!/www.jstor.org/stable/2372673}.

\bibitem[Lau96]{Laumon} G.~Laumon: \emph{Cohomology of Drinfeld Modular Varieties I\/}, Cambridge Studies in Advanced Mathematics {\bfseries 41}, Cambridge University Press, Cambridge, 1996. 

\bibitem[Laz62]{Lazard} M.~Lazard: \emph{Les z\'eros des fonctions analytiques d'une variable sur un corps valu\'e complet}, Inst.\ Hautes \'Etudes Sci.\ Publ.\ Math.\ {\bfseries 14} (1962), 47--75; also available at \href{http://www.numdam.org/item?id=PMIHES_1962__14__47_0}{http:/\!/www.numdam.org/item?id=PMIHES\_1962\_\_14\_\_47\_0}.

\bibitem[Luc78]{Lucas1878} \'E.~Lucas: \emph{Th\'eorie des Fonctions Num\'eriques Simplement P\'eriodiques}, Amer.\ J.\ Math.\ {\bfseries 1} (1878), 184--196, 197--240, 289--321; available at \href{http://www.jstor.org/action/doAdvancedSearch?q0=edouard+lucas&f0=au&c1=AND&q1=+Th%C3%A9orie+des+Fonctions+Num%C3%A9riques+Simplement+P%C3%A9riodiques&f1=ti&acc=on&wc=on&fc=off&Search=Search&sd=&ed=&la=&pt=&isbn=}{http:/\!/www.jstor.org/}.

\bibitem[L{\"u}t90]{Lubo} W.~L{\"u}tkebohmert: \emph{Formal-algebraic and rigid-analytic geometry}, Math.\ Ann.\ {\bfseries 286} (1990), 341--371.

\bibitem[Man63]{Manin} Y.I.~Manin: \emph{The theory of commutative formal groups over fields of finite characteristic}, Russian Math.\ Surveys {\bfseries 18} No.~6 (1963), 1--83.

\bibitem[Man68]{ManinMotives} Y.I.~Manin: \emph{Correspondences, motifs and monoidal transformations}, Mat.\ Sb.\ (N.S.) {\bfseries 77 (119)}, (1968), 475--507. 

\bibitem[Mat86]{MatsumuraRingTh} H.~Matsumura: \emph{Commutative ring theory}, Cambridge Studies in Advanced Mathematics {\bfseries 8}, Cambridge University Press, Cambridge 1986. 

\bibitem[Mil92]{Milne92} J.~Milne: \emph{The points on a Shimura variety modulo a prime of good reduction}, in: The zeta functions of Picard modular surfaces, pp.~151--253, Univ. Montréal, Montreal, QC, 1992; available at \href{http://www.jmilne.org/math/articles/abstracts.html#1992}{http:/\!/www.jmilne.org/math/articles/abstracts.html\#1992}. 

\bibitem[Mis12]{Mishiba12} Y.~Mishiba: \emph{On $v$-adic periods of $t$-motives}, J.~Number Theory {\bfseries 132} (2012), no.~10, 2132--2165; also available as \href{http://arxiv.org/abs/1105.6243}{arXiv:1105.6243}.

\bibitem[Mis14]{MishibaMZV} Y.~Mishiba: \emph{On algebraic independence of certain multizeta values in characteristic $p$}, preprint 2014 on \href{http://arxiv.org/abs/1401.3628}{arXiv:1401.3628}

\bibitem[Pap08]{Papanikolas} M.\ Papanikolas: \emph{Tannakian duality for Anderson-Drinfeld motives and algebraic independence of Carlitz logarithms}, Invent.\ Math.\ {\bfseries 171} (2008), 123--174; also available as \href{http://arxiv.org/abs/math.NT/0506078}{arXiv:math.NT/0506078}.

\bibitem[Pel08]{pellarin-08} F.~Pellarin: \emph{Aspects de l'ind\'ependance alg\'ebrique en caract\'eristique non nulle (d'apr\`es {A}nderson, {B}rownawell, {D}enis, {P}apanikolas, {T}hakur, {Y}u, et al.)}, S{\'e}minaire Bourbaki, Vol.~2006/2007, Exp.\ No.~973, Ast\'erisque {\bfseries 317} (2008), 205--242. 

\bibitem[Pin97a]{Pink97a} R. Pink: \emph{The Mumford-Tate conjecture for Drinfeld modules}, Publ.\ Res.\ Inst.\ Math.\ Sci.\ {\bfseries 33} (3) (1997) 393--425; available at \href{https://people.math.ethz.ch/~pinkri/ftp/dmt-v8.pdf}{http:/\!/www.math.ethz.ch/$\sim$pinkri/}.

\bibitem[Pin97b]{PinkHodge} R.~Pink: \emph{Hodge Structures over Function Fields}, Preprint 1997; available at \href{http://www.math.ethz.ch/~pinkri/ftp/HS.pdf}{http:/\!/www.math.ethz.ch/$\sim$pinkri/}.

\bibitem[PR09a]{PR09a} R.~Pink, E.~R\"utsche, \emph{Image of the group ring of the Galois representation associated to Drinfeld modules}, J.~Number Theory {\bfseries 129} (2009), no.~4, 866--881; also available at \href{http://www.math.ethz.ch/~pinkri/ftp/Pink-Ruetsche-1.pdf}{http:/\!/www.math.ethz.ch/$\sim$pinkri/}.

\bibitem[PR09b]{PR09b} R.~Pink, E.~R\"utsche, \emph{Adelic openness for Drinfeld modules in generic characteristic}, J.~Number Theory {\bfseries 129} (2009), no.~4, 882--907; also available at \href{http://www.math.ethz.ch/~pinkri/ftp/Pink-Ruetsche-2.pdf}{http:/\!/www.math.ethz.ch/$\sim$pinkri/}.

\bibitem[PT06]{PT06b} R.~Pink, M.~Traulsen, \emph{The Galois representations associated to a Drinfeld module in special characteristic, III. Image of the group ring}, J.~Number Theory {\bfseries 116} (2006), no.~2, 373--395; also available at \href{http://www.math.ethz.ch/~pinkri/ftp/PT1-20040318.pdf}{http:/\!/www.math.ethz.ch/$\sim$pinkri/}. 

\bibitem[Qui73]{Quillen} D.\ Quillen: Higher algebraic K-theory I, in \emph{Algebraic K-Theory, I: Higher $K$-theories (Proc.\ Conf., Battelle Memorial Inst., Seattle, Wash., 1972)},  pp. 85--147, LNM {\bfseries 341}, Springer-Verlag, Berlin 1973.

\bibitem[Sch84]{Schikhof} W.H.~Schikhof: \emph{Ultrametric calculus. An introduction to $p$-adic analysis}, Cambridge Studies in Advanced Mathematics {\bfseries 4}, Cambridge University Press, Cambridge, 1984.

\bibitem[Ser88]{Serre} J.-P.\ Serre: \emph{Algebraic Groups and Class fields}, GTM {\bfseries 117}, Springer-Verlag, Berlin etc.\ 1988.

\bibitem[Sil86]{Silverman}J.~Silverman: \emph{The Arithmetic of Elliptic Curves}, GTM {\bfseries 106}, Springer-Verlag, Berlin etc.\ 1986.

\bibitem[Sin97]{Sinha97} S.~Sinha: \emph{Periods of $t$-motives and transcendence}, Duke Math.\ J.~{\bfseries 88} (1997), no.~3, 465--535. 

\bibitem[Tae09a]{Taelman} L.~Taelman: \emph{Artin $t$-motives}, J.~of Number Theory {\bfseries 129} (2009), 142--157; also available as \href{http://arxiv.org/abs/0809.4351}{arXiv:math.NT/0809.4351}.

\bibitem[Tae09b]{Taelman09} L.~Taelman: \emph{Special $L$-values of $t$-motives: a conjecture}, Int.\ Math.\ Res.\ Not.\ {\bfseries 16} (2009), 2957--2977; also available as \href{http://arxiv.org/abs/0811.4522}{arXiv:math/0811.4522}. 

\bibitem[Tae20]{Taelman16} L.~Taelman: \emph{$1$-$t$-motives}, in ``$t$-motives: Hodge structures, transcendence and other motivic aspects'', eds.\ G.~B\"ockle, D.~Goss, U.~Hartl, M.~Papanikolas, EMS 2020, pp.~417--440; also available as \href{http://arxiv.org/abs/0908.1503}{arXiv:math/0908.1503}.


\bibitem[Tag95b]{Taguchi95b}Y.~Taguchi: \emph{The Tate conjecture for $t$-motives}, Proc.\ Amer.\ Math.\ Soc.\ {\bfseries 123} (1995), no.~11, 3285--3287; also available at \href{http://www2.math.kyushu-u.ac.jp/~taguchi/bib/tmottate.pdf}{http:/\!/www2.math.kyushu-u.ac.jp/$\sim$taguchi/bib/}.

\bibitem[TW96]{TW} Y.~Taguchi, D.~Wan: \emph{$L$-functions of $\phi$-sheaves and Drinfeld modules}, J.~Amer.\ Math.\ Soc.\ {\bfseries 9} (1996), no.~3, 755--781; also available at \href{http://www2.math.kyushu-u.ac.jp/~taguchi/bib/lfunction.pdf}{http:/\!/www2.math.kyushu-u.ac.jp/$\sim$taguchi/bib/}.

\bibitem[Tam94]{Tamagawa} A. Tamagawa: \emph{Generalization of Anderson's $t$-motives and Tate conjecture}, in ``Moduli Spaces, Galois Representations and $L$-Functions'', S\=urikaisekikenky\=usho K\=oky\=uroku, n.~884, Kyoto (1994), 154--159.

\bibitem[Tha04]{Thakur} D.~Thakur: \emph{Function field arithmetic}, World Scientific Publishing Co., Inc., River Edge, NJ, 2004.

\bibitem[Vil06]{VillaSalvador} G.D.~Villa Salvador: \emph{Topics in the theory of algebraic function fields}, Mathematics: Theory \& Applications, Birkh\"auser Verlag, Boston, MA, 2006.

\bibitem[Wad41]{Wade41} L.I.~Wade: \emph{Certain quantities transcendental over $GF(pn,x)$}, Duke Math.\ J.~{\bfseries 8} (1941), 701--720. 


\bibitem[Yu90]{Yu90} J.~Yu: \emph{On periods and quasi-periods of Drinfel'd modules}, Compositio Math.\ {\bfseries 74} (1990), no.~3, 235--245. 

\end{thebibliography}
\end{document}